\documentclass[letterpaper,11pt,titlepage,oneside]{book}

\usepackage{amsmath,amssymb,amstext,amsfonts,amsthm} 
\usepackage{algpseudocode}
\usepackage[algoruled,boxed,lined]{algorithm2e}
\usepackage[usenames,dvipsnames,svgnames,table]{xcolor}
\usepackage{epic}
\usepackage{mathrsfs}
\usepackage[utf8]{inputenc}
\usepackage[T1]{fontenc}
\usepackage{enumerate}
\usepackage{graphicx}
\usepackage{footnote,pifont,dsfont}
\usepackage{tikz}
\usetikzlibrary{patterns,decorations.pathreplacing}
\usepackage{epigraph}

\usepackage{imakeidx}
\usepackage{bm}
\usepackage{thmtools}
\usepackage{caption}

\usepackage[
    hyperref,
    backend=biber,
    backref,
    style=numeric,
    citestyle=numeric,
    maxbibnames=99
]{biblatex}
\renewbibmacro*{doi+eprint+url}{%
  \iftoggle{bbx:url}
    {\iffieldundef{doi}{\printfield{url}}{}}
    {}%
  \newunit\newblock
  \iftoggle{bbx:eprint} 
    {\usebibmacro{eprint}}
    {}%
  \newunit\newblock
  \iftoggle{bbx:doi}
    {\printfield{doi}}
    {}}
 
\AtEveryBibitem{
 \clearlist{address}
 \clearfield{date}
 \clearfield{eprint}
 \clearfield{isbn}
 \clearfield{issn}
 \clearlist{location}
 \clearfield{month}
 \clearfield{series}
 
 \ifentrytype{book}{}{
  \clearlist{publisher}
  \clearname{editor}
 }
}

\usepackage[pdftex,pagebackref=false]{hyperref} 
\hypersetup{
    plainpages=false,       % needed if Roman numbers in frontpages
    unicode=false,          % non-Latin characters in Acrobat’s bookmarks
    pdftoolbar=true,        % show Acrobat’s toolbar?
    pdfmenubar=true,        % show Acrobat’s menu?
    pdffitwindow=false,     % window fit to page when opened
    pdfstartview={FitH},    % fits the width of the page to the window
    pdftitle={Analytic Combinatorics in Several Variables: Effective Asymptotics and Lattice Path Enumeration},    % title: CHANGE THIS TEXT!
    pdfauthor={Stephen Melczer},    % author: CHANGE THIS TEXT! and uncomment this line
    pdfsubject={Mathematics, Computer Science},  % subject: CHANGE THIS TEXT! and uncomment this line
    pdfnewwindow=true,      % links in new window
    colorlinks=true,        % false: boxed links; true: colored links
    linkcolor=blue,         % color of internal links
    citecolor=DarkGreen,        % color of links to bibliography
    filecolor=magenta,      % color of file links
    urlcolor=cyan           % color of external links
}

\usepackage[automake,toc]{glossaries-extra} % Exception to the rule of hyperref being the last add-on package

\let\origdoublepage\cleardoublepage
\newcommand{\clearemptydoublepage}{%
  \clearpage{\pagestyle{empty}\origdoublepage}}
\let\cleardoublepage\clearemptydoublepage

\renewcommand{\geq}{\geqslant}
\renewcommand{\leq}{\leqslant}
\newtheorem{theorem}{Theorem}
\newtheorem*{theorem*}{Theorem}
\newtheorem{definition}[theorem]{Definition}
\newtheorem{corollary}[theorem]{Corollary}
\newtheorem{proposition}[theorem]{Proposition}
\newtheorem*{proposition*}{Proposition}
\newtheorem{lemma}[theorem]{Lemma}

\newtheorem{conjecture}[theorem]{Conjecture}

\theoremstyle{definition}
\newtheorem{examplex}[theorem]{Example}
\newenvironment{example}
  {\pushQED{\qed}\examplex}
  {\popQED\endexamplex}

\newtheorem{remark}[theorem]{Remark}

\newcommand{\websiteurl}{\url{https://github.com/smelczer/thesis}}

\newcommand{\Relog}{\textsl{Relog}}
\newcommand{\amoeba}{\textsl{amoeba}}
\newcommand{\mS}{\ensuremath{\mathcal{S}}}
\newcommand{\mR}{\ensuremath{\mathcal{R}}}
\newcommand{\mG}{\ensuremath{\mathcal{G}}}
\newcommand{\mM}{\ensuremath{\mathcal{M}}}
\newcommand{\mN}{\ensuremath{\mathcal{N}}}
\newcommand{\mV}{\ensuremath{\mathcal{V}}}
\newcommand{\mH}{\ensuremath{\mathcal{H}}}
\newcommand{\mO}{\ensuremath{\mathcal{O}}}
\newcommand{\bz}{\ensuremath{\mathbf{z}}}
\newcommand{\bd}{\ensuremath{\mathbf{d}}}
\newcommand{\bw}{\ensuremath{\mathbf{w}}}
\newcommand{\bx}{\ensuremath{\mathbf{x}}}
\newcommand{\by}{\ensuremath{\mathbf{y}}}
\newcommand{\ba}{\ensuremath{\mathbf{a}}}
\newcommand{\bb}{\ensuremath{\mathbf{b}}}
\newcommand{\obz}{\ensuremath{\overline{\mathbf{z}}}}

\newcommand{\bi}{\ensuremath{\mathbf{i}}}
\newcommand{\bv}{\ensuremath{\mathbf{v}}}
\newcommand{\bj}{\ensuremath{\mathbf{j}}}
\newcommand{\bn}{\ensuremath{\mathbf{n}}}
\newcommand{\bss}{\ensuremath{\mathbf{s}}}
\newcommand{\br}{\ensuremath{\mathbf{r}}}

\newcommand{\ox}{\ensuremath{\overline{x}}}
\newcommand{\oz}{\ensuremath{\overline{z}}}

\newcommand{\oy}{\ensuremath{\overline{y}}}
\newcommand{\of}{\ensuremath{\overline{f}}}

\newcommand{\sgn}{\operatorname{sgn}}

\newcommand{\bg}{\mbox{\boldmath$\gamma$}}
\newcommand{\bp}{\mbox{\boldmath$\rho$}}
\newcommand{\bETA}{\mbox{\boldmath$\eta$}}
\newcommand{\mD}{\ensuremath{\mathcal{D}}}
\newcommand{\mT}{\ensuremath{\mathcal{T}}}
\newcommand{\mW}{\ensuremath{\mathcal{W}}}
\newcommand{\mWR}{\ensuremath{\mathcal{W}_{\mathbb{R}}}}
\newcommand{\mWRs}{\ensuremath{\mathcal{W}_{\mathbb{R}^*}}}
\newcommand{\mE}{\ensuremath{\mathcal{E}}}
\newcommand{\bzeta}{\ensuremath{\boldsymbol \zeta}}

\newcommand{\bzer}{\ensuremath{\mathbf{0}}}
\newcommand{\bone}{\ensuremath{\mathbf{1}}}

\newcommand{\bs}{\ensuremath{\boldsymbol \sigma}}
\newcommand{\bt}{\ensuremath{\boldsymbol \theta}}
\newcommand{\btt}{\ensuremath{\mathbf{t}}}
\newcommand{\bzhat}{\bz_{\hat{k}}}
\newcommand{\bzht}[1]{\bz_{\hat{#1}}}
\newcommand{\bzhtn}{\bz_{\hat{n}}}

\newcommand{\bwhtn}{\bw_{\hat{n}}}

\newcommand{\bbf}{\ensuremath{\mathbf{f}}}

\newcommand{\bZ}{\ensuremath{\mathbf{Z}}}
\newcommand{\bU}{\ensuremath{\mathbf{U}}}
\newcommand{\bQ}{\ensuremath{\mathbf{Q}}}
\newcommand{\sH}{\ensuremath{\mathscr{H}}}
\newcommand{\sC}{\ensuremath{\mathscr{C}}}

\newcommand{\xc}{x_c}
\newcommand{\yc}{y_c}
\newcommand{\bxun}{\bar x_1}
\newcommand{\bxde}{\bar x_2}

\def\Res{\operatorname{Res}}

\newglossaryentry{bz}
{
name={$\bz$:},
sort={a1},
description={shorthand for vector $(z_1,...,z_n)$ of length $n$}
}

\newglossaryentry{bzi}
{
name={$\bz^{\bi}$:},
sort={a2},
description={shorthand for $z_1^{i_1}\cdots z_n^{i_n}$}
}

\newglossaryentry{bzhat}
{
name={$\bzht{k}$:},
sort={a3},
description={shorthand for vector $(z_1,...,z_{k-1},z_{k+1},...,z_n)$ of length $n-1$}
}

\newglossaryentry{Relog}
{
name={$\Relog$:},
sort={a4},
description={the map $\bz \mapsto (\log|z_1|,...,\log|z_n|)$}
}

\newglossaryentry{amoeba}
{
name={$\amoeba(H)$:},
sort={a5},
description={amoeba of the polynomial $H(\bz)$}
}

\newglossaryentry{Dz}
{
name={$D(\bw)$:},
sort={a6},
description={polydisk defined by $\bw \in \mathbb{C}^n$}
}

\newglossaryentry{Tz}
{
name={$T(\bw)$:},
sort={a7},
description={polytorus defined by $\bw \in \mathbb{C}^n$}
}

\newglossaryentry{V2}
{
name={$\mV$:},
sort={a8},
description={singular variety}
}

\newglossaryentry{mD}
{
name={$\mD$:},
sort={b1},
description={open power series domain of convergence}
}

\newglossaryentry{mDbd}
{
name={$\partial\mD$:},
sort={b2},
description={boundary of power series domain of convergence}
}

\newglossaryentry{clmD}
{
name={$\overline{\mD}$:},
sort={b3},
description={closure of power series domain of convergence}
}

\newglossaryentry{Vstar}
{
name={$\mV^*$:},
sort={b4},
description={elements of the singular variety with non-negative coordinates}
}

\newglossaryentry{KronRep}
{
name={$[P(u),\bQ]$:},
sort={b5},
description={Kronecker representation}
}

\newglossaryentry{NumerKronRep}
{
name={$[P(u),\bQ,\bU]$:},
sort={b6},
description={numerical Kronecker representation}
}

\newglossaryentry{V}
{
name={$\mV(f_1,...,f_r)$:},
sort={b7},
description={complex-valued solutions of $f_1=\cdots=f_r=0$}
}

\newglossaryentry{localring}
{
name={$\mO_{\bw}$:},
sort={b8},
description={local ring at $\bw \in \mathbb{C}^n$}
}

\makeglossaries

\makeindex

\newcommand{\diagrF}[1]{
  \begin{tikzpicture}[scale=0.6]\makediag{#1}\end{tikzpicture}
}

\newcommand{\diagrFb}[1]{
  \begin{tikzpicture}[scale=0.3]\makediag{#1}\end{tikzpicture}
}

\makeatletter
\def\testb#1{\testb@i#1,,\@nil}%
\def\testb@i#1,#2,#3\@nil{%
  \draw[->, thick] (O) --++(#1);
  \ifx\relax#2\relax\else\testb@i#2,#3\@nil\fi}
\makeatother   

\newcommand{\makediag}[1]{
    \coordinate (O) at (0,0); \coordinate (N) at (0,0.8);
    \coordinate (NE) at (0.8,0.8); \coordinate (E) at (0.8,0);
    \coordinate (SE) at (0.8,-0.8); \coordinate (S) at (0,-0.8);
    \coordinate (SW) at (-0.8,-0.8);\coordinate (W) at (-0.8,0);
    \coordinate (NW) at (-0.8,0.8); \coordinate (B1) at (1.2,1.2);
    \coordinate (B2) at (-1.2,-1.2);
    
    \testb{#1}
} 

\newcommand{\diag}[1]{
  \begin{tikzpicture}[scale=0.2]\makediagb{#1}\end{tikzpicture}
}

\makeatletter
\def\testbb#1{\testbb@i#1,,\@nil}%
\def\testbb@i#1,#2,#3\@nil{%
  \draw (O) --++(#1);
  \ifx\relax#2\relax\else\testbb@i#2,#3\@nil\fi}
\makeatother   

\newcommand{\makediagb}[1]{
    \coordinate (O) at (0,0); \coordinate (N) at (0,1);
    \coordinate (NE) at (1,1); \coordinate (E) at (1,0);
    \coordinate (SE) at (1,-1); \coordinate (S) at (0,-1);
    \coordinate (SW) at (-1,-1);\coordinate (W) at (-1,0);
    \coordinate (NW) at (-1,1); \coordinate (B1) at (1.2,1.2);
    \coordinate (B2) at (-1.2,-1.2);
    
    \draw (B1) --++(0,-2.4); \draw (B1) --++ (-2.4,0);
    \draw (B2) --++(0,2.4);  \draw (B2) --++ (2.4,0);   
    \testbb{#1}
}

\begin{document}

%----------------------------------------------------------------------
% FRONT MATERIAL
%----------------------------------------------------------------------
\pagestyle{empty}
\pagenumbering{roman}

\begin{titlepage}
        \begin{center}
        \vspace*{1.0cm}

        \Huge
        {\bf Analytic Combinatorics in Several Variables: Effective Asymptotics and Lattice Path Enumeration}

        \vspace*{1.0cm}

        \normalsize
        by \\

        \vspace*{1.0cm}

        \Large
        Stephen Melczer \\

        \vspace*{3.0cm}

        \normalsize
        A thesis \\
        presented to the University of Waterloo \\ 
        and the École normale supérieure de Lyon \\ 
        in fulfillment of the \\
        thesis requirement for the degree of \\
        Doctor of Philosophy \\
        in \\
        Computer Science \\

        \vspace*{2.0cm}

        Waterloo, Ontario, Canada, 2017 \\

        \vspace*{1.0cm}

        \copyright\ Stephen Melczer 2017 \\
        \end{center}
\end{titlepage}

\pagestyle{plain}
\setcounter{page}{2}

\cleardoublepage

% COMMITTEE 
% -------------------------------

\noindent
\textbf{Examining Committee Membership} \\\\
The following served on the Examining Committee for this thesis. The decision of the Examining Committee is by majority vote.
\bigskip

\begin{tabular}{ll}
Co-Supervisor & George Labahn \\
& Professor, University of Waterloo \\\\\\
Co-Supervisor & Bruno Salvy \\
& Director of Research, INRIA and ENS Lyon \\\\\\
Examining Member & Jason Bell \\
(Waterloo Internal-external) & Professor, University of Waterloo \\\\\\
Examining Member & Sylvie Corteel \\
& Director of Research, CNRS and University of Paris 7 Diderot \\\\\\
Examining Member \qquad\qquad\qquad & Michael Drmota \\
(Waterloo External Examiner) & Professor, TU Vienna \\\\\\
Examining Member & Éric Schost \\
(Waterloo Internal) & Associate Professor, University of Waterloo 
\end{tabular}
\vspace{0.3in}

\noindent
\textbf{Rapporteurs} \\\\
The following were rapporteurs for the École normale supérieure de Lyon.
\bigskip

\begin{tabular}{ll}
Rapporteure \qquad\qquad\qquad & Sylvie Corteel \\
& Director of Research, CNRS and University of Paris 7 Diderot \\\\\\
Rapporteur \qquad\qquad\qquad & Ira Gessel \\
& Professor Emeritus, University of Brandeis
\end{tabular}

\cleardoublepage

\begin{center}\textbf{Statement of Contributions}\end{center}

I am the sole author of Chapters 1, 3, 4, 5, 6, 9, and 12.  Chapter 2 was written by me but translated from English to French with the help of Bruno Salvy.  Chapter 7 is partially based on an article co-authored with Marni Mishna.  Chapter 8 is partially based on an article co-authored with Bruno Salvy.  Chapter 10 is partially based on an article co-authored with Mark Wilson.  Chapter 11 is partially based on an article co-authored with Julien Courtiel, Marni Mishna, and Kilian Raschel.

\cleardoublepage

% A B S T R A C T
% ---------------

\begin{center}\textbf{Abstract}\end{center}

The field of analytic combinatorics, which studies the asymptotic behaviour of sequences through analytic properties of their generating functions, has led to the development of deep and powerful tools with applications across mathematics and the natural sciences.  In addition to the now classical univariate theory, recent work in the study of analytic combinatorics in several variables (ACSV) has shown how to derive asymptotics for the coefficients of certain D-finite functions represented by diagonals of multivariate rational functions. This thesis examines the methods of ACSV from a computer algebra viewpoint, developing rigorous algorithms and giving the first complexity results in this area under conditions which are broadly satisfied. Furthermore, this thesis gives several new applications of ACSV to the enumeration of lattice walks restricted to certain regions.  In addition to proving several open conjectures on the asymptotics of such walks, a detailed study of lattice walk models with weighted steps is undertaken.
\vspace{0.5in}

La combinatoire analytique étudie le comportement asymptotique des suites à travers les propriétés analytiques de leurs fonctions génératrices. Ce domaine a conduit au développement d'outils profonds et puissants avec de nombreuses applications. Au-delà de la théorie univariée désormais classique, des travaux récents en combinatoire analytique en plusieurs variables (ACSV) ont montré comment calculer le comportement asymptotique d'une grande classe de fonctions différentiellement finies: les diagonales de fractions rationnelles. Cette thèse examine les méthodes de l'ACSV du point de vue du calcul formel, développe des algorithmes rigoureux et donne les premiers résultats de complexité dans ce domaine sous des hypothèses très faibles. En outre, cette thèse donne plusieurs nouvelles applications de l'ACSV à l'énumération des marches sur des réseaux restreintes à certaines régions : elle apporte la preuve de plusieurs conjectures ouvertes sur les comportements asymptotiques de telles marches, et une étude détaillée de modèles de marche sur des réseaux avec des étapes pondérées.

\cleardoublepage

% A C K N O W L E D G E M E N T S
% -------------------------------

\begin{center}\textbf{Acknowledgements}\end{center}

I would like to thank:

\begin{itemize}
     \itemsep1.3em
\item[] My supervisors, Bruno Salvy and George Labahn, for all their support, encouragement, editing, and signing of paperwork, and all of the fascinating research we worked on together;
\item[] My collaborators and co-authors, Alin Bostan, Mireille Bousquet-Mélou, Sophie Burrill, Julien Courtiel, Éric Fusy, Manuel Kauers, Kilian Raschel, Mark Wilson, and (several times over) Marni Mishna, for their sage wisdom, advise, and mentoring;
\item[] All of my family and friends in Vancouver, Waterloo, and Lyon, for their support;
\item[] Brett Nasserden for our discussions on some of the more algebraic aspects of this work;
\item[] Jason Bell, Sylvie Corteel, Michael Drmota, and Éric Schost for their role on my jury. Thanks also to Ira Gessel for providing a report on this thesis for the French system;
\item[] Boris Adamczewski, Benoit Charbonneau, Ruxandra Moraru, and Mohab Safey El Din for letting me sit in on classes which helped inform some of the background on this thesis (and were just plain interesting);
\item[] The Natural Sciences and Engineering Research Council of Canada, the David R. Cheriton School of Computer Science, the French Ministry of Foreign Affairs and International Development, the France-Canada Research Fund, the University of Waterloo, and Inria, for supporting the research conducted during this degree;
\item[] Finally, Celia, for all the love she's given me.
\end{itemize}

\cleardoublepage

% T A B L E   O F   C O N T E N T S
% ---------------------------------
\renewcommand\contentsname{Table of Contents}
\tableofcontents
\cleardoublepage
\phantomsection    % allows hyperref to link to the correct page

% L I S T   O F   F I G U R E S
% -----------------------------
\addcontentsline{toc}{chapter}{List of Figures}
\listoffigures
\cleardoublepage
\phantomsection		% allows hyperref to link to the correct page

% L I S T   O F   T A B L E S
% ---------------------------
\addcontentsline{toc}{chapter}{List of Tables}
\listoftables
\cleardoublepage
\phantomsection		% allows hyperref to link to the correct page

% GLOSSARIES 
% -----------------------------
%\printglossaries 
\printglossary[title={List of Symbols and Notation}]
\cleardoublepage
\phantomsection		% allows hyperref to link to the correct page

% Epigraph
% -----------------------------
\addcontentsline{toc}{chapter}{Epigraph}

\setlength{\epigraphwidth}{3.7in}

\epigraph{The question you raise ``how can such a formulation lead to computations'' doesn't bother me in the least! Throughout my whole life as a mathematician, the possibility of making explicit, elegant computations has always come out by itself, as a byproduct of a thorough conceptual understanding of what was going on. Thus I never bothered about whether what would come out would be suitable for this or that, but just tried to understand -- and it always turned out that understanding was all that mattered.}{Alexander Grothendieck, letter to Ronnie Brown dated 12.04.1983}

\epigraph{...in an ideal world, people would learn this material over many years, after having background courses in commutative algebra, algebraic topology, differential geometry, complex analysis, homological algebra, number theory, and French literature. We do not live in an ideal world.}{Ravi Vakil, \emph{The Rising Sea: Foundations of Algebraic Geometry}}
\cleardoublepage
\phantomsection     % allows hyperref to link to the correct page

% Change page numbering back to Arabic numerals
\pagenumbering{arabic}

%----------------------------------------------------------------------
% MAIN BODY
%----------------------------------------------------------------------

\part{Background and Motivation}
\label{part:Background}

%%%%%%%%%%%%%%%%%%%%%%%%%%%%%
% CHAPTER 1
%%%%%%%%%%%%%%%%%%%%%%%%%%%%%
\chapter{Introduction}
%======================================================================
\vspace{-0.2in}

\setlength{\epigraphwidth}{3.9in}
\epigraph{Often I have considered the fact that most of the difficulties which block the progress of students trying to learn analysis stem from this: that although they understand little of ordinary algebra, still they attempt this more subtle art.\footnotemark}{Leonhard Euler, \emph{Introductio in analysin infinitorum}}
\footnotetext{Translated from the Latin by John D. Blanton.}
\vspace{-0.2in}

\epigraph{For it is unworthy of excellent men to lose hours like slaves in the labor of calculation which could safely be relegated to anyone else if the machine were used.\footnotemark}{Gottfried Wilhelm Leibniz, \emph{Machina arithmetica in qua non additio tantum et subtractio sed et multiplicatio \dots}}
% nullo, diviso vero paene nullo animi labore peragantur}}
\footnotetext{Translated from the Latin by Mark Kormes.}

A fundamental problem in mathematics is how to efficiently encode mathematical objects and, from such encodings, determine their underlying properties.  Dating back at least to the seventeenth century work of Leibniz\footnote{On February 1, 1673 Leibniz (originally inspired by the sight of a pedometer in Paris) presented to the Royal Society of London a machine which could add, subtract, multiply, and divide numbers.  In 1674 Leibniz outlined a machine capable of solving certain algebraic equations, and later went on to write about topics such as the mechanization of logical reason and rules of deduction, properties of binary arithmetic, and the encoding of all human knowledge in symbolic form. See Davis~\cite[Chapter 1]{Davis2001} for more information.}, many mathematicians and scholars have been enthralled by the possibility of mechanizing the rules of logical reasoning and systematizing mathematical discovery.  In the twentieth century, leaps in the study of formal logic, the rise of computer science, and the formalization of computability and complexity theory helped to illustrate the power of such thinking.  Unfortunately, these developments also led to the discovery of undecidability results at the heart of computational mathematics, such as the following.

\begin{theorem*}[{Matiyasevich~\cite[Section 9.2]{Matiyasevich1993}}]
Let $\mathcal{F}$ denote the class of all functions of one variable $x$ that can be constructed using composition from $x$, the constant 1, addition, subtraction, multiplication, and the functions sine and absolute value. Then there is no method for determining for an arbitrary given function $f$ in the
class $\mathcal{F}$ whether $f(x)$ is identically zero.\footnote{See also the notes to Chapter 1 of Bostan et al.~\cite{BostanChyzakGiustiLebretonLecerfSalvySchost2017} for historical remarks on this result.}
\end{theorem*}

This poses a challenge for the modern study of computer algebra, where mathematical theory and computational tools are brought together to design and analyze mathematical algorithms.
%\footnote{According to Loos~\cite[Page 1]{BuchbergerCollinsLoosAlbrecht1983}, ``Computer algebra is that part of computer science which designs, analyzes, implements and applies algebraic algorithms\dots  Negatively, one could say that computer algebra treats those subjects which are too computational to appear in an algebra text and which are too algebraic to be presented in a usual computer science text book.''}.  
In particular, due to undecidability results, there are simply stated problems which cannot be solved computationally.  An algebraic structure $A$ (such as a ring, field, vector space, etc.) is called \emph{effective} if each element can be represented by some finite data structure and there are algorithms to carry out the operations of $A$ and to test predicates such as equalities\footnote{See Chapter 1 of Bostan et al.~\cite{BostanChyzakGiustiLebretonLecerfSalvySchost2017} for more information about effective objects in computer algebra.  All of the results on effectiveness listed here can be found in that source.}.  Easy examples include the ring of integers modulo a fixed positive integer $B$ (which contains only a finite number of elements), the ring of integers (whose elements can be encoded by their base $B$ representations for some fixed positive integer $B$), the field of rational numbers (whose elements can be encoded by pairs of integers), and the ring of polynomials with rational coefficients (whose elements can be encoded by arrays of rational numbers). A less obvious, but still classical, example of an effective field is the field of algebraic numbers, whose elements are represented by their minimal polynomials and isolating disks.
%$m_{\alpha}(x) \in \mathbb{Q}[x]$ together with an open disk $U$ in the complex plane with rational center and radius such that $\alpha$ is the only root of $m_{\alpha}$ in $U$. In order to determine the encoding of a sum, difference, product, or division of two elements, tools from algebraic elimination theory, such as the resultant, can be used.
%\footnote{See Bostan et al.~\cite[Proposition 6.10]{BostanChyzakGiustiLebretonLecerfSalvySchost2017} for a precise result.}  
%The effectiveness of $\mathbb{A}$ means that one can automatically and rigorously prove equalities between two expressions involving algebraic numbers.  

Given an effective ring $A$, any matrix ring with entries in $A$ is effective, as is the ring of polynomials with coefficients in $A$; when $A$ is an integral domain its field of fractions is effective, and when $A$ is a field its algebraic closure is effective.  Once a structure is known to be effective, which is in essence a decidability result, it is natural to wonder about the complexity of performing operations with elements of the structure.  These computability and complexity problems lie at the heart of computer algebra.

\subsection*{Generating Functions and Effective Enumeration}

In this thesis we study problems arising in enumerative combinatorics from a computer algebra perspective.  Given a sequence $(f_n)_{n \geq 0}$, our aim is to determine either: a simple closed form expression for the element $f_n$ as a function of $n$, or a simple representation of the asymptotic behaviour of $f_n$ as $n$ approaches infinity\footnote{Of course, the notion of a ``simple'' closed form expression is subjective, and thus open to interpretation.  We do not touch on this topic here, but refer the interested reader to the discussion in Section 1.1 of Stanley~\cite{Stanley2012}. The sequences we encounter in this thesis will have their dominant asymptotics specified by a finite collection of algebraic numbers and rational evaluations of the gamma function $\Gamma(s)$.  By a representation of asymptotics we thus mean a determination of this finite set of information; see Chapter~\ref{ch:Background} for more information.}.  We focus mainly on problems where exact enumeration is difficult and asymptotics are desired; our main tool will be the use of generating functions.  Given a sequence $(f_n)_{n \geq 0}$ of elements in a ring $A$, the \emph{generating function} of $(f_n)$ is the formal power series
\begin{equation} F(z) = \sum_{n \geq 0} f_n z^n \in A[[z]]. \label{eq:introFz} \end{equation}
Effectiveness of the ring $A$ does not imply effectiveness of the ring $A[[z]]$, as to be effective the power series under consideration must be encoded by a finite amount of information.  When $A$ is effective the ring of formal power series which satisfy algebraic equations, and the ring of formal power series which satisfy linear differential equations\footnote{The (formal) derivative of a formal power series $\sum_{n \geq 0} f_n z^n$ is defined as the formal power series $\sum_{n \geq 1}n z^{n-1}$.  When a formal power series defines an analytic function at the origin, this definition matches with the usual analytic derivative.} with polynomial coefficients, are effective.  Given a formal power series, specified by equations over some effective ring, our goal is to determine asymptotics of its coefficient sequence (or determine when such a task is undecidable).

Suppose now that $A \subset \mathbb{C}$ and there exists a constant $K>0$ such that $|f_n| \leq K^n$ for all $n \in \mathbb{N}$.  Then the power series in Equation~\eqref{eq:introFz} defines an analytic function when $z$ is restricted to a neighbourhood of the origin, and the powerful tools of complex analysis can be applied to $F(z)$. In particular, Cauchy's residue theorem implies
\[ f_n = \int_C \frac{F(z)}{z^{n+1}} dz, \]
where $C$ is a counter-clockwise circle in the complex plane sufficiently close to the origin.  This equality relates the coefficients of $F$ to an analytic object, and allows one to determine asymptotics of $f_n$ by determining asymptotics of a parametrized integral in the complex plane.  The systematic use of analytic techniques to study the asymptotic behaviour of sequences is known as the study of \emph{analytic combinatorics}~\cite{FlajoletSedgewick2009}, and the main results of analytic combinatorics illustrate strong links between the singularities of an analytic generating function and asymptotics of its coefficients. 

When the power series coefficients of $F(z)$ do not decay super-exponentially, $F$ admits at least one singularity in the complex plane; the singularities of $F$ with minimum modulus are known as \emph{dominant singularities}.  If the dominant singularities of $F$ have modulus $r>0$ then the \emph{exponential growth} $\rho = \limsup_{n \rightarrow \infty} |f_n|^{1/n}$ of the coefficients $f_n$, which is the coarsest measure of their asymptotics, satisfies $\rho=1/r$.  To completely determine the dominant asymptotics of $f_n$ one usually finds the dominant singularities of $F$, giving the exponential growth of $f_n$, and then performs a local analysis at each of these singularities (when they are finite in number).  For most examples encountered in applications, it is sufficient to determine the type\footnote{For example, is each dominant singularity a simple pole, higher order pole, an algebraic branch cut, a logarithmic branch cut, etc.} of each dominant singularity together with small amount of additional information (such as the residue at a pole) which can then be substituted into known formulas.

\paragraph{Generating Function Classes}

The universality of many properties of analytic functions often allows for an automated asymptotic analysis for generating functions fitting into certain classes.  As a first example, the generating function $F(z)$ of any sequence satisfying a linear recurrence relation with integer coefficients is rational\footnote{The use of generating functions as formal series whose coefficients encode sequences of interest dates back to the eighteenth century work of de Moivre, who showed~\cite[Theorem V]{Moivre1730} that the generating function of any linear recurrence relation with constant coefficients is rational.  Although hinted at in the work of de Moivre, Euler~\cite[page 201]{Euler1988} was among the first to explicitly consider such formal series as functions which could be evaluated using these representations as rational functions.}, and using a partial fraction decomposition one can automatically determine asymptotics of such a sequence $(f_n)$ from any linear recurrence relation satisfied by $f_n$ together with a finite number of initial terms (see Section~\ref{sec:RatPS} below).  In a similar manner, an algebraic power series $F(z)$ over the rational numbers which is analytic at the origin can be encoded by its minimal polynomial and a finite number of initial coefficients, and from such an encoding it is possible to automatically determine asymptotics of its coefficient sequence (see Section~\ref{sec:AlgPS} below).  These two classes of functions contain the generating functions of many sequences arising in applications. For example, the sequence counting the number of words in a rational language by length is always rational, and sequences enumerating unambiguous context-free languages, many types of trees, pattern-avoiding permutations, certain planar maps, and triangulations have algebraic generating functions (in addition to many other examples, see Stanley~\cite[Chapter 6]{Stanley1999}). 

The rings of rational and algebraic generating functions mirror the rings of rational and algebraic numbers, and this is reflected in the way these objects can be encoded.  Under our assumptions a generating function defines an analytic function at the origin, and one can additionally consider acting on these functions with operations from calculus.  In particular, the ring of analytic \emph{D-finite} functions (which contains the ring of algebraic power series which are analytic at the origin) consists of all analytic power series which satisfy linear differential equations with polynomial coefficients.  A D-finite function can be encoded by an annihilating linear differential equation together with initial conditions, and the ring of analytic D-finite functions with rational coefficients is effective. An analytic function is D-finite if and only if its coefficient sequence satisfies a linear recurrence relation with polynomial coefficients, and D-finite functions occur in many applications\footnote{Examples of D-finite functions include generalized hypergeometric functions (with fixed parameters), Bessel functions and many other special functions, and all examples of rational diagonal functions given later in this thesis; the class of D-finite functions is also closed under several natural operations.  Additional information is given in Section~\ref{sec:DFin}.}.  Although this is an effective class of generating functions, it is currently unknown whether or not it is decidable to determine coefficient asymptotics of an arbitrary D-finite function (see Section~\ref{sec:DFin} below).

In this thesis we focus on coefficient asymptotics for a sub-class of D-finite functions called \emph{multivariate rational diagonals}.  Given an $n$-variate rational function $F(\bz)$ with power series expansion
\[ F(\bz) = \sum_{\bi \in \mathbb{N}^n} f_{\bi} \bz^{\bi} = \sum_{i_1,\dots,i_n \in \mathbb{N}} f_{i_1,\dots,i_n} z_1^{i_1} \cdots z_n^{i_n} \]
at the origin, the \emph{diagonal} of $F(\bz)$ is the univariate function obtained by taking the coefficients where all variable exponents are equal:
\[ (\Delta F)(z) := \sum_{k \geq 0} f_{k,k,\dots,k} z^k. \]
The diagonal of any rational function is D-finite, and every algebraic function can be realized as the diagonal of a bivariate rational function (see Section~\ref{sec:mvratDiag} below).  Because the ring of multivariate rational diagonals lies between the class of algebraic functions, where coefficient asymptotics can be determined automatically, and the ring of D-finite functions, where this problem is still open, they make a prime subject on which to study effective coefficient asymptotics.  Many problems in combinatorics (lattice path enumeration, statistics on trees, irrational tilings of rectangles), probability theory (random walk models), number theory (binomial sums such as Apéry's sequence, used in his proof of the irrationality of $\zeta(3)$) and physics (the Ising model) appear naturally as questions about rational diagonals.  In order to study the asymptotics of rational diagonal coefficient sequences we use results from the new field of \emph{analytic combinatorics in several variables} (which we often abbreviate as ACSV).  

\paragraph{Effective Enumerative Results}

This thesis gives the first fully rigorous algorithms and complexity results for determining the asymptotics of non-algebraic rational diagonal coefficient sequences under conditions which are broadly satisfied.  In addition, we take a look at several applications of ACSV to problems arising in lattice path enumeration.  One motivation for this study was a set of conjectured asymptotics by Bostan and Kauers~\cite{BostanKauers2009}, who found annihilating linear differential equations for the generating functions of certain lattice path sequences but were unable to prove asymptotics for the sequences.  Using the results of ACSV we are able to prove asymptotics of these sequences for the first time, explain observed asymptotic behaviour analytically, and study much more general classes of lattice path problems.  Lattice path enumeration also provides a rich family of problems to help illustrate the theory of ACSV, providing a wealth of concrete examples for those wanting to learn its methods and possibly hinting at future directions for research\footnote{It is interesting to note that the development of complex analysis was greatly inspired by the study of elliptic functions, while the theory of complex analysis in several variables suffered due to lack of concrete problems.  To quote work of Blumenthal~\cite{Blumenthal1903} from 1903, ``If up till now the theory of functions of several variables has lagged behind the widely extended and highly developed theory of functions of a single complex variable, this can essentially be attributed to the lack of interesting and appropriate examples with which a general theory could connect.'' (translated from the German by Bottazzini and Gray~\cite[page 679]{BottazziniGray2013}).}.

There are several (potentially overlapping) audiences for this thesis: mathematicians interested in the behaviour of functions satisfying certain algebraic, differential, or functional equations; combinatorialists interested in learning the new theory of analytic combinatorics in several variables; computer scientists interested in new applications of computer algebra and real algebraic geometry; and researchers from a variety of domains with an interest in lattice path enumeration.  

We first give a broad overview and history of the theory of analytic combinatorics in several variables and the study of lattice path enumeration, before highlighting our original research contributions and going into specifics on the content in each chapter.  In this thesis we deal mainly with (rational, algebraic, and D-finite) generating functions directly and, outside of lattice path enumeration, do not say much about how one goes from a combinatorial specification of a problem to a description of its generating function. There are several large theories built around this topic including the `Symbolic Method' described in Flajolet and Sedgewick~\cite{FlajoletSedgewick2009}, Joyal's Theory of Species~\cite{Joyal1981,BergeronLabelleLeroux1998}, and the Delest-Viennot-Schützenberger methodology~\cite{Delest1996} for context-free languages.

\section{Analytic Combinatorics in Several Variables}

We now describe the theory of analytic combinatorics in several variables, as it has been developed by Pemantle and Wilson~\cite{PemantleWilson2013}, and their collaborators.  Suppose $F(\bz) = G(\bz)/H(\bz)$, where $G,H \in \mathbb{Z}[z_1,\dots,z_n]$ are co-prime polynomials.  When $H(\bzer)$ is non-zero, $F$ is analytic at the origin and thus admits a power series expansion
\[ F(\bz) = \sum_{\bi \in \mathbb{N}^n} f_{\bi} \bz^{\bi}, \]
valid in some open domain of convergence $\mD$. As in the univariate case, there is a strong link between the singularities of $F(\bz)$, which are the elements of the \emph{singular variety} $\mV = \{\bz:H(\bz)=0\}$, and asymptotics of the diagonal sequence $f_{k,\dots,k}$ as $k\rightarrow\infty$.  Singularities $\bw \in \mV$ which are on the boundary of the domain of convergence $\bw \in \mV \cap \partial \mD$ are known as \emph{minimal points}, and are a generalization of dominant singularities in the univariate case.

The study of analytic combinatorics becomes much more difficult in several variables\footnote{There was not even a clear definition of an analytic function in several variables for half a century.  Undertaking some preliminary studies on multivariate complex functions (including generalizations of the Cauchy integral formula) in the 1830s, Cauchy considered a multivariate function to be analytic over a domain $\mD$ if it was analytic as a univariate function of each variable at every point in $\mD$, and this definition was also used by Jordan.  Weierstrass, on the other hand, called a multivariate function analytic in a domain $\mD$ if it had a power series representation in the neighbourhood of any point in the domain (Poincaré also used this definition in this doctoral thesis in 1879).  These two definitions were not shown to be equivalent until work of Hartogs~\cite{Hartogs1906} in 1906. See Bottazzini and Gray~\cite[Chapter 9]{BottazziniGray2013} for additional historical information on the development of complex analysis in several variables.}.  Although many\footnote{For instance, any meromorphic function has a finite number of dominant singularities, and any rational, algebraic, or D-finite function has a finite number of singularities in the complex plane.} univariate functions which are analytic at the origin admit a finite number of dominant singularities, in the multivariate case (when $n \geq 2$) there will always be an infinite number of minimal points unless $F(\bz)$ is a polynomial.  The ultimate goal, following the univariate case, is to determine a finite number of minimal points where a local singularity analysis of $F(\bz)$ allows one to determine asymptotics of the diagonal sequence.  The fact that this is not always possible is a reflection of the pathologies which can arise dealing with the singularities of multivariate functions.

\paragraph{Critical Points}
Similar to the univariate case, in order to determine asymptotics of the diagonal sequence of $F(\bz)$ one begins with the multivariate Cauchy integral formula
\begin{equation} f_{k,k,\dots,k} = \frac{1}{(2\pi i)^n} \int_C F(\bz) \frac{dz_1 \cdots dz_n}{z_1^{k+1} \cdots z_n^{k+1}}, \label{eq:introCIF} \end{equation}
where $C$ is a product of circles sufficiently close to the origin. Using standard integral bounds it can (and, in Chapter~\ref{ch:SmoothACSV}, will) be shown that every minimal point $\bw \in \mV \cap \partial \mD$ gives an upper bound 
\[ \rho \leq |w_1 \cdots w_n|^{-1}\] 
on the exponential growth $\rho := \limsup_{k \rightarrow \infty} |f_{k,\dots,k}|^{1/k}$ of the diagonal sequence.  To find a set of minimal points where a local singularity analysis of $F(\bz)$ determines asymptotics, it makes sense to look for the minimal points minimizing this upper bound as these are the only ones where the integrand of Equation~\eqref{eq:introCIF} could have the same exponential growth as the diagonal sequence.

Suppose first that $H$ is square-free and $\mV$ is a complex manifold (i.e., that $H$ and its partial derivatives do not simultaneously vanish).  To minimize the upper bound on exponential growth, it is sufficient to consider points with non-zero coordinates.  The map $h(\bz) = -\log |z_1 \cdots z_n|$ from the points in $\mV$ with non-zero coordinates to the real numbers is a smooth map of manifolds, and basic results in differential geometry imply that any local extremum of this map must be a critical point (that is, a point where the differential of $\phi$ is zero).  In Chapter~\ref{ch:SmoothACSV} we show that such points correspond to the solutions of the algebraic system of \emph{smooth critical point equations}
\[ H(\bz) = 0, \qquad z_1(\partial H/\partial z_1)(\bz) = \cdots = z_n(\partial H/\partial z_n)(\bz), \]
and when $\mV$ is a manifold such points are called \emph{critical points} of $F(\bz)$.  

When $\mV$ is not a manifold one must partition $\mV$ into a collection of manifolds called \emph{strata} and examine critical points of the map $\bz \mapsto -\log |z_1 \cdots z_n|$ when restricted to each stratum.  In Chapter~\ref{ch:NonSmoothACSV} we discuss how the critical points on any stratum can always be defined by an algebraic system of equations.  The equations defining critical points depend on the local geometry of $\mV$, and when $\mV$ is a manifold in a neighbourhood of a point $\bw$ then $\bw$ is critical if and only if it satisfies the smooth critical point equations.  In practice, it is usually easy to characterize the critical points of $F(\bz)$, but much more difficult to decide which (if any) are minimal. 

When there are minimal critical points where the singular variety is locally a manifold such points must minimize the upper bound $|z_1\cdots z_n|^{-1}$ on $\rho$, however this is not true for non-smooth minimal critical points.  Even when $\mV$ is a manifold and $|z_1\cdots z_n|^{-1}$ achieves its minimum over the set of minimal points it is not necessary to have minimal critical points (see Example~\ref{ex:smoothnomincrit}).

\paragraph{Determining Asymptotics}

Analogously to the univariate case, to determine asymptotics one tries to deform the contour of integration $C$ in the multivariate Cauchy residue integral~\eqref{eq:introCIF} until it reaches the singularities of $F(\bz)$, and then attempts to perform a local singularity analysis.  Intuitively, minimal points are those to which the contour $C$ can be easily deformed, as they are on the boundary of the domain of convergence, while critical points are those where such a singularity analysis can be performed to determine asymptotics. As in the univariate case, the nature of the singular variety at minimal critical points is important to the determination of asymptotics. When dealing with multivariate rational functions only polar singularities arise, but a multivariate rational function can exhibit a wide range of singular behaviour depending on the geometry of $\mV$.  

The easiest case is when $\mV$ admits a single minimal critical point $\bw$, around which $\mV$ is locally a complex manifold.  Assuming an extra condition on the local geometry of $\mV$ at $\bw$, which is typically satisfied in applications, one can determine asymptotics of the diagonal sequence by computing a univariate residue integral followed by an $n-1$ dimensional saddle-point integral whose domain of integration can be made arbitrarily close to $\bw$.  When $\mV$ has a finite number of such minimal critical points, one can determine diagonal asymptotics by computing saddle-point integrals around each of these points.  Theorem~\ref{thm:smoothAsm} and Corollary~\ref{cor:smoothAsm} in Chapter~\ref{ch:SmoothACSV} give explicit formulas for diagonal asymptotics in such a situation, which depend only on the minimal critical points and evaluations of the partial derivatives of $G(\bz)$ and $H(\bz)$.

A \emph{transverse multiple point} of $\mV$ is a point where $\mV$ locally is the intersection of manifolds whose tangent planes are linearly independent.  In Chapter~\ref{ch:NonSmoothACSV} we consider dominant asymptotics when $\mV$ admits minimal critical points which are also transverse multiple points.  Under certain conditions which often hold, and which are sufficient for the purposes of this thesis, diagonal asymptotics can again be computed through explicit formulas.  The main asymptotic results of this chapter are Theorems~\ref{thm:mpAsm}, ~\ref{thm:compintasm}, and~\ref{thm:resasm}.

\paragraph{History of Analytic Combinatorics in Several Variables}

Early examples of multivariate generating function analyses include work by Bender, Richmond, Gao, and collaborators~\cite{BenderRichmond1983,BenderRichmondWilliamson1983,GaoRichmond1992,BenderRichmond1999} dating back to the 1980s\footnote{See Section 1.2 of Pemantle and Wilson~\cite{PemantleWilson2013} for additional information on these early works.}.  More recently, the work of Pemantle and Wilson, and collaborators, highlighted above, has brought together results from several different mathematical disciplines in such a way as to develop a large-scale systematic theory of multivariate asymptotics for combinatorial purposes.  The first work of Pemantle and Wilson~\cite{PemantleWilson2002} on this subject described a method for determining asymptotics of $F(\bz)$ when $\mV$ is a complex manifold, and stated ``an ultimate goal\dots is to systematize the extraction of multivariate asymptotics sufficiently that it may be automated, say in Maple''\footnote{Quotation from page 131 of Pemantle and Wilson~\cite{PemantleWilson2002}.}.  This thesis contains the first algorithms and complexity results working towards that goal.

Two years after their first paper, Pemantle and Wilson~\cite{PemantleWilson2004} extended their results to cover certain minimal critical points which are also transverse multiple points.  These early results developing the theory of ACSV used explicit deformations of the multivariate Cauchy residue integral which allowed Pemantle and Wilson to calculate a univariate residue integral followed by a multivariate saddle-point integral.  More recently, Baryshnikov and Pemantle~\cite{BaryshnikovPemantle2011} used more complicated deformations of the domain of integration in the multivariate Cauchy integral to extend these results.  This work shows how the methods of ACSV fit into the very general framework of stratified Morse theory, and Pemantle and Wilson~\cite{PemantleWilson2013} later incorporated additional homological tools, such as multivariate complex residues\footnote{Multivariate complex residues were previously applied to determine coefficient asymptotics when the denominator of $F(\bz)$ is a product of linear factors~\cite{Lichtin1991,BertozziMcKenna1993} and when $F$ is bivariate~\cite{Lichtin1995}.}.

The work of Baryshnikov and Pemantle shows that although minimal points are the more natural generalization of dominant singularities from the univariate case, critical points are the ones which determine diagonal asymptotics (when they exist).  In theory, these results allow one to determine diagonal asymptotics in some cases when no critical points are minimal, but the results are less explicit.  The Morse theoretic approach to ACSV also shows that diagonal asymptotics can be determined in several situations when there are an infinite number of minimal points minimizing $|z_1 \cdots z_n|^{-1}$ but only a finite number of them are critical.  A recent textbook by Pemantle and Wilson~\cite{PemantleWilson2013} collects these results, but its focus on the homological viewpoint makes it difficult to follow for first time readers.  This thesis aims to give a general presentation of the results of ACSV which focuses more on explicit calculations (although we will still make use of some of the more advanced results).

\section{Lattice Path Models}
Roughly speaking, a lattice path model is a combinatorial class which encodes the number of ways to ``move'' on a lattice subject to certain constraints.  More precisely, given a dimension $n \in \mathbb{N}$, a finite \emph{set of allowable steps} $\mS \subseteq\mathbb{Z}^n$, and a \emph{restricting region} $\mR \subseteq\mathbb{Z}^n$, the \emph{integer lattice path model} taking steps in $\mS$ and restricted to $\mR$ is the combinatorial class consisting of sequences of the form $(s_1,\dots,s_k)$, where $s_j \in \mS$ for $1 \leq j \leq k$ and every partial sum $s_1 + \cdots + s_r \in \mR$ for $1 \leq r \leq k$ (addition is performed component-wise in $\mathbb{Z}^n$).  The size of an element in this class is the length of the sequence (the number of steps it contains), and by convention we add a single sequence of length zero representing an empty walk.  We view such a sequence as a \emph{path} or \emph{walk} starting at the origin in $\mathbb{Z}^n$ which successively takes steps from $\mS$ and always stays in the region $\mR$ by drawing line segments between the endpoints of the partial sums of the sequence.  We may also restrict the class further by adding other constraints, for instance only admitting sequences which end in some terminal set $\mathcal{T} \subseteq\mathbb{Z}^n$ (the element sum of each sequence in the class lies in $\mathcal{T}$).  

\begin{figure}
\centering
\includegraphics[width=0.7\linewidth]{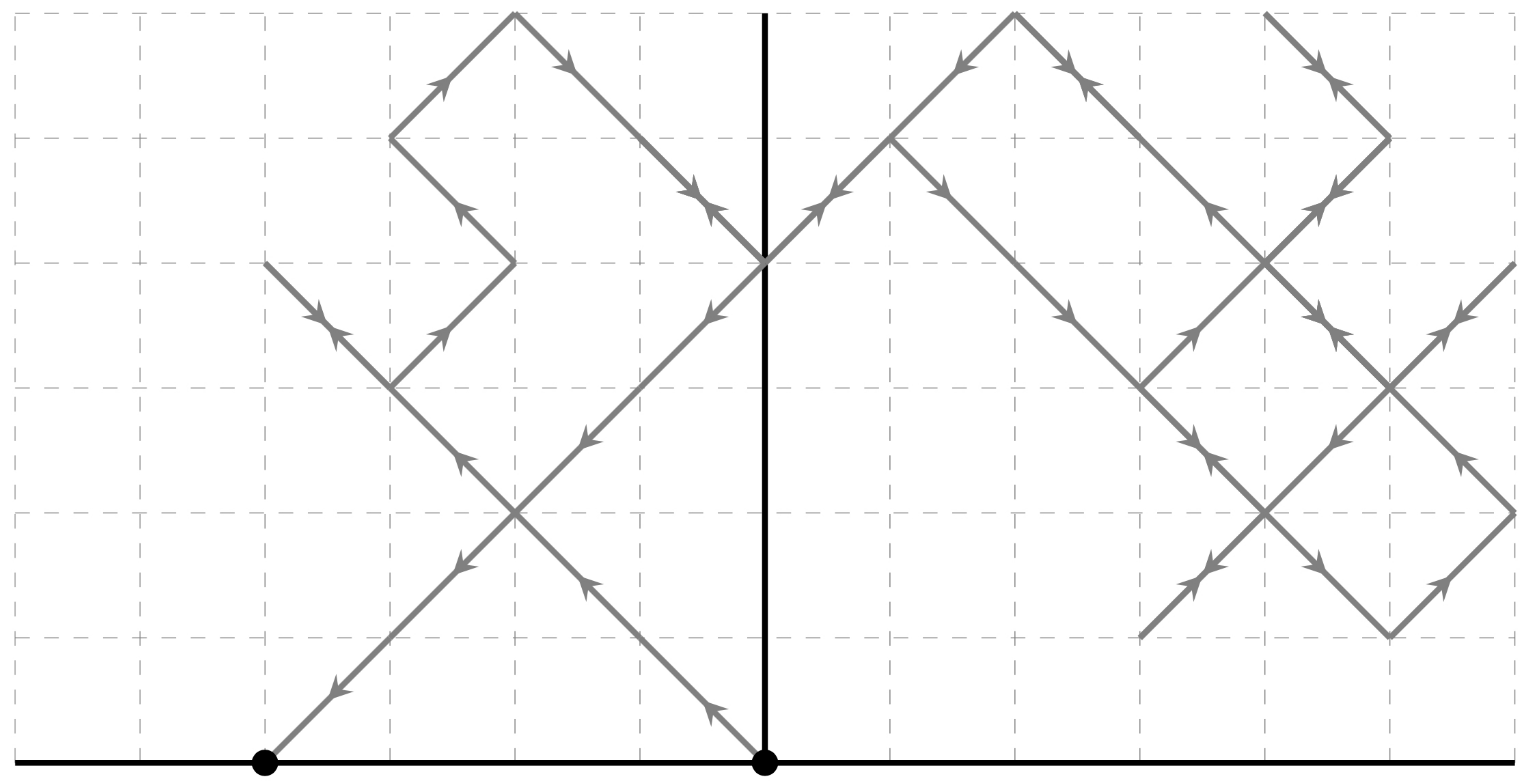}
\caption[A lattice walk of length 50 restricted to a half-space]{A lattice walk of length 50 on the steps $\mS = \{(-1,-1),(-1,1),(1,-1),(1,1)\}$ restricted to a half-space and ending on the x-axis.}
\end{figure}

As laid out in the historical survey of Humphreys~\cite{Humphreys2010}, the earliest accounts of what are now considered lattice path problems arose in probabilistic contexts as far back as the seventeenth century studies of Pascal and Fermat, including examples analogous to the ballot problem in the work of de Moivre~\cite{Moivre1711} in 1711.  An 1878 work of Whitworth~\cite{Whitworth1878} uses explicit lattice path terminology (for instance ``paces'' from an origin) to consider ``Arrangements of $m$ things of one sort and $n$ things of another sort under certain conditions of priority'', and answered questions posed by the Educational Times in 1878 including the probability of drinking $k$ glasses of wine and $k$ glasses of water in a random order while never drinking more wine than water.  

Lattice walks in the early twentieth century were considered by many to be a recreational topic, as exemplified by an article of Grossman~\cite{Grossman1950} entitled ``Fun with lattice points'' and published in the journal Scripta Mathematica aimed at the layperson.  The mid twentieth century saw strong interest in lattice walks and the related topic of random walks from the field of physics~\cite{Mohanty1979}.  Lattice path models are able to model physical phenomena through their application to statistical mechanics, for instance in the study of polymers in a solution~\cite{Rensburg2000}.  Modern applications include results in statistical mechanics, probability theory, formal language theory~\cite{Bousquet-Melou2005a}, queuing theory~\cite{Bohm2010}, the analysis of data structures~\cite{BruijnKnuthRice1972}, mathematical art~\cite{Irvine2016}, and the study of other combinatorial structures such as plane partitions~\cite{AlegriBrietzkeSantosSilva2011} or sequences of Young tableaux~\cite{BurrillCourtielFusyMelczerMishna2016}.  

\paragraph{The Kernel Method and Walks in a Quadrant}
A now classic technique in the study of $n$-dimensional lattice walks restricted to a region is to introduce an $(n+1)$-variate generating function $Q(\bz,t)$ whose $t$ variable tracks the length of a walk and whose first $n$ variables track the endpoint of a walk.  The recursive nature of a walk of length $k$ as a walk of length $k-1$ followed by a single step results in a functional equation satisfied by the generating function.  A procedure known as the \emph{kernel method} often allows one to obtain an expression for the generating function for the total number of walks in a model of a given length---or those ending in certain sets---as an explicit diagonal.  Although similar techniques appeared early in the study of random walks and statistical physics, the origin of the kernel method is often attributed to the 1968 textbook of Knuth~\cite{Knuth1997}.  Well-known examples of the kernel method which helped to modernize and develop it as a distinct strategy of proof include Bousquet-M{\'e}lou and Petkov{\v{s}}ek~\cite{Bousquet-MelouPetkovsek2000}, Banderier et al.~\cite{BanderierBousquet-MelouDeniseFlajoletGardyGouyou-Beauchamps2002}, Bousquet-M{\'e}lou~\cite{Bousquet-Melou2006}, and van Rensburg et al.~\cite{RensburgPrellbergRechnitzer2008}; see also Prodinger~\cite{Prodinger2003} for additional examples.

Knuth's early use of what would become the kernel method was applied to the ballot problem, which can be posed as the enumeration of one-dimensional lattice paths in the half-space $\mathbb{N} \subset \mathbb{Z}$ beginning and ending at the origin and taking the steps $\mS = \{-1,1\}$.  Knuth's approach was greatly generalized by Banderier and Flajolet~\cite{BanderierFlajolet2002}, who proved that the generating function for any lattice path model restricted to a half-space is algebraic, gave explicit representations of these generating functions, and determined asymptotics for such models.  Asymptotics for the number of \emph{excursions}, which are the number of walks beginning and ending at the origin, and walks with weighted steps, were also derived.

A natural next step is the study of two-dimensional lattice path models in a quadrant (or, in higher dimensions, lattice path models in an orthant).  Although the generating functions of models restricted to a half-space are always algebraic, the generating functions of models in a quadrant can exhibit a wide variety of behaviour and have thus become an object of great study.  Much of this work has focused on models with \emph{short} step sets $\mS$, which are those where $\mS \subset \{\pm1,0\}^2$. The class of models restricted to the quarter plane with short step sets already admit generating functions which can be rational, algebraic~\cite{Gessel1986}, (transcendental and) D-finite~\cite{Bousquet-Melou2002}, (non-D-finite but) differentially algebraic\footnote{A power series $F(\bz)$ is \emph{differentially algebraic} if there exists a multivariate polynomial $P$ such that $F$ and some finite set of its derivatives $F',\dots,F^{(k)}$ satisfy $P(z,F',\dots,F^{(k)})=0$; a power series which is not differentially algebraic is called \emph{hypertranscendental}.  The first result exhibiting a lattice path model in a quadrant with non-D-finite generating function was given by Bousquet-Mélou and Petkov{\v{s}}ek~\cite{Bousquet-MelouPetkovsek2003}, although the model they considered starts at the point $(1,1)$ and has non-short step set $\mS = \{(-2,1),(1,-2)\}$.}~\cite{BernardiBousquet-MelouRaschel2016}, and hypertranscendental~\cite{DreyfusHardouinRoquesSinger2017}.  

The systematic enumeration of such models was begun by Bousquet-Mélou~\cite{Bousquet-Melou2002}, following probabilistic results of Fayolle and Iasnogorodski~\cite{FayolleIasnogorodski1979} and Fayolle et al.~\cite{FayolleIasnogorodskiMalyshev1999}, and greatly developed in work of Bousquet-Mélou and Mishna~\cite{Bousquet-MelouMishna2010}.  The work of Bousquet-Mélou and Mishna showed that there are 79 non-isomorphic quarter plane models with short steps which are not equivalent to models restricted to a half-space.  The last several decades have seen progress made on the study of these models using tools from the theory of algebraic curves, formal power series approaches to discrete differential equations, probability theory, computer algebra, boundary value problems, potential theory, differential Galois theory, the study of hypergeometric functions, and several branches of complex analysis (further details on these approaches are given in Chapter~\ref{ch:KernelMethod}).

The enumeration of lattice path models in a quadrant thus lies at the boundary of what is currently solvable and what is still open.  Around the same time
%\footnote{Although the publication date of Bousquet-M{\'e}lou and Mishna~\cite{Bousquet-MelouMishna2010} is one year after that of Bostan and Kauers~\cite{BostanKauers2009}, the work of Bousquet-M{\'e}lou and Mishna was available online two years before publication and is cited by Bostan and Kauers.} 
as the work of Bousquet-M{\'e}lou and Mishna, Bostan and Kauers used computer algebra techniques to guess linear differential equations for these 79 models, finding likely differential equations for 23 of the models\footnote{It is conjectured, although still not fully proven, that the remaining 56 models have non-D-finite (univariate) generating functions; the generating functions for the number of walks returning to the origin are non-D-finite, for instance.  See Chapter~\ref{ch:KernelMethod} for more details.} and guessing asymptotics which are displayed in Table~\ref{tab:shortQPasm} of Chapter~\ref{ch:KernelMethod}.  These guessed differential equations have now been proven, but problems related to the effectiveness of D-finite coefficient asymptotics have led to difficulties proving the guessed asymptotics.  For example, the dominant asymptotics of such walks are given by a finite sum of terms of the form $a_n = Cn^{\alpha}\rho^n$ for algebraic constants $C,\alpha,$ and $\rho$.  To determine each leading constant $C$, Bostan and Kauers determined the possible values of $\alpha$ and $\rho$ from a guessed differential equation, computationally generated the number of walks up to length ten thousand, and used numerical approximations of $C$ obtained from this data to guess its minimal polynomial.  We use the methods of ACSV to prove asymptotics of these models in Chapter~\ref{ch:QuadrantLattice}.

\section{Original Contributions}

\subsection{Effective Asymptotics}
Chapter~\ref{ch:EffectiveACSV} contains the first rigorous effective algorithms and complexity results for rational diagonal asymptotics in any dimension, under assumptions which are often satisfied in applications.  This chapter develops a collection of symbolic-numeric results from polynomial system solving and related areas which are then combined with results from the theory of ACSV. A multivariate rational function $F(\bz)$ is called \emph{combinatorial} if all coefficients in its power series expansion are non-negative, and a property of rational functions is said to hold \emph{generically} if it holds for all rational functions except those whose coefficients satisfy a polynomial relation depending only on the degrees of the numerator and denominator of the rational function.  The main result of this chapter, stated exactly in Theorem~\ref{thm:EffectiveCombAsm}, is the following.

\begin{theorem*}
Let $F(\bz) \in \mathbb{Z}(z_1,\dots,z_n)$ be a rational function with numerator and denominator of degrees at most $d$ and coefficients of absolute value at most $2^h$.  Assume that $F$ is combinatorial, has a minimal critical point, and satisfies additional restrictions\footnote{See Section~\ref{sec:assumptions_comb} of Chapter~\ref{ch:EffectiveACSV}.} which hold generically.  Then there exists a probabilistic algorithm computing dominant asymptotics of the diagonal sequence in $\tilde{O}(hd^{4n+5})$ bit operations\footnote{We write $f = \tilde{O}(g)$ when $f=O(g\log^kg)$ for some $k\ge0$; see Section~\ref{sec:complexity_measure} of Chapter~\ref{ch:EffectiveACSV} for more information on our complexity model and notation.}.  The algorithm returns three rational functions $A,B,C \in \mathbb{Z}(u)$, a square-free polynomial $P \in \mathbb{Z}[u]$ and a list $U$ of roots of $P(u)$ (specified by isolating regions) such that
\[ f_{k,\dots,k} = (2\pi)^{(1-n)/2}\left(\sum_{u \in U} A(u)\sqrt{B(u)} \cdot C(u)^k \right)k^{(1-n)/2}\left(1 + O\left(\frac{1}{k}\right) \right). \]
The values of $A(u), B(u),$ and $C(u)$ can be determined to precision $2^{-\kappa}$ at all elements of $U$ in $\tilde{O}(d^{n+1}\kappa + hd^{3n+3})$ bit operations.
\end{theorem*}
A high-level description of this algorithm is given in Algorithm~\ref{alg:EffectiveCombAsm}, which originally appeared in an article of Melczer and Salvy~\cite{MelczerSalvy2016}.  A preliminary implementation of this work\footnote{Available at \websiteurl.} has been developed which can rigorously prove asymptotic results contained in recent publications, and has already been used by other researchers~\cite{Pantone2017}.  

The strongest assumption we require is that $F(\bz)$ is combinatorial, which greatly helps to determine when critical points are minimal.  Theorem~\ref{thm:GenEffMinCrit} describes how to determine minimal critical points without this assumption, and appears for the first time in this work.  In order to prove minimality in the non-combinatorial case we use a critical point method inspired by techniques from real algebraic geometry.  

\subsection{Lattice Path Asymptotics}
This thesis contains several new applications of the theory of ACSV to the study of lattice path enumeration.

\paragraph{Highly Symmetric Models} 
Chapter~\ref{ch:SymmetricWalks}, which is based on an article of Melczer and Mishna~\cite{MelczerMishna2016}, describes how to enumerate lattice path models restricted to an orthant in any dimension whose step sets are symmetric over every axis.  Our work establishes strong asymptotic results and provides an extended application illustrating the methods of ACSV in the smooth case.  Theorem~\ref{thm:symAsm} gives an explicit formula for dominant asymptotics of the number of walks from quantities which can be immediately read off of a model's step set.  Theorem~\ref{thm:symAsmEx} gives an asymptotic bound on the number of walks returning to the origin, and the number of walks returning to any fixed set of boundary hyperplanes.  Some of this work was originally contained in the Masters thesis of the author~\cite{Melczer2014}, but Theorem~\ref{thm:symAsmEx}, extensions to models with symmetrically weighted step sets, and applications to the connection problem for D-finite functions were completed after that publication.

\paragraph{Lattice Walks in a Quadrant} 
In Chapter~\ref{ch:QuadrantLattice}, which is based on an article of Melczer and Wilson~\cite{MelczerWilson2016}, we give the first full proof of the conjectures of Bostan and Kauers~\cite{BostanKauers2009} for asymptotics of lattice path models restricted to a quadrant.  Our approach shows the link between combinatorial properties of a lattice path model, such as symmetries in its set of steps, and features of its asymptotics. In addition, some asymptotics for the number of walks which begin at the origin and return to the origin, the $x$-axis, or the $y$-axis are derived, and previously observed links between these quantities are explained analytically through a multivariate singularity analysis.  The results of this chapter require the use of ACSV when there are minimal critical points where the singular variety $\mV$ is not locally a manifold.

\paragraph{Centrally Weighted Models} 
Chapter~\ref{ch:WeightedWalks}, which is based on an article of Courtiel, Melczer, Mishna, and Raschel~\cite{CourtielMelczerMishnaRaschel2016}, considers aspects of weighted walks restricted to orthants.  The first half of the chapter explores asymptotics of a particular model, known as the Gouyou-Beauchamps model, under weightings of its step set which allow for a parametrized rational diagonal expression.  When the weights satisfy certain algebraic equations the geometry of the singular set changes, resulting in sharp \emph{phase transitions} in asymptotics as the weights vary continuously. Theorem~\ref{thm:GB_main_asymptotic_result} determines the asymptotics for the number of weighted walks in a model as a function of the weights.   

In order to determine such a parametrized diagonal expression, the step set under consideration must be weighted so that the weight of any path between two fixed points depends only on its length.  We call such a weighting \emph{central}, and the second part of this chapter characterizes the central weightings of any $n$-dimensional model restricted to the orthant $\mathbb{N}^n \subset \mathbb{Z}^n$.  Among other results, this allows one to associate weighted lattice path models with D-finite generating functions to a single unweighted model with a D-finite generating function.  Kauers and Yatchak~\cite{KauersYatchak2015} computationally investigated weighted lattice path models with short steps in the quarter plane, and found what they conjectured to be a finite list of families containing all models with (weighted) D-finite generating functions; all but one of these families can be characterized using our work.

Finally, a connection between these parametrized asymptotics and recent conjectures of Garbit, Mustapha, and Raschel~\cite{GarbitMustaphaRaschel2017} on the exit times of random walks in cones is discussed.  These conjectures, which are very general and apply to lattice path models with non-D-finite step sets, hint at future possibilities for lattice path enumeration using multivariate singularity analyses.  We prove these conjectures for all centrally weighted two-dimensional models whose underlying set of steps is symmetric over both axes, a result presented here for the first time.

\subsection{Thesis Publications}
The original research presented in this thesis is contained in the following publications.

\begin{enumerate}
    \item \textsl{Asymptotic lattice path enumeration using diagonals.} \\
    S. Melczer and M. Mishna.
    Algorithmica, Volume 75(4), 782-811, 2016.\\
    \url{http://dx.doi.org/10.1007/s00453-015-0063-1}\\
    \url{http://arxiv.org/abs/1402.1230}

	\item\textsl{Symbolic-Numeric Tools for Analytic Combinatorics in Several Variables.} \\
	S. Melczer and B. Salvy.
	Proceedings of the ACM on ISSAC 2016, 333-340, 2016.\\
	\url{http://dx.doi.org/10.1145/2930889.2930913}\\
	\url{http://arxiv.org/abs/1605.00402}

	\item\textsl{Asymptotics of lattice walks via analytic combinatorics in several variables.} \\
	S. Melczer and M. C. Wilson.
	Proceedings of FPSAC 2016, DMTCS proc. 863-874, 2016.\\
	\url{http://fpsac2016.sciencesconf.org/114341}\\
	\url{http://arxiv.org/abs/1511.02527}

	\item \textsl{Weighted Lattice Walks and Universality Classes.} \\
  	J. Courtiel, S. Melczer, M. Mishna, and K. Raschel.
  	Accepted to Journal of Combinatorial Theory, Series A, June 2017. \\
  	\url{http://arxiv.org/abs/1609.05839}
\end{enumerate}

\subsection{Additional Publications During this Thesis}

In addition to the above works, two other papers of the author were published during this doctoral program.   The research contained in these works was completed after the author's Masters thesis but before the start of this doctoral program, and because these papers focus on exact enumeration instead of asymptotics we simply summarize them here.

\begin{enumerate}\itemsep=0.5em
  \item \textsl{On 3-dimensional lattice walks confined to the positive octant.}
    A. Bostan, M. Bousquet-M\'elou, M. Kauers, and S. Melczer.
    Annals of Combinatorics, Volume 20(4), 661--704, 2016.\\
    \url{http://dx.doi.org/10.1007/s00026-016-0328-7}\\
    \url{http://arxiv.org/abs/1409.3669}

    \item \textsl{Tableau sequences, open diagrams, and Baxter families.}
    S. Burrill, J. Courtiel, E.~Fusy, S. Melczer, M. Mishna.
    European Journal of Combinatorics, Volume 58, 144-165, 2016.\\
    \url{http://dx.doi.org/10.1016/j.ejc.2016.05.011}\\
    \url{http://arxiv.org/abs/1506.03544}
\end{enumerate}

The first paper, Bostan et al.~\cite{BostanBousquet-MelouKauersMelczer2016}, began the systematic study of three-dimensional lattice path models with short step sets $\mS \subset \{\pm1,0\}^3$ which are restricted to an octant, with a focus on determining which models have D-finite generating functions.  Although there are only 79 non-isomorphic models with short steps in two dimensions, in three dimensions there are $11,074,225$ step sets of interest. This work studied the $35,548$ models with at most 6 steps, experimentally trying to determine which admitted D-finite generating functions and then verifying those guesses rigorously.  In addition to applications of the kernel method, to prove D-finiteness or algebraicity of the generating functions which arise we identified models admitting a special \emph{Hadamard decomposition}, which can be reduced to lattice path models in lower dimensions.  Rigorous computer algebraic proofs of algebraicity and transcendental D-finiteness of several generating functions were also given.  

The kernel method in the quadrant and octant usually works by associating to each lattice path model a finite group of transformations.  In the two-dimensional case, this group is finite whenever the generating function of a model is D-finite, and when the group is infinite the associated generating function appears to be non-D-finite.  Our work in three dimensions, however, found 19 models with finite groups whose generating functions appear to be non-D-finite.  The nature of these generating functions is still unknown, despite interest from researchers after these results were announced.  Bacher et al.~\cite{BacherKauersYatchak2016} later performed additional computations to experimentally determine octant models with larger than 6 steps which admit D-finite generating functions, and Berthomieu and Faugère~\cite{BerthomieuFaugere2016} applied fast Gröbner basis techniques to generate relations satisfied by the multivariate sequences tracking length and endpoint for some models contained in our work.
\bigskip

The second paper, Burrill et al.~\cite{BurrillCourtielFusyMelczerMishna2016}, studied connections between walks on Young's lattice of integer partitions, certain sequences of Young tableaux, and combinatorial objects known as arc diagrams.  The main result of that work gives a bijection between standard Young tableaux of bounded height and walks on Young's lattice starting at the empty partition, ending in a row shape, and visiting only partitions of bounded height.  As a corollary, the generating function for the number of Young tableaux of bounded height is given as an explicit rational diagonal.  A new combinatorial family enumerated by the Baxter numbers is also described.  The interested reader is referred to that work for more information.

\section{Thesis Organization}

This thesis is divided into four parts, with Part~\ref{part:Background} covering additional background and motivation for our work on rational diagonal asymptotics, Part~\ref{part:SmoothACSV} covering the theory and applications of ACSV in the smooth case, Part~\ref{part:NonSmoothACSV} discussing the theory and applications of ACSV for some non-smooth cases, and Part~\ref{part:Conclusion} concluding and summarizing the thesis. A detailed chapter breakdown, not including this introduction, is as follows:
\begin{itemize}
	\item[] Chapter~\ref{ch:FrenchSummary} contains a French summary of the results contained in this thesis.
	\item[] Chapter~\ref{ch:Background} contains a detailed background on generating functions and coefficient asymptotics.  After describing the basic principles of analytic combinatorics, the classes of rational, algebraic, and D-finite power series are detailed, including results on coefficient asymptotics and the complexity of determining coefficients exactly.  This is followed by the introduction of multivariate rational diagonals, as well as results on formal and convergent Laurent series expansions, amoebas of Laurent polynomials, and multivariate series sub-extractions which will be useful in later chapters.
	\item[] Chapter~\ref{ch:KernelMethod} contains a presentation of the kernel method for lattice path enumeration.  Beginning with the easy case of unrestricted lattice path models, the mechanics of the kernel method are built up for one-dimensional walks restricted to a half-space and two-dimensional walks restricted to a quadrant.  After describing this incredibly effective machinery, the current state of results enumerating lattice paths in a quadrant are discussed and the asymptotic conjectures of Bostan and Kauers~\cite{BostanKauers2009} are introduced.  
	\item[] Chapter~\ref{ch:OtherSources} describes several domains of mathematics and the sciences where rational diagonals arise.  In addition to showing the importance of rational diagonals, the examples discussed in this chapter are used to illustrate the methods of ACSV in later chapters.  
	\item[] Chapter~\ref{ch:SmoothACSV} describes the basics of analytic combinatorics in several variables, and shows how to derive asymptotics for many rational functions which admit singular varieties that are complex manifolds.  After an extended example which concretely illustrates the methods of ACSV in the smooth case from start to finish, the general theory is developed.  Many examples are given and general strategies for applying the tools of analytic combinatorics are demonstrated.  
	\item[] Chapter~\ref{ch:SymmetricWalks} contains our results on lattice path models with symmetric step sets.
	\item[] Chapter~\ref{ch:EffectiveACSV} contains our results on effective methods for analytic combinatorics in several variables. 
	\item[] Chapter~\ref{ch:NonSmoothACSV} describes the theory of analytic combinatorics in several variables when the singular variety is no longer a manifold.  After an extended example illustrating how the theory can be applied to transverse multiple points, the background necessary to use the methods of ACSV in this more complicated case is described and asymptotic results are given. 
	\item[] Chapter~\ref{ch:QuadrantLattice} proves the conjectured asymptotics of Bostan and Kauers~\cite{BostanKauers2009} for D-finite lattice path problems in a quadrant.
	\item[] Chapter~\ref{ch:WeightedWalks} contains our results on families of weighted lattice path models.
	\item[] Chapter~\ref{ch:Summary} concludes the thesis.
\end{itemize}

Some of the background material in Part~\ref{part:Background} was adapted from the author's Masters thesis~\cite{Melczer2014}.

%%%%%%%%%%%%%%%%%%%%%%%%%%%%%%%%%%%%
% CHAPTER 2
%%%%%%%%%%%%%%%%%%%%%%%%%%%%%%%%%%%%
\chapter{R{\'e}sum{\'e} en Fran{\c{c}}ais}
\label{ch:FrenchSummary}

\setlength{\epigraphwidth}{3.7in}
\epigraph{Le génie n'est que l'enfance retrouvée à volonté, l'enfance douée maintenant, pour s'exprimer, d'organes virils et de l'esprit analytique qui lui permet d'ordonner la somme de matériaux involontairement amassée.}{Charles Baudelaire, \emph{Le Peintre de la vie moderne}}

\epigraph{Géomètre de premier rang, Laplace ne tarda pas à se montrer administrateur plus que médiocre; dès son premier travail nous reconnûmes que nous nous étions trompé. Laplace ne saisissait aucune question sous son véritable point de vue: il cherchait des subtilités partout, n'avait que des idées problématiques, et portait enfin l'esprit des `infiniment petits' jusque dans l'administration.}{Napoléon Bonaparte, \emph{Mémoires de Napoléon Bonaparte}}

Une méthodologie extrêmement utile dans plusieurs domaines de la combinatoire a été l'adoption de techniques analytiques dans l'étude de l'asymptotique en combinatoire énumérative et en probabilités. Étant donnée une suite $(f_n)_{n \geq 0}$, la \emph{fonction génératrice} associée à la suite est la série formelle \[ F(z) = \sum_{n \geq 0}f_n z^n = f_0 + f_1z + f_2z^2 + \cdots. \]
Bien que la fonction génératrice soit \emph{a priori} un objet formel, dans de nombreuses applications (par exemple, quand $f_n$ est n'importe quelle suite qui croît au maximum exponentiellement) la série $ F(z)$ définit une fonction analytique dans un voisinage de l'origine. Il existe une large gamme de résultats, remontant aux XVIIIe et XIXe siècles, reliant le comportement analytique d'une fonction proche de ses singularités et l'asymptotique des coefficients de sa série de Taylor.

Les résultats dans ce domaine, appelés théorèmes de transfert, sont puissants et largement applicables, l'universalité de nombreuses propriétés des fonctions analytiques permettant souvent l'automatisation des analyses asymptotiques. Par exemple, la détermination asymptotique des coefficients des fonctions algébriques est effective, en ce sens qu'il existe un algorithme qui prend un nombre fini de termes initiaux d'une suite combinatoire $(f_n)$ ayant une fonction génératrice algébrique $F(z)$, avec le polynôme minimal de $F(z)$, et renvoie le comportement asymptotique dominant de $f_n$.

Une autre propriété qui se présente souvent dans les applications combinatoires est celle de la D-finitude. Une fonction analytique $ F (z)$ est \emph{D-finie} lorsqu'elle satisfait une équation différentielle linéaire avec des coefficients polynomiaux. Contrairement au cas des fonctions algébriques, dans lesquelles le coefficient asymptotique est totalement effectif, on ne sait pas encore comment dériver des asymptotiques de coefficients pour une fonction D-finie à partir d'une liste de coefficients initiaux et d'une équation différentielle annulat la fonction génératrice.

Une grande partie de ce projet de thèse aborde le problème de la détermination des coefficients asymptotiques pour une sous-classe de fonctions D-finies appelées \emph{diagonales rationnelles multivariées}. Soit la fonction de $n$ variables $F(\bz)$ avec un développement à l'origine en série
\[ F(\bz) = \sum_{\bi \in \mathbb{N}^n} c_{\bi} \bz^{\bi} = \sum_{i_1,\dots,i_n \in \mathbb{N}} c_{i_1,\dots,i_n} z_1^{i_1} \cdots z_n^{i_n}, \]
alors la \emph{diagonale} de $ F (\ bz) $ est la fonction univariée obtenue en prenant les coefficients où tous les exposants sont égaux, 
\[ (\Delta F)(z) := \sum_{k \geq 0} c_{k,k,\dots,k} z^k. \]

La diagonale de toute fonction rationnelle est D-finie~\cite{Christol1988,Lipshitz1989}, et toute fonction algébrique est la diagonale d'une fonction rationnelle bivariée~\cite{DenefLipshitz1987}. Au cours de la dernière décennie, une série de résultats de Pemantle, Wilson et de ses collaborateurs, recueillis dans leur récent ouvrage~\cite{PemantleWilson2013}, a utilisé des résultats de l'analyse complexe en plusieurs variables pour établir les bases des méthodes d'asymptotique pour les diagonales de fractions rationnelles multivariées. C'est ce que l'on appelle l'étude de la combinatoire analytique en plusieurs variables, que nous abrégeons souvent sous le nom <<ACSV>>.

Cette thèse donne les premiers algorithmes et résultats de complexité entièrement rigoureux pour déterminer les asymptotiques des suites de coefficients de diagonales non algébriques de fractions rationnelles dans des conditions qui sont largement satisfaites. De plus, nous examinons plusieurs applications de l'ACSV aux problèmes qui se posent dans l'énumération des marches dans des réseaux. Une des motivations de cette étude était un ensemble de conjectures de Bostan et Kauers~\cite{BostanKauers2009}, qui ont trouvé des équations différentielles linéaires annulant les fonctions génératrices de certaines suites d'énumération de marches sur des réseaux, mais n'ont pas été en mesure de prouver les comportements asymptotiques. En utilisant les résultats de l'ACSV, nous pouvons démontrer ces asymptotiques pour la première fois, expliquer le comportement asymptotique observé de manière analytique et étudier des classes beaucoup plus générales de problèmes de marches sur des réseaux. L'énumération des marches sur des réseaux fournit également une famille riche de problèmes permettant d'illustrer la théorie de l'ACSV, fournissant une variété d'exemples concrets pour ceux qui veulent apprendre ses méthodes. Cette thèse d'adresse à plusieurs publics (avec potentiellement des intersections): les mathématiciens intéressés par le comportement de fonctions satisfaisant certaines équations algébriques, différentielles ou fonctionnelles; les combinatoriciens intéressés à apprendre la nouvelle théorie de la combinatoire analytique en plusieurs variables; les informaticiens intéressés par de nouvelles applications du calcul formel et de la géométrie algébrique réelle; et des chercheurs d'une variété de domaines avec un intérêt pour l'énumération de marches sur des réseaux.

\section*{Contributions originales}

\subsection*{Asymptotique effective}
Le chapitre~\ref{ch:EffectiveACSV} contient les premiers algorithmes efficaces rigoureux et les résultats de complexité correspondants pour le calcul du comportement asymptotique des diagonales de fractions rationnelles en dimension arbitraire, sous des hypothèses qui sont souvent satisfaites dans les applications. Ce chapitre développe une collection de résultats symboliques-numériques issus de la résolution de systèmes polynomiaux et des domaines connexes qui sont ensuite combinés avec les résultats de la théorie de l'ACSV. Nos principaux résultats sur l'asymptotique effective et sa complexité sont donnés dans les théorèmes~\ref{thm:EffectiveCombAsm} et~\ref{thm:GenEffMinCrit}, ainsi que l'algorithme~\ref{alg:EffectiveCombAsm}. Ce travail a d'abord paru dans un article de Melczer et Salvy~\cite{MelczerSalvy2016}.

Une mise en œuvre préliminaire de ce travail\footnote{Disponible à \websiteurl.} a été implantée et peut être utilisée pour démontrer rigoureusement les résultats asymptotiques contenus dans des publications récentes, et a déjà été utilisée par d'autres chercheurs~\cite{Pantone2017}.

\subsection*{Asymptotique des marches}
Informellement, un modèle de marche sur un réseau est une classe combinatoire qui encode le nombre de manières de «se déplacer» sur un réseau sous certaines contraintes. Ici, nous nous concentrons sur les modèles de marche sur un réseau qui commencent à l'origine, restent dans $\mathbb {N}^n \subset \mathbb{Z}^n$ et prennent des étapes dans un ensemble fini $\mS \subset \{\pm1,0\}^n$, pour un certain entier naturel fixe $n$. Une technique appelée <<méthode du noyau>> permet de déterminer des expressions en diagonales de fractions rationnelles pour les fonctions génératrices de beaucoup de tels modèles, qui sont ensuite analysées pour déterminer les asymptotiques.

\paragraph{Modèles hautement symétriques}
Le chapitre~\ref{ch:SymmetricWalks}, qui est basé sur un article de Melczer et Mishna~\cite{MelczerMishna2016}, décrit comment énumérer les modèles de marches sur un réseau, restreintes à un orthant, en dimension arbitraire, et dont les jeux de pas sont symétriques sur chaque axe. Notre travail établit l'asymptotique et fournit une application illustrant les méthodes de l'ACSV de manière étendue. Le théorème~\ref{thm:symAsm} donne une formule explicite pour l'asymptotique dominante d'un modèle à partir de quantités qui peuvent être immédiatement lues sur l'ensemble de pas du modèle. Le théorème~\ref{thm:symAsmEx} donne une borne asymptotique sur le nombre de marches revenant à l'origine et le nombre de marches revenant à n'importe quel hyperplan frontière. 

\paragraph{Lattice Walks dans un quadrant}
Dans le chapitre~\ref{ch:QuadrantLattice}, qui est fondé sur un article de Melczer et Wilson, nous donnons la première preuve complète des conjectures de Bostan et Kauers~\cite{BostanKauers2009} pour les asymptotiques de modèles de marches sur un réseau restreintes à un quadrant. Notre approche montre le lien entre les propriétés combinatoires d'un modèle de marches sur un réseau, telles que les symétries de son ensemble de pas, et les caractéristiques de ses asymptotiques. De plus, les asymptotiques pour le nombre de marches qui commencent à l'origine et retournent à l'origine, sur l'axe des $x$ ou des $y$ sont déduits, et les liens précédemment observés entre ces quantités sont expliqués analytiquement à travers une analyse analyse de singularité multivariée.

\paragraph{Modèles à pondération centrale}
Le chapitre~\ref{ch:WeightedWalks}, qui est basé sur un article de Courtiel, Melczer, Mishna et Raschel~\cite{CourtielMelczerMishnaRaschel2016}, considère des marches pondérées restreintes aux orthants. La première moitié du chapitre explore l'asymptotique d'un modèle particulier, connu sous le nom de modèle Gouyou-Beauchamps, sous les pondérations de son ensemble de pas qui permettent une expression comme diagonale rationnelle paramétrée. Lorsque les poids satisfont certaines équations algébriques, la géométrie de l'ensemble singulier change, ce qui entraîne des \emph{transitions de phase} nettes en asymptotiques, car les poids varient en continu. Le théorème~\ref{thm:GB_main_asymptotic_result} détermine les asymptotiques pour le nombre de marches pondérées dans un modèle en fonction des poids. Pour déterminer une telle expression diagonale paramétrée, l'étape considérée doit être pondérée de sorte que le poids de n'importe quel trajet entre deux points fixes dépend uniquement de sa longueur. Nous appelons une telle pondération \emph{centrale} et la deuxième partie de ce chapitre caractérise les pondérations centrales de tout modèle $n$-dimensionnel restreint à l'orthant $ \mathbb{N}^n \subset \mathbb {Z}^n$. Enfin, un lien entre ces asymptotiques paramétrées et les conjectures récentes de Garbit, Mustapha et Raschel~\cite{GarbitMustaphaRaschel2017} sur les temps de sortie de randonnées aléatoires en cônes est discuté.

\section*{Organisation de la thèse}

Cette thèse est divisée en quatre parties, avec la Partie~\ref{part:Background} couvrant les définitions et propriétés de base, et la motivation de notre travail sur l'asymptotique des diagonales rationnelles, la Partie~\ref{part:SmoothACSV} couvrant la théorie et les applications de l'ACSV lorsque l'ensemble des singularités de la fonction rationnelle $F(\bz)$ forme une variété lisse, la Partie~\ref{part:NonSmoothACSV} discutant la théorie et les applications de l'ACSV dans des cas plus généraux, et la Partie~\ref{part:Conclusion} concluant et résumant la thèse. Une ventilation détaillée des chapitres est la suivante:
\begin{itemize}
  \item[] Le chapitre~\ref{ch:Background} contient un historique détaillé sur les fonctions génératrices et les asymptotiques de leurs coefficients. Après avoir décrit les principes de base de la combinatoire analytique, on décrit les classes des séries rationnelles, algébriques et D-finies, y compris les résultats sur l'asymptotique des coefficients et la difficulté de la détermination exacte des coefficients de ces asymptotiques. Ensuite, on introduit les diagonales rationnelles multivariées, ainsi que des résultats sur les développements formels et convergents en série de Laurent, sur les amibes de polynômes de Laurent et les extractions de séries multivariées qui seront utiles dans les chapitres suivants. 
  \item[] Le chapitre~\ref{ch:KernelMethod} contient une présentation de la méthode du noyau pour l'énumération des marches sur des réseaux. En commençant par le cas facile des modèles de marches dans des réseaus sans restriction, la mécanique de la méthode du noyau est construite pour des marches unidimensionnelles limitées à un demi-espace et des marches bidimensionnelles restreintes à un quadrant. Après avoir décrit ce mécanisme incroyablement efficace, l'état actuel des résultats énumérant les marches dans des réseaux dans un quadrant est discuté et les conjectures asymptotiques de Bostan et Kauers~\cite{BostanKauers2009} sont introduites. 
  \item[] Le chapitre~\ref{ch:OtherSources} décrit plusieurs domaines des mathématiques et des sciences où apparaissent des diagonales rationnelles. En plus de montrer l'importance des diagonales rationnelles, les exemples présentés dans ce chapitre sont utilisés pour illustrer les méthodes de l'ACSV dans les chapitres suivants.
  \item[] Le chapitre~\ref{ch:SmoothACSV} décrit les bases de la combinatoire analytique en plusieurs variables et montre comment dériver l'asymptotique pour de nombreuses fonctions rationnelles dont les ensembles singuliers forment des variétés lisses. Après un exemple étendu qui illustre concrètement les méthodes d'ACSV dans ce cas du début à la fin, la théorie générale est développée. De nombreux exemples sont donnés et des stratégies générales d'application des outils de la combinatoire analytique sont présentées. 
  \item[] Le chapitre~\ref{ch:SymmetricWalks} contient nos résultats sur des modèles de marches sur des réseaux avec des jeux de pas symétriques. 
  \item[] Le chapitre~\ref{ch:EffectiveACSV} contient nos résultats sur les méthodes effectives de combinatoire analytique en plusieurs variables.
  \item[] Le chapitre~\ref{ch:NonSmoothACSV} décrit la théorie de la combinatoire analytique en plusieurs variables lorsque la variété singulière n'est plus lisse. Après un exemple étendu, le contexte nécessaire pour utiliser les méthodes de l'ACSV dans ce cas plus compliqué est décrit et des résultats asymptotiques sont donnés. 
  \item[] Le chapitre~\ref{ch:QuadrantLattice} prouve les asymptotiques conjecturées de Bostan et Kauers pour les problèmes de marches dans des réseaux D-fini dans un quadrant. 
  \item[] Le chapitre~\ref{ch:WeightedWalks} contient nos résultats sur les familles de modèles pondérés de marches sur des réseaux.
  \item[] Le chapitre~\ref{ch:Summary} conclut la thèse.
\end{itemize}

%%%%%%%%%%%%%%%%%%%%%%%%%%%%%%%%%%%%
% CHAPTER 3
%%%%%%%%%%%%%%%%%%%%%%%%%%%%%%%%%%%%
\chapter{Background on Generating Functions and Asymptotics}
\label{ch:Background}

\vspace{-0.2in}

\setlength{\epigraphwidth}{3.7in}
\epigraph{Since there is a great conformity between the Operations in Species, and the same Operations in common Numbers \dots I cannot but wonder that no body has thought of accommodating the lately discover'd Doctrine of Decimal Fractions in like manner to Species \dots especially since it might have open'd a way to more abstruse Discoveries.}{Sir Issac Newton, \emph{The Method of Fluxions and Infinite Series}}

\vspace{-0.2in}

\epigraph{The problem about finding the middle coefficient in a very large power of the binomial had been solved by De Moivre some years before I considered it: And it is probable that to this very day I would not have thought about it, unless that most esteemed man, Mr Alex. Cuming, had not stated that he very much doubted that it could be solved by Newton's Method of Differences.\footnotemark}{James Stirling, \emph{Methodus Differentialis}}
\footnotetext{Translated from the Latin by Ian Tweddle~\cite{Tweddle2003}.}

Given a ring $R$ we let $R[[z]]$ denote the usual ring of formal power series with coefficients in $R$, and $R[[z_1,\dots,z_n]]$ denote the ring of multivariate power series in the variables $z_1,\dots,z_n$ (see Stanley~\cite{Stanley2012} or Lang~\cite[Section IV.9]{Lang2002} for background on formal power series and their properties).  Throughout this document we use multi-index notation, with bold letters denoting multivariate quantities, so that for $\bz = (z_1,\dots,z_n) \in \mathbb{C}^n$ we define \glsadd{bz} \glsadd{bzi} \glsadd{bzhat}
\[\bz^{\bi} := z_1^{i_1} \cdots z_n^{i_n} \in \mathbb{C} \qquad \text{and} \qquad \bzhat = (z_1,\dots,z_{k-1},z_{k+1},\dots,z_n) \in \mathbb{C}^{n-1}. \]
For a power series $F(\bz) = \sum_{\bi \in \mathbb{N}^n}f_{\bi}\bz^{\bi}$, we use $[\bz^{\bi}]F(\bz)$ to denote the coefficient $f_{\bi}$.

A \emph{combinatorial class}\index{combinatorial class} is a countable set of objects $\mathcal{C}$, together with a size-function $|\cdot|:\mathcal{C}\rightarrow \mathbb{N}$ such that there are a finite number of objects of any given size (the inverse image of any natural number under $|\cdot|$ is finite).  The \emph{counting sequence} associated to $\mathcal{C}$ is the sequence $(c_n)_{n \geq 0}$ whose $n$th term is the number of objects in $\mathcal{C}$ of size $n$, $c_n = \#\{ x \in \mathcal{C} : |x|=n \}$, and the \emph{generating function} of $\mathcal{C}$ is the formal series
\begin{equation} C(z) = \sum_{n \geq 0}c_nz^n. \label{eq:genseries} \end{equation}

We begin with a description of the univariate study of analytic combinatorics, a beautiful theory that will serve as a guide to the more complicated multivariate version.  The now standard reference for this material is the compendium work of Flajolet and Sedgewick~\cite{FlajoletSedgewick2009}, in which most of the following results can be found in more detail.  When referring to a power series as analytic, we assume the ring $R$ from which it takes its coefficients is a subring of the complex numbers.

\section{The Basics of Analytic Combinatorics}

Dating back to their origins in the early eighteenth century work of de Moivre, generating functions have provided an invaluable formal framework for the manipulation of counting sequences.  In 1730, the year that de Moivre's first work~\cite{Moivre1730} devoted to generating functions was printed, Stirling~\cite{Stirling1730} published his own work combining various aspects of de Moivre's theory of series coefficients, Newton's `Method of Differences'~\cite{Newton1711}, and Taylor's work on series approximations~\cite{Taylor1715}.  Thus, for three centuries there has been work studying the link between generating functions and calculus on power series\footnote{Although there was less of a distinction between convergent and formal series at this time, according to the introduction of Hardy~\cite{Hardy1949} ``$\dots$ all the greatest mathematicians of the seventeenth and eighteenth centuries, however recklessly they may seem to have manipulated series, knew well enough whether the series which they used were convergent.''  Many of the modern notions of analyticity and convergence that are now standard come from Cauchy's work in the early nineteenth century (around this time Cauchy also developed the `method of majorants' for detecting situations where formal series solutions obtained from Newton's method of indeterminate coefficients represent analytic functions).  Chapter 3 of Bottazzini~\cite{Bottazzini1986} contains a deep account of the development of rigour in analysis during the time of Cauchy.}. For a detailed historical background on this early work see the introduction to Tweddle's recent translation~\cite{Stirling1730} of Stirling's Methodus Differentialis.

The most basic link between the analytic behaviour of a function and asymptotic behaviour of its coefficients is found in the following theorem.
\begin{proposition}[{Cauchy's Root Test~\cite[Theorem 1, Ch. VI.2]{Cauchy1821}}] 
\label{prop:expGrowth}
Suppose that $F(z) = \sum_{n\geq0}f_nz^n$ is a power series with non-zero finite radius of convergence $\rho$.  Then 
\[ \limsup_{n \rightarrow \infty} |f_n|^{1/n} = \rho^{-1}. \]
\end{proposition}
In combinatorial contexts the reason that the limsup is not a limit typically comes from periodicity restraints in the underlying enumeration problem. Note that an analytic function defined by $F(z)$ in a neighbourhood of the origin necessarily has a singularity\footnote{If $f(z)$ is an analytic function in the disk $|z|\leq R+\epsilon$ for some $R,\epsilon>0$ with power series coefficients $(f_n)$ at the origin, then bounds following from the Cauchy Integral Formula yield $|f_n| = O\left((R+\epsilon)^{-n} \right)$ and the ratio test then implies that the series $\sum_{n \geq0}f_nz^n$ converges in the disk $|z| \leq R$.  Thus, $f(z)$ must admit a singularity whose modulus equals the radius of convergence of $\sum_{n \geq0}f_nz^n$.} on the circle $|z|=\rho$.  In this way, we meet Flajolet and Sedgewick's~\cite{FlajoletSedgewick2009} first principle of coefficient asymptotics.
\vspace{0.1in}

\hspace{0.5in}\begin{minipage}{\dimexpr\textwidth-3cm}
\emph{(First Principle of Coefficient Asymptotics)} The \emph{location} of a function's singularities
dictates the exponential growth of its coefficients.
\end{minipage}
\vspace{0.1in}

From now on we abuse notation slightly and write $F(z)$ to mean both an analytic function at the origin and the power series expansion of that function at the origin.  When given a function $F(z)$ which is analytic at the origin, we sometimes refer to the coefficients of the power series expansion of $F(z)$ at 0 as the coefficients of $F(z)$.  The singularities of $F(z)$ with minimum modulus are called its \emph{dominant singularities}, and are the only ones which determine asymptotics up to dominant exponential growth.  

From Proposition~\ref{prop:expGrowth}, we see that if $F(z)$ has radius of convergence $\rho$ then $f_n$ has an asymptotic expansion 
\[ f_n = \rho^{-n}\cdot\theta(n) + O(\alpha^n),\] 
where $0 \leq \alpha < \rho^{-1}$ and $\theta(n)$ grows sub-exponentially.  In order to determine information about $\theta(n)$ a closer analysis is required.  The starting point is Cauchy's Integral Formula, which implies 
\begin{equation} f_n = \frac{1}{2\pi i}\int_{C}\frac{F(z)}{z^{n+1}}dz, \label{eq:CIF} \end{equation}
where $C$ is any positively oriented circle around the origin of radius less than the radius of convergence $\rho$.  When $F(z)$ admits a finite number of singularities, all of which are poles, then one can immediately characterize the sub-exponential factor $\theta(n)$ using complex analysis.

\begin{theorem}[{Flajolet and Sedgewick~\cite[Theorem IV.10]{FlajoletSedgewick2009}}] \label{thm:meroAsm}
Suppose that $F(z)$ is analytic on the circle $|z|=R$ and has a finite number of polar singularities $\sigma_1,\dots,\sigma_m$ in the disk $|z|<R$.  Then there exist polynomials $P_1(n),\dots,P_m(n)$ such that
\[ f_n = \sum_{j=1}^m P_j(n) \sigma_j^{-n} + O(R^{-n}). \]
The degree of $P_j$ is one less than the order of the pole of $F(z)$ at $z=\sigma_j$.
\end{theorem}

% \begin{proof}
% The only singularities of the integrand in Equation~\eqref{eq:CIF} occur at zero and $\sigma_1,\dots,\sigma_m$.  Thus, the Cauchy Residue Theorem implies
% \[ \frac{1}{2\pi i} \int_{|z|=R} \frac{F(z)}{z^{n+1}}dz = \sum_{\zeta \in \{0,\sigma_1,\dots,\sigma_n\}} \text{Res}\left( \frac{F(z)}{z^{n+1}}; z = \zeta \right), \]
% where we note that the residue at $z=0$ is $f_n$.  If $F(z)$ has a pole of order $r$ at $z=\sigma_j$ then in a neighbourhood of $\zeta$ one can write $F(z)=\frac{g(z)}{(z-\sigma_j)^r}$ with $g(z)$ analytic and non-zero at $\sigma_j$.  Standard results from complex analysis imply
% \[ \text{Res}\left( \frac{g(z)}{z^{n+1}(z-\sigma_j)^r} ; z=\sigma_j\right) = \frac{1}{(r-1)!} \lim_{z\rightarrow\sigma_j} \frac{d^{r-1}}{dz^{r-1}}\left(\frac{g(z)}{z^{n+1}}\right) = P_j(n) \cdot \sigma_j^{-n}, \]
% where $P_j(n)$ is a polynomial in $n$ of degree $r-1$.  Finally, integration bounds imply that
% \[ \left| \frac{1}{2\pi i} \int_{|z|=R} \frac{F(z)}{z^{n+1}}dz \right| \leq \frac{1}{2\pi} \int_{|z|=R} \left|\frac{F(z)}{z^{n+1}}\right|dz \leq \frac{M \cdot 2\pi}{2\pi R^{n+1}} = O\left(R^{-n}\right), \]
% for $M$ the maximum of $F(z)$ on $|z|=R$, finite as $F(z)$ is analytic for $|z|=R$.
% \end{proof}

Hence, for meromorphic functions, the order of the poles comprising the dominant singularities of $F(z)$ describe the sub-exponential growth $\theta(n)$.  This kind of result can be greatly generalized.  Flajolet and Odlyzko~\cite{FlajoletOdlyzko1990} coined the term \emph{singularity analysis} for the process of analyzing the local singular behaviour of a function at its (dominant) singularities and translating the result into asymptotic results on coefficients through the use of \emph{transfer theorems}.  These results include many cases with algebraic and logarithmic singularities, and typically require analyticity in \emph{delta domains} of the form $\Delta_{\zeta} = \{|z| < R : R > |\zeta|, \,\, |\arg(z-\zeta)|>\phi, \,\, z \neq \zeta\}$ for some $\phi \in (0,\pi/2)$, which look like a circle with a wedge removed (to account for branch cuts).  The strong connection between singular structure and full asymptotic behaviour motivates Flajolet and Sedgewick's second principle of coefficient asymptotics.
\vspace{0.1in}

\hspace{0.5in}\begin{minipage}{\dimexpr\textwidth-3cm}
\emph{(Second Principle of Coefficient Asymptotics)} The \emph{nature} of a function's singularities
determines the associated sub-exponential growth $\theta(n)$.
\end{minipage}
\vspace{0.1in}

In later chapters we deal mainly with multivariate rational functions, meaning that algebraic and logarithmic singularities will not arise.  However, in contrast to a univariate rational function in which the `nature' of its singularities is completely described by the orders of its isolated poles, a multivariate rational function can exhibit a large range of singular structure depending on the geometry of its (algebraic) set of singularities.  

When the power series coefficients of $F(z)$ are all non-negative, as is the case for generating functions of counting sequences, finding dominant singularities is simplified by the following result\footnote{Although commonly referred to as Pringsheim's Theorem, Hadamard~\cite{Hadamard1954} credits this result to Émile Borel.}. 

\begin{proposition}[Pringsheim's Theorem; {Flajolet and Sedgewick~\cite[Theorem IV.6]{FlajoletSedgewick2009}}] If $F(z)$ is represented at the origin by a series expansion that has non-negative coefficients and finite radius of convergence $\rho > 0$, then $z=\rho$ is a singularity of $F(z)$.  
\end{proposition}
\noindent
In particular, $F(z)$ has a dominant singularity which is real and positive (there may be other dominant singularities of the same modulus).

Before describing the multivariate case we look in detail at determining asymptotics for several classes of generating functions.

%%%%%%%%%%%%%%%%%%%%%%%%%%%
% Rational Power Series
%%%%%%%%%%%%%%%%%%%%%%%%%%%
\section{Rational Power Series}
\label{sec:RatPS}

We begin with a study of univariate rational functions and their coefficients.

\subsubsection{Coefficient Properties} 
The types of sequences which arise as coefficient sequences of rational functions have a nice characterization.  Given a natural number $r \in \mathbb{N}$, we say that a sequence $(f_n)_{n\geq0}$ of elements in a field $K$ is a \emph{linear recurrence of order $r$ with constant coefficients} over $K$ if there exist constants $a_0,\dots,a_{r-1} \in K$ with $a_0\neq0$ such that for all $n \geq 0$,
\begin{equation} f_{n+r} = a_{r-1}f_{n-r-1} + a_{r-2}f_{n-r-2} + \cdots + a_0f_n; \label{eq:clinrec} \end{equation}
such a sequence is clearly determined by its first $r-1$ terms $f_0,\dots,f_{r-1}$.  In fact, de Moivre used generating functions in his seminal work~\cite{Moivre1730} to solve linear recurrence relations with constant coefficients. The following result follows from basic generating function manipulations.

\begin{theorem}
Suppose that $(f_n)_{n\geq0}$ is a linear recurrence relation with constant coefficients satisfying Equation~\eqref{eq:clinrec} above.  Then the generating function $F(z) = \sum_{n \geq 0} f_nz^n$ is a rational function of the form
\begin{equation} F(z) = \frac{A(z)}{1- (a_{r-1}z + a_{r-2}z^2 + \cdots + a_0z^r)}, \label{eq:ratGF} \end{equation}
where the degree of $A(z)$ is at most $r$.
\end{theorem}

% \begin{proof}
% The result follows from noting that for any natural number $n$
% {\small\begin{align*} 
% [z^{n+r}]\left(\sum_{n=0}^{\infty} f_nz^n\right)\left(1- (a_{r-1}z + a_{r-2}z^2 + \cdots + a_0z^r)\right) &= f_{n+r} - \left(a_{r-1}f_{n-r-1} + a_{r-2}f_{n-r-2} + \cdots + a_0f_n\right) \\
% &= 0.
% \end{align*}}
% \vspace{-0.2in}
% \end{proof}

\subsubsection{Asymptotics} 
When considering a sequence over the complex numbers, $F(z)$ is an analytic function with a finite number of polar singularities given by the roots of its denominator $H(z)$.  One can then recover Theorem~\ref{thm:meroAsm} by computing a partial fraction decomposition of $F(z)$ over the complex numbers.  Furthermore, the partial fraction decomposition makes finding the polynomials $P_j(n)$ effective in the following sense. 

\begin{theorem}[{Gourdon and Salvy~\cite[Algorithm 1]{GourdonSalvy1996}}] \label{thm:ratAsm}
Suppose that $F(z)=G(z)/H(z) \in \mathbb{Q}(z)$ is a rational function with $H(0)\neq0$.  Let $d$ denote the degree of $H(z)$ and $\alpha_1,\dots,\alpha_m$ be the distinct roots of $H(z)$ in the complex plane.  Then there exist polynomials $P_1(n,x),\dots,P_m(n,x)$ in $\mathbb{Q}[n,x]$, whose degrees in $x$ are at most $d$, such that for all $n$ larger than some fixed natural number the Taylor coefficients of $F(z)$ satisfy
\[ f_n = \sum_{j=1}^m P_j(n,\alpha_j) \alpha_j^{-n}. \]
The polynomials $P_1,\dots,P_m$ can be determined explicitly in polynomial time (with respect to $d$), and the degree of $P_j(n,x)$ in $n$ is one less than the order of the pole of $F(z)$ at $z=\alpha_j$.
\end{theorem}

As described in the computational work of Bronstein and Salvy~\cite{BronsteinSalvy1993}, the numerators appearing in the partial fraction decomposition of $F(z)$ can be computed symbolically, after which Newton's generalized binomial theorem can be applied to obtain asymptotic results on coefficients.  Although this is an explicit asymptotic result, there are several subtleties present.  To determine dominant asymptotics, one must determine the roots of $H(z)$ with smallest modulus and isolate their contribution to the asymptotics; this can be done using algorithms related to real root isolation which will also be utilized in the multivariate situation (see Gourdon and Salvy~\cite{GourdonSalvy1996} for details).  The issue is that there can be cancellation in the sub-exponential asymptotic behaviour: in fact, when there is more than one dominant singularity there are still several simple properties of rational coefficient sequences, such as determining whether a sequence has an infinite number of zeroes, which are not known to be decidable\footnote{Ouaknine and Worrell~\cite{OuaknineWorrell2012,OuaknineWorrell2014} survey some of the related decision problems.}. 

Luckily, precise results can often be obtained in combinatorial contexts.  Proposition IV.3 of Flajolet and Sedgewick~\cite{FlajoletSedgewick2009} shows that any generating function of a combinatorial class obtained from a wide variety of recursive `constructions' must have an explicit periodic behaviour; i.e., there exists a natural number $r$ such that for each $k=0,\dots,r-1$ the coefficient sub-sequence $(f_{rn+k})_{n\geq0}$ has dominant asymptotics of the form $C_k\cdot n^{\alpha_k} \cdot \rho_k^n$ for $\alpha_k$ an explicit natural number and algebraic constants $C_k$ and $\rho_k$ whose minimal polynomials can be determined as in Theorem~\ref{thm:ratAsm}.

\subsubsection{Generation of Terms}  
We end this section by noting that it is extremely efficient to calculate coefficients of rational functions (and thus terms of sequences satisfying linear recurrences with constant coefficients).

\begin{proposition} Suppose that $F(z)=G(z)/H(z)=\sum_{n\geq0}f_nz^n$ is a rational function over the field $K$, such that the degrees of $G$ and $H$ are bounded by $d$ and $H(0) \neq 0$. Then 
\begin{itemize}
\item[a)] The $N^{\text{th}}$ term $f_N$ can be calculated in $O(\log N \cdot d \log d \log \log d)$ operations in $K$;
\item[b)] The first $N$ terms $f_0,\dots,f_N$ can be calculated in $O(N\cdot \log d \log \log d)$ operations in $K$, using no divisions.
\end{itemize}
\end{proposition}
Part (a) comes from Fiduccia~\cite{Fiduccia1985} while (b) follows from results in Shoup~\cite{Shoup1991}; see also Bostan et al.~\cite[Corollaries 4.8 and 4.11]{BostanChyzakGiustiLebretonLecerfSalvySchost2017}.

%%%%%%%%%%%%%%%%%%%%%%%%%%%
% Algebraic Power Series
%%%%%%%%%%%%%%%%%%%%%%%%%%%
\section{Algebraic Power Series}
\label{sec:AlgPS}

Let $K$ be a field.  A formal power series $F(z) \in K[[z]]$ is called \emph{algebraic} if there exist polynomials $p_0(z),\dots,p_d(z)$, not all zero, such that
\[ p_d(z)F(z)^d + p_{d-1}(z)F(z)^{d-1} + \cdots + p_0(z) = 0. \]
Given algebraic $F(z)$, the \emph{minimal polynomial} of $F$ is the unique\footnote{The minimal polynomial, as defined here, is unique up to a non-zero multiple of $K$.} polynomial $P(z,y) \in K[z][y]$ of minimal degree in $y$ with co-prime coefficients (over $K[z]$) such that $P\left(z,F(z)\right)=0$.

\subsubsection{Asymptotics}
Suppose that the function $F(z)$ is analytic at the origin and is a root of the polynomial
\[ P(z,y) = p_d(z)y^d + p_{d-1}(z)y^{d-1} + \cdots + p_0(z) \in \mathbb{Q}[z][y].  \]
Following the principles of singularity analysis, in order to determine asymptotics for the coefficient sequence of $F(z)$ we must determine the location and nature of its singularities. The implicit function theorem implies that any singularity $z=\zeta$ of $F(z)$ lies in the set
\[ \Xi := \left\{ \zeta : p_d(\zeta)=0 \text{ or }  \text{disc}_y(P)(\zeta) = 0 \right\}, \]
where disc$_y(P)$ is the \emph{discriminant of} $P(z,y)$ with respect to $y$ (the resultant of $P$ and $\partial P/\partial y$, up to a constant).  The points where $p_d(\zeta)=0$ correspond to points where branches\footnote{At any point $z_c \in \mathbb{C} \setminus \Xi$ the equation $P(z_c,y)=0$ has $d$ solutions $y_1,\dots,y_d$ in $y$, and the implicit function theorem implies that each point $(z_c,y_j)$ lies on the graph $(z,y_j(z))$ of an analytic function $y_j(z)$ defined in a neighbourhood of $z_c$.  Each $y_j(z)$ defines a \emph{branch} of $P(z,y)=0$ in the largest simply connected region of $\mathbb{C}$ containing $z_c$ where it is analytic.} of the equation $P(z,y)=0$ approach infinity, while the discriminant vanishing characterizes points where two branches of $P(z,y)=0$ collide.  

To determine asymptotics of the power series coefficients of $F(z)$ at the origin, one must determine which points in the finite algebraic set $\Xi$ are actually singularities of $F$, and find the corresponding local singular behaviour.  This can be resolved computationally, leading to an algorithm for determining the asymptotics of algebraic function coefficients.  The key to identifying which points in $\Xi$ are singularities of $F(z)$ is that for any non-singular point $\eta$ it is possible to calculate a bound separating distinct solutions of $P(\eta,y)=0$.  

The algorithm \emph{roughly} proceeds as follows.  First, list the elements of $\Xi$ in terms of increasing modulus.  Next, iterate through $\Xi$ and for each element~$\zeta$:
\begin{itemize} \itemsep=1mm
\item[1)] let $y_1(z),\dots,y_r(z)$ be the branches of $P(z,y)=0$ defined and analytic in a punctured disk $0<|z-\zeta| < \epsilon$ minus a ray from $\zeta$ to infinity avoiding the origin (to account for branch cuts);
\item[2)] determine numerical approximations to $y_1(z),\dots,y_r(z)$ at $\eta=(1-\frac{\epsilon}{2|\zeta|})\zeta$, where each branch is analytic;
\item[3)] determine a numerical approximation to $F(\eta)$;
\item[4)] if $F(z)$ is sufficiently close to one of the branches which are singular at $\eta$, return $\zeta$ as a dominant singularity of $F(z)$ and repeat steps (1) -- (3) on the remaining elements of $\Xi$ with the same modulus as $\zeta$.
\end{itemize}

Since we deal only with rational functions in the multivariate case we do not go into more details on this algorithm, which can be found in Section VII.36 of Flajolet and Sedgewick~\cite{FlajoletSedgewick2009}, but do note that finding the numerical approximations to the required accuracy can be done rigorously~\cite[Chapter VI]{Chabaud2002}.  Singular behaviour is determined through the theory of Newton-Puiseux expansions and the method of undetermined coefficients, which goes back to Newton~\cite{Newton1736} and was later studied by Cramer~\cite{Cramer1750} and, of course, Puiseux~\cite{Puiseux1850}. This theory characterizes the singular behaviour which can occur for algebraic functions, thus determining the types of asymptotic growth their coefficients can admit.  The following result was originally investigated in the 19th century by Darboux~\cite{Darboux1878}.

\begin{theorem}[{Darboux's Method~\cite[Theorem 11.10b]{Henrici1977}}] \label{thm:algAsm}
Let $F(z)=\sum_{n\geq0}f_nz^n$ be an algebraic function over $\mathbb{Q}$ which is analytic at the origin, and let $\omega_0\beta,\dots,\omega_{m-1}\beta$ be the singularities of $F(z)$ on its circle of convergence (so $|\omega_j|=1$ for each $j$) where $\beta>0$.  Then
\[ f_n = \frac{\beta^n n^s}{\Gamma(s+1)} \cdot \sum_{j = 0}^{m-1} C_j\omega_j^n + O(\beta^n n^t), \]
where $s \in \mathbb{Q} \setminus \{-1,-2,\dots\}$, $t<s$, and $\beta$, the $\omega_j$, and the $C_j$ are algebraic.
\end{theorem}

\subsubsection{Generation of Terms}  
Fast calculation of algebraic function coefficients is slightly more awkward than for rational or D-finite functions (discussed below) as the condition of being algebraic does not directly correspond to a linear constraint on coefficients.  The idea behind fast algebraic coefficient generation is thus to use the fact that any algebraic series is D-finite\footnote{This result has been rediscovered several times going back to the $19^\text{th}$ century, including works by Abel, Tannery, Cockle, Harley, and Comtet.  Chudnovsky and Chudnovsky~\cite{ChudnovskyChudnovsky1986} and Bostan et al.~\cite{BostanChyzakSalvyLecerfSchost2007} studied this result from a complexity viewpoint; see the introduction to Bostan et al. for more historical remarks.}, and to use efficient algorithms for generating terms from an annihilating differential equation (see Proposition~\ref{prop:fastDFinite} below for further details).

\begin{proposition} Suppose that $F(z)=\sum_{n\geq0}f_nz^n$ is algebraic with minimal polynomial $P(z,y)$ in $\mathbb{Q}[z,y]$ of degree $d_z$ in $z$ and degree $d_y$ in $y$.  Then 
\begin{itemize} \itemsep=2mm
\item[a)] The coefficient $f_N$ can be calculated in $O(\sqrt{N})$ rational operations;
\item[b)] The coefficients $f_0,\dots,f_N$ can be calculated in $O\left(Nd_z(d_y + d_z\log d_z \log\log d_z)\right)$ rational operations.
\end{itemize}
\end{proposition}
Statement (a) comes from Chudnovsky and Chudnovsky~\cite{ChudnovskyChudnovsky1987}, and (b) from Bostan et al.~\cite{BostanChyzakSalvyLecerfSchost2007}.  The key result of Bostan et al. is to derive explicit---and experimentally close to optimal---bounds on the degree and order of an annihilating differential equation from the degrees of the minimal polynomial.

%%%%%%%%%%%%%%%%%%%%%%%%%%%
% D-finite Power Series
%%%%%%%%%%%%%%%%%%%%%%%%%%%
\section{D-Finite Power Series}
\label{sec:DFin}

Let $K$ be a field.  A power series $F(z) \in K[[z]]$ is said to be \emph{differentially finite (D-finite) of order $r$ and degree $d$} if there exist polynomials $a_0(z),\dots,a_r(z)$, one of which has degree $d$ and all of which have degree at most $d$, such that
\begin{equation} a_0(z)\frac{d^r}{dz^r}F(z) + a_{1}(z)\frac{d^{r-1}}{dz^{r-1}}F(z) + \cdots + a_{r-1}(z)\frac{d}{dz}F(z) + a_r(z)F(z) = 0, \label{eq:Dfinite} \end{equation}
and $a_0(z)$ is not $0$. Note that when the power series $F(z) = \sum_{n \geq 0}f_nz^n$ does not represent an analytic function at the origin, or when $K$ is not a sub-field of the complex numbers, the (formal) derivative is defined as $\frac{d}{dz}F(z) := \sum_{n \geq 1}(nf_n)z^{n-1}$. When $F(z)$ does represent an analytic function at the origin then this formal derivative agrees with the usual analytic derivative inside the domain of convergence of the power series at the origin.  We call an analytic function \emph{D-finite} if it satisfies a linear differential equation with polynomial coefficients.

\subsubsection{Coefficient Properties} 

Just like rational functions, D-finite functions are important to the study of linear recurrences.  A sequence $(f_n)_{n \geq0}$ satisfies a \emph{linear recurrence relation with polynomial coefficients (of order $r$ and degree $d$)} if there exist polynomials $c_0(n),\dots,c_r(n)$, one of which has degree $d$ and all of which have degree at most $d$, such that for all $n\geq 0$,
\begin{equation} c_0(n)f_{n+r} + c_1(n)f_{n-r-1} +  \cdots + c_r(n)f_n = 0, \label{eq:linRec} \end{equation}
and $c_0(n)$ and $c_r(n)$ are not identically 0.  We call a sequence satisfying some linear recurrence relation with polynomial coefficients \emph{P-recursive}.

\begin{proposition} \label{thm:DfinEq}
Let $F(z) = \sum_{n \geq 0}f_nz^n$.  Then
\begin{enumerate}
\item If $F(z)$ is D-finite of order $r$ and degree $d$ then $(f_n)_{n \geq 0}$ is P-recursive of order at most $r+d$ and degree at most $r$.
\item If $(f_n)_{n \geq 0}$ is P-recursive of order $r$ and degree $d$ then $F(z)$ is D-finite of order at most $d$ and degree at most $r+d$.
\end{enumerate}
\end{proposition}

See Theorem 14.1 of Bostan et al.~\cite{BostanChyzakGiustiLebretonLecerfSalvySchost2017} for a short proof of this result. 

% \begin{proof}
% Basic manipulations imply that for any $j,k \geq 0$,
% \[ z^j \frac{d^k}{dz^k} F(z) = \sum_{n=j}^{\infty}\left[ (n-j+k)(n-j+(k-1))\cdots(n-j+1)f_{n-j+k} \right]z^n. \]
% This allows one to translate an annihilating differential equation for a power series into a recurrence on its coefficients, and tracking the orders and degrees of the correspondence gives (1).  To obtain (2) one can multiply the linear recurrence in Equation~\eqref{eq:linRec} by $z^n$, formally sum over all $n \geq 0$ and use that for any $k\in\mathbb{N}$,
% \[ \sum_{n \geq0}n^kf_nz^n = \left(z \frac{d}{dz} \right)^k F(z) \qquad \text{ and }\qquad \sum_{n \geq 0}f_{n+k}z^n = \frac{F(z)-(f_0 + \cdots + f_{k-1})z^{k-1}}{z^k}  \]
% to obtain a differential equation of the required degree and order for $F(z)$, where $\left(z \frac{d}{dz} \right)^kF(z)$ denotes the result of differentiating and multiplying by $z$ iteratively $k$ times. See Section 14.1 of Bostan et al.~XXX for details.
% \end{proof}

The study of D-finite functions in combinatorial contexts was popularized by Stanley~\cite{Stanley1980}, who surveyed many examples and closure properties (see also Stanley~\cite{Stanley1999}).  Later work of Zeilberger~\cite{Zeilberger1990} continued this interest by describing automatic means of proving identities of D-finite sequences; this area of research is still extremely active\footnote{For example, see Chyzak and Salvy~\cite{ChyzakSalvy1998}, Chyzak~\cite{Chyzak2000}, Chen et al.~\cite{ChenKauersSinger2012}, Bostan et al.~\cite{BostanChenChyzakLiXin2013}, Bostan et al.~\cite{BostanLairezSalvy2013}, and the references therein.}.

\subsubsection{Asymptotics}

Ideally, one would like to have a method for calculating asymptotics of the coefficients of D-finite functions which is similar to the one for coefficients of algebraic functions.  That is, one could hope for an algorithm which takes as input an annihilating differential equation of a D-finite function $F(z)$ and initial conditions which distinguish it as a solution of the differential equation, and return asymptotics.  The first part of the theory is similar: given that $F(z)$ satisfies a linear differential equation
\[ a_0(z)\frac{d^r}{dz^r}y(z) + a_{1}(z)\frac{d^{r-1}}{dz^{r-1}}y(z) + \cdots + a_{r-1}(z)\frac{d}{dz}y(z) + a_r(z)y(z) = 0, \] 
the singularities of $F(z)$ must lie in the finite algebraic set
\[ \Xi := \{ \zeta : a_0(\zeta) = 0 \}. \]
The difficulty is in trying to determine which elements of $\Xi$ are actually singularities of $F$.  Unlike the algebraic case, where there are a finite number of branches determined by one minimal polynomial, a linear differential equation with polynomial coefficients of order $r$ determines an infinite $r$-dimensional vector space, complicating matters.  Sometimes the nature of a problem can help determine which elements of $\Xi$ correspond to singularities of $F(z)$, for instance if combinatorial arguments can establish the exponential growth of the coefficient sequence and there is only one element of $\Xi$ with the correct modulus to give that growth.

The next step is to perform a local analysis at the singularities\footnote{Determining asymptotics of local solutions to differential equations is a classical topic in several areas of pure and applied mathematics.  Standard references include Wasow~\cite{Wasow1965}, Hille~\cite{Hille1976}, and Olver~\cite{Olver1997}.}.   To begin, we re-write the above differential equation in the form
\begin{equation} 
\frac{d^r}{dz^r}y(z) + b_{1}(z)\frac{d^{r-1}}{dz^{r-1}}y(z) + \cdots + b_{r-1}(z)\frac{d}{dz}y(z) + b_r(z)y(z) = 0, \label{eq:diffEq2} 
\end{equation}
where $b_j(z)$ is the rational function $a_j(z)/a_0(z)$.  Pick $\zeta \in \Xi$ and for any rational function $R(z)$ let $\omega_{\zeta}(R)$ denote the order of the pole of $R$ at $z=\zeta$ (which is $0$ if $R$ is analytic at $\zeta$).  The singularity $\zeta$ is called \emph{regular} if 
\[ \omega_{\zeta}(b_1) \leq 1, \quad \omega_{\zeta}(b_2) \leq 2, \quad \dots \quad, \omega_{\zeta}(b_r) \leq r,  \]
and a linear differential equation with polynomial coefficients admitting only regular singularities is called \emph{Fuchsian}\footnote{In 1866 Lazarus Fuchs~\cite{Fuchs1868} published a study of singularities of solutions to linear differential equations, examining which differential equations do not admit solutions with essential singularities.  Soon after, Frobenius~\cite{Frobenius1873} worked on finding the general forms of solutions to such differential equations.  In fact, much of the theory seems to be contained in an unpublished manuscript of Riemann from February 20, 1857, which is now available as paper XXI in his collected works~\cite{Riemann1990}.  More historical details of the development of singular solutions to differential equations, with a focus on Fuchs, are given by Gray~\cite{Gray1984}.}.  The \emph{indicial polynomial}\footnote{Although Fuchs used indicial polynomials in his work, the name was coined by Cayley; see Cayley~\cite[Section 9]{Cayley1886}, for instance.} of the differential equation~\eqref{eq:diffEq2} at the point $z=\zeta$ is the polynomial
\[ I(\theta) = (\theta)_{(r)} + \delta_1 (\theta)_{(r-1)} + \cdots + \delta_r, \]
where $(\theta)_{(j)} = \theta(\theta-1)\cdots(\theta-j+1)$ and $\delta_j = \lim_{z \rightarrow \zeta} (z-\zeta)^j d_j(z)$.

The roots of the indicial polynomial help determine the form of a local basis of solutions for the differential equation.  In the case of a unique dominant singularity which is regular, this can be transfered directly into coefficient asymptotics.

\begin{theorem}[{Flajolet and Sedgewick~\cite[Theorem VII.10]{FlajoletSedgewick2009}}] 
\label{thm:DfinAsm}
Assume that the coefficients $b_j(z)$ of the differential equation~\eqref{eq:diffEq2} are analytic in a disk $|z| < \rho$, except at a unique pole $\zeta$ with $0<|\zeta|<\rho$. Suppose also that $\zeta$ is a regular singular point of the differential equation~\eqref{eq:diffEq2}, and that $F(z)$ is a solution which is analytic at the origin.  If none of the solutions $\theta_1,\dots,\theta_r$ to the indicial equation $I(\theta)=0$ at $\zeta$ differ by an integer then there exist constants $\lambda_1,\dots,\lambda_r \in \mathbb{C}$ such that for any $\rho_0$ with $|\zeta|<\rho_0 < \rho$
\[ [z^n]F(z) = \sum_{j=1}^r \lambda_j \Delta_j(n) + O(\rho_0^{-n}), \]
where
\[ \Delta_j(n) = \frac{n^{-\theta_j-1}}{\Gamma(-\theta_j)}\zeta^{-n} \left(1 + O\left(\frac{1}{n}\right)\right) \]
if $\theta_j \notin \mathbb{N}$ and $\Delta_j(n) = 0$ if $\theta_j \in \mathbb{N}$.  When the roots of the indicial polynomial differ by an integer (including the case of multiple roots) then the dominant asymptotics of $[z^n]F(z)$ can be expressed as a $\mathbb{C}-$linear combination of terms of the form $\zeta^{-n}n^{-1-\theta_j}(\log n)^l$, where $l$ is a non-negative integer.
\end{theorem}

When there are several dominant singularities, all of which satisfy the conditions of Theorem~\ref{thm:DfinAsm}, then one can compute the contribution of each using the theorem and sum the results to determine dominant asymptotics.  Note that the possible form of the asymptotic growth is less restricted than in the algebraic case: here the polynomial factors $n^{-\theta_j-1}$ can have negative integer exponents, and there can be logarithmic terms present.  Determining the constants $\lambda_j$ is necessary for finding dominant asymptotics---in particular, one needs to determine when the constants are non-zero to even determine exponential growth---but it is currently unknown how to rigorously accomplish this in general\footnote{Given a differential equation with only regular single points, numerical approximations of the connection coefficients with rigorous error bounds can be effectively computed, but it is unknown how to determine the coefficients exactly.  For instance, if the constant corresponding to the dominant asymptotic term is 0 this cannot be proven (without access to a bound on the tolerance needed to decide equality to zero).  The complete algorithm required to compute such approximations was given by van der Hoeven~\cite{Hoeven2001}, based on work of Chudnovsky and Chudnovsky~\cite{ChudnovskyChudnovsky1986,ChudnovskyChudnovsky1987}. More results on this topic can be found in the PhD thesis of Mezzarobba~\cite{Mezzarobba2011}, who has created a Sage package which can (among other tasks) compute numerical connection coefficients with rigorous error bounds~\cite{Mezzarobba2016}.}.  This is referred to as the \emph{connection problem}, and the $\lambda_j$ are called \emph{connection coefficients}.

The relatively nice asymptotic growth given in Theorem~\ref{thm:DfinAsm} comes from the fact that any solution to a linear differential equation at a regular singular point can have only a finite pole, branch cut, or logarithmic singularity (or be analytic) at that point --- no solution has an essential singularity at a regular singular point.  This is not the case with irregular singularities, and complicates the analysis.  

We do not examine irregular singular points further for a simple reason: they cannot arise when dealing with analytic power series with integer coefficients!  An analytic power series $F(z)$ with rational coefficients is called a \emph{G-function}\footnote{G-functions were introduced by Siegel~\cite{Siegel2014} in his studies on number theory and elliptic integrals.} if $F(z)$ is D-finite and there exists a constant $C>0$ such that for all $n$ both $|a_n|$ and the least common denominator of $a_0,\dots,a_n$ are bounded by $C^n$.  

\begin{proposition}[Chudnovsky and Chudnovsky~\cite{ChudnovskyChudnovsky1985}] \label{prop:Gfun}
If $F(z)$ is a G-function then the indicial equation $I(\theta)=0$ of the minimal order differential equation annihilating $F(z)$ has only rational solutions.
\end{proposition}

Proposition~\ref{prop:Gfun}, which follows from a result often referred to as the André-Chudnovsky-Katz Theorem\footnote{Chudnovsky and Chudnovsky~\cite[Theorem III]{ChudnovskyChudnovsky1985} showed that if $F(z)$ is a G-function then its minimal order annihilating differential equation is \emph{globally nilpotent}, meaning that the $p^\text{th}$ iterate of a linear operator related to the annihilating differential equation is nilpotent mod $p$ for all but a finite number of primes $p$.  A previous result of Katz~\cite{Katz1970} then restricts the singular behaviour of any solution to the differential equation, yielding Proposition~\ref{prop:Gfun}. The original proof of this result in Chudnovsky and Chudnovsky~\cite{ChudnovskyChudnovsky1985} contained a small flaw which was corrected by Andr{\'e}~\cite[Section VI]{Andre1989}.}, gives properties of a differential equation from properties of a single solution, which is very strong.

\begin{corollary} \label{cor:ratDFin}
Suppose that $F(z)$ is a D-finite function with integer coefficients which is analytic at the origin.  Then dominant asymptotics of the coefficient sequence $[z^n]F(z)$ is given by a finite sum of terms of the form $C n^{\alpha} (\log n)^l \zeta^n$, where $C$ is a constant, $\alpha$ is a rational number, $l$ is a non-negative integer, and $\zeta$ is algebraic.
\end{corollary}

Fischler and Rivoal~\cite[Theorems 1 and 2]{FischlerRivoal2014} show that any leading constant appearing in such an asymptotic expansion is the evaluation $f(1)$ of a G-function $ f \in \mathbb{Q}(i)[[z]]$ whose radius of convergence can be made arbitrarily large.

More details can be found in work of Andr{\'e}~\cite{Andre1989,Andre2000}, and some combinatorial applications to the asymptotics of D-finite function coefficients were discussed by Garoufalidis~\cite{Garoufalidis2009}.  Corollary~\ref{cor:ratDFin} was recently used by Bostan et al.~\cite{BostanRaschelSalvy2014} to show the non-D-finiteness of a class of generating functions arising in the study of lattice walks restricted to the first quadrant.

\subsubsection{Generation of Terms}
Suppose $F(z)$ is a D-finite function with rational coefficients whose P-recursive coefficients satisfy a linear recurrence relation of the form given in Equation~\eqref{eq:linRec}.  The set of sequences solving this recurrence form a vector space over the rational numbers.  Unfortunately, because the leading polynomial coefficient $c_0(n)$ of the recurrence may vanish at positive integers, the dimension of this vector space may be larger than the order $r$ of the recurrence.  

Let $H$ be the set of positive integer solutions to $c_0(n)=0$ and 
\[ J = \{0,\dots,r-1\} \cup \{h+r : h \in H\}.\]  
Specifying the values $f_j$ for $j \in J$ of a solution $(f_n)$ to Equation~\eqref{eq:linRec} uniquely determines the sequence $(f_n)$ (see Bostan et. al~\cite[Proposition 15.5]{BostanChyzakGiustiLebretonLecerfSalvySchost2017} for details).  The values of the $f_j$ for $j \in J$ are known as \emph{generalized initial conditions} of the recurrence in Equation~\eqref{eq:linRec}.  Note that all elements of $H$ are bounded by the maximum absolute value of the coefficients of $c_0(n)$ plus one (see Lemma~\ref{lemma:roots} below).

\begin{proposition}[{Bostan et al.~\cite[Propositions 15.6 and 15.7]{BostanChyzakGiustiLebretonLecerfSalvySchost2017}}] \label{prop:fastDFinite}
Suppose that the polynomials in Equation~\eqref{eq:linRec} have degree at most $d$ and integer coefficients of bit-size at most $l$.  If a solution $(f_n)$ of this recurrence relation is specified by generalized initial conditions consisting of integers of bit-size\footnote{We write $a = \tilde{O}(b)$ when $a=O(b\log^kb)$ for some $k\ge0$; see Section~\ref{sec:complexity_measure} of Chapter~\ref{ch:EffectiveACSV} for more information.} $\tilde{O}(dN + lN + N\log r)$, 
\begin{itemize} \itemsep=2mm
\item[i)] the terms $f_0,\dots,f_N$ can be calculated in $\tilde{O}\left(rN^2(d+l)\right)$ binary operations;
\item[ii)] the term $f_N$ can be calculated in $\tilde{O}\left(N(d+l)(r^{\theta}+dr) \right)$ binary operations, 
\end{itemize}
where $\theta$ is any positive number such that matrix multiplication of two $\rho \times \rho$ integer matrices can be computed in $O(\rho^\theta)$ integer operations\footnote{The currently optimal result on matrix multiplication by Le Gall~\cite{Le-Gall2014} implies that one can take $\theta < 2.3728639$.}. 
\end{proposition}
%These results are obtained by calculating the coefficient recurrence corresponding to the annihilating differential equation and using the recurrence to generate terms. Part (a) comes from a naive generation of terms in the recurrence, while part (b) utilizes a `baby step / giant step' approach where one combines fast polynomial matrix multiplication with fast multipoint evaluation\footnote{The baby step / giant step method in computer algebra was pioneered by Strassen~\cite{Strassen1976} during his work on the problem of integer factorization.  This approach was first applied to the problem of computing terms of D-finite sequences by Chudnovsky and Chudnovsky~\cite[Section 6]{ChudnovskyChudnovsky1988}.}; see Chapter $15$ of Bostan et al.~\cite{BostanChyzakGiustiLebretonLecerfSalvySchost2017} for details.

%%%%%%%%%%%%%%%%%%%%%%%%%%%
% Rational Diagonals
%%%%%%%%%%%%%%%%%%%%%%%%%%%
\section{Multivariate Rational Diagonals}
\label{sec:mvratDiag}

Let $R$ be a ring.  Given a power series
\[ F(\bz) = \sum_{\bi \in \mathbb{N}^d} f_{\bi}\bz^{\bi} \]
in $R[[\bz]]$, the \emph{(complete) diagonal} of $F(\bz)$ is the power series $(\Delta F)(z) \in R[[z]]$ defined by
\[(\Delta F)(z) := \sum_{k \geq 0} f_{k,k,\dots,k} z^k; \]
that is, the diagonal is the formal power series defined by the terms of $F(\bz)$ where all variables have the same exponent.  In this thesis we will focus on diagonals of multivariate rational functions.

\begin{example}
The binomial theorem shows that in a neighbourhood of the origin the rational function $F(x,y) = \frac{1}{1-x-y}$ has the power series expansion
\[ F(x,y) = \sum_{i,j \geq 0} \binom{i+j}{i}x^iy^j, \]
so that the diagonal $(\Delta F)(z) = \sum_{n \geq 0}\binom{2n}{n}z^n$ is the generating function of the central binomial coefficients.
\end{example}

\begin{example}
Although the coefficients of a multivariate rational function must satisfy a (multivariate) linear recurrence relation with constant coefficients, the converse is decidedly false.  For example, consider the bivariate sequence $(a_{m,n})$ defined by the linear recurrence relation
\[ a_{m,n} = a_{m+1,n-2} + a_{m-2,n+1} - a_{m-1,n-1}, \quad n,m \geq 2, \]
$a_{1,1} = -1$, and $a_{m,n} = 0$ in all other cases.  Bousquet-Mélou and Petkov{\v{s}}ek~\cite[Example 6]{Bousquet-MelouPetkovsek2000} show that the \emph{section} $G(x) = \sum_{m \geq 2} a_{m,2}x^{m+1}$ is \emph{hypertranscendental}, meaning it does not satisfy any algebraic differential equation $P(x,G(x),G'(x),\dots,G^{(k)}(x))=0$ with $P$ polynomial.  Furthermore, the bivariate generating function $F(x,y) = \sum_{m,n \geq 2} a_{m,n} x^{m-2}y^{m-2}$ satisfies $F(x,y) = \frac{xy-G(x)-G(y)}{(x-y^2)(y-x^2)}$.  The sequence $(a_{m,n})$ only takes on the values $-1,0,$ and $1$.
\end{example}

One of the key properties of the diagonal operator is how it affects the various classes of generating functions discussed earlier in this chapter.

\begin{theorem}[Hautus-Klarner~\cite{HautusKlarner1971}, Furstenberg~\cite{Furstenberg1967}, Polya~\cite{Polya1921}] \label{thm:bivarRat}
Let the rational function $F(x,y) = P(x,y)/Q(x,y) \in \mathbb{Q}(z)$ define a power series at the origin.  Then $\Delta F$ is algebraic.
\end{theorem}
\begin{proof}[Proof Sketch]
As $F(x,y)$ is rational it represents the Taylor series of an analytic function in some neighbourhood of the origin, thus for $|y|$ sufficiently small the complex-valued function $F(x,y/x)$ in the variable $x$ converges uniformly in an annulus $A_{|y|} := \{|y|/2 < x < |y|\}$ around the origin.  Let $C_{|y|}$ be a positively oriented circle around the origin staying in $A_{|y|}$.  Then the Cauchy Residue Theorem implies
\[ (\Delta F)(y) = \frac{1}{2\pi i} \int_{C_{|y|}} \frac{P(x,y/x)}{xQ(x,y/x)}dx = \sum_{i=1}^n \text{Res}\left( \frac{P(x,y/x)}{xQ(x,y/x)} ; x = \rho_i \right), \]
where $\rho_1(y),\dots,\rho_n(y)$ are the roots of $Q(x,y/x)$ inside $C_{|y|}$.  For $|y|$ sufficiently close to $0$ the collection of roots of $Q(x,y/x)$ interior to the curve $C_{|y|}$ stabilizes to include only those which approach $0$ as $y\rightarrow0$, giving the diagonal as a finite sum of algebraic residues near the origin.
\end{proof}
An effective algorithm for computing the minimal polynomial of a bivariate rational diagonal, and a bound on its degree, is given by Bostan et al.~\cite{BostanDumontSalvy2015} (see also the Introduction of that paper for some historical remarks on Theorem~\ref{thm:bivarRat}).

\begin{example} \label{ex:infCrit1}
Consider the rational function 
\[ F(x,y) := \frac{1}{(1-9xy)(1-x-y)}, \]
with $\frac{1}{x}F(x,y/x) = \frac{1}{(9y-1)(x^2-x+y)}$.  The equation $(9y-1)(x^2-x+y)=0$ has the two roots
\[  r_1(y) = \frac{1-\sqrt{1-4y}}{2} \qquad\quad r_2(y) = \frac{1+\sqrt{1-4y}}{2} \]
in $x$, of which only $r_1(y)\rightarrow0$ as $y\rightarrow0$. Thus, the residue computation detailed in the proof of Theorem~\ref{thm:bivarRat} implies
\[(\Delta F)(y) = \text{Res}\left( \frac{1}{(9y-1)(x^2-x+y)} ; x = r_1 \right) = \frac{1}{(9y-1)(r_2-r_1)} = \frac{1}{(1-9y)\sqrt{1-4y}}, \]
for $y$ in a neighbourhood of the origin. 
\\
\end{example}

Over a field of characteristic zero, the diagonal of a rational function in more than two variables may not be algebraic.  Over a field of positive characteristic, however, the diagonal of an algebraic function in any number of variables must be algebraic, a result shown for rational power series by Furstenberg~\cite{Furstenberg1967} and algebraic power series by Deligne~\cite{Deligne1984}.  Adamczewski and Bell~\cite{AdamczewskiBell2013} give an effective version of the result by explicitly bounding the degree and height of the minimal polynomial for the diagonal ($\Delta F$ mod $p$) in terms of the prime $p$, the degree and height of the minimal polynomial of $F(\bz)$, and the number of variables.

A power series $F(\bz) \in \mathbb{Q}[[\bz]]$ or analytic function $F(\bz)$ is called \emph{D-finite} if the $\mathbb{Q}(\bz)-$vector space generated by $F$ and its partial derivatives is finite dimensional.  Although the class of rational (or even algebraic) functions is not generally closed under taking diagonals, a result of Lipshitz shows that the class of D-finite functions is closed under this operation.  In a sense, this makes D-finite functions the closure of rational functions under the diagonal operation (Christol~\cite{Christol1988,Christol1990} was the first to show that the diagonal of a rational function is always D-finite).

\begin{theorem}[Lipshitz~\cite{Lipshitz1988}]
\label{thm:diagDfin}
If $F(\bz)$ is D-finite, then the diagonal $(\Delta F)(z)$ is D-finite.
\end{theorem}

Every rational diagonal is a G-function, so combining Theorem~\ref{thm:diagDfin} with Theorem~\ref{thm:DfinAsm} and Proposition~\ref{prop:Gfun} gives a characterization of diagonal coefficient sequence asymptotics.

\begin{corollary} \label{cor:diagAsm}
Suppose $F(\bz) \in \mathbb{Q}(\bz)$ is analytic at the origin.  Then the dominant asymptotics of its diagonal coefficient sequence $[z_1^k\cdots z_n^k]F(\bz)$ is a finite sum of terms of the form $C k^{\alpha} (\log k)^l \zeta^k$, where $C$ is a constant, $\alpha$ is a rational number, $l$ is a non-negative integer, and $\zeta$ is algebraic.
\end{corollary}

\begin{example}
\label{ex:binomsquare}
Let 
\[ F(z_1,z_2,z_3,z_4) = \frac{1}{1-z_1-z_2} \cdot \frac{1}{1-z_3-z_4} = 
\sum_{\bi \in \mathbb{N}^4} \binom{i_1+i_2}{i_1}\binom{i_3+i_4}{i_3} \bz^{\bi}.\]
Then 
\[ (\Delta F)(z) = \sum_{n \geq 0}\binom{2n}{n}^2 z^n \]
is \emph{not} algebraic as Stirling's approximation implies that its coefficients grow asymptotically as $16^n / (\pi n)$, violating the constraints of Theorem~\ref{thm:algAsm}.  Binomial identities imply that the coefficients of the diagonal satisfy the recurrence
\[ (n+1)^2f_{n+1} - 4(2n+1)^2f_n = 0 \]
so, using a constructive proof~\cite[Theorem 14.1]{BostanChyzakGiustiLebretonLecerfSalvySchost2017} of Proposition~\ref{thm:DfinEq}, $(\Delta F)(z)$ satisfies the differential equation
\[ (z-16z^2)\frac{d^2}{dz^2} y(z) + (1-32z)\frac{d}{dz}y(z) - 4y(z) = 0. \]
\end{example}

The process of going from a rational function to an annihilating differential equation of the diagonal forms part of the theory of \emph{Creative Telescoping}.  The fastest known algorithm for determining such a differential equation is given by Lairez~\cite{Lairez2016}, following work of Bostan et al.~\cite{BostanLairezSalvy2013}, and an extensive history of Creative Telescoping can be found in the Habilitation thesis of Chyzak~\cite{Chyzak2014}.  In terms of computing asymptotics, going through an annihilating differential equation can run into the connection problem for D-finite functions discussed above.  

It is also interesting to know when a function belonging to one of the above classes can be written as the diagonal of a rational function.  A first result in this area is the following.

\begin{theorem}[Furstenberg~\cite{Furstenberg1967}, Denef and Lipshitz~\cite{DenefLipshitz1987}] \label{thm:algToDiag}
Let $A$ be an integral domain and $P(z,y) \in A[z,y]$ with $(\partial P/\partial y)(0,0)$ a unit in $A$.  If $F(z) \in A[[z]]$ has no constant term and $P(z,F(z))=0$ then
\[ F(z) = \Delta\left( \frac{y^2 (\partial P/\partial y)(zy,y)}{P(zy,y)} \right); \]
i.e., $F$ is the diagonal of a bivariate rational function.
\end{theorem}

\begin{proof}
Writing $P(z,y) = (y-F(z))g(z,y)$ for $g \in A[[z]][y]$, we have
\[ (\partial P/\partial y)(z,y) = g(z,y) + (y-F(z))(\partial g/\partial y)(z,y), \]
so that 
\begin{equation} \frac{y^2 (\partial P/\partial y)(zy,y)}{P(zy,y)} = \frac{y^2}{y-F(zy)} + \frac{y^2 (\partial g/\partial y)(zy,y)}{g(zy,y)}. \label{eq:diagBi} \end{equation}
As $F(z)$ has no constant term,
\[ \frac{y^2}{y-F(zy)} = \frac{y}{1-F(zy)/y} \]
is a power series whose diagonal is $F(z)$.  Furthermore, the second summand in Equation~\eqref{eq:diagBi} is a power series by the assumption on $(\partial P/\partial y)(0,0)$, and has a diagonal of zero.  The result follows from the distributivity of the diagonal operator over addition.
\end{proof}

\begin{example} \label{ex:infCrit2}
The function 
\[ g(z) = \frac{1}{(1-9z)\sqrt{1-4z}}\]
was obtained in Example~\ref{ex:infCrit1} as the diagonal of a rational function.  Note that $g(z)$ satisfies the algebraic equation
\[ (1-9z)^2(1-4z)g(z)^2 - 1 = 0, \]
however $g(z)$ has a non-zero constant term.  Subtracting off the constant term one obtains $P(z,g(z)-1)=0$, where 
\[ P(z,y) = y^2(9z-1)^2(1-4z)-2(9z-1)^2(1-4z)y+(9z-1)^2(1-4z)-1. \]
Theorem~\ref{thm:algToDiag} applies, and adding the constant term back to $g(z)$ yields
\begin{equation} g(z) = \Delta \left( 
\frac{-z(y-1)(2y^2-y+1)(324y^2z^2-153yz+22)+2y^2-3y+2}
{324y^2(y-1)^2z^3-153y(y-1)^2z^2+22(y-1)^2z-y+2} \right) \label{eq:infCrRat2} \end{equation}
where the numerator and denominator inside the diagonal are co-prime. Note that this expression is very different from the original bivariate rational function in Example~\ref{ex:infCrit1}.  
\end{example}

The property of being a rational function diagonal continues to hold when the minimal polynomial of $F(z)$ has a vanishing partial derivative at the origin, but the result is no longer as simple (there are also generalizations to algebraic functions in more variables).

\begin{theorem}[{Denef and Lipshitz~\cite[Theorem 6.2]{DenefLipshitz1987}}]
Let $F(z)$ be an algebraic power series over a field $K$.  Then there exists a bivariate rational power series $R(z_1,z_2)$ such that $F(z)=(\Delta R)(z)$.
\end{theorem}
\noindent
A nice discussion of this result can be found in Section $3$ of Adamczewski and Bell~\cite{AdamczewskiBell2013}.
\\

It is natural to wonder whether a similar result holds for more general families of D-finite functions; indeed, such a characterization was conjectured by Christol in 1990 and remains open.  A power series $F(z)$ in $\mathbb{Q}[[z]]$ is called \emph{globally bounded} if $F(z)$ represents the Taylor series of an analytic function in a neighbourhood of the origin and there exist non-zero $a,b \in \mathbb{Q}$ such that $aF(bz)$ has integer coefficients.  It is an easy exercise to show that every rational diagonal which is analytic at the origin is globally bounded.

\begin{conjecture}[{Christol~\cite[Conjecture 4]{Christol1988}}]
Every globally bounded D-finite function is the diagonal of a rational function.
\end{conjecture}
\noindent
Christol~\cite{Christol2015} provides a recent survey of approaches to the conjecture, and its connection to related results.

\section{Laurent Expansions and Sub-Series Extractions} 
\label{sec:LaurentExp}
In later chapters of this thesis it will be necessary to consider expansions (and diagonals) of functions which cannot be represented by power series at the origin.  Recall that over a ring $R$ the \emph{ring of formal Laurent series} in the variable $z$ is the set
\[ R((z)) = \left\{ \sum_{i \geq q} a_iz^i : q \in \mathbb{Z}, a_i \in R  \right\}, \]
equipped with the usual Cauchy product\footnote{The Cauchy product of two formal Laurent series $F(z) = \sum_j f_jz^j$ and $G(z) = \sum_j g_jz^j$ is the formal Laurent series $F(z) \cdot G(z) = \sum_j h_j z^j$ where $h_j = \sum_{m+n=j}f_mg_n$.} and term-wise sum for infinite series.  When $R$ is a field then $R((z))$ is a field (in fact, it is the field of fractions of the ring of formal power series over $R$).  For more than one variable, the ring of \emph{formal iterated Laurent series} in the variables $z_1,\dots,z_n$ is defined inductively by $R((z_1,\dots,z_n)) := R((z_1,\dots,z_{n-1}))((z_n))$.  Note that the order of the variables used in the definition of the ring of iterated Laurent series is important.

\begin{example}
\label{ex:binLaurent}
Consider the rational function $\frac{1}{1-x-y}$ which has power series expansion 
\[\frac{1}{1-x-y} = \sum_{i,j \geq 0}\binom{i+j}{i}x^iy^j\]
at the origin. In the ring $\mathbb{Q}((x,y)) = \mathbb{Q}((x))((y))$ one can compute the expansion
\[ \frac{1}{1-x-y} = \frac{-1/x}{1-(1-y)/x} = \sum_{i,j \geq 0} \binom{i}{j}(-1)^{j+1}y^j x^{-i-1},  \]
where the binomial coefficient is 0 if $j>i$, while in the ring $\mathbb{Q}((y,x)) = \mathbb{Q}((y))((x))$ one obtains
\[ \frac{1}{1-x-y} = \frac{-1/y}{1-(1-x)/y} = \sum_{i,j \geq 0} \binom{i}{j}(-1)^{j+1}x^j y^{-i-1}.  \]
\vspace{-0.3in}

\end{example}

For a variable $z$, let $\oz = 1/z$.  The \emph{ring of Laurent polynomials} in the variables $z_1,\dots,z_n$ over the field $R$, denoted $R[\bz,\obz]$, is the subset of $R((z_1,\dots,z_n))$ consisting of elements with a finite number of non-zero coefficients.  Note that the ring of Laurent polynomials does not depend on the order of the variables used to define it (up to isomorphism).  Iterated Laurent series are studied in great detail from a formal point of view by Xin~\cite{Xin2004}, while Aparicio-Monforte and Kauers~\cite{Aparicio-MonforteKauers2013} discuss constructions of other rings of formal series expansions.  

We now turn to the study of convergent Laurent series rings over the complex numbers.  Given an open and simply connected subset $\mD \subset \mathbb{C}^n$, the set $\mathbb{C}_{\mD}\{\bz\}$ of \emph{convergent Laurent series on $\mD$} consists of series $\sum_{\bi \in \mathbb{Z}^n} a_{\bi}\bz^{\bi}$ with $a_{\bi} \in \mathbb{C}$ which are absolutely convergent at each point of $\mD$ and uniformly convergent on compact subset of $\mD$.  Given a point $\bz \in \mathbb{C}$ define \glsadd{Relog}
\[ \Relog(\bz) := \left(\log|z_1|,\dots,\log|z_n| \right). \]
The following classic result, which characterizes domains of convergence for multivariate Laurent series, is proven in Pemantle and Wilson~\cite{PemantleWilson2013}.

\begin{proposition}[{Pemantle and Wilson~\cite[Theorem 7.2.2]{PemantleWilson2013}}]
\label{prop:convLaurent}
If $F(\bz)$ is defined by the sum $\sum_{\bi \in \mathbb{Z}^n} f_{\bi}\bz^{\bi}$ then the open domain of convergence of $F$ has the form $\mD = \Relog^{-1}(B)$ for some open convex subset $B \subset \mathbb{R}^n$, and $F$ defines an analytic function on $\mD$.  Conversely, if $f(\bz)$ is an analytic function on $\mD = \Relog^{-1}(B)$ with $B \subset\mathbb{R}^n$ open and convex then there exists a unique element $F \in \mathbb{C}_{\mD}\{\bz\}$ converging to $f$, whose coefficients are given by 
\[ [\bz^{\bi}]F = \frac{1}{(2\pi i)^n} \int_{\Relog^{-1}(\bx)} \frac{f(\bz)}{z_1^{i_1} \cdots z_n^{i_n}} \cdot \frac{dz_1 \cdots dz_n}{z_1 \cdots z_n}, \]
for any $\bx \in B$.
\end{proposition}

The set of formal expressions $\sum_{\bi \in \mathbb{Z}^n} a_{\bi}\bz^{\bi}$ does not have a natural ring structure as, for instance, the Cauchy product of two series can be undefined\footnote{Term-wise addition is well defined for the set of formal expressions $\sum_{\bi \in \mathbb{Z}^n} a_{\bi}\bz^{\bi}$, however, and this set can be made into a module over the ring of Laurent polynomials.}.  The set $\mathbb{C}_{\mD}\{\bz\}$ is, however, a ring when addition is defined term-wise and the multiplication of elements $F,G \in \mathbb{C}_{\mD}\{\bz\}$ defining analytic functions $f(\bz)$ and $g(\bz)$ on $\mD$ is defined as the unique element of $\mathbb{C}_{\mD}\{\bz\}$ converging to $f(\bz)g(\bz)$ (whose coefficients can be determined using Proposition~\ref{prop:convLaurent}).  Similarly, given any two convergent Laurent series $F \in \mathbb{C}_{\mD_1}\{\bz\}$ and $G \in \mathbb{C}_{\mD_2}\{\bz\}$ whose domains $\mD_1$ and $\mD_2$ have non-empty intersection, there is a unique Laurent series in $\mathbb{C}_{\mD_1 \cap \mD_2}\{\bz\}$ corresponding to the product of the analytic functions $G$ and $H$, both of which are defined in $\mD_1 \cap \mD_2$.  

Given a function $f(\bz)$ we define \glsadd{amoeba}
\[ \amoeba(f) := \left\{\Relog(\bz) : \bz \in \left(\mathbb{C}^*\right)^n, f(\bz) = 0 \right\} \subset \mathbb{R}^n. \]
This set was introduced to the study of algebraic varieties by Bergman~\cite{Bergman1971}, and the name \emph{amoeba} was coined by Gelfand, Kapranov, and Zelevinsky~\cite{GelfandKapranovZelevinsky2008} as two dimensional amoebas in the plane resemble cellular amoebas with tentacles going off to infinity.  The next result follows from Proposition~\ref{prop:convLaurent} and helps characterize the Laurent expansions of a fixed ratio of Laurent polynomials. 

\begin{proposition}[{Gelfand, Kapranov, and Zelevinsky~\cite[Corollary 1.6]{GelfandKapranovZelevinsky2008}}]
If $f(\bz)$ is a Laurent polynomial then all connected components of the set $\mathbb{R}^n \setminus \amoeba(f)$ are convex subsets of $\mathbb{R}^n$.  These real convex sets are in bijection with the Laurent series expansions of the rational function $1/f(\bz)$.  When $1/f$ has a power series expansion, then it corresponds to the component of $\mathbb{R}^n \setminus \amoeba(f)$ containing all points $(-N,\dots,-N)$ for $N$ sufficiently large.
\end{proposition}

Further results on convergent Laurent expansions of $1/f$, including their strong connection to properties of the Newton polygon $\mN(f)$ of $f$, can be found in Chapter 6 of Gelfand, Kapranov, and Zelevinsky~\cite{GelfandKapranovZelevinsky2008} or Chapter 7 of Pemantle and Wilson~\cite{PemantleWilson2013}.  In particular, we mention that each connected component of $\mathbb{R}^n \setminus \amoeba(f)$ corresponds to an integer point in $\mN(f)$, so that the number of integer points in $\mN(f)$ gives an upper bound on the number of convergent Laurent expansions of $1/f$.  There always exist connected components corresponding to vertices (extreme points) of the Newton polytope, but whether or not there are any components corresponding to the other integer points of the polytope depends on the coefficients of $f$.

\begin{example}
\label{ex:binDomainConverge}
In Example~\ref{ex:binLaurent} we saw three formal Laurent expansions for the rational function $F(x,y) = 1/(1-x-y)$.  In fact, each of these formal expansions are convergent Laurent expansions and, as the Newton polytope $\mN(1-x-y)$ consists of the three integer points $(0,0),(0,1),(1,0)$ (which are vertices), they make up all convergent Laurent expansions of $F$.  To determine the associated domains of absolute convergence we note that by the binomial theorem
\begin{align*}
\sum_{i,j \geq 0} \binom{i+j}{i}|x|^i|y|^j &= \frac{1}{1-|x|-|y|} \\
\sum_{i,j \geq 0} \binom{i}{j}|y|^j |x|^{-i-1} &= \frac{-1/|x|}{1-(1+|y|)/|x|} \\
\sum_{i,j \geq 0} \binom{i}{j}|x|^j |y|^{-i-1} &= \frac{-1/|y|}{1-(1+|x|)/|y|}
\end{align*}
so that the domains of absolute convergence are
\[ \mD_1 = \{(x,y) : |x|+|y|<1 \}, \qquad \mD_2 = \{(x,y) : 1+|y|<|x| \}, \qquad \mD_3 = \{(x,y) : 1+|x|<|y| \}. \]
The amoeba of $1-x-y$ is shown\footnote{To determine the amoeba of $1-x-y$, note that its points can be described by $\left(\log |x|, \log|1-x|\right)$ for $x \in \mathbb{C}$.  The boundary points in the first and fourth quadrant of Figure~\ref{fig:amoeba1xy} below the line $y=x$ are given by $\left(\log x, \log (x-1)\right)$ for $x \in (1,\infty)$, and the boundary points in the third quadrant are given by $\left(\log x, \log (1-x) \right)$ for $x \in (0,1)$.  Finally, the boundary points in the first and second quadrant above the line $y=x$ are determined by $\left(\log x, \log (1+x) \right)$ for $x \in (0,\infty)$.} in Figure~\ref{fig:amoeba1xy}, along with its Newton polygon.
\end{example}

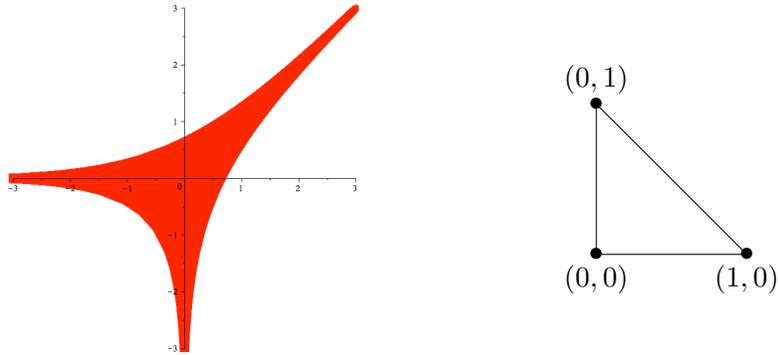
\begin{figure}
\centering
	\begin{minipage}{.4\textwidth}
        \centering
        \includegraphics[width=0.8\linewidth]{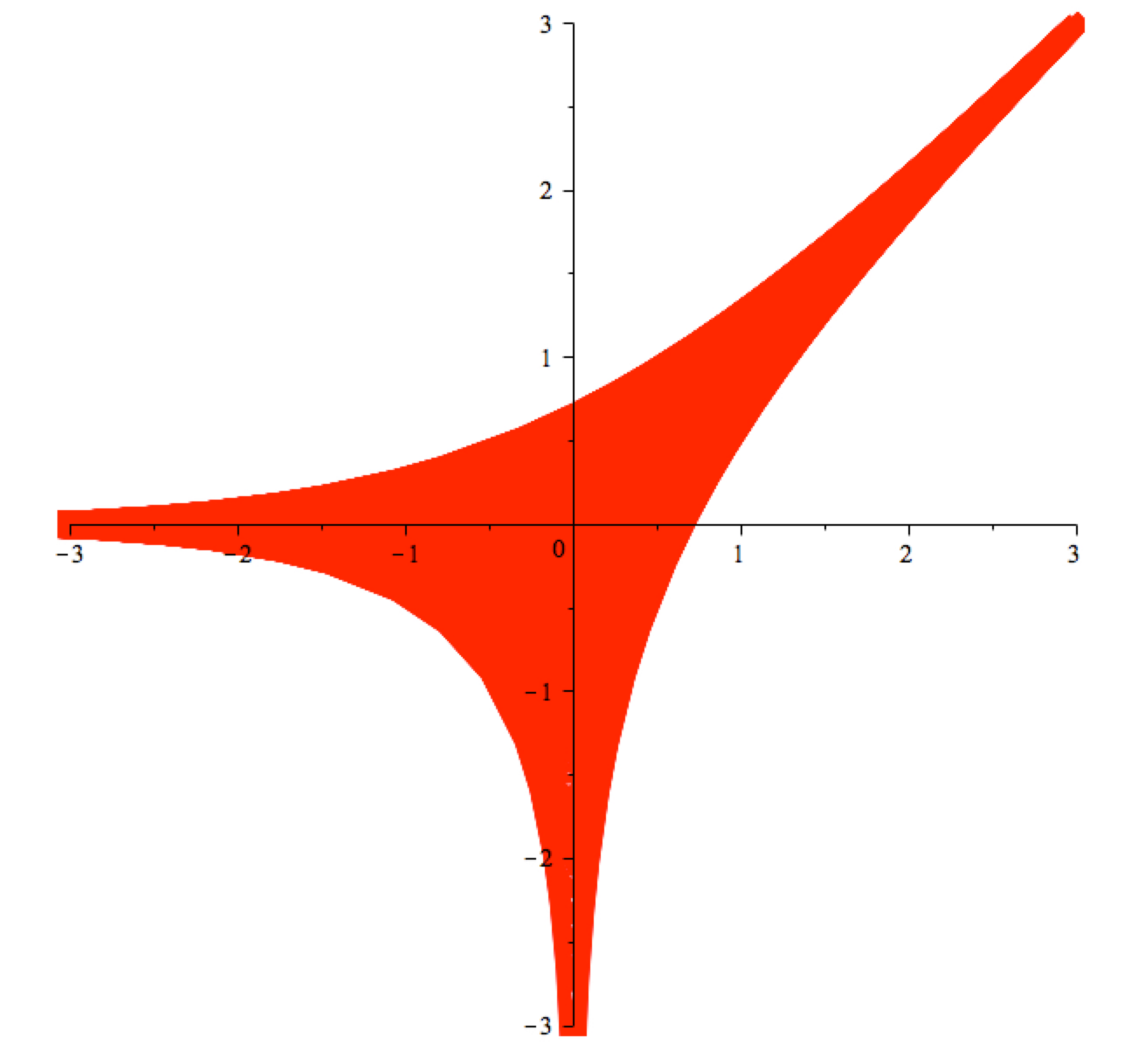}
    \end{minipage}%
    \begin{minipage}{0.4\textwidth}
        \centering
	\begin{tikzpicture}
	\draw (0,0) node[anchor=north]{$(0,0)$}
	-- (2,0) node[anchor=north]{$(1,0)$}
	-- (0,2) node[anchor=south]{$(0,1)$}
 	-- cycle;
 	\foreach \Point in {(0,0), (0,2), (2,0)}{
    	\node at \Point {\textbullet};
	}
	\end{tikzpicture}
    \end{minipage}
    \caption[The amoeba and Newton polygon of $1-x-y$]{The amoeba (left) and Newton polygon (right) of $1-x-y$.}
    \label{fig:amoeba1xy}
\end{figure}

Some computational questions related to amoebas, including drawing them in two dimensions and determining their boundary, are addressed by Theobald~\cite{Theobald2002} and de Wolff~\cite{Wolff2013}.

\subsubsection{Diagonals of Laurent Expansions}
Given a formal Laurent series
\[ F(\bz) = \sum f_{\bi} \bz^{\bi} \in R((\bz)) \]
or convergent Laurent series
\[ F(\bz) = \sum f_{\bi} \bz^{\bi} \in \mathbb{C}_{\mD}\{\bz\} \]
in some domain $\mD$, the diagonal of $F$ is simply the univariate series
\[ (\Delta F)(z) = \sum_{k \geq 0} f_{k,\dots,k}z^k. \]
Given a function $f(\bz)$ over the complex numbers one can compute the diagonal of $f$ for any of its convergent Laurent series.  Thus, one must specify a domain of convergence in order to define the diagonal $\Delta f$, which by Proposition~\ref{prop:convLaurent} can be done by specifying any point in the domain.  Unless explicitly noted, when given a function which is analytic at the origin we always consider the diagonal of the power series expansion of the function.  

Most of the results discussed above for diagonals of rational functions with power series expansions hold for all convergent Laurent series expansions of rational functions.  In particular, the diagonal of a convergent Laurent expansion is still D-finite.  In fact, there is a differential operator which annihilates all convergent Laurent expansions of a rational function, and this operator can be found using the creative telescoping algorithm of Lairez~\cite{Lairez2016}.  Thus, the diagonal of any convergent Laurent expansion of a rational function is still a G-function and we obtain the following analogue of Corollary~\ref{cor:diagAsm}.

\begin{corollary} \label{cor:diagAsmLaurent}
Let $F(\bz) \in \mathbb{Q}(\bz)$ be a rational function.  Then dominant asymptotics of the diagonal coefficient sequence of any convergent \emph{Laurent} series expansion of $F(\bz)$ is a finite sum of terms of the form $C k^{\alpha} (\log k)^l \zeta^k$, where $C$ is a constant, $\alpha$ is a rational number, $l$ is a non-negative integer, and $\zeta$ is algebraic.
\end{corollary}

\subsubsection{Non-Negative Series Extractions}

Over any ring $R$ the \emph{non-negative series extraction} operator with respect to the variables $z_1,\dots,z_n,t$ is the operator $[z_1^{\geq 0} \cdots z_n^{\geq 0}]: R((\bz))[[t]] \rightarrow R[[\bz,t]]$ which takes an element
\[ F(\bz,t) = \sum_{k \geq 0} \left( \sum_{\bi \in \mathbb{Z}^n} f_{\bi,k}\bz^{\bi} \right)t^k \]
of $R((\bz))[[t]]$ and returns
\[ [z_1^{\geq 0} \cdots z_n^{\geq 0}] F(\bz,t) = \sum_{k \geq 0} \left( \sum_{\bi \in \mathbb{N}^n} f_{\bi,k}\bz^{\bi} \right)t^k. \]
Note that $R((\bz))$ implicitly comes with an ordering of the variables $z_1,\dots,z_n$, and even though the image of $F$ under $[z_1^{\geq 0} \cdots z_n^{\geq 0}]$ is a power series it will depend on this underlying ordering. When $F(\bz,t) \in R[\bz,\overline{\bz}][[t]]$, the image of $F$ under $[z_1^{\geq 0} \cdots z_n^{\geq 0}]$ is independent of how the variables are ordered.

Certain variants of the kernel method (to be described in Chapter~\ref{ch:KernelMethod}) rely heavily on generating function representations using non-negative series extractions of rational functions.  The following result gives a relationship between a multivariate function encoded as the non-negative series extraction of a Laurent series and a diagonal representation of evaluations of that function.

\begin{proposition}
\label{prop:postodiag}
Let $F(\bz,t) \in R[\bz,\overline{\bz}][[t]]$.  Then for $\ba \in \{0,1\}^n$,
\begin{equation} 
\label{eq:postodiag}
[z_1^{\geq}]\cdots[z_n^{\geq}]F(\bz,t) \bigg|_{z_1=a_1,\dots,z_n=a_n} = \Delta \left(\frac{F\left(\oz_1,\dots,\oz_n,z_1\cdots z_n\cdot t\right)}{(1-z_1)^{a_1}\cdots(1-z_n)^{a_n}}\right).
\end{equation}
\end{proposition}

Note that the specialization of variables on the left-hand side of Equation~\eqref{eq:postodiag}, and the substitution on its right-hand side, are well defined as each coefficient of $F(\bz,t)$ with respect to $t$ is a Laurent polynomial.

\begin{proof}
The right-hand side of Equation~\eqref{eq:postodiag} is given by 
\begin{align*}
&\Delta \left[ \left(\sum_{k \geq 0} z_1^k \right)^{a_1} \cdots\left(\sum_{k \geq 0} z_n^k \right)^{a_d}
\left(\sum_{k \geq 0} \left( \sum_{\bi \in \mathbb{Z}^n} f_{\bi,k} z_1^{k-i_1}\cdots z_n^{k-i_n} \right)t^k\right) \right] \\[+2mm]
&= \Delta \left[  \sum_{k \geq 0} \left(\sum_{\bj \in \mathbb{N}^n}  \sum_{\bi \in \mathbb{Z}^n} f_{\bi,k} z_1^{a_1j_1-i_1}\cdots z_n^{a_nj_n-i_n} \right)(z_1 \cdots z_n t)^k \right].
\end{align*}
Thus, if $a_j=0$ for $1 \leq j \leq n$ then $i_j = 0$ in the inner sum for any term on the diagonal.  If, however, $a_j=1$ then any terms with $i_j$ non-negative in the inner sum lie on the diagonal.  Evaluating the non-negative series extraction at $z_j=0$ removes all terms with positive powers of $z_j$, while evaluating at $z_j=1$ sums all coefficients with non-negative powers of $z_j$, and the result follows.
\end{proof}

Proposition~\ref{prop:postodiag} will be the key result which allows us to obtain asymptotics by combining the kernel method and analytic combinatorics in several variables.

%%%%%%%%%%%%%%%%%%%%%%%%%%%
% Chapter 4
%%%%%%%%%%%%%%%%%%%%%%%%%%%
\chapter{Lattice Path Enumeration and The Kernel Method}
\label{ch:KernelMethod}

\setlength{\epigraphwidth}{2.3in}
\vspace{-0.2in}

\epigraph{We present here a new method for solving the ballot problem with the use of double generating functions, since this method lends itself to the solution of more difficult problems\dots}{Don Knuth, \emph{The art of computer programming. Vol 1. }}

\vspace{-0.2in}

\epigraph{But I love your feet \\
only because they walked \\
upon the earth and upon \\
the wind and upon the waters, \\
until they found me.\footnotemark}{Pablo Neruda, \emph{Los versos del Capitán}}
\footnotetext{Translated from the Spanish by Donald D. Walsh.}

As described in the introduction to this thesis, given a dimension $n \in \mathbb{N}$, a finite \emph{step set} $\mS \subset \mathbb{Z}^n$, and a \emph{restricting region} $\mR \subset \mathbb{Z}^n$ the \emph{integer lattice path model} taking steps in $\mS$ restricted to $\mR$ is the combinatorial class consisting of sequences of the form $(\bss_1,\dots,\bss_k)$, where $\bss_j \in \mS$ for $1 \leq j \leq k$ and every partial sum $\bss_1 + \cdots + \bss_r \in \mR$ for $1 \leq r \leq k$ (addition is performed component-wise).  We enumerate the objects in this class by the number of steps they contain, and add a single sequence of length zero representing an empty walk.

We begin this chapter by discussing models whose walks are unrestricted, which always have rational generating functions, followed by models whose walks are restricted to a half-space, which always have algebraic generating functions.  Finally, we consider models whose walks are restricted to an orthant, which can admit generating functions with a wide variety of behaviour.

\section{Unrestricted Lattice Walks}

Consider first a lattice path model with step set $\mS \subset \mathbb{Z}^n$ and $\mR = \mathbb{Z}^n$, so that there is no restriction on where the walks in the model can move.  As a walk of length $k$ can have any of the $|\mS|$ steps in each of its $k$ coordinates, we have the generating function identity
\[ C(t) = \sum_{k \geq 0} |\mS|^k t^k = \frac{1}{1-|\mS|t}. \]

Although unrestricted models are simple to enumerate, we use them as an opportunity to set up the basics of the kernel method.  Instead of looking at the univariate generating function enumerating the total number of walks in the model, the key of the method is to use a multivariate generating function to additionally keep track of each walk's endpoint.  With this in mind, we define the formal series
\[ C(\bz,t) := \sum_{k \geq 0}  \left( \sum_{\bi \in \mathbb{Z}^n}  c_{\bi,k}\bz^{\bi}\right) t^k, \]
where $c_{\bi,k}$ denotes the number of walks of length $k$ which end at the point $\bi \in \mathbb{Z}^n$.  As there are a finite number of walks of any fixed length, this formal series is well defined as an element of the ring $\mathbb{Q}[\bz,\overline{\bz}][[t]]$.  Let 
\[ S(\bz) := \sum_{ \bi \in \mS} \bz^\bi,\]
which is called the \emph{characteristic polynomial} of the model, and define 
\[ C_k(\bz) := [t^k]C(\bz,t) = \sum_{\bi \in \mathbb{Z}^n}  c_{\bi,k}\bz^{\bi}, \]
for each $k \geq 0$.  Combinatorially, a walk of length $k+1$ is a walk of length $k$ followed by a step in $\mS$.  Updating the endpoint of a walk appropriately, we get the recurrence
\begin{equation} 
\label{eq:unrec} 
C_0(\bz) = 1, \qquad C_{k+1}(\bz) = S(\bz) C_k(\bz) \quad \text{ for } k \geq 0,
\end{equation}
which, upon multiplying by $t^{k+1}$ and summing for $k=0,1,\dots$, yields the expression
\begin{equation} 
\label{eq:unker}
(1-tS(\bz))C(\bz,t) = 1.
\end{equation}
Equation~\eqref{eq:unker} is called the \emph{kernel equation}, with the Laurent polynomial $K(\bz,t)=1-tS(\bz)$ known as the \emph{kernel}.  In the following sections we show that similar equations can be set up for lattice paths restricted to other regions, however the right-hand side of the resulting equations will rely on evaluations and coefficient extractions of (a priori unknown in explicit form) multivariate generating functions.  In this easy case we can simply solve the kernel equation to obtain
\[ C(\bz,t) = \sum_{k \geq 0} C_k(\bz)t^k = \sum_{k \geq 0} S(\bz)^k t^k = \frac{1}{1-tS(\bz)}. \]
Note that $C(\bone,t) = C(t)$, the univariate generating function counting the total number of walks.

An additional benefit of the kernel method is that it often yields generating function expressions for walks returning to the origin, or those ending on certain axes.  For example, the generating function for the number of walks returning to the origin can be expressed as
\[ B(t) = [z_1^0 \cdots z_n^0] C(\bz,t) = [z_1^0 \cdots z_n^0] \frac{1}{1-tS(\bz)} = \Delta \left(\frac{1}{1-t(z_1 \cdots z_n)S(\bz)}\right). \]
In the one dimensional case this generating function, as the diagonal of a bivariate rational function, is algebraic\footnote{Although Theorem~\ref{thm:bivarRat} is stated only for diagonals of bivariate power series, an analogous result holds for diagonals of bivariate Laurent expansions (see, for instance, Pochekutov~\cite[Theorem 1]{Pochekutov2009}).}. Banderier and Flajolet~\cite[Theorem 1]{BanderierFlajolet2002} give the explicit representation
\[ B(t) = t \sum_{j=1}^p \frac{z'_j(t)}{z_j(t)}, \]
where $z_1(t),\dots,z_p(t)$ are the algebraic roots of $1-tP(z)=0$ in $z$ which approach zero as $t$ approaches zero.

\section{Lattice Walks in a Half-space} 
\label{sec:half-space}

Given a fixed multiset $\mS \subseteq\mathbb{Z}$ (possibly with repeated elements), we now consider the lattice path model taking steps in $\mS$ whose walks are restricted to $\mR = \mathbb{N}$.  More generally, one can consider walks in $(n+1)$ dimensions restricted to the half-space $\mathbb{Z}^n \times \mathbb{N}$, but every higher dimensional model has the same counting sequence as the one dimensional model obtained by projecting each step onto its $(n+1)^\text{st}$ coordinate.  Let
\[ H(z,t) := \sum_{k \geq 0}  \left( \sum_{i \in \mathbb{Z}}  h_{i,k}z^i\right) t^k, \]
where $h_{i,k}$ denotes the number of walks of length $k$ which end at the point $i \in \mathbb{Z}$, and let $-a$ and $b$ denote the minimum and maximum of the elements of $\mS$.  If $a\leq0$ we are in the unrestricted case of the previous section, and if $b\leq 0$ then there are no valid walks of non-zero length.  Thus, we assume $a,b>0$.
\smallskip

Defining $H_k(z) := [t^k]H(z,t)$, our goal is to obtain a recurrence for $H_k$.  The recurrence will not be the same as the one for $C_k$ in the unrestricted case as we must take into account the restriction to a half-space, but taking this into consideration is easy as we track the endpoint of a walk.  If
\[ S(z) = \sum_{i \in \mS} z^i = s_{-a}z^{-a} + \cdots + s_b z^b\] 
with each $s_j \in \mathbb{N}$, then defining 
\[ [z^{< -j}]S(z) := s_{-a}z^{-a} + \cdots + s_{-j-1}z^{-j-1} \]
one obtains the recurrence 
\[ H_{k+1}(z) = S(z)H_k(z) - \sum_{j=0}^{a-1} [z^{< -j}]S(z) \cdot [z^j]H_k(z) \]
for $k \geq 0$, since the subtracted terms enforce the restriction to a half-space.  Multiplying by $t^{k+1}$ and summing over $k$ gives, after some rearrangement,
\begin{equation} (1-tS(z))H(z,t) = 1 - t\sum_{j=0}^{a-1} [z^{< -j}]S(z) \cdot [z^j]H(z,t). \label{eq:halfkernel} \end{equation}

This functional equation is more difficult to deal with than the unrestricted kernel equation~\eqref{eq:unker} because there are $a$ sub-series extractions of the unknown function $H(z,t)$ on the right-hand side.   A detailed analysis of the kernel $1-tP(z)=0$, carried out by Banderier and Flajolet~\cite{BanderierFlajolet2002}, shows that there are precisely $a$ roots $z_1(t),\dots,z_a(t)$ in $z$ which are analytic in a slit-neighbourhood\footnote{A slit-neighbourhood of the origin is a neighbourhood of the origin with a line segment from the origin to infinity removed (to allow for branch cuts).} of the origin and have a constant term of zero in their Puiseux series expansions (there are an additional $b$ branches which approach infinity as $t$ approaches zero).  Substituting each of these into Equation~\eqref{eq:halfkernel} gives a system of $a$ equations with $a$ unknown terms $[z^j]H(z,t)$. The form of the system shows that it can always be solved for the $[z^j]H(z,t)$, giving an explicit expression for the generating function.

\begin{theorem}[{Banderier and Flajolet~\cite[Theorem 2]{BanderierFlajolet2002}}]
\label{thm:halfalg}
The bivariate generating function $H(z,t)$ is algebraic, and has the representation 
\[ H(z,t) = \frac{\prod_{j=1}^a(z-z_j(t))}{z^a(1-tS(z))}. \]
In particular, the generating function for the total number of walks of length $k$ is the algebraic function
\[ H(1,t) = \frac{1}{1-t|\mS|} \displaystyle\prod_{j=1}^a (1-z_j(t)), \]
and the generating function for the number of walks of length $k$ which end at $z=0$ is the algebraic function
\[ H(0,t) = \frac{(-1)^{a-1}}{s_{-a}t} \displaystyle\prod_{j=1}^a z_j(t). \]
\end{theorem}

Note that the generating functions $H(1,t)$ and $H(0,t)$ are analytic at the origin as the $z_j(t)$ are algebraic conjugates.
\smallskip

\begin{example}[Dyck Paths] 
\label{ex:freeDyck}
 Let $\mS = \{-1,1\}$ and consider walks which stay in the half-space $\mR = \mathbb{N}$.  Here we have $S(x) = x^{-1} + x$, and the kernel equation~\eqref{eq:halfkernel} becomes
\[ (1-t(\ox+x))H(x,t) = 1 - t \ox H(0,t) \]
or, equivalently,
\begin{equation} (x-t(1+x^2))H(x,t) = x - tH(0,t). \label{eq:ballot}\end{equation}
Solving $x-t(1+x^2) = 0$ for $x$ using the quadratic formula gives two solutions
\[ x_1(t) = \frac{1 - \sqrt{1-4t^2}}{2t} \qquad \qquad x_2(t) = \frac{1 + \sqrt{1-4t^2}}{2t},\]
of which $x_1(t)$ is a power series in $t$ with a constant term of zero.  Thus, one can substitute $x=x_1(t)$ into Equation~\eqref{eq:ballot} to obtain
\[ H(0,t) =  \frac{1 - \sqrt{1-4t^2}}{2t^2}, \]
so 
\[ H(x,t) = \frac{x-\frac{1 - \sqrt{1-4t^2}}{2t}}{x-t(1+x^2)} = \frac{1-2xt-\sqrt{1-4t^2}}{2t(t+tx^2-x)}. \]
Since
\[ H(t) := H(1,t) = \frac{1}{2t}\left(\frac{\sqrt{1-4t^2}}{1-2t}-1\right)\]
has the local expansion $\sqrt{2}(1-2t)^{-1/2} + O(1-2t)$ at its dominant singularity $t=1/2$, the methods of analytic combinatorics imply that the sequence counting the number of walks in this model ending anywhere has the asymptotic expansion
\[ [t^k]H(t) = 2^k \cdot k^{-1/2}\left(\frac{\sqrt{2}}{\sqrt{\pi}} + O(k^{-1/2})\right). \]
\vspace{-0.3in}

\end{example}

Theorem~\ref{thm:halfalg} gives an explicit representation of the generating functions for the total number of walks (and walks returning to the boundary of $\mR$) for half-plane models, and allows one to determine asymptotics of the associated counting sequences.  In this sense, the study of lattice path models in a half-space is essentially solved.  The next natural generalization is to consider walks in a quarter plane, and as mentioned in the introduction of this thesis such models have a wide array of applications.

\section{Lattice Walks in the Quarter Plane} 
\label{sec:QPwalks}

\begin{figure}
    \centering
    \includegraphics[width=0.7\linewidth]{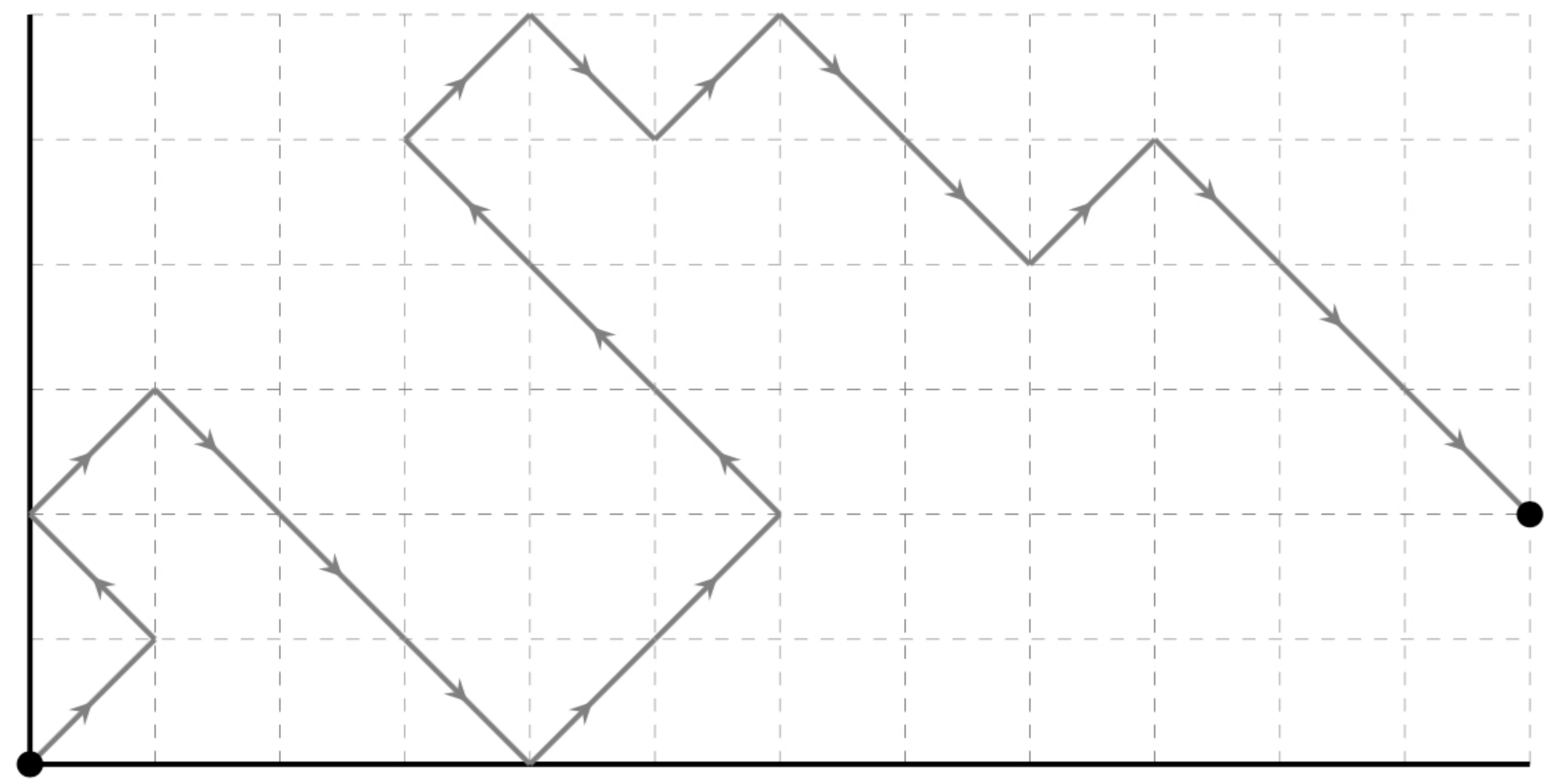}
    \caption[A lattice walk of length 20 restricted to the quarter plane.]{A lattice walk of length 20 using the steps $\{(1,-1),(1,1),(-1,1)\}$, restricted to the quarter plane.}
\end{figure}

Consider now a model defined by step set $\mS \subset \mathbb{Z}^2$ restricted to the quarter plane $\mR=\mathbb{N}^2$.  Although lattice walks in a half-space always have algebraic generating functions, it was shown by Bousquet-Mélou and Petkov{\v{s}}ek~\cite{Bousquet-MelouPetkovsek2003} that there are models staying in the quarter plane whose generating functions are non-D-finite. Much work in this area has focused on walks which take \emph{unit} or \emph{short} steps; that is, models with step sets $\mS \subseteq\{\pm1,0\}^2$.    The restriction to short steps bounds the degree of the kernel to be at most two, allowing its roots to be determined explicitly and causing the kernel equation to have a relatively simple form.  These models were originally studied via kernel method techniques in a probabilistic context by Fayolle and Iasnogorodski~\cite{FayolleIasnogorodski1979} and Fayolle et al.~\cite{FayolleIasnogorodskiMalyshev1999}, and the treatment below begins by following a now standard method of argument popularized for quarter plane walks by the combinatorial work of Bousquet-Mélou and Mishna~\cite{Bousquet-MelouMishna2010}.  We deal exclusively with short step models in this chapter.

Many step sets $\mS \subset \{\pm1,0\}^2$ lead to models which contain no non-empty walks, or models which are combinatorially isomorphic to models restricted only to a half-space\footnote{For instance, the model defined by $\mS = \{(-1,0), (0,1), (0,-1)\}$ can never take its first step and is thus isomorphic to the model solved in Example~\ref{ex:freeDyck}.}.  Thus, in the following we assume that any set of steps $\mS$ contains some steps with $\pm1$ in their first coordinates and some steps (possibly the same) with $\pm1$ in their second coordinates.

\subsection{The Kernel Equation and Group of a Walk}

Analogously to the unrestricted and half-space cases, given a step set $\mS \subseteq\{\pm1,0\}^2$ we use the multivariate generating function
\[ Q(x,y,t) :=\sum_{i,j,k \geq 0} q_{i,j,k} x^iy^j t^k,\]
where $q_{i,j,k}$ counts the number of walks of length $k$ taking steps from $\mS$ which stay in the non-negative quadrant and end at the point $(i,j)$, along with the characteristic polynomial 
\[ S(x,y) :=\sum_{(i,j) \in \mS}x^i y^j \in \mathbb{Q}\left[x,\overline{x},y,\overline{y}\right]. \]
Again the recursive structure of a walk of length $k$ gives a kernel equation satisfied by $Q(x,y,t)$.  As $Q(0,y,t)$ (respectively $Q(x,0,t)$) gives the generating function of walks ending on the $y$-axis (respectively $x$-axis), and $\mS$ is restricted to contain only unit steps, the kernel equation in the quarter plane becomes
\begin{equation} xy(1 - tS(x,y))Q(x,y,t) = xy - tI(y) - tJ(x) + \epsilon t Q(0,0,t),\label{eq:QPkernel} \end{equation}
where 
\[ I(y) = y\left([x^{-1}]S(x,y)\right)Q(0,y,t), \qquad J(x) = x\left([y^{-1}]S(x,y)\right)Q(x,0,t), \]
and
\[  \epsilon = \begin{cases} 1 & \text{ if } (-1,-1) \in \mS \\ 0 & \text{ otherwise } \end{cases} \]
($\epsilon$ compensates for subtracting off walks ending at the origin twice in other terms).  The additional variable present in the kernel $K(x,y,t) = 1 - tS(x,y)$ complicates the analysis by forcing one to consider algebraic surfaces solving $K(x,y,t)$ instead of algebraic curves (as in the half-space case).  

\begin{figure}
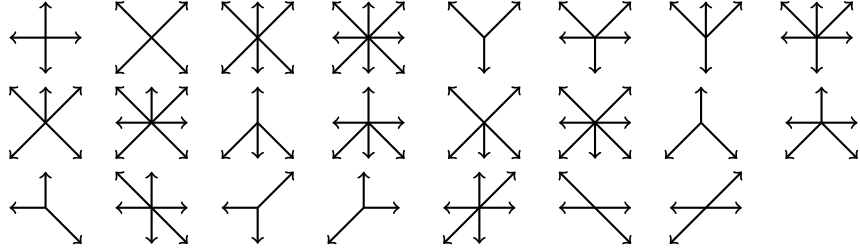

      \center 
      \begin{tabular}{cccccccc}
          \diagrF{N,S,E,W}& \diagrF{NE,SE,NW,SW}& \diagrF{N,S,NE,SE,NW,SW}&\diagrF{N,S,E,W,NW,SW,SE,NE}& \diagrF{NE,NW,S}&\diagrF{NE,NW,E,W,S}&\diagrF{N,NW,NE,S}& \diagrF{N,NW,NE,E,W,S} \\
          \diagrF{N,NE,NW,SE,SW}&\diagrF{N,E,W,NE,NW,SE,SW}&\diagrF{N,S,SE,SW}&\diagrF{N,E,W,S,SW,SE}&\diagrF{NE,NW,SE,SW,S}& \diagrF{NE,NW,E,W,SE,SW,S}& \diagrF{N,SE,SW} & \diagrF{N,E,W,SE,SW}\\ 
          \diagrF{N,W,SE}& \diagrF{NW,SE,N,S,E,W}& \diagrF{NE,W,S}&\diagrF{N,E,SW} &\diagrF{N,NE,E,S,SW,W} & \diagrF{E,SE,W,NW}&\diagrF{NE,E,SW,W}\\
      \end{tabular} 
      \caption[The 23 short step sets defining non-isomorphic quarter plane models with finite group]{The 23 short step sets defining non-isomorphic quarter plane models with finite group $\mG$.  Each half-arrow represents an element of the set $\{\pm1,0\}^2$, and each collection of half-arrows with common base point defines a step set.} \label{tab:finQP}
\end{figure}

This complication led Bousquet-Mélou~\cite{Bousquet-Melou2005}, followed by Bousquet-Mélou and Mishna~\cite{Bousquet-MelouMishna2010}, to borrow the notion of the \emph{group of a model} from the probabilistic studies of Fayolle et al.~\cite{FayolleIasnogorodskiMalyshev1999}.  If we define the Laurent polynomials $A_j(y)$ and $B_j(x)$ for $j \in \{-1,0,1\}$ by
\[ S(x,y) =  xA_1(y) + A_0(y) + \overline{x}A_{-1}(y) = yB_1(x) + B_0(x) + \overline{y} B_{-1}(x), \]
then the bi-rational transformations $\Psi$ and $\Phi$ of the plane defined by 
\[ \Psi : (x,y) \mapsto \left( \overline{x} \frac{A_{-1}(y)}{A_1(y)}, y \right) \qquad\qquad \Phi : (x,y) \mapsto \left( x, \overline{y} \frac{B_{-1}(x)}{B_1(x)} \right), \]
fix $S(x,y)$ and thus $K(x,y,t)$.  The group $\mG$ of a model is defined to be the group of bi-rational transformations of the $(x,y)-$plane generated by the involutions $\Psi$ and $\Phi$.   Bousquet-M\'elou and Mishna showed that, up to isomorphism, there are only 79 distinct models with unit steps: 23 whose corresponding group is finite and 56 whose corresponding group is infinite (see Figures~\ref{tab:finQP} and~\ref{tab:infQP}).

As both $\Psi$ and $\Phi$ are involutions, to prove that $\mG$ is finite it is sufficient to find a natural number $n$ such that composing the group element $\Psi \circ \Phi$ with itself $n$ times yields the identity, a feat easily accomplished in a computer algebra system (assuming such $n$ exists and is of reasonable size).  To prove that a group is of infinite order one can find an explicit point $(x_0,y_0)$ in the plane whose image under $\mG$ has infinite size, or show that the mapping $\Psi \circ \Phi$ never composes to the identity by analyzing its Jacobian at fixed points (Bousquet-M\'elou and Mishna do both for the various cases in Figure~\ref{tab:infQP}).

\begin{figure}
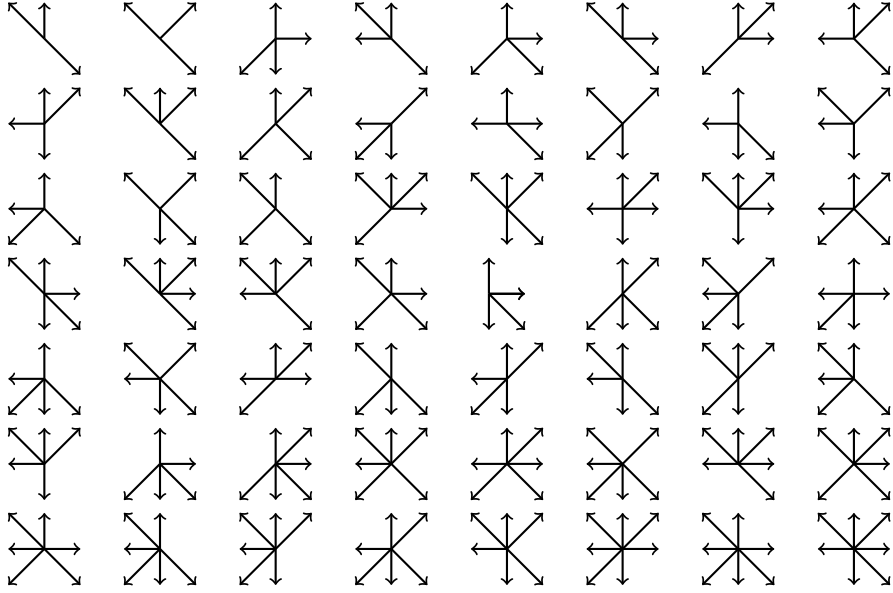

    \center
    \begin{tabular}{cccccccc}
    \diagrF{NW,N,SE} &  \diagrF{NW,NE,SE} &  \diagrF{N,E,S,SW} &  \diagrF{N,SE,W,NW} &  \diagrF{N,E,SE,SW} & \diagrF{N,E,SE,NW} & \diagrF{N,NE,E,SW} & \diagrF{N,NE,SE,W} \\
    \diagrF{N,NE,S,W} & \diagrF{N,NE,SE,NW} & \diagrF{N,NE,SE,SW} & \diagrF{NE,S,SW,W} & \diagrF{N,E,SE,W} & \diagrF{NE,S,SW,NW} & \diagrF{N,SE,S,W} & \diagrF{NE,S,W,NW} \\
    \diagrF{N,SE,SW,W} & \diagrF{NE,SE,S,NW} & \diagrF{N,SE,SW,NW} & \diagrF{N,NE,E,SW,NW} &\diagrF{N,NE,SE,S,NW} & \diagrF{N,NE,E,S,W} & \diagrF{N,NE,E,S,NW} & \diagrF{N,NE,SE,SW,W} \\
    \diagrF{N,E,SE,S,NW} & \diagrF{N,NE,E,SE,NW} & \diagrF{N,NE,SE,W,NW} & \diagrF{N,E,SE,SW,NW} &\diagrF{N,E,SE,S,E} & \diagrF{N,NE,SE,S,SW} & \diagrF{NE,S,SW,W,NW} & \diagrF{N,E,S,SW,W} \\ 
    \diagrF{N,SE,S,SW,W} & \diagrF{NE,SE,S,W,NW} & \diagrF{N,NE,E,SW,W} & \diagrF{N,SE,S,SW,NW} & \diagrF{N,NE,S,SW,W} & \diagrF{N,SE,S,W,NW} & \diagrF{N,NE,S,SW,NW} & \diagrF{N,SE,SW,W,NW} \\ 
    \diagrF{N,NE,S,W,NW} & \diagrF{N,E,SE,S,SW} & \diagrF{N,NE,E,SE,S,SW} & \diagrF{N,NE,SE,SW,W,NW} & \diagrF{N,NE,E,SE,SW,W} & \diagrF{NE,SE,S,SW,W,NW} & \diagrF{N,NE,E,SE,W,NW} & \diagrF{N,NE,E,SE,SW,NW} \\
    \diagrF{N,E,SE,SW,W,NW} & \diagrF{N,SE,S,SW,W,NW} & \diagrF{N,NE,S,SW,W,NW} & \diagrF{N,NE,SE,S,SW,W} & \diagrF{N,NE,SE,S,W,NW} & \diagrF{N,NE,E,S,SW,W,NW} & \diagrF{N,E,SE,S,SW,W,NW} & \diagrF{N,NE,E,SE,S,W,NW} \\
    \end{tabular} 
    \caption[The 56 short step sets defining non-isomorphic quarter plane models with infinite group]{The 56 short step sets defining non-isomorphic quarter plane models with infinite group $\mG$.}  \label{tab:infQP}
\end{figure}

\subsection{Generating Function Representations}
For notational convenience, given $g \in \mG$ and a Laurent polynomial $A(x,y)$ we define $g(A(x,y)) := A(g(x,y))$.  If $\mG$ has size $2k$ then any element $g \in \mG$ can be written uniquely as either 
\[ g = \Psi \circ \Phi \circ \cdots \circ \Psi \circ \Phi \qquad \text{ or } \qquad g = \Psi \circ \Phi \circ \cdots \circ \Psi \circ \Phi \circ \Psi,\] 
where there are $0 \leq r < 2k$ terms in the composition, and we define $\sgn(g) := (-1)^r$.  In addition to determining whether the group of each model is finite, Bousquet-M\'elou and Mishna also proved that 22 of the 23 models with finite group admit D-finite generating functions.  Central to their argument is the following result.

\begin{proposition}[{Bousquet-M\'elou and Mishna~\cite[Proposition 5]{Bousquet-MelouMishna2010}}] 
\label{thm:osQP}
Assume that the group $\mG$ is finite.  Then 
\begin{equation} \sum_{g \in \mG} \sgn(g)g(xy Q(x,y,t)) = \frac{1}{K(x,y,t)} \sum_{g \in \mG} \sgn(g)g(xy) \label{eq:osQP}. \end{equation}
\end{proposition}

\begin{proof}
Define $x'$ and $y'$ by  $(x',y) = \Psi(x,y)$ and $(x,y') = \Phi(x,y)$.  Applying the maps $\Psi$ and $\Phi$ successively to the kernel equation gives
\begin{align*}  
(id)&& xyK(x,y)Q(x,y,t) &= xy - tI(y) - tJ(x) + \epsilon \cdot t Q(0,0,t) \\
(\Psi)&& x'yK(x,y)Q(x',y,t) &= x'y - tI(y) - tJ(x') + \epsilon \cdot t Q(0,0,t) \\
(\Phi\Psi)&& x'y'K(x,y)Q(x',y',t) &= x'y' - ty'I(y') - tJ(x') + \epsilon \cdot t Q(0,0,t),
\end{align*}
as both $\Psi$ and $\Phi$ fix $K(x,y)$.  Note that $-tI(y)$ and $-tJ(x')$ both appear on the right-hand sides of successive equations, so taking an alternating sum of these three equations cancels those terms.  In fact, since $\Psi$ and $\Phi$ each fix one variable, continuing to compose the group generators in this manner and taking an alternating sum of the resulting equations cancels each unknown function of the form $I(Y)$ or $J(X)$ arising on the right-hand side.  This follows from the finiteness of the group, as the compositions of group elements eventually return to the identity.  The $\epsilon t Q(0,0,t)$ term is also canceled as the group has even order, and the resulting equation gives the theorem.
\end{proof}

The procedure described in the proof of Theorem~\ref{thm:osQP} is known in the literature as the \emph{orbit sum method}, as one sums the kernel equation over orbits of the group generators.  Examining Equation~\eqref{eq:osQP} for 19 of the 23 cases with finite group shows that the only term on the left-hand side with non-negative powers of $x$ and $y$ is $xyQ(x,y,t)$.  A short argument then shows the following.

\begin{theorem}[{Bousquet-M\'elou and Mishna~\cite[Proposition 8]{Bousquet-MelouMishna2010}}]
\label{thm:pospt}
Let $\mS$ be one of the 19 step sets with finite group which is not listed in Figure~\ref{fig:zsQP}.  Then $Q(x,y,t) = [x^{\geq 0}][y^{\geq0}]R(x,y,t)$, where $R(x,y,t)$ is the rational function
\[ R(x,y,t) = \frac{1}{K(x,y,t)} \sum_{g \in \mG} \sgn(g) g(xy), \]
and is thus D-finite. 
\end{theorem}

\begin{figure}
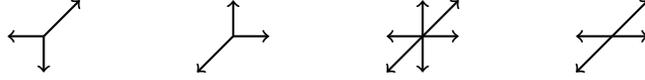

\[\diagrF{S,W,NE} \qquad\qquad \diagrF{N,E,SW} \qquad\qquad \diagrF{N,E,S,W,NE,SW} \qquad\qquad \diagrF{W,E,SW,NE}\]
\caption{The four walks to which Theorem~\ref{thm:pospt} does not apply} \label{fig:zsQP}
\end{figure}

The four walks in Figure~\eqref{fig:zsQP} have both sides of their associated orbit sum equation~\eqref{eq:osQP} identically zero due to an element of the group fixing the product $xy$ while having negative sign.   Bousquet-M\'elou and Mishna proved that the first three walks in Figure~\ref{fig:zsQP} are algebraic (and thus can be written as diagonals of rational functions) by taking a modified `half-orbit sum' and performing a detailed analysis, but the final model---known as \emph{Gessel's model}---was classified by a computational approach outlined below.  Bousquet-M\'elou and Mishna conjectured\footnote{Previous work of Mishna~\cite{Mishna2009} also made this conjecture.} that all 56 walks with an infinite group had non-D-finite univariate generating functions, but did not prove this for any model.

Combining Theorem~\ref{thm:pospt} with Proposition~\ref{prop:postodiag} gives diagonal representations for the generating functions of the number of walks ending anywhere in the quarter plane, returning to the origin, or ending on either boundary axis.

\begin{theorem}
\label{thm:nzOrbDiag}
Let $\mS$ be one of the 19 step sets with finite group which is not listed in Figure~\ref{fig:zsQP}.  Then for $a,b \in \{0,1\}$,
\[ Q(a,b,t) = \Delta \left( \frac{O(\ox,\oy)}{(1-x)^a(1-y)^b(1-txyS(\ox,\oy))} \right), \]
where $O$ is the orbit sum 
\[ O(x,y) = \sum_{g \in \mG} \sgn(g) g(xy). \]
\end{theorem}

Thus, we have represented the generating functions of these lattice path models by explicit rational diagonals.  From this expression an annihilating linear differential equation of each generating function can be computed using creative telescoping.
\smallskip

\begin{example}[Simple Walks in the Quarter Plane]
\label{ex:simpleqp}
Consider the quarter plane model defined by the steps $\mS = \left\{(\pm1,0),(0,\pm1)\right\}$.  Here the kernel equation is
\[ K(x,y,t)xyQ(x,y,t) = xy - tx Q(x,0,t) - ty Q(0,y,t) \]
and the group $\mG$ is the dihedral group of order 4 generated by the involutions
\[ \Phi :(x,y) \mapsto (\ox,y) \qquad\qquad \Psi:(x,y) \mapsto (x,\oy). \]
Taking an orbit sum of the kernel equation gives
\[ xyQ(x,y,t) - (\ox y)Q(\ox,y,t) + (\ox\oy)Q(\ox,\oy,t) - (x\oy)Q(x,\oy,t) = \frac{xy - \ox y + \ox\oy - x\oy}{1-t(x+y+\ox+\oy)}. \]
Since the only term on the left-hand side of this equation with non-negative powers of $x$ and $y$ is $xyQ(x,y,t)$, it follows that
\[ Q(x,y,t) = [x^{\geq 0}y^{\geq 0}] \frac{xy + \ox y + x\oy + \ox\oy}{xy(1-t(x+y+\ox+\oy))} =  [x^{\geq 0}y^{\geq 0}] \frac{(x-\ox)(y-\oy)}{xy(1-t(x+y+\ox+\oy))} \]
and we obtain
\[ Q(1,1,t) = \Delta\left( \frac{(x-\ox)(y-\oy)}{(1-x)(1-y)\ox\,\oy(1-txy(x+y+\ox+\oy))} \right) = \Delta\left( \frac{(1+x)(1+y)}{1-txy(x+y+\ox+\oy)} \right). \]
Using the Mathematica package of Koutschan~\cite{Koutschan2010}, which implements creative telescoping algorithms, we can use this diagonal expression to compute a differential operator
\[ \mathcal{L} := t^2(4t-1)(4t+1)\partial_t^3 + 2t(4t+1)(16t-3)\partial_t^2 + (224t^2+28t-6)\partial_t + (12+64t) \]
which annihilates $Q(1,1,t)$.  We will show that the number of lattice walks in the class has asymptotics of the form $\frac{4}{\pi} \cdot \frac{4^k}{k}$, meaning $Q(1,1,t)$ is transcendental by Theorem~\ref{thm:algAsm}.
\end{example}

\subsection{A Computer Algebra Approach} 
\label{sec:CompAlg}

We now describe a computational approach to asymptotics and the classification of generating functions, which has been applied to several problems in lattice path combinatorics~\cite{KauersZeilberger2008,KauersZeilberger2011,KauersKoutschanZeilberger2009,BostanKauers2010} and was used by Bostan and Kauers~\cite{BostanKauers2009} to conjecture asymptotics for the 23 quarter plane lattice path models in Figure~\ref{tab:finQP}.  The basic idea is to use the recurrence relation
\begin{equation} 
q_{i,j,k} = \sum_{(a,b) \in \mS} \epsilon_{i-a,j-b} q_{i-a,j-b,k-1}, \qquad\qquad \epsilon_{i-a,j-b} = \left\{
    \begin{array}{ll}
      0 & : i-a < 0 \text{ or } j-b<0\\
      1 & : \text{ otherwise}
  \end{array}
\right. \label{eq:seqrec}\end{equation}
to generate a truncation of the generating function $Q(x,y,t)$ and use this truncation to guess an algebraic or differential equation which the full generating function will satisfy.  Such guessing can be done efficiently through fast algorithms for Pad\'e-Hermite Approximants.

\subsubsection{Padé-Hermite Approximants}
Let $K$ be a field.  Given a vector of formal power series 
\[ \mathbf{F} = (F_1,\dots,F_r) \in K[[t]]^r \]
and a vector of natural numbers 
\[ \mathbf{d} = (d_1,\dots,d_r) \in \mathbb{N}^r,\]
a vector $\mathbf{P} = (P_1,\dots,P_r) \neq 0$ of polynomials in $K[t]$  is called a \emph{Padé-Hermite approximant of type $\mathbf{d}$ for $\mathbf{F}$} if
\begin{enumerate}
\item $\mathbf{P} \cdot \mathbf{F} = O(t^\sigma)$ where $\sigma = \sum_{i=0}^r (d_i+1) - 1$ (i.e., the lowest non-zero term in the dot product has exponent at least $\sigma$);
\item the degree of $P_i$ is at most $d_i$ for all $1 \leq i \leq n$.
\end{enumerate}

\begin{example}
Given a power series $F(t) \in \mathbb{Q}[[t]]$ and $d\in\mathbb{N}$, if one takes $F_k = F(t)^{k-1}$ and $\mathbf{d} = (d,\dots,d)$ in the above definition then $\mathbf{P}$ is a Pad\'e-Hermite approximant of type $\mathbf{d}$ if and only if $F(t)$ satisfies an algebraic equation of degree $r$ with coefficients of degree at most $d$, up to order $t^{rd+r-1}$:
\[ P_r(t)F(t)^{r-1} + \cdots + P_2(t)F(t) + P_1(t) = 0 \mod{t^{rd+r-1}}. \]
\end{example}
  
\begin{example}
Given a power series $F(t) \in \mathbb{Q}[[t]]$ and $d\in\mathbb{N}$, if one takes $F_k = \frac{d^k}{dt^k}F(t)$ and $\mathbf{d} = (d,\dots,d)$ in the above definition then $\mathbf{P}$ is a Pad\'e-Hermite approximant of type $\mathbf{d}$ if and only if $F(t)$ satisfies an linear differential equation of order $r$ with coefficients of degree at most $d$, up to order $t^{rd+r-1}$:
\[ P_r(t)\frac{d^r}{dt^r}F(t) + \cdots + P_2(t)\frac{d^2}{dt^2}F(t) + P_1(t)\frac{d}{dt}C(t) = 0 \mod{t^{rd+r-1}}. \]
\end{example}

% \begin{proposition}
% Every vector $\mathbf{F}$ admits a Pad\'e-Hermite approximant of type $\mathbf{d} = (d_1,\dots,d_r)$.
% \end{proposition}
% \begin{proof}
% By writing $P_i = \sum_{j=0}^{d_i} p_{i,j} t^j,$ the condition $\mathbf{P} \cdot \mathbf{F} = O(t^\sigma)$ gives a homogeneous linear system with $\sigma = \sum_{i=1}^r (d_i+1) - 1$ equations and $\sigma+1$ unknowns $p_{i,j}$.
% \end{proof}

Pad\'e-Hermite approximants always exist, and can be computed efficiently.

\begin{theorem}[{Beckermann and Labahn~\cite{BeckermannLabahn1994}}] 
\label{thm:becklab}
Given the vector $\mathbf{F}$ it is possible to calculate a Pad\'e-Hermite approximant of type $\mathbf{d}$ in $O(MM(r,\sigma)\log \sigma)$ operations in the field $K$, where $MM(r,\sigma)$ is the number of operations required to multiply two $r \times r$ matrices whose entries are polynomials of degree at most $\sigma$ modulo $t^{\sigma+1}$.  
%The Coppersmith--Winograd algorithm for matrix multiplication implies $O(MM(r,\sigma)\log \sigma) = O(r^{2.376} M(\sigma)\log \sigma)$, where $M(\sigma)$ is the cost of multiplying two degree $\sigma$ polynomials over the field $K$.  If $K=\mathbb{Q}$ and $\mathbf{d} = (d,d,\dots,d)$ then $O(MM(r,\sigma)\log \sigma) = O(r^{3.376} d \log^2 (rd) \log\log (rd))$.
\end{theorem}

After determining an algebraic or differential equation for a truncated series, several techniques can be used to give confidence that the generating function under consideration satisfies this equation.  In addition to simply computing additional terms of the generating function and verifying that the additional terms also satisfy the equation, Section 2.4 of Bostan and Kauers~\cite{BostanKauers2009} gives several algebraic and analytic heuristics.  Guessing algebraic or differential equations from a list of initial coefficients has been implemented in the Maple package GFUN~\cite{SalvyZimmermann1994}, and the interested reader can refer to Chapter 7 of Bostan et al.~\cite{BostanChyzakGiustiLebretonLecerfSalvySchost2017} for additional details on algorithms for Padé-Hermite approximations.

\subsection{Asymptotics of D-Finite Quarter Plane Models}

\begin{table}
\centering
\begin{tabular}{ | c | c @{ \hspace{0.01in} }@{\vrule width 1.2pt }@{ \hspace{0.01in} } c | c @{ \hspace{0.01in} }@{\vrule width 1.2pt }@{ \hspace{0.01in} } c | c |  }
  \hline
   $S$ & Asymptotics & $S$ & Asymptotics & $S$ & Asymptotics \\ \hline
  &&&&& \\[-5pt] 
  \diag{N,S,E,W}  & $\frac4\pi \cdot \frac{4^k}k$ &
  \diag{N,NE,NW,SE,SW}  & $\frac{\sqrt{5}}{3\sqrt{2\pi}} \cdot \frac{5^k}{\sqrt{k}}$ &
  \diag{NE,SE,NW,SW}  & $\frac2\pi \cdot \frac{4^k}k$ \\ [+2mm]
  %%%%%%%%%%%%%
  \diag{N,E,W,NE,NW,SE,SW} & $\frac{\sqrt{7}}{3\sqrt{3\pi}} \cdot \frac{7^k}{\sqrt{k}}$ & 
  \diag{N,S,NE,SE,NW,SW} & $\frac{\sqrt{6}}\pi \cdot \frac{6^k}k$  &
  \diag{N,S,SE,SW} & $\frac{3\sqrt{3}\cdot B_k}{\pi} \cdot \frac{(2\sqrt{3})^k}{k^2}$ \\ [+2mm]
  %%%%%%%%%%%%%
  \diag{N,S,E,W,NW,SW,SE,NE}  & $\frac{8}{3\pi} \cdot \frac{8^k}k$ &
  \diag{N,E,W,S,SW,SE}  & $\frac{\sqrt{3}(1+\sqrt{3})^{7/2}}{2\pi} \cdot \frac{(2+2\sqrt{3})^k}{k^2}$  &
  \diag{NE,NW,S}  & $\frac{\sqrt{3}}{2\sqrt{\pi}} \cdot \frac{3^k}{\sqrt{k}}$ \\ [+2mm]
  %%%%%%%%%%%%%
  \diag{NE,NW,SE,SW,S}  & $\frac{12\sqrt{30}}{\pi} \cdot \frac{(2\sqrt{6})^k}{k^2}$ &
  \diag{NE,NW,E,W,S}  & $\frac{\sqrt{5}}{2\sqrt{2\pi}} \cdot \frac{5^k}{\sqrt{k}}$ &
  \diag{NE,NW,E,W,SE,SW,S}  & ${\scriptstyle \frac{\sqrt{6(379+156\sqrt{6})(1+\sqrt{6})^7}}{5\sqrt{95}\pi} \cdot \frac{(2+2\sqrt{6})^k}{k^2}}$ \\ [+2mm]
  %%%%%%%%%%%%%
  \diag{N,NW,NE,S}  & $\frac4{3\sqrt{\pi}} \cdot \frac{4^k}{\sqrt{k}}$ &
  \diag{N,SE,SW} & $\frac{4 \cdot A_k}{\pi} \cdot \frac{(2\sqrt{2})^k}{k^2}$  &
  \diag{N,NW,NE,E,W,S} & $\frac{2\sqrt{3}}{3\sqrt{\pi}} \cdot \frac{6^k}{\sqrt{k}}$ \\ [+2mm]
  %%%%%%%%%%%%%
  \diag{N,E,W,SE,SW}  & $\frac{\sqrt{8}(1+\sqrt{2})^{7/2}}{\pi} \cdot \frac{(2+2\sqrt{2})^k}{k^2}$ &
  \diag{N,W,SE}   & $\frac{3\sqrt{3}}{2\sqrt{\pi}} \cdot \frac{3^k}{k^{3/2}}$ &
  \diag{N,E,SW}   & $\frac{3\sqrt{3}}{\sqrt2\Gamma(1/4)} \cdot \frac{3^k}{k^{3/4}}$ \\ [+2mm]
  %%%%%%%%%%%%%
  \diag{NW,SE,N,S,E,W}   & $\frac{3\sqrt{3}}{2\sqrt{\pi}} \cdot \frac{6^k}{k^{3/2}}$  &
  \diag{N,NE,E,S,SW,W}   & $\frac{\sqrt{6\sqrt{3}}}{\Gamma(1/4)} \cdot \frac{6^k}{k^{3/4}}$ &
  \diag{NE,W,S}   & $\frac{2\sqrt{2}}{\Gamma(1/4)} \cdot \frac{3^k}{k^{3/4}}$ \\ [+2mm]
  %%%%%%%%%%%%%
  \diag{E,SE,W,NW}  & $\frac{8}{\pi} \cdot \frac{4^k}{k^2}$ &
  \diag{NE,E,SW,W}  & $\frac{4\sqrt3}{3\Gamma(1/3)} \cdot \frac{4^k}{k^{2/3}}$ &&\\ 
\hline
\end{tabular}
\vspace{-0.3in}

\[ {\scriptstyle A_k = 4(1-(-1)^k)+3\sqrt{2}(1+(-1)^k), \quad B_k = \sqrt{3}(1-(-1)^k)+2(1+(-1)^k), \quad C_k = 12/\sqrt{5}(1-(-1)^k)+\sqrt{30}(1+(-1)^k)} \]
{}\vspace{-0.3in}

\caption[Asymptotics of D-finite short step quarter plane models]{Asymptotics for the 23 D-finite models; these are proven in Chapter~\ref{ch:QuadrantLattice}.} \label{tab:shortQPasm}
\end{table} 

Bostan and Kauers~\cite[Table 1]{BostanKauers2009} were able to guess algebraic and/or differential equations for each of the 23 models with finite group in Figure~\ref{tab:finQP} (and could not find such equations for the 56 models with infinite group in Figure~\ref{tab:infQP}).  From this they conjectured asymptotics of the form $q_k \sim C \cdot k^{\alpha}\cdot \rho^k$ for the total number of walks in each class, getting around the connection problem by using numerical approximations to guess the constant $C$.  Their results\footnote{Three of the models involve periodic terms $A_k, B_k, $ and $C_k$, and the guesses of Bostan and Kauers only included the values of these constants when $k$ is even.} are presented in Table~\ref{tab:shortQPasm}.  Bostan and Kauers~\cite{BostanKauers2010} also used computer-algebraic tools to prove that the right-most model of Figure~\ref{fig:zsQP} (Gessel's model, the only model with finite group whose generating function was not proven to be D-finite by Bousquet-Mélou and Mishna) has an algebraic generating function, which they determined explicitly.  The generating function of Gessel's model was later proven to be algebraic by several other arguments~\cite{BostanKurkovaRaschel2016,Bousquet-Melou2016a,BernardiBousquet-MelouRaschel2016}.

\subsubsection{Rigorous Results on Asymptotics}

Bousquet-Mélou and Mishna~\cite{Bousquet-MelouMishna2010} determined explicit expressions for the number of walks in the models defined by step sets\footnote{The first three models here have algebraic generating functions, while the next two have transcendental trivariate generating functions $Q(x,y,t)$, with algebraic specializations $Q(1,1,t)$.  The final model does not have an algebraic specialization, but the coefficients $q_{i,j,k}$ are Gosper summable (see Bousquet-Mélou and Mishna~\cite[Proposition 11]{Bousquet-MelouMishna2010} for details).} 
\[ \diagrF{S,W,NE} \qquad\qquad \diagrF{N,E,SW} \qquad\qquad \diagrF{N,E,S,W,NE,SW}  \qquad\qquad \diagrF{N,S,E,W,NW,SE} \qquad\qquad \diagrF{W,N,SE}
\qquad\qquad \diagrF{E,W,NW,SE} \]
and asymptotics of Gessel's model follows from the work of Bostan and Kauers~\cite{BostanKauers2010}.  Fayolle and Raschel~\cite{FayolleRaschel2012} outline a method which in principle allows one to determine the exponential growth rate $\rho$ of the 23 D-finite models (and many of the 56 models with infinite group), and found $\rho$ in several examples.  The models 
\[\diagrF{NW,NE,SE} \qquad\qquad
\diagrF{NW,N,E,SE} \qquad\qquad
\diagrF{NW,N,NE,E,SE} \qquad\qquad
\diagrF{NW,N,SE} \qquad\qquad
\diagrF{NW,N,NE,SE} \]
which admit non-D-finite generating functions $Q(1,1,t)$ are known as \emph{singular models}, and asymptotics of their counting sequences were worked out by Mishna and Rechnitzer~\cite{MishnaRechnitzer2009} and Melczer and Mishna~\cite{MelczerMishna2014}.  Exponential growth of the 74 non-singular models with short steps can be determined from the work of Garbit and Raschel~\cite{GarbitRaschel2016}, which applies in much more general contexts.  

The probabilistic work of Denisov and Wachtel~\cite{DenisovWachtel2015} gave rise to a method which can be used to compute the exponential growth constant $\rho$ and growth exponent $\alpha$ for the number of walks which return to the origin, for the 74 non-singular models.  An algorithm for determining these constants was given by Bostan et al.~\cite{BostanRaschelSalvy2014}, who showed that the generating function for the number of walks returning to the origin is non-D-finite for the 51 non-singular models with infinite group.  Given a step set $\mS$ whose vector sum contains two negative coordinates\footnote{This does not occur for any of the D-finite models.}, work of Duraj~\cite[Example 7]{Duraj2014} implies that the constants $\rho$ and $\alpha$ are the same when enumerating walks returning to the origin and the total number of walks (ending anywhere) defined by the model.

\subsubsection{Asymptotics via Analytic Combinatorics in Several Variables}

Aside from the 7 D-finite models discussed in the last section, the conjectures of Bostan and Kauers were largely open for several years\footnote{Around the same time as the ACSV approach to lattice path asymptotics was being developed, Bostan et al.~\cite{BostanChyzakHoeijKauersPech2017} proved the guessed annihilating differential equations of Bostan and Kauers~\cite{BostanKauers2009} and used this to represent the generating functions of walks restricted to the quarter plane in terms of integrals of algebraic and ${}_2F_1$ hypergeometric functions.  These representations allow the asymptotics of some, but not all, of the short step quarter plane models to be determined; see Section 4.3 of Bostan et al.~\cite{BostanChyzakHoeijKauersPech2017} for details.}.  Melczer and Mishna~\cite{MelczerMishna2016} determined asymptotics using ACSV for models (in any dimension) restricted to an orthant whose step sets are symmetric over every axis.  This result applies to the models\footnote{Asymptotics of the model with step set $\{(\pm1,1),(\pm1,-1)\}$ can also be determined through a decomposition into two one-dimensional models.}
\[ \diagrF{N,S,E,W} \qquad\qquad \diagrF{NE,NW,SE,SW} \qquad\qquad \diagrF{NE,NW,SE,SW,N,S}  \qquad\qquad \diagrF{NE,NW,SE,SW,N,S,E,W} \]
in the quarter plane and is described in Chapter~\ref{ch:SymmetricWalks}.  More recently, Melczer and Wilson~\cite{MelczerWilson2016} generalized these results to determine asymptotics for all remaining D-finite models with short steps in the quarter plane.  This work is discussed in Chapter~\ref{ch:QuadrantLattice}, and also gives some asymptotic results for walks returning to either axis or the origin.

%%%%%%%%%%%%%%%%%%%%%%%%%%%
% Chapter 5
%%%%%%%%%%%%%%%%%%%%%%%%%%%
\chapter{Other Sources of Rational Diagonals}
\label{ch:OtherSources}

\setlength{\epigraphwidth}{4in}
\epigraph{The interplay between generality and individuality, deduction and construction, logic and imagination -- this is the profound essence of live mathematics\dots~In brief, the flight into abstract generality must start from and return again to the concrete and specific.}{Richard Courant, \emph{Mathematics in the Modern World}}

\epigraph{To many, mathematics is a collection of theorems. For me, mathematics is a collection of examples; a theorem is a statement about a collection of examples and the purpose of proving theorems is to classify and explain the examples\dots}{John B. Conway, \emph{Subnormal Operators}}

In order to further motivate the theory of analytic combinatorics in several variables, and provide examples beyond lattice path enumeration for further chapters, we now describe several domains of mathematics where rational diagonals appear.

\section{Binomial Sums}

One of the simplest examples of a rational diagonal is the bivariate function 
\[ F(x,y)= \frac{1}{1-x-y} = \sum_{(i,j) \in \mathbb{N}^2}\binom{i+j}{i}x^iy^j \]
seen in Chapter~\ref{ch:Background}.  The diagonal sequence of $F$, composed of the central binomial coefficients, is an elementary example of a \emph{binomial sum}.  Informally, the \emph{class of binomial sums} over a field $K$ is the smallest $K$-algebra of (possibly multivariate) sequences which:
\begin{itemize} 
	\item contains the geometric and binomial coefficient sequences;
	\item contains the Kronecker delta sequence $(1,0,0,0,\dots)$;
	\item is closed under indefinite summation;
	\item is closed under affine maps on the sequence indices.
\end{itemize}

For a formal construction, see Definition 1.1 of Bostan et al.~\cite{BostanLairezSalvy2017}.  One main result of that paper is the following.

\begin{theorem*}[{Bostan et al.~\cite[Theorem 3.5]{BostanLairezSalvy2017}}]
A univariate sequence $(u_k)$ is a binomial sum if and only if the generating function $U(z) = \sum_{k \geq 0} u_kz^k$ is the diagonal of a rational power series.
\end{theorem*}

The results of Bostan et al.~give an algorithm\footnote{Algorithm 1 of Bostan et al.~\cite{BostanLairezSalvy2017} gives a rational function $R(y_1,\dots,y_n,z) \in \mathbb{Q}(\by,z)$ such that the generating function of $(u_k)$ is the \emph{constant term extraction} $U(z) = [y_1^0 \cdots y_n^0]R(\by,z)$.  A simple argument then shows that $U(z)$ is the diagonal $\Delta R(\by, y_1y_2\cdots y_n \cdot z)$.  The Maple package of Lairez contains the command {\textsf sumtores} which returns a rational function $R(\by,z) \in \mathbb{Q}(\by,z)$ such that $U(z) = [y_1^{-1} \cdots y_n^{-1}]R(\by,z) = \Delta\left( y_1\cdots y_n R(\by, y_1y_2\cdots y_n \cdot z) \right)$.} which takes a univariate binomial sum $(u_k)$ and returns a rational function $F(\bz) \in \mathbb{Z}(\bz)$ such that the generating function of $(u_k)$ is $(\Delta F)(z)$, and a Maple package implementing these results was developed by Lairez\footnote{This Maple package is available from \url{https://github.com/lairez/binomsums}.}.

\begin{example}[Apéry Numbers]
\label{ex:Apery}
Apéry's celebrated proof~\cite{Apery1979} of the irrationality of $\zeta(3)$ relies on constructing two sequences of rational numbers whose ratios converge to $\zeta(3)$ at a rate which implies that $\zeta(3)$ is irrational.  Alfred van der Poorten's canonical report~\cite{Poorten1978} on the proof gives the following exercise: ``Be the first in your block to prove by a 2-line argument that $\zeta(3)$ is irrational'' by determining an algebraic relationship between the two sequences and determining the exponential growth of the sequence $(b_n)$ defined by
\[ b_n := \sum_{k=0}^n \binom{n}{k}^2\binom{n+k}{k}^2. \]
The integers $b_n$ are often referred to as the Apéry numbers (OEIS entry \href{https://oeis.org/A005259}{A005259}) and the Maple package of Lairez shows that their generating function satisfies
\[ B(z) = \Delta\left(\frac{1}{1-t(1+x)(1+y)(1+z)(1+y+z+yz+xyz)}\right). \]
Bostan et al.~\cite[Appendix B]{BostanBoukraaChristolHassaniMaillard2013} list four different rational diagonal expressions for the generating function of $b_n$ (two containing 5 variables, one containing 6 variables, and one containing 8 variables, none of which is the representation given here).  Apéry also presented a new elementary proof of the irrationality of $\zeta(2)$ which relies on asymptotics of the sequence $(c_n)$ defined by
\[ c_n := \sum_{k=0}^n \binom{n}{k}^2\binom{n+k}{k}. \]
The $c_n$ are also referred to as Apéry numbers (OEIS entry \href{https://oeis.org/A005258}{A005258}).  Apéry himself~\cite{Apery1983} noted that the generating function $C(z)$ of $(c_n)$ is the diagonal of two trivariate rational functions
\[ C(z) = \Delta\left(\frac{1}{1-(1+z)(x+y+xy)}\right) = \Delta\left(\frac{1}{1-x-y-z(1-x)(1-y)}\right),\]
and the Maple package of Lairez shows
\[C(z) = \Delta\left( \frac{1}{1-z(1+x)(1+y)(xy+y+1)}\right) .\]  
A recent paper of Hirschhorn~\cite{Hirschhorn2015} finds ``an expression for $\pi$ as a limit involving the golden ratio $\phi$'' by finding the dominant asymptotics of the Apéry numbers $c_n$.  We use the tools of analytic combinatorics in several variables to determine asymptotics of these sequences in Example~\ref{ex:Apery2}.  Furthermore, the results of Chapter~\ref{ch:EffectiveACSV} will be able to rigorously and automatically determine asymptotics for $(b_n)$ and $(c_n)$, which will be done in Examples~\ref{ex:Apery3a} and~\ref{ex:Apery3b}.
\end{example}

\section{Irrational Tilings}
We will see in Part~\ref{part:SmoothACSV} that the techniques of ACSV are (theoretically and computationally) simpler for multivariate rational functions whose power series expansions have all but a finite number of non-negative coefficients.  Unfortunately it is still unknown, even in the univariate case, how to decide this \emph{ultimate positivity problem}\footnote{Ouaknine and Worrell~\cite{OuaknineWorrell2014} have shown the decidability of the ultimate positivity problem for univariate rational functions with square-free denominators, but the general univariate case is still open.}.  We now discuss an important class of rational functions with non-negative coefficients.

\begin{definition}
The set of $n$-variate \emph{$\mathbb{N}$-rational functions} is the smallest set of rational functions containing $0,z_1,\dots,z_n$ which is closed under addition, multiplication, and pseudo-inverse (the operation $G \mapsto 1/(1-G)$ for $G$ with a constant term of 0). 
\end{definition}

The set of univariate $\mathbb{N}$-rational functions consists of the generating functions of rational languages over finite alphabets, and given $F \in \mathbb{N}(z)$ it is effective~\cite{Soittola1976,Koutschan2008} to determine whether or not it is $\mathbb{N}$-rational (and to decompose it in terms of additions, multiplications, and pseudo-inverses when it is).  An influential principle, sometimes referred to as the \emph{Schützenberger methodology}\footnote{Although this term, and the related expression ``Delest-Schützenberger-Viennot methodology'', is usually used to describe the more general philosophy that algebraic generating functions be studied through bijections to context-free languages~\cite{Delest1996}.}, states that ``every'' rational generating function of a naturally occurring combinatorial class is $\mathbb{N}$-rational, and enumerative properties of such classes can be determined through bijections to suitable regular languages (see Bousquet-Mélou~\cite[Section 2.4.]{Bousquet-Melou2006} and Gessel~\cite{Gessel2003}).

Garrabrant and Pak~\cite{GarrabrantPak2014} give a combinatorial characterization of the rational diagonal sequences which are diagonals of $\mathbb{N}$-rational functions.  A \emph{tile} is an axis-parallel simply connected closed polygon in the plane of height 1, and a \emph{tiling} of a rectangle $R$ of height 1 with the set of tiles $T$ is a sequence of tiles in $T$, overlapping only on their boundaries, which cover $R$ (see Figure~\ref{fig:GarPak}).  For a set of tiles $T$ and fixed $\epsilon>0$ define $f_{T,\epsilon}(n)$ to be the number of tilings of a $1 \times (n+\epsilon)$ rectangle using the elements of $T$ for all $n \in \mathbb{N}$.  Let $\mathcal{F}$ be the set of all such \emph{tile-counting functions} $f_{T,\epsilon}:\mathbb{N}\rightarrow\mathbb{N}$ as $T$ and $\epsilon$ vary. 

\begin{figure}
\centering
\includegraphics[width=0.3\linewidth]{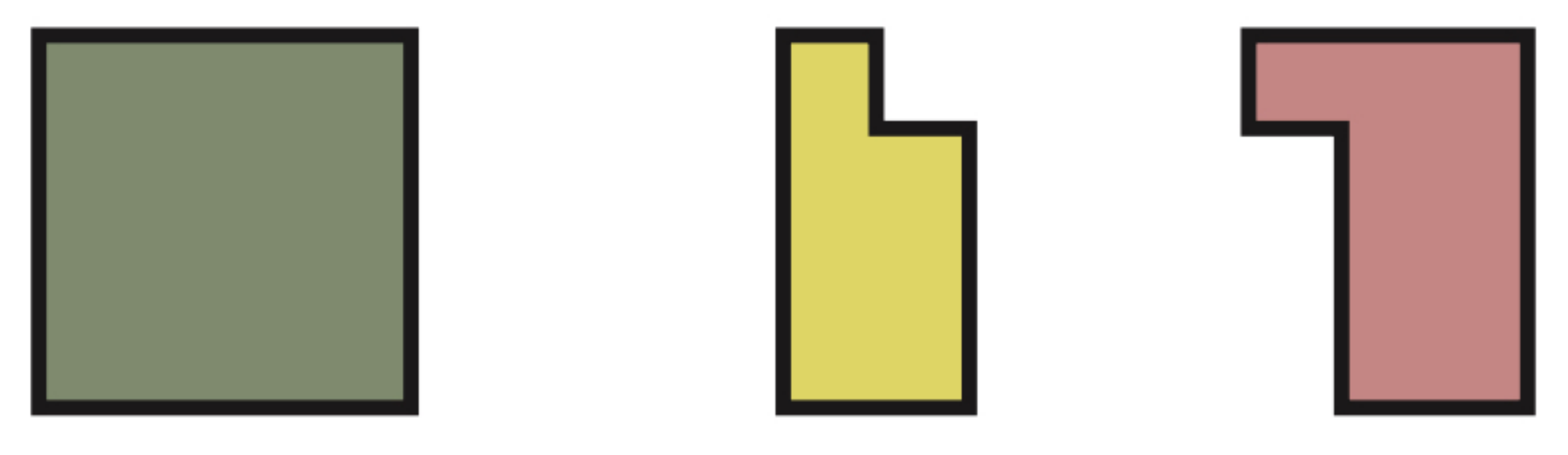}
\caption[A set of irrational tiles]{A set of tiles from Garrabrant and Pak~\cite[Figure 4]{GarrabrantPak2014}: the green square has length 1 and the yellow and pink polygons are such that setting them beside each other gives a square of length 1.  The number of tilings of a $1 \times n$ rectangle is $2^n$.}
\label{fig:GarPak}
\end{figure}
 
\begin{proposition}[{Garrabrant and Pak~\cite[Main Theorem 1.2]{GarrabrantPak2014}}]
The function $f(n) \in \mathcal{F}$ if and only if the generating function $\sum_{n \geq 0} f(n)z^n$ is the diagonal of an $\mathbb{N}$-rational function.
\end{proposition}

Note that the theorem does not show how many variables an $\mathbb{N}$-rational function whose diagonal sequence equals $f(n)$ will contain, and given $f(n) \in \mathcal{F}$ it is not currently known how to determine the smallest number of variables needed to express $f(n)$ as the diagonal of an $\mathbb{N}$-rational function.  Garrabrant and Pak establish this result through a connection to a sub-family of binomial sums.  We say that a \emph{restricted binomial multisum} is the family of functions
\[ f(n) = \sum_{\bv \in \mathbb{Z}^d}\left( \prod_{i=1}^r \binom{\mathbf{a}^{(i)} \cdot \bv + a'_in + a_i''}{\mathbf{b}^{(i)} \cdot \bv + b'_in + b_i''}   \right) \] 
where $r,d \in \mathbb{N}$, $\mathbf{a}^{(i)},\mathbf{b}^{(i)} \in \mathbb{Z}^d$, and $a_i',b_i',a_i'',b_i'' \in \mathbb{Z}$ for all $1 \leq i \leq r$.  

\begin{proposition}[{Garrabrant and Pak~\cite[Main Theorem 1.3]{GarrabrantPak2014}}]
The set of restricted binomial multisums is equal to $\mathcal{F}$ (and thus the set of diagonals of $\mathbb{N}$-rational functions).
\end{proposition}

Although it is decidable to determine when a univariate function is $\mathbb{N}$-rational, it is currently unknown how to characterize $\mathbb{N}$-rationality in higher dimensions.  For example, Garrabrant and Pak conjecture that the generating function for the Catalan numbers $C_n = \frac{1}{n+1}\binom{2n}{n}$ is not the diagonal of an $\mathbb{N}$-rational function (in any number of variables) while it is the diagonal of a bivariate rational function as it is algebraic\footnote{In fact, they show~\cite[Proposition 4.7]{GarrabrantPak2014} that for any $\epsilon>0$ there exists a constant $A \in (1-\epsilon,1+\epsilon)$ and sequence $f_n$ which is the diagonal of an $\mathbb{N}$-rational function such that $f_n \sim A \cdot C_n$, so this conjecture cannot be resolved by asymptotic means.  They also show~\cite[Propositions 4.8 and 4.9]{GarrabrantPak2014} that it cannot be resolved by arithmetic means (for instance, for any $m \in \mathbb{N}$ there is an $\mathbb{N}$-rational function whose diagonal sequence is the same as the Catalan numbers modulo $m$).}.  The univariate characterization of $\mathbb{N}$-rationality relies heavily on a singularity analysis which does not easily translate into the multivariate case.  The field of analytic combinatorics in several variables provides a potential source of tools to examine this problem, although significant progress on such a deep question will be challenging.

\section{Period Integrals}
A \emph{period number}, in the sense of Kontsevich and Zagier~\cite{KontsevichZagier2001}, is any complex number whose real and imaginary parts can be expressed as absolutely convergent integrals of the form
\[ \int_{\Gamma} F(\bz) d\bz, \]
where $F(\bz) \in \mathbb{Q}(\bz)$ and $\Gamma \subset\mathbb{R}^n$ is defined by polynomial inequalities with rational coefficients. The collection of period numbers includes all algebraic numbers, logarithms of algebraic numbers, $\pi$, and all multiple zeta values, however Kontsevich and Zagier conjecture that $e$, Euler's constant $\gamma$, and $1/\pi$ are not period numbers.  It seems to be difficult to find an explicit example of a number which is not a period, although the set of period numbers is countable.  

Closely related to period numbers are period integrals of rational functions depending on a parameter; that is, integrals of the form
\[ \int_{\Gamma} F(\bz,t) d\bz, \]
where $F(\bz,t)$ is a rational function with parameter $t$ and $\Gamma$ is an appropriate domain of integration (so that, for example, the integral is absolutely convergent for all values of $t$ in some open subset of the complex plane). 

If $F(\bz) = \sum f_{\bi}\bz^{\bi}$ is a rational function which is analytic at the origin, then the multivariate Cauchy Integral Formula (described in Theorem~\ref{thm:mCIF} below) implies
\begin{align*} 
(\Delta F)(t) = \sum_{k \geq 0} f_{k,\dots,k}t^k 
&= \frac{1}{(2\pi i)^n} \sum_{k \geq 0} \int_{\Gamma}\frac{F(\bz)}{(z_1 \cdots z_n t)^k} \frac{dz_1 \cdots dz_n}{z_1 \cdots z_n} \\
&= \frac{1}{(2\pi i)^n} \int_{\Gamma} \left(\sum_{k \geq 0} \frac{F(\bz)}{(z_1 \cdots z_n t)^k}\right)\frac{dz_1 \cdots dz_n}{z_1 \cdots z_n} \\
&= \frac{1}{(2\pi i)^n} \int_{\Gamma} \frac{F(\bz)}{1-t(z_1 \cdots z_n)} \frac{dz_1 \cdots dz_n}{z_1 \cdots z_n},
\end{align*}
where $\Gamma$ is a product of circles in the complex plane sufficiently close to the origin (the summation and integration can be exchanged as a power series converges absolutely and uniformly on the interior of its domain of convergence).  This shows that rational diagonals are examples of period integrals\footnote{Although the integral representation given here is taken over the complex plane, one can make the substitution $z_j=x_j + iy_j$ and use the fact that the circle $|z_j|=\epsilon$ is parametrized by $x_j^2+y_j^2=\epsilon$ when $x_j$ and $y_j$ are real.}, up to powers of (the conjecturally not a period number) $1/\pi$. 

The functions defined by period integrals with parameters satisfy a family of differential equations known as \emph{Picard-Fuchs differential equations}~\cite[Chapter 2]{KontsevichZagier2001}; period numbers then arise as evaluations of solutions of Picard-Fuchs differential equations at algebraic arguments.  For example, following the conjectures of Kontsevich and Zagier on period numbers it is tempting to conjecture that Euler's constant $\gamma$ cannot arise as the evaluation of a rational diagonal (over the rational numbers) at an algebraic argument.

\begin{example}[Periods on Calabi-Yau 3-Folds]
\label{ex:CalabiYau}
Period integrals with parameters defined over cycles on certain algebraic varieties are known to encode important information about the algebraic varieties.  For instance, such period integrals can be used to count the number of rational curves of fixed degree on quintic 3-folds (hypersurfaces with degree 5 and dimension 3)~\cite{Morrison1993}.  Much of this theory has been developed for Calabi-Yau 3-folds through the use of ``mirror symmetry'' (see Cox and Katz~\cite{CoxKatz1999} for details and definitions).

In a recent paper, Batyrev and Kreuzer~\cite{BatyrevKreuzer2010} determined a family of Calabi-Yau threefolds, identified by polytopes $P_j \subset\mathbb{Z}^4$, and studied their \emph{principal periods}
\[ \overline{\omega}_0(t) = \int_C \frac{1}{1-t\sum_{\mathbf{v} \in P_j} \bz^{\mathbf{v}}} \frac{dz_1 dz_2 dz_3 dz_4}{z_1 z_2 z_3 z_4} \]
where $C$ is a product of circles in the complex plane sufficiently close to the origin. Batyrev and Kreuzer were interested in properties of the Picard-Fuchs differential equations annihilating these integrals: the models break down into 68 classes, of which they were able to guess such equations for the models in 28 classes.  Lairez~\cite{Lairez2016} used a fast creative telescoping algorithm to rigorously compute annihilating differential operators for all models\footnote{Lairez's complete list of Laurent polynomials $\sum_{\mathbf{v} \in P_j} \bz^{\mathbf{v}}$ and their annihilating differential operators can be found at~\url{http://pierre.lairez.fr/supp/periods/}.}.  For each polytope $P_j$, the principal period can be expressed as the diagonal
\[ \overline{\omega}_0(t) = \Delta \left( \frac{1}{1-t(z_1z_2z_3z_4)\sum_{\mathbf{v} \in P_j} \bz^{\mathbf{v}}} \right),\]
where the rational function is expanded in the ring $\mathbb{Q}[\bz,\overline{\bz}][[t]]$.  Asymptotics for one of these diagonals is computed in Example~\ref{ex:CalabiYau2}.
\end{example}

\section{Further Examples}
\subsection[n-fold Ising Integrals]{$n$-fold Ising Integrals}
The Ising model is an important model in statistical physics, introduced by Lenz~\cite{Lenz1920} and studied in the one dimensional case by his PhD student Ising~\cite{Ising1925}.  Roughly speaking, the model considers the spins of particles arranged on a lattice with respect to an external magnetic field.  Such spins can take the values $\pm1$ and, possibly in the presence of interactions between the particles or outside forces, one wants to determine information for different configurations of spins after certain parameters are fixed.  Many of the desired properties can be expressed as sums of $n$-fold integrals, and Bostan et al.~\cite[Section 3]{BostanBoukraaChristolHassaniMaillard2013} show that the integrals which arise can often be written as diagonals of explicit $n$-variate algebraic functions (meaning they are diagonals of $2n$-variate rational functions).

\begin{example}[{Bostan et al.~\cite[Appendix C]{BostanBoukraaChristol2012}}]
\label{ex:nPartEx}
Bostan et al. consider a family of integrals $\Phi_D^{(n)}(w)$ related to the ``$n$-particle contribution to the diagonal magnetic susceptibility of the Ising model'' and give the explicit example
\[ \Phi_D^{(3)}(w) = \Delta \left( \frac{1-2w+\sqrt{(1-2w)^2-4w^2t^2}}{2\sqrt{1-t^2}\sqrt{(1-2w)^2-4w^2t^2}} - \frac{1}{2}  \right). \]
Using the methods presented in Lemma 6.3 and the proof of Theorem 6.2(ii) of Denef and Lipshitz~\cite{DenefLipshitz1987}, one can construct\footnote{The rational function is available on \websiteurl.} a 4 variable rational function $F(t,w,u,v)$ whose diagonal gives $\Phi_D^{(3)}(w)$.  The rational function has a numerator of (total) degree 55 in its variables, and a denominator of degree 54.  An explicit expression for $\Phi_D^{(n)}(w)$ in terms of ${}_4F_3$ hypergeometric series is given in Section 4 of Boukraa et al.~\cite{BoukraaHassaniMaillardZenine2007}.
\end{example}

Many of the objects appearing in the Ising model are similar to those appearing in lattice path enumeration, although such objects often naturally arise as diagonals of multivariate \emph{algebraic} functions. This makes them a good potential source of study for new work looking to apply the methods of analytic combinatorics in several variables.

\subsection{Enumerating Simple Singular Vector Tuples of Generic Tensors}
\label{sec:SingularTuple}
In a study of rank-1 approximations of tensors, Friedland and Ottaviani proved the following result\footnote{See Friedland and Ottaviani~\cite{FriedlandOttaviani2014} for all relevant definitions.}.

\begin{proposition}[{Friedland and Ottaviani~\cite[Theorem 1]{FriedlandOttaviani2014}}]
Let $c(i_1,\dots,i_n)$ denote the number of simple singular vector tuples of a generic complex $m_1\times\cdots\times m_n$ tensor.  Then
\[ c(i_1,\dots,i_n) = [t_1^{i_1} \cdots t_n^{i_n}] \prod_{i=1}^d \frac{\tau_i^{m_i} - t_i^{m_i}}{\tau_i-t_i}, \quad \text{for } \tau_i = \sum_{1 \leq j \neq i \leq n}t_j.\]
\end{proposition}

Based on this result, Ekhad and Zeilberger~\cite[Proposition 1]{EkhadZeilberger2016} observed that the multivariate generating function $F(\bz)$ for $c(i_1,\dots,i_n)$ can be written
\[ F(\bz) = \sum_{\bi \in \mathbb{N}^n} c(\bi) \bz^{\bi} = \frac{z_1 \cdots z_n}{(1-z_1) \cdots (1-z_n)\left( 1 - \sum_{i=2}^n (i-1)e_i(\bz) \right)}, \]
where $e_i(\bz)$ is the $i$th elementary symmetric function 
\[ e_i(\bz) = \sum_{1 \leq j_1 < \cdots < j_i \leq n} z_{j_1} \cdots z_{j_i}.\]
For all natural numbers $n$ define
\[  C_n(k) = c(k,k,\dots,k). \]
Ekhad and Zeilberger used creative telescoping methods to determine a linear recurrence relation with polynomial coefficients for $C_3(k)$ and used that to deduce asymptotics.  Of particular interest is the fact that there are solutions of the linear recurrence with larger exponential growth than $C_3(k)$ (so that some of the connection coefficients of the associated generating function differential equation vanish).  Ekhad and Zeilberger could not determine a recurrence for $C_4(k)$, but conjectured asymptotics from a large number of available terms.  Using the methods of analytic combinatorics in several variables, Pantone~\cite{Pantone2017} recently gave asymptotics of $C_n(k)$ for all $n \geq 3$.

\begin{proposition}[{Pantone~\cite[Theorem 1.3]{Pantone2017}}]
For $n \geq 3$,
\[ C_n(k) = \frac{(n-1)^{n-1}}{(2\pi)^{(n-1)/2}n^{(n-2)/2}(n-2)^{(3n-1)/2}} \cdot \left((n-1)^n\right)^k \cdot k^{(1-n)/2}\left( 1 + O\left(\frac{1}{n}\right)\right), \]
as $k\rightarrow \infty$.
\end{proposition}

This result is re-derived in Example~\ref{sec:SingularTuple} below.

\subsection{Examples from Pemantle and Wilson}
A survey paper by Pemantle and Wilson~\cite{PemantleWilson2008}, together with their textbook~\cite{PemantleWilson2013}, highlights a large range of multivariate generating functions whose asymptotics can be calculated through the theory of ACSV. These include examples from the study of trees and graphs, quantum random walks, Chebyshev polynomial coefficients, Gaussian weak and central limit laws, queuing theory\footnote{One of the pioneering applications which used complex analysis in several variables to compute asymptotics was the work of Bertozzi and McKenna~\cite{BertozziMcKenna1993} on queuing theory problems.}, integer solutions to linear equations, tilings of the Aztec Diamond, sequences defined by Riordan arrays, convex polyominoes, symmetric Eulerian numbers, and strings with forbidden patterns.  
\smallskip

We try as much as possible to give new examples in this thesis, so that those wanting to learn the theory have a wider selection of samples to help guide their understanding.   The reader looking for more information on these examples can consult the work of Pemantle and Wilson.

\part[Smooth ACSV and Applications to Lattice Paths]{Smooth Analytic Combinatorics in Several Variables and Applications to Lattice Paths}
\label{part:SmoothACSV}

%%%%%%%%%%%%%%%%%%%%%%%%%%%
% Chapter 6
%%%%%%%%%%%%%%%%%%%%%%%%%%%
\chapter[The Theory of ACSV for Smooth Points]{The Theory of Analytic Combinatorics in Several Variables for Smooth Points}
\label{ch:SmoothACSV}

\setlength{\epigraphwidth}{3.4in}
\epigraph{One might be tempted to think of the analysis of several complex variables \dots as being essentially one variable theory with the additional complication of multi-indices. This perception turns out to be incorrect.  Deep new phenomena and profound (as yet unsolved) problems present themselves in the theory of several variables.}{Steven G. Krantz, \emph{Function Theory of Several Complex Variables}}

\epigraph{Tomas did not realize at the time that metaphors are dangerous. Metaphors are not to be trifled with. A single metaphor can give birth to love.\footnotemark}{Milan Kundera, \emph{The Unbearable Lightness of Being}}
\footnotetext{Translated from the Czech by Michael Henry Heim.}

In this chapter we describe the theory of analytic combinatorics in several variables under a set of assumptions which simplify the analysis.  We take an example based approach, interspersing specific cases with general theory.  For the simplest examples one encounters, the analysis requires little background beyond basic complex analysis and a knowledge of the saddle-point method, both of which are crucial to univariate analytic combinatorics. Dealing with more complicated examples, however, will require advanced results from algebraic and differential geometry, topology, and analysis.  In any case, one should always keep the univariate approach in mind: examine the singularities of the function under consideration, determine a finite set of those singularities which dictate exponential growth, then perform a local analysis of the function at those points to determine dominant asymptotics.  As our goal is to be pedagogical, instead of rigorously exhaustive, we give rigorous derivations of asymptotics for our examples and refer to the work of Pemantle and Wilson for full proofs of the general statements.

\subsubsection*{Setup}

Given $\bw \in \mathbb{C}^n$, we let $D(\bw)$ \glsadd{Dz} denote the closed \emph{polydisk} 
\[ D(\bw) := \{ \bz : |z_j| \leq |w_j|, \quad j=1,\dots,n \}, \]
and $T(\bw)$ \glsadd{Tz} denote the \emph{polytorus}
\[ T(\bw) := \{ \bz : |z_j| = |w_j|, \quad j=1,\dots,n \}. \]
Similar to many results in univariate analytic combinatorics, our analysis will rest on an integral representation for power series coefficients coming from a multivariate generalization of Cauchy's integral formula.

\begin{theorem}[Multivariate Cauchy Integral Formula]
\label{thm:mCIF}
Let $C$ be a torus around the origin such that the complex-valued function $F(\bz)$ is analytic inside and on $C$, and let $F(\bz) = \sum_{\bi \in \mathbb{N}^n}f_{\bi}\bz^{\bi}$ be its power series expansion at the origin.  Then for every natural number $k$,
\begin{equation} 
f_{k,k,\dots,k} = \frac{1}{(2\pi i)^n} \int_C F(\bz) \frac{dz_1 \cdots dz_n}{z_1^{k+1} \cdots z_n^{k+1}}. \label{eq:mCIF}
\end{equation}
\end{theorem}
This standard result follows from the univariate Cauchy integral formula by induction, and can be found as Proposition 7.2.6 in Pemantle and Wilson~\cite{PemantleWilson2013}. This is a particular instance of Proposition~\ref{prop:convLaurent} in Chapter~\ref{ch:Background}, on convergent Laurent expansions of multivariate functions.

\section{Central Binomial Coefficient Asymptotics}
We begin with the simple rational function 
\[ F(x,y) = \frac{1}{1-x-y} = \sum_{(i,j) \in \mathbb{N}^2} \binom{i+j}{i}x^iy^j,\]
and let $\mD$ denote the open domain of convergence of this power series at the origin.  As shown in Example~\ref{ex:binDomainConverge} of Chapter~\ref{ch:Background}, 
\[ \mD = \{(x,y) \in \mathbb{C} : |x| + |y| < 1\}. \]
The Cauchy integral formula then implies
\begin{equation} \binom{2k}{k} = \frac{1}{(2\pi i)^2}\int_{T(a,b)} \frac{1}{1-x-y} \cdot \frac{dxdy}{x^{k+1}y^{k+1}}, \label{eq:cbinCIF} \end{equation}
for any $(a,b) \in \mD$.  

\subsubsection*{Step 1: Bound Exponential Growth}
By Corollary~\ref{cor:diagAsm}, the `coarsest' measure of asymptotics for a rational diagonal coefficient sequence $f_{k,\dots,k}$ is its exponential growth 
\[ \rho := \limsup_{k \rightarrow \infty}|f_{k,\dots,k}|^{1/k}.\]  
We thus begin by seeing how much information can be obtained about $\rho$ from the analytic properties of the rational function $F(x,y)$. Recall that in the univariate case, the exponential growth of a sequence is obtained by finding the minimal modulus of its generating function's singularities and taking the reciprocal.

In our example, $|a|+|b| < 1$ whenever $(a,b) \in \mD$ so that 
\[ \left| \frac{1}{1-x-y} \right| \leq \frac{1}{1-|a|-|b|} \quad \text{ for all } (x,y) \in T(a,b).\]  
A standard result in complex analysis states that an upper bound on the modulus of an integral is obtained by multiplying an upper bound for the modulus of the integrand by the area of the domain of integration. Applied here, this bound gives
\begin{equation} 
\binom{2k}{k} = \left| \frac{1}{(2\pi i)^2}\int_{T(a,b)} \frac{1}{1-x-y} \cdot \frac{dxdy}{x^{k+1}y^{k+1}} \right| \leq \frac{|ab|^{-k}}{1-|a|-|b|} 
\label{eq:binbound}
\end{equation}
for all $(a,b) \in \mD$.

Equation~\eqref{eq:binbound} gives a family of bounds 
\[\limsup_{k \rightarrow \infty} \binom{2k}{k}^{1/k} \leq |ab|^{-1} \]
on the exponential growth of the central binomial coefficients, one for each pair of points $(a,b) \in \mD$.  In fact, allowing $(a,b)$ to approach the boundary $\partial \mD$ shows that the exponential growth is bounded above by $|ab|^{-1}$ for all points $(a,b)$ in the closure $\overline{\mD}$.  It is natural to wonder which points give the best upper bound, and whether that bound is tight (indeed, answering these two questions turns out to be the hardest step of most multivariate singularity analyses).

Since 
\[ \overline{\mD} = \{(a,b) \in \mathbb{C}^2 : |a| + |b| \leq 1\},\] 
the minimum\footnote{The minimum of $|ab|^{-1}$ on $\overline{\mD}$ is equal to the maximum of $|ab|$, and will occur on the boundary $|a|+|b|=1$.  Thus, one seeks to maximize the function $a(1-a)$ over the domain $a \in (0,1)$.} of $|ab|^{-1}$ on $\overline{\mD}$ is 4, achieved exactly when $|a|=|b|=1/2$.  Going back to our bound in Equation~\eqref{eq:binbound}, we have shown that for every $\epsilon>0$ there exists a constant $C_{\epsilon}$ such that 
\[ \binom{2k}{k} \leq C_{\epsilon} \cdot (4+\epsilon)^k \]
for all natural numbers $k$. Stirling's formula implies 
\[ \binom{2k}{k} = \frac{4^k}{\sqrt{\pi k}}\left(1+O\left(\frac{1}{k}\right)\right)\] 
as $k\rightarrow\infty$, so our upper bound of 4 on the exponential growth of the central binomial sequence is in fact tight.

\subsubsection*{Step 2: Determine Contributing Singularities}
Of course, we want to completely determine dominant asymptotics, not just bound exponential growth.  Analogously to the univariate case, this will involve a local analysis of $F(x,y)$ near some of its singularities.  But which ones should we study? The minimum on $\overline{\mD}$ of the quantity $|ab|^{-1}$ bounding exponential growth occurred on the boundary $\partial\mD$, and the point $(x,y)=(1/2,1/2)$ is the unique singularity of $F(x,y)$ whose coordinates' moduli give this minimum.  

As the integrand of the Cauchy integral grows like $4^k$ when $(x,y)$ is in a neighbourhood of $(1/2,1/2)$, and nowhere else on $\overline{\mD}$, one would expect that local behaviour of the integrand near $(1/2,1/2)$ is important to the asymptotics of the binomial coefficients.  Furthermore, the domain of integration in the Cauchy integral formula can be deformed arbitrarily close to $(1/2,1/2)$ as it lies on the boundary $\partial\mD$.  Thus, we will attempt to determine dominant asymptotics by manipulating the Cauchy integral in Equation~\eqref{eq:cbinCIF} into an integral whose domain stays near this singularity.

\subsubsection*{Step 3: Localize the Cauchy Integral and Compute a Residue}

By the Cauchy integral formula, and the fact that $(1/2,1/2)$ is on the boundary of the domain of convergence $\partial\mD$, we see $\binom{2k}{k} = I$ for
\[  I := \frac{1}{(2\pi i)^2} \int_{|x| = 1/2} \left( \int_{|y| = 1/4} \frac{1}{1-x-y} \cdot \frac{dy}{y^{k+1}} \right) \frac{dx}{x^{k+1}}. \]
Let 
\[ \mN := \{ |x| = 1/2 : \arg(x) \in (-\pi/4,\pi/4) \} \quad\text{and}\quad \mN' := \{ |x| = 1/2 \} \setminus \mN. \] 
Basic arguments show that 
\[ |1-x| < \underbrace{\left|1-\frac{e^{i\pi/4}}{2}\right|}_{\rho} \approx 0.7368\dots \]
for $x \in \mN$ and $|1-x| \geq \rho$ for $x \in \mN'$ (see Figure~\ref{fig:smoothAsmDist}). 

\begin{figure}
\centering
\includegraphics[width=0.8\linewidth]{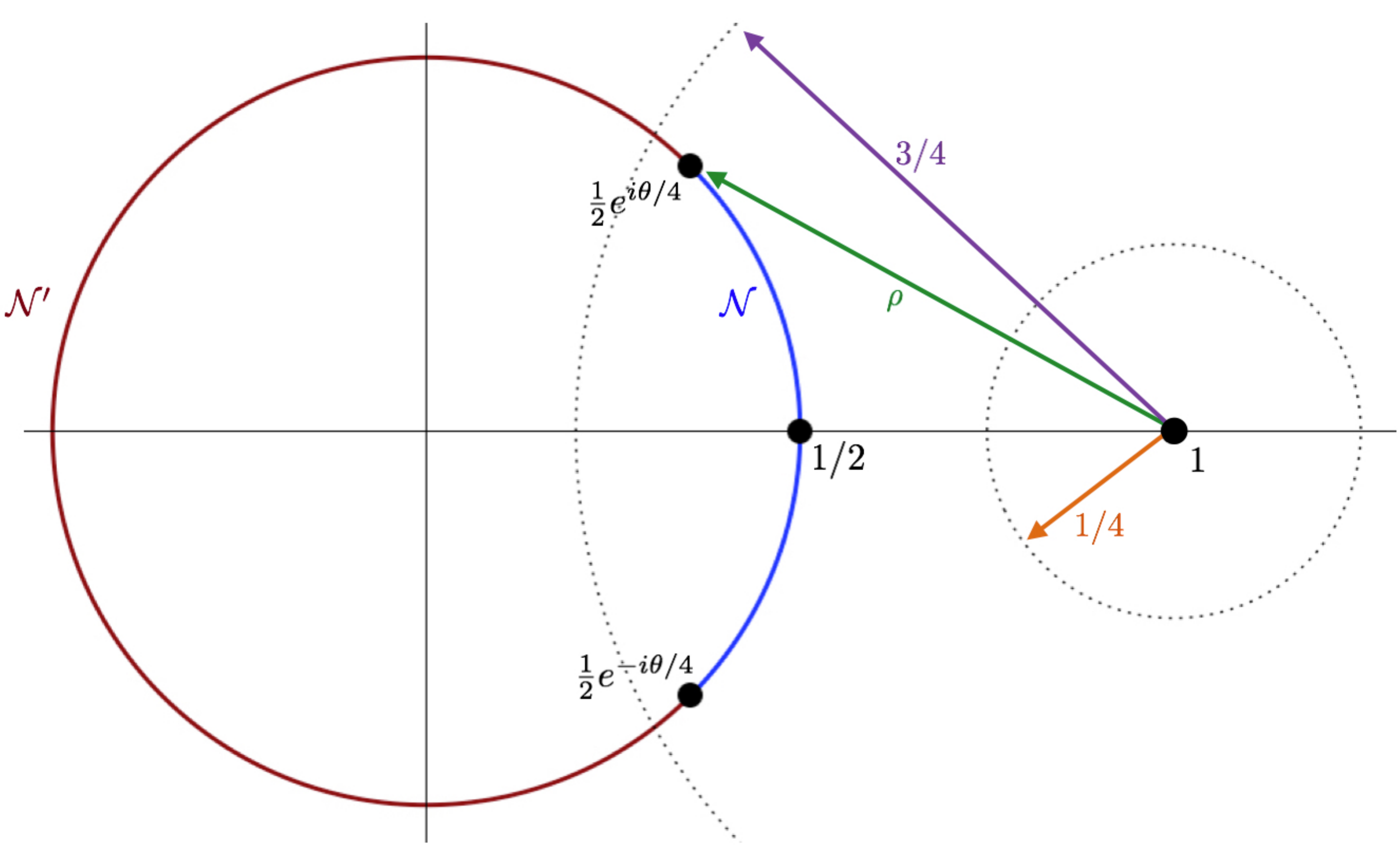}
\caption[Domains of integration for central binomial coefficient asymptotics]{The sets $\mN$ and $\mN'$, to be used as domains of integration. Note that $|1-x|<\rho$ for $x \in \mN$ and $|1-x|\geq \rho>1/4$ for $x \in \mN'$.}
\label{fig:smoothAsmDist}
\end{figure}

We now compare the integral $I$ to the ``localized'' integral
\[ I_{loc} := \frac{1}{(2\pi i)^2} \int_{\mN} \left( \int_{|y| = 1/4} \frac{1}{1-x-y} \cdot \frac{dy}{y^{k+1}} \right) \frac{dx}{x^{k+1}} \]
whose domain of integration is restricted to an $x$-neighbourhood of 1/2 (replacing $\mN$ by its intersection with an arbitrarily small neighbourhood of 1/2 would not change any of the following arguments).  For fixed $x \in \mN'$, one has 
\begin{align*} 
\left| \int_{|y| = 1/4} \frac{1}{1-x-y} \cdot \frac{dy}{y^{k+1}} \right| 
&= \left| \int_{|y| = 1/4} \frac{1/(1-x)}{1-\frac{y}{1-x}} \cdot \frac{dy}{y^{k+1}} \right|  \\
&= \left|[y^k] \sum_{j \geq 0} (1-x)^{-(j+1)} y^j \right| &&\text{(as $|y|=1/4 < |1-x|$ when $x \in \mN'$)} \\
&= |1-x|^{-(k+1)} \\
& \leq \rho^{-(k+1)}
\end{align*}
so that
\[
|I - I_{loc}| 
= \frac{1}{(2\pi)^2} \left| \int_{\mN'} \left( \int_{|y| = 1/4} \frac{1}{1-x-y} \cdot \frac{dy}{y^{k+1}} \right) \frac{dx}{x^{k+1}}\right|
\leq \frac{(3/2) \pi}{(2\pi)^2}\cdot \rho^{-(k+1)} \cdot 2^k 
= \frac{3}{8\rho\pi} \left(\frac{2}{\rho}\right)^k,
\]
where $2/\rho \leq 2.72$. In particular, one can replace the integral $I$ with $I_{loc}$ and introduce an error which grows at an exponentially smaller rate than the central binomial coefficients.  The next step is to consider the integral
\[ I_{out} :=  \frac{1}{(2\pi i)^2} \int_{\mN} \left( \int_{|y| = 3/4} \frac{1}{1-x-y} \cdot \frac{dy}{y^{k+1}} \right) \frac{dx}{x^{k+1}} \]
whose domain of integration is outside the domain of convergence $\mD$.  For $x \in \mN$, the quantity $|1-x|$ is bounded away from 3/4 so that $\left|\frac{1}{1-x-y}\right|$ is bounded\footnote{It is not true that $\left|\frac{1}{1-x-y}\right|$ is bounded for all $|x|=1/2$ and $|y|=3/4$ (see Figure~\ref{fig:smoothAsmDist}), which is why we must first localize $I$ to the $x$-neighbourhood $\mN$ of 1/2.} when $|y|=3/4$.  The integral bounds discussed above then imply 
\[ |I_{out}| = O\left(\left(\frac{8}{3}\right)^k\right). \]

Define
\[ \chi := I_{loc} - I_{out} = \frac{-1}{2\pi i } \int_{\mN}\,\, \frac{1}{2\pi i}\left( \int_{|y| = 3/4} \frac{1}{1-x-y} \cdot \frac{dy}{y^{k+1}} - \int_{|y| = 1/4} \frac{1}{1-x-y} \cdot \frac{dy}{y^{k+1}} \right) \frac{dx}{x^{k+1}}. \]
For each $x \in \mN$, the function $F(x,y) = (1-x-y)^{-1}$ has a unique pole between the curves $\{|y|=1/4\}$ and $\{|y|=3/4\}$, at $y=1-x$.  Thus, the Cauchy residue theorem implies that 
\[  \frac{1}{2\pi i}\left( \int_{|y| = 3/4} \frac{1}{1-x-y} \cdot \frac{dy}{y^{k+1}} - \int_{|y| = 1/4} \frac{1}{1-x-y} \cdot \frac{dy}{y^{k+1}}\right) = -(1-x)^{-(k+1)},\]
and
\begin{equation} \chi = \frac{1}{2\pi i} \int_{\mN} \frac{dx}{(1-x)^{k+1}x^{k+1}}. \label{eq:cbinchi} \end{equation}
From our above reasoning, we know
\[ \left|\binom{2k}{k}-\chi\right| = \left|I - \left(I_{loc} - I_{out}\right)\right| \leq |I - I_{loc}| + |I_{out}| = O\left(\left(\frac{2}{\rho}\right)^k\right).\]
Parameterizing the domain of integration $\mN$ in Equation~\eqref{eq:cbinchi} as $\{ e^{i \theta}/2 : \theta \in (-\pi/4,\pi/4) \}$ one obtains, after some simplification,
\begin{equation} \chi = \frac{4^k}{2 \pi} \int_{-\pi /4}^{\pi/4} A(\theta) e^{-k \phi(\theta)} d\theta, \label{eq:cbinchiFL} \end{equation}
where 
\[ A(\theta) = \frac{1}{1-e^{i\theta}/2} \qquad \text{and} \qquad \phi(\theta) = \log(2-e^{i\theta}) + i\theta.  \]

\subsubsection*{Step 4: Find Asymptotics using Laplace's Method}
Asymptotics of the integral appearing in Equation~\eqref{eq:cbinchiFL} can be determined by a method of Laplace~\cite{Laplace1774} which dates back to 1774 and is named in his honour (see Section 4.4 of de Bruijn~\cite{Bruijn1958} for details).  This type of integral is known as a \emph{Fourier-Laplace} integral, and we give a result for determining asymptotics of multivariate Fourier-Laplace integrals in Proposition~\ref{prop:HighAsm} below.  In this case, we obtain
\[ \binom{2k}{k} = \chi + O\left(\left(\frac{2}{\rho}\right)^k\right) = \frac{4^k}{\sqrt{\pi k}}\left(1+O\left(\frac{1}{k}\right)\right) \]
as $k\rightarrow\infty$.  The key properties of the integral in Equation~\eqref{eq:cbinchiFL} which allow for such an analysis are that $A$ and $\phi$ are analytic at the origin and:
\begin{itemize}
	\item $\phi(0)=\phi'(0) = 0$;
	\item $\phi'(\theta) \neq 0$ on $(-\pi/4,\pi/4)$ unless $\theta=0$;
	\item $\phi''(0) \neq 0$;
	\item the real part of $\phi$ is non-negative on the domain of integration.
\end{itemize}
The idea is that under these hypotheses one can asymptotically approximate
\[ \int_{-\pi /4}^{\pi/4} A(\theta) e^{-k \phi(\theta)} d\theta \quad \approx \quad A(0)\int_{-\infty}^{\infty} e^{-k\frac{\phi''(0)}{2} \theta^2}d\theta, \]
and the Gaussian integral which arises can be calculated explicitly.  The miracle which underlies analytic combinatorics in several variables is the fact that, by making a natural choice of singularities to study (those giving the best bound on exponential growth), one typically ends up with Fourier-Laplace integrals satisfying these (or analogous) restrictions.  Note that when dealing with rational functions of $n$ variables the Fourier-Laplace integral expressions that will be obtained are $n-1$ dimensional.  They have no analogue in the meromorphic univariate case, where one is finished after computing the residue in Step 3.

\section{The Smooth Case}
In this section we will see that the approach taken for the central binomial coefficients generalizes to an amazing degree. Suppose that we have some fixed rational function 
\[ F(\bz) = \frac{G(\bz)}{H(\bz)}\] 
which is analytic at the origin, and let $\mD$ denote the domain of convergence of its power series expansion \glsadd{mD}
\[ F(\bz) = \sum_{\bi \in \mathbb{N}^n} f_{\bi}\bz^{\bi}\]  
at the origin. We may assume that the denominator $H$ is dependent on each variable $z_1,\dots,z_n$, as otherwise the diagonal sequence will eventually become the zero sequence. The set of singularities of $F(\bz)$ is known as its \emph{singular variety}, and denoted $\mV$.  The singular variety of any rational function is an algebraic set. \glsadd{V2}

\begin{lemma}
\label{lemma:singGH}
Suppose $F(\bz) = \frac{G(\bz)}{H(\bz)}$ for co-prime polynomials $G,H \in \mathbb{Q}[\bz]$ with $H$ not identically zero.  Then the singular variety of $F(\bz)$ is the algebraic set $\mV(H) := \{ \bz : H(\bz)=0 \}$.
\end{lemma}

We will show that local containment of an irreducible algebraic variety in an algebraic set implies containment on all of $\mathbb{C}^n$. Our proof uses properties of the dimension of an algebraic set, which can be found in Chapter 1A of Mumford~\cite{Mumford1976}.
%For any open set $O \subset \mathbb{C}^n$, let $\mV_O(G)$ and $\mV_O(H)$ denote the zeroes of $G(\bz)$ and $H(\bz)$ in $O$. 

\begin{proof}
First, we note that any singularity of $F(\bz)$ must be a zero of $H(\bz)$. Suppose now that $\bw$ is a zero of $H(\bz)$ and not a singularity of $F(\bz)$. Then the modulus of $F(\bz)$ must be bounded in an open ball $\mO \subset \mathbb{C}^n$ centered at $\bw$, meaning 
\[ \varnothing \quad \neq \quad \mV(H) \cap \mO \quad \subset \quad \mV(G) \cap \mO.\] 
We may assume that $H$ is irreducible, as otherwise it can be replaced by one of its irreducible factors.

Let $\mW = \mV(H) \cap \mV(G)$.  The dimension of an (irreducible) algebraic variety is the dimension of its tangent space at any smooth point, and the dimension of any algebraic set is the maximum of the dimensions of its tangent spaces at its smooth points.  Furthermore, the set of singular points of an algebraic set forms a closed algebraic subset, so there exists a smooth point of $\mW$ and $\mV(H)$ in $\mO$.  

Since $\mW \cap \mO = \mV(H) \cap \mO$ and tangent spaces are defined locally, this implies the dimension of $\mW$ is at least the dimension of $\mV(H)$, which is $n-1$ as $\mV(H)$ is a proper algebraic variety.  Thus, $\mW \neq \mathbb{C}^n$ has dimension $n-1$.  The only algebraic varieties of dimension $n-1$ are those of the form $\mV(P)$ where $P$ is an irreducible polynomial, so $H$ must be an irreducible factor of $G$.
\end{proof}
\smallskip

When writing $F(\bz)=G(\bz)/H(\bz)$ we always assume that $G$ and $H$ are co-prime polynomials.  We begin this chapter by assuming that the singular variety $\mV$ is a complex (analytic) manifold and that $H(\bz)$ is square-free, which is equivalent to the fact that $H(\bz)$ and its partial derivatives do not simultaneously vanish at any point. 

Any $m$-dimensional complex manifold has an underlying $2m$-dimensional real smooth manifold structure, which is determined by setting $z_j=x_j+iy_j$ for real variables $x_j$ and $y_j$.  Pemantle and Wilson refer to the analysis when $\mV$ is a complex manifold as the \emph{smooth case}, and a point $\bw \in \mV$ is called \emph{locally smooth} if some open neighbourhood of $\bw$ in $\mV$ is a complex manifold.  The requisite background for the differential geometry discussed here can be found in Chapter 0 of Griffiths and Harris~\cite{GriffithsHarris1978} (or, with a more introductory presentation, in Chapter 3 of Ebeling~\cite{Ebeling2007}).

\subsubsection*{Step 1: Bound Exponential Growth}
Just as in the example of the central binomial coefficients, the Cauchy integral formula implies that for every point $\bw \in \mD$ 
\[ |f_{k,\dots,k}| = \left| \frac{1}{(2\pi i)^n}\int_{T(\bw)} F(\bz) \cdot \frac{d\bz}{z_1^{k+1} \cdots z_n^{k+1}} \right| \leq C_{\bw} \cdot |w_1\cdots w_n|^{-k}, \]
where $C_{\bw} = \max_{\bz \in T(\bw)}|F(\bz)|$ is finite.  Thus, there is an exponential growth bound
\[ \limsup_{k\rightarrow\infty} |f_{k,\dots,k}|^{1/k} \leq |w_1 \cdots w_n|^{-1} \]
for every $\bw \in \overline{\mD}$.  It will always be the case that the minimum of $|w_1 \cdots w_n|^{-1}$ on $\overline{\mD}$ occurs at the boundary $\partial \mD$ \glsadd{mDbd} \glsadd{clmD} when it is achieved\footnote{This holds since $|w_1 \dots w_n|^{-1}$ decreases as the point $\bw$ moves away from the origin.}. Furthermore, it can be shown from the Cauchy integral formula that $\bw \in \partial \mD$ if and only if the intersection $T(\bw) \cap \mV$ is non-empty.  Singularities on the boundary $\partial \mD \cap \mV$ are called \emph{minimal points}, and being the singularities of $F(\bz)$ which are closest to the origin they are the most natural generalization of dominant singularities in the univariate case. 

\subsubsection*{Step 2: Determine Critical Points}

The above argument shows that to minimize $|w_1 \cdots w_n|^{-1}$ on $\overline{\mD}$ it is sufficient to consider only the set of minimal points $\mV \cap \partial\mD$.  In fact, as our objective function becomes arbitrarily large as any coordinate approaches 0 (and the others are fixed) we can replace $\mV$ with its subset $\mV^* = \mV \setminus \{\bz : z_1 \dotsm z_n=0\}$ of points whose coordinates are non-zero.  Since $\mV^*$ is an open subset of $\mV$, it is also a complex manifold under our assumptions. \glsadd{Vstar}

Define the polynomial map $\phi(\bz) := z_1 \cdots z_n$ from $\mV^*$ to $\mathbb{C}$.  If we consider $\phi$ to be an analytic mapping from the complex manifold $\mV^*$ to the complex manifold $\mathbb{C}$, then the critical points of $\phi$ (i.e., the points where the differential of $\phi$ is 0) help to characterize minimizers of our upper bound on exponential growth.

\begin{lemma}
\label{lem:smoothCPphi}
When $\mV$ is a complex manifold then any local extremum of $|z_1 \cdots z_n|^{-1}$ on $\mV^*$ is a critical point of the map $\phi(\bz) = z_1 \cdots z_n$ from $\mV^*$ to $\mathbb{C}$. 
\end{lemma}

\noindent
Lemma~\ref{lem:smoothCPphi} follows from Section 8.3 of Pemantle and Wilson~\cite{PemantleWilson2013}, but we sketch its proof here.

\begin{proof}[Proof Sketch]
Let $h(\bz) = \log|z_1 \cdots z_n|$, so that any local extremum of $|z_1 \cdots z_n|^{-1}$ on $\mV^*$ is a local extremum of $h(\bz)$.  The complex manifold $\mV^*\subset\mathbb{C}^n$ gives rise to an underlying real smooth manifold $\mW^* \subset\mathbb{R}^{2n}$, obtained by setting $z_j = x_j + iy_j$ for real variables $x_j$ and $y_j$.  If $h$ is considered as a smooth mapping from $\mW^*$ to $\mathbb{R}$, then any local extremum of $h$ must occur\footnote{Let $\mM$ be a real manifold and suppose the smooth map $f:\mM\rightarrow\mathbb{R}$ has an extremum at $\bw$.  Then any curve $\gamma:(-\epsilon,\epsilon)\rightarrow\mM$ through $\bw$ satisfies $\left.\frac{d}{dt} f(\gamma(t))\right|_{t=0} = 0$ and the differential of $f$ at $\bw$ can be represented, in any set of coordinates, as a linear combination whose coefficients are given by such derivatives~\cite[Proposition 8.18]{Tu2011}.} at one of its critical points. 

As $\phi$ does not vanish on $\mV^*$, one can define a branch of the logarithm $\log(\phi) = \log(z_1 \cdots z_n)$ on $\mV^*$.  Considering $\phi$ and $\log(\phi)$ to be analytic maps between complex manifolds, the chain rule implies that the set of critical points of $\phi$ equals the set of critical points of $\log(\phi)$. Furthermore, the map $\log(\phi)$ can be considered as a smooth map from $\mW^*$ to $\mathbb{R}^2$ (decomposing $\log(\phi)$ into its real and imaginary components), and the relationship between $\mV^*$ and $\mW^*$ implies that the set of critical points of $\log(\phi)$ is the same when considering it as a smooth or analytic mapping\footnote{In particular, the rank of the differential of $\log(\phi)$ as a smooth map of real manifolds is twice the rank of the differential of $\log(\phi)$ as an analytic map of complex manifolds, and a critical point is a point where the rank of the differential is zero. See Griffiths and Harris~\cite[Page 18]{GriffithsHarris1978} for details.}.  Since $h(\bz)$ is the real part of the map $\log(\phi)$, the Cauchy-Riemann equations then imply that any critical point of $\log(\phi)$ must be a critical point of $h(\bz)$.

Putting everything together, the set of critical points of $\phi(\bz)$ equals the set of critical points of $h(\bz)$, which contains all local extrema of $|z_1 \cdots z_n|^{-1}$ on $\mV^*$.
\end{proof}

The critical points of $\phi$ are called the \emph{(smooth) critical points} of $F(\bz)$. The critical points of $F$ form an algebraic set which is easily characterized.

\begin{proposition}[{Pemantle and Wilson~\cite[Section 8.3]{PemantleWilson2013}}]
\label{prop:SmoothCrit}
When $\mV$ is a complex manifold and $H$ is square-free then $\bw \in \mV^*$ is a critical point if and only if 
\begin{equation} H(\bw)=0, \qquad w_1\left(\frac{\partial H}{\partial z_1}\right)(\bw) = \cdots = w_n\left(\frac{\partial H}{\partial z_n}\right)(\bw) \label{eq:critpt}. \end{equation} 
\end{proposition}

\begin{proof}
Let $H = H_1 \cdots H_r$ be a factorization of $H$ into distinct irreducible polynomials.  As $\mV$ is a manifold, given $\bw \in \mV$ there is a unique index $1 \leq j \leq r$ such that $H_j(\bw) = 0$.  Furthermore, $(\nabla H)(\bw)$ is a non-zero scalar multiple of $(\nabla H_j)(\bw)$ so the tangent space of $\mV$ at $\bw$ is the hyperplane with normal $(\nabla H)(\bw)$.  A critical point of the map $\phi:\mV^*\rightarrow\mathbb{C}$ is one where the projection of $\nabla \phi$ to the tangent space of $\mV^*$ is zero, meaning the critical points of $\phi$ are those where the modified Jacobian matrix 
\[ \begin{pmatrix} 
\nabla H \\
\nabla \phi
\end{pmatrix}
=
\begin{pmatrix} 
\partial H/\partial z_1 & \partial H/\partial z_2 & \cdots & \partial H/\partial z_d \\
z_2 \cdots z_n & z_1z_3\cdots z_n & \cdots & z_1 \cdots z_{n-1}
\end{pmatrix} \]
is rank deficient.  The set of equations generated by the vanishing of the $2\times2$ minors of this matrix result in Equations~\eqref{eq:critpt}, since $z_1 \dotsm z_n \neq 0$.
\end{proof}

Equations~\eqref{eq:critpt} are known as the \emph{smooth critical point equations}, and when $\mV$ is a complex manifold they can be taken to define critical points.  When $\mV$ is a manifold but $H$ is not square-free, the critical points of $F(\bz)$ can be found by replacing $H$ with the product of its distinct irreducible factors in Equations~\eqref{eq:critpt}.

In the case of the central binomial coefficients, the point $(1/2,1/2)$ was both critical and minimal which allowed us to perform our analysis.  Likewise, the existence of a finite number of minimal critical points in the general case usually (although not always) means that they are the ones where local behaviour determines dominant asymptotics.  This is made more precise in Theorem~\ref{thm:chiint} below.

\subsubsection*{Step 3: Localize the Cauchy Integral and Compute a Residue}

A point $\bz\in\mV$ is minimal if and only if $D(\bz) \cap \mV \subset \partial D(\bz)$, where $D(\bz)$ is the polydisk defined by $\bz$. When $D(\bz) \cap \mV$ is a finite subset of $\partial D(\bz)$ we call $\bz$ a \emph{finitely minimal} point, and when $D(\bz) \cap \mV = \{\bz\}$ we call $\bz$ a \emph{strictly minimal} point. %In fact, when $\mV$ is smooth and then a critical point is minimal if and only if the stronger condition $D(\bz) \cap \mV \subset T(\bz)$ is satisfied\footnote{See Definition 3.7 of Pemantle and Wilson~\cite{PemantleWilson2008} and the following remark.}.

Suppose now that $F(\bz)$ admits a strictly minimal critical point $\bw$.  As $\mV$ is a complex manifold and $H$ is square-free, there exists an index $j$ such that $(\partial H/\partial z_j)(\bw) \neq 0$ and without loss of generality we assume this holds for $j=n$.  Let $\rho = |w_n|$ and $\mT = T(\bwhtn)$, where we recall the notation $\bwhtn = (w_1,\dots,w_{n-1})$. 

Pemantle and Wilson~\cite[Section 9.2]{PemantleWilson2013} use the implicit function theorem to show the existence of $\delta \in (0,\rho)$, a neighbourhood $\mN$ of $\bwhtn$ in $\mT$, and an analytic function $g:\mN\rightarrow\mathbb{C}$ parameterizing $w_n$ on $\mV$, such that for $\bzhtn \in \mN$
\begin{enumerate}
	\item $H(\bzhtn,g(\bzhtn)) = 0$
	\item $\rho \leq |g(\bzhtn)| < \rho + \delta$
	\item $\rho = |g(\bzhtn)|$ if and only if $\bzhtn = \bwhtn$
	\item $H(\bzhtn,w)\neq0$ whenever $w \neq g(\bzhtn)$ and $|w|<\rho+\delta$.
\end{enumerate}
Define the integrals
\begin{align*}
I &:= \frac{1}{(2\pi i)^n} \int_\mT \left(\int_{|z_n|=\rho-\delta} F(\bz) \cdot \frac{dz_n}{z_n^{k+1}}\right) \frac{dz_1 \cdots dz_{n-1}}{z_1^{k+1}\cdots z_{n-1}^{k+1}} \\
I_{loc} &:= \frac{1}{(2\pi i)^n} \int_\mN \left(\int_{|z_n|=\rho-\delta} F(\bz) \cdot \frac{dz_n}{z_n^{k+1}}\right) \frac{dz_1 \cdots dz_{n-1}}{z_1^{k+1}\cdots z_{n-1}^{k+1}} \\
I_{out} &:= \frac{1}{(2\pi i)^n} \int_\mN \left(\int_{|z_n|=\rho+\delta} F(\bz) \cdot \frac{dz_n}{z_n^{k+1}}\right) \frac{dz_1 \cdots dz_{n-1}}{z_1^{k+1}\cdots z_{n-1}^{k+1}} \\
\chi &:= I_{loc}-I_{out} = \frac{-1}{(2\pi i)^n} \int_\mN \left(\int_{|z_n|=\rho+\delta} F(\bz) \cdot \frac{dz_n}{z_n^{k+1}} - \int_{|z_n|=\rho-\delta} F(\bz) \cdot \frac{dz_n}{z_n^{k+1}} \right) \frac{dz_1 \cdots dz_{n-1}}{z_1^{k+1}\cdots z_{n-1}^{k+1}}.
\end{align*}
By minimality of $\bw$, the Cauchy integral formula implies $f_{k,\dots,k} = I$.  Following arguments similar to the ones for the central binomial coefficients, it can be shown that $|I-I_{loc}|$ and $|I_{out}|$ grow exponentially slower than $|w_1 \cdots w_n|^{-1}$, so that
\[ f_{k,\dots,k} = \chi + O\left(\left(|w_1\cdots w_n|+\epsilon\right)^{-k} \right) \]
for some $\epsilon>0$.  For each $\bzhtn \in \mN$ the function $F(\bzhtn,w)$ has a unique singularity between the curves $|w|=\rho-\delta$ and $|w|=\rho+\delta$, which is a simple pole at the point $w = g(\bzhtn)$.  The Cauchy residue theorem then implies
\begin{equation} \chi = \frac{1}{(2\pi i)^{n-1}} \int_\mN \frac{-G(\bzhtn,g(\bzhtn))}{(\partial H/\partial z_n)(\bzhtn,g(\bzhtn))} \cdot \frac{dz_1 \cdots dz_{n-1}}{z_1^{k+1}\cdots z_{n-1}^{k+1}\cdot g(\bzhtn)^{k+1}},
\label{eq:genchi}
\end{equation}
as $\frac{G(\bzhtn,g(\bzhtn))}{(\partial H/\partial z_n)(\bzhtn,g(\bzhtn))}$ is the residue of $F(\bzhtn,w)$ at $w = g(\bzhtn)$. 
\\ 

To convert the expression in Equation~\eqref{eq:genchi} into a Fourier-Laplace integral we make the change of coordinates $z_j = w_j e^{i \theta_j}$ for $j=1,\dots,n-1$.  Let $\mN' \subset \mathbb{R}^{n-1}$ be the image of $\mN$ under this change of variables, which will be a neighbourhood of the origin.  To lighten notation we define $\bt = (\theta_1,\dots,\theta_{n-1})$ and write $\bwhtn e^{i\bt} := \left(w_1e^{i\theta_1},\dots,w_{n-1}e^{i\theta_{n-1}}\right)$.  After some simplification we obtain the following result.

\begin{theorem}[Theorem 9.2.1 and Proposition 9.2.5 of Pemantle and Wilson~\cite{PemantleWilson2013}] 
\label{thm:chiint}
Suppose that $\mV$ is smooth, $H$ is square-free, and $\bw$ is a strictly minimal critical point of $F(\bz)$.  Then there exists $\epsilon>0$ such that 
\[ f_{k,\dots,k} = \chi + O\left(\left(|w_1\cdots w_n|+\epsilon\right)^{-k} \right), \]
where
\begin{equation} 
\chi = (w_1\cdots w_n)^{-k} \cdot \frac{1}{(2\pi)^{n-1}} \int_{\mN'} A(\bt)\,e^{-k \phi(\bt)} d\bt
\label{eq:genFL}
\end{equation}
for
\begin{align}
\begin{split}
A(\bt) &= \frac{-G\left(\bwhtn e^{i\bt},g\left(\bwhtn e^{i\bt}\right)\right)}{g\left(\bwhtn e^{i\bt}\right) \cdot (\partial H/\partial z_n)\left(\bwhtn e^{i\bt},g\left(\bwhtn e^{i\bt}\right)\right)} \\[+2mm]
\phi(\bt) &= \log\left( \frac{g\left(\bwhtn e^{i\bt}\right)}{g(\bwhtn)} \right) + i(\theta_1+\cdots+\theta_{n-1}).
\end{split} \label{eq:Aphismooth}
\end{align}
The function $\phi(\bt)$ vanishes to order at least 2 at the origin.
\end{theorem}

When $\bw$ is a finitely minimal point, then one can perform the above analysis at each minimal critical point and obtain an asymptotic expression for the diagonal coefficient sequence as a finite sum of integrals having the form of Equation~\eqref{eq:genFL}, up to an exponentially small error term as above.

\subsubsection*{Step 4: Find Asymptotics using the Saddle-Point Method}

To determine asymptotics of the Fourier-Laplace integral in Equation~\eqref{eq:genFL} we will use a result of H{\"o}rmander\footnote{
The history of determining asymptotics of Fourier-Laplace integrals begins with the previously mentioned 18th century work of Laplace, who computed asymptotics for integrals of the form $\int_a^b A(x)e^{-k\phi(x)}$ where $A$ and $\phi$ are sufficiently smooth real valued functions.  The \emph{saddle-point method} determines asymptotics of integrals having the form $\int_{\gamma} A(z)e^{-k\phi(z)}$, with $\gamma$ a contour in the complex plane and $A$ and $\phi$ analytic functions.  It essentially works by deforming the domain of integration $\gamma$ into another contour $\gamma'$ where the real part of $\phi(z)$ is minimized at critical points of $\phi$ or end points of $\gamma'$, and was first published by Deybe~\cite{Debye1909} in 1909 who cited unpublished work of Riemann from 1863 (now paper XXIII in his collected works~\cite{Riemann1990}).  The related \emph{method of stationary phase} dates back to Stokes and Kelvin and computes asymptotics for integrals of the form $\int_a^b A(x)e^{i k \phi(x)}dx$ where $A$ and $\phi$ are sufficiently smooth real valued functions (when $A$ and $\phi$ are analytic then the method of stationary phase is essentially an example of the saddle-point method). Fedoryuk~\cite[Theorem 2.3]{Fedorjuk1971} gave the result in Proposition~\ref{prop:HighAsm} for multivariate real smooth functions using a generalization of the method of stationary phase, before H{\"o}rmander~\cite[Theorem 7.7.5]{Hormander1990a} gave the more general result where $\phi$ maps into the complex numbers. More recent literature on this subject is the focus of Section 7 in Pemantle and Wilson~\cite{PemantleWilson2010}.
}~\cite[Theorem 7.7.5]{Hormander1990a} from his asymptotic study of linear PDE solutions (using Laplace, Fourier, and Mellin transforms to solve differential equations often results in Fourier-Laplace integrals).  
%Lemma 13.3.2 in Pemantle and Wilson~\cite{PemantleWilson2013} is missing a minus sign on the differential operator $\mE$.

\begin{proposition}[Asymptotics of Nondegenerate Multivariate Fourier-Laplace Integrals] 
\label{prop:HighAsm}
Suppose that the functions $A(\bt)$ and $\phi(\bt)$ from $\mathbb{R}^d$ to $\mathbb{C}$ are smooth in a neighbourhood $\mN$ of the origin and let $\mH$ be the Hessian of $\phi$ evaluated at the point $\bt=\bzer$.  If
\begin{itemize}
	\item $\phi(\bzer)=0$ and $(\nabla \phi)(\bzer)=\bzer$
	\item the origin is the only point of $\mN$ where $\nabla \phi$ is 0
	\item $\mH$ is non-singular
	\item the real part of $\phi(\bt)$ is non-negative on $\mN$,
\end{itemize} 
then for any nonnegative integer $M$ there exist effective constants $C_0,\dots,C_M$ such that
\begin{equation}
\label{eq:Goal}
\int_{\mN} A(\bt)\,e^{-k \phi(\bt)} d\bt = \left(\frac{2\pi}{k}\right)^{d/2} \det(\mH)^{-1/2} \cdot \sum_{j=0}^M C_j k^{-j} + O\left(k^{-M-1}\right). 
\end{equation}
The constant $C_0$ is equal to $A(\bzer)$ and if $A(\bt)$ vanishes to order $L\geq1$ at the origin then (at least) the constants $C_0,\dots, C_{\lfloor\frac{L}{2}\rfloor}$ are all zero.  More precisely, define the differential operator 
\[\mE := - \sum_{1 \leq i,j \leq d} \left(\mH^{-1}\right)_{ij}\partial_i\partial_j\] 
where $\partial_j$ denotes differentiation with respect to the variable $\theta_j$ and $\mH^{-1}$ is the inverse of $\mH$.  Let
\[ \tilde{\phi}(\bt) := \phi(\bt) - (1/2)\bt \cdot \mH \cdot \bt^T, \]  
which is a scalar function vanishing to order 3 at the origin.  Then
\begin{equation} C_j = (-1)^j \sum_{0 \leq l \leq 2j} \frac{\left. \mE^{l+j}\left(A(\bt)\tilde{\phi}(\bt)^l\right)\right|_{\bt=\bzer}}{2^{l+j}l!(l+j)!}. \label{eq:FLconstant} \end{equation}
Due to the order of vanishing of $\tilde{\phi}$, to determine $C_j$ one only needs to calculate evaluations at $\bzer$ of the derivatives of $A$ of order at most $2j$ and the derivatives of $\phi$ of order at most $2j+2$.
\end{proposition}

In order to apply Proposition~\ref{prop:HighAsm} to the integral representation for the diagonal coefficients given in Theorem~\ref{thm:chiint} we require that the Hessian $\mH$ of $\phi$ at the origin is nonsingular, and that the real part of $\phi$ is non-negative on $\mN'$.

Given critical point $\bzeta$, define $\lambda$ to be the common value of $\zeta_k(\partial H/\partial_i)(\bzeta)$ for $1 \leq i \leq n$, and for $1\leq i,j \leq n$ define
\[ U_{i,j}:= \zeta_i\zeta_j \frac{\partial^2H}{\partial z_i\partial z_j}(\bzeta).\]
Basic multivariate calculus shows that the $(n-1)\times(n-1)$ Hessian matrix $\mH$ of $\phi$ in Equation~\eqref{eq:Aphismooth} at the origin has $(i,j)^\text{th}$ entry
\begin{equation} 
\label{eq:Hess}
\mH_{i,j} = 
\begin{cases}
1 + \frac{1}{\lambda}\left(U_{i,j} - U_{i,n} - U_{j,n} + U_{n,n}\right) &: i \neq j \\[+2mm]
2 + \frac{1}{\lambda}\left(U_{i,i} - 2U_{i,n} + U_{n,n}\right) &: i=j
\end{cases}
\end{equation}
We say that the critical point $\bzeta$ is \emph{nondegenerate} if this matrix is nonsingular; we will require that all minimal critical points are nondegenerate.  Furthermore, we note that the real part of $\phi$ can be expressed as
\[ \Re (\phi) = \Re \log\left( \frac{g\left(\bwhtn e^{i\bt}\right)}{g(\bwhtn)} \right) = \log \left|g\left(\bwhtn e^{i\bt}\right)\right| - \log |g(\bwhtn)|, \]
so that it is non-negative if and only if $|g(\bwhtn)| \leq \left|g\left(\bwhtn e^{i\bt}\right)\right|$ for $\bt$ in the neighbourhood $\mN'$ of the origin.  But when $\bw$ is strictly or finitely minimal this will hold for any sufficiently small neighbourhood $\mN'$, since $g\left(\bwhtn e^{i\bt}\right)$ gives the $z_n$ value of a point on $\mV$ whose first $n-1$ coordinates have the same coordinate-wise modulus as $\bwhtn$.  Thus, we obtain the following theorem, which is the main result of ACSV when $\mV$ is smooth.

\begin{theorem}[{Pemantle and Wilson~\cite[Theorem 9.2.7]{PemantleWilson2013}}]
\label{thm:smoothAsm}
Let $F(\bz)$ be a rational function with square-free denominator which is analytic at the origin and has a smooth singular variety $\mV$.  Assume that $F$ admits a nondegenerate strictly minimal critical point $\bw$ and that $(\partial H/\partial z_n)(\bw)\neq0$.  Then for any nonnegative integer $M$,
\begin{equation} 
f_{k,\dots,k} = (w_1\cdots w_n)^{-k} \cdot k^{(1-n)/2}\cdot (2\pi)^{(1-n)/2} \det(\mH)^{-1/2}  \left(\sum_{j=0}^M C_j k^{-j} + O\left(k^{-M-1}\right)\right), 
\label{eq:smoothAsm}
\end{equation}
where $\mH$ is the matrix defined by Equation~\eqref{eq:Hess} and $C_0,\dots,C_M$ are determined by Equations~\eqref{eq:Aphismooth} and~\eqref{eq:FLconstant}.  The leading constant $C_0$ in this series has the value
\[ C_0 = \frac{-G(\bw)}{w_n (\partial H/\partial z_n)(\bw)}, \]
which is nonzero whenever $G(\bw) \neq 0$.
\end{theorem}

Although the constants appearing in Equation~\eqref{eq:smoothAsm} are defined in terms of partial derivatives of the parametrization $g(\bzhtn)$ for $z_n$ on $\mV$, implicitly differentiating the equation $H(\bzhtn,g(\bzhtn))=0$ allows one to determine the partial derivatives of $g$ at $\bwhtn$ from the partial derivatives of $H$ at $\bw$.  Thus, the only pieces of information needed to determine the constants $C_0,\dots,C_M$ appearing in Theorem~\ref{thm:smoothAsm} are the evaluations at $\bz=\bw$ of the partial derivatives of $G(\bz)$ up to order $2M$ and the partial derivatives of $H(\bz)$ up to order $2M+2$.

The argument above can be easily adapted to a finitely minimal critical point $\bw$, when each minimal critical point with the same coordinate-wise modulus as $\bp$ satisfies the conditions of Theorem~\ref{eq:smoothAsm}. In that case one can simply compute the asymptotic contribution of each minimal critical point and add them up to determine dominant asymptotics. 

\begin{corollary}[{Pemantle and Wilson~\cite[Corollary 9.2.3]{PemantleWilson2013}}]
\label{cor:smoothAsm}
Let $F(\bz)$ be a rational function with square-free denominator which is analytic at the origin and has a smooth singular variety $\mV$.  Suppose $\bx$ is a finitely minimal critical point, let $E$ be the set of critical points in $T(\bx)$, and suppose all elements of $E$ are nondegenerate. For some nonnegative integer $M$, let $\Phi_{\bw}$ denote the right-hand side of Equation~\eqref{eq:smoothAsm} calculated at $\bw \in E$.  Then the equation
\[ f_{k,\dots,k} = \sum_{\bw \in E} \Phi_{\bw} \]
gives an asymptotic expansion of $f_{k,\dots,k}$ as $k\rightarrow\infty$.
\end{corollary}

\subsection*{A Multivariate Residue Approach}

The presentation above, which uses the Cauchy residue theorem to construct an explicit sum of Fourier-Laplace integrals, originates in the work of Pemantle and Wilson~\cite{PemantleWilson2002}.  A more recent approach, first discussed by Baryshnikov and Pemantle~\cite{BaryshnikovPemantle2011}, uses deep homological and cohomological tools to give methods dealing with certain situations where $F(\bz)$ has minimal critical points which are not finitely minimal.  We will require these results, which make use of the theory of multivariate complex residues\footnote{The theory of multivariate complex residues has its origins in work of Poincaré~\cite{Poincare1887} on integrals of bivariate functions and was investigated by Kodaira, Schwartz, and Dolbeault before a framework was fully developed by Leray~\cite{Leray1959} and Norguet~\cite{Norguet1959a}.  Details on the history of multivariate complex residues can be found in Dolbeault~\cite[Article V]{Vitushkin1990}, and a detailed treatment of multivariate residues is given in Chapter III of A{\u\i}zenberg and Yuzhakov~\cite{AuizenbergYuzhakov1983}.}, for our work on lattice path asymptotics. Pemantle and Wilson call the use of univariate residues the \emph{surgery approach} to analytic combinatorics in several variables, and the use of multivariate residues the \emph{residue approach}. 

Let $F(\bz)$ be a rational function with square-free denominator.  Suppose that $\bx \in \partial \mD$ minimizes $|z_1 \cdots z_n|^{-1}$ on $\overline{\mD}$, and that all minimizers of $|z_1 \cdots z_n|^{-1}$ on $\overline{\mD}$ lie in $T(\bx)$ (i.e., have the same coordinate-wise modulus as $\bx$). If $c$ is the minimum of $|z_1 \cdots z_n|^{-1}$ on $\overline{\mD}$, achieved at $\bx$, we assume that the set 
\[ \mV^{c-\epsilon} := \{ \bz \in \mV : |z_1 \cdots z_n|^{-1} \geq c - \epsilon \} \]
contains only smooth points for some $\epsilon>0$.  Finally, we let $E$ denote the set of critical points in $\mV \cap T(\bx)$, which we further assume is non-empty, finite, and contains only smooth nondegenerate critical points.  The main result of the residue approach is the following.

\begin{proposition}[{Pemantle and Wilson~\cite[Theorems 9.3.7 and 9.4.2]{PemantleWilson2013}}]
\label{prop:smoothAsmCP}
Suppose that $\bx \in \partial \mD$ minimizes $|z_1 \cdots z_n|^{-1}$ on $\overline{\mD}$, all minimizers of $|z_1 \cdots z_n|^{-1}$ on $\overline{\mD}$ lie in $T(\bx)$, the set $\mV^{c - \epsilon}$ contains only smooth points for some $\epsilon>0$, and that $E$ contains a single nondegenerate smooth critical point.  Then for any nonnegative integer $M$ there exist constants $C_0, \dots, C_M$ such that
\begin{equation} 
f_{k,\dots,k} = (w_1\cdots w_n)^{-k} \cdot k^{(1-n)/2}\cdot (2\pi)^{(1-n)/2} \det(\mH)^{-1/2}  \left(\sum_{j=0}^M C_j k^{-j} + O\left(k^{-M-1}\right)\right), 
\label{eq:smoothAsmRes}
\end{equation}
where 
\[ C_0 = \frac{-G(\bw)}{w_n (\partial H/\partial z_n)(\bw)}. \]
When $E$ contains a finite number of nondegenerate smooth critical points, one obtains an asymptotic expansion for $f_{k,\dots,k}$ by summing the right-hand side of Equation~\eqref{eq:smoothAsmRes} at each $\bw \in E$.
\end{proposition}

The residue approach relies on determining a homological object called the \emph{intersection class} related to $\mV$ and the minimal critical points, which is well understood in the smooth case.  In particular, Section 9.3 of Pemantle and Wilson~\cite{PemantleWilson2013} shows that one obtains an expression for diagonal coefficient asymptotics in terms of Fourier-Laplace integrals over chains of integration $C(\bw)$ sufficiently close to each $\bw \in \mV$.  Theorem 9.4.2 of that text, together with the results of its Appendix B, show that one can take any submanifolds $C(\bw) \subset \mV$ which are diffeomorphic to open disks of real dimension $n-1$ such that $|z_1 \cdots z_n|^{-1}$ is strictly maximized on $C(\bw)$ at $\bz=\bw$.  This allows one to obtain the higher order constants\footnote{In the non-smooth case, as we will see in Chapter~\ref{ch:NonSmoothACSV}, less is known about the intersection class and explicit formulas for the higher order constants are harder to derive for non-finitely minimal critical points.} $C_j$ in Proposition~\ref{prop:smoothAsmCP}.

Suppose there exists an open $\mV$-neighbourhood $\mN_{\bw}$ of each $\bw \in E$ such that $\mN_{\bw} \cap T(\bx) = \{\bw\}$.  When $(\partial H/\partial z_n)(\bw) \neq 0$, then there exists an analytic parametrization $z_n = g(\bzhtn)$ in a neighbourhood of $\bw$ in $\mV$ and one can take\footnote{On $C(\bw)$, one has $|z_1 \cdots z_n|^{-1} = |w_1 \cdots w_{n-1}|^{-1} \cdot |g(\bwhtn e^{i\bt})|^{-1}$ and every point of $\mN_{\bw}$ whose first $n-1$ coordinates have the same coordinate-wise modulus as $\bwhtn$ have final coordinate with modulus larger than $w_n$.  In addition, Pemantle and Wilson show that one can always take $C(\bw)$ to be the \emph{downwards subspace} of $\mV$ at $\bw$ with respect to $h(\bz)$, obtained by writing $z_j=x_j + iy_j$ for real variables $x_j$ and $y_j$ and examining the Hessian of the map $(x_1^2+y_1^2)^{-1} \cdots (x_n^2+y_n^2)^{-1}$ restricted to the underlying real smooth manifold of $\mV$ (see Pemantle and Wilson~\cite[Section 8.5]{PemantleWilson2013} for details).}
\[ C(\bw) = \left\{ \left( \bwhtn e^{i\bt}, g(\bwhtn e^{i\bt}) \right) : \bt \in (-\epsilon,\epsilon)^{n-1} \right\} \]
for sufficiently small $\epsilon>0$.  In this situation, one ultimately derives the same Fourier-Laplace integrals which are used to determine asymptotics in the finitely minimal case.  Thus, we obtain the following result.

\begin{corollary}[{Pemantle and Wilson~\cite[Theorems 9.3.2 and 9.4.2]{PemantleWilson2013}}]
\label{cor:smoothAsmCP}
Suppose that the assumptions of Proposition~\ref{prop:smoothAsmCP} hold, and that the points of $E$ are isolated points of $\mV \cap T(\bx)$.  For every positive integer $M>0$, an asymptotic expansion of $f_{k,\dots,k}$ is obtained by summing the right-hand side of Equation~\eqref{eq:smoothAsm} in Theorem~\ref{thm:smoothAsm} at each $\bw \in E$.
\end{corollary}

Finally, we note that one does not need to prove that a smooth minimal critical point minimizes $|z_1 \cdots z_n|^{-1}$ on $\overline{\mD}$ (only that there are no other minimizers with different coordinate-wise moduli) in light of the following result.

\begin{proposition}
\label{prop:smoothmincrit}
If $\bw \in \mV \cap \partial \mD$ is a smooth minimal critical point then $\bw$ is a minimizer of $|z_1 \cdots z_n|^{-1}$ on $\overline{\mD}$.
\end{proposition}
\begin{proof}
Any minimizer of the map $|z_1 \cdots z_n|^{-1}$ on $\overline{\mD} \cap \left(\mathbb{C}^*\right)^n$ is a maximizer of $\log|z_1 \cdots z_n|$ on the same domain. As described in Proposition~\ref{prop:convLaurent} above, the image of $\overline{\mD}$ under the $\Relog$ map is a convex set $B \subset \mathbb{R}^n$, and under the change of coordinates $\bx = \Relog(\bz)$ the function $\log|z_1 \cdots z_n|$ becomes the linear function $\bone \cdot \bx$.  It can be shown (see, for example, Pemantle and Wilson~\cite[Proposition 3.12]{PemantleWilson2008}) that $\bw \in \mV \cap \partial \mD$ is a smooth minimal critical point if and only if the hyperplane with normal $\bone$ containing the point $\Relog(\bw)$ is an outwardly oriented support hyperplane to the convex set $B$.  By definition, this means that $\bone \cdot \bx \leq \bone \cdot \Relog(\bw)$ for all $\bx \in B$ so that $|z_1 \cdots z_n| \leq |w_1 \cdots w_n|$ for all $\bz \in \overline{\mD}$.
\end{proof}
\smallskip

\section{Applying the Theory in the Smooth Case}

Given a rational function $F(\bz)$, it is easy to check if the singular variety $\mV$ is everywhere smooth by computing a Gr\"obner Basis of the system of equations $H, z_1(\partial H/\partial z_1), \dots, z_n(\partial H/\partial z_n)$ (or by using the multivariate resultant, which will be discussed in Chapter~\ref{ch:EffectiveACSV}).  In fact, `almost all' rational functions admit smooth singular varieties and have a finite set of critical points\footnote{This is made precise in Chapter~\ref{ch:EffectiveACSV}.}.  Thus, the most difficult step in trying to apply the above results is often determining which of a finite set of critical points is minimal.  When $H(\bz)$ is simple enough, direct arguments can be used.

\begin{example}[Central Binomial Coefficients Revisited]
Consider again the bivariate rational function
\[ F(x,y) = \frac{1}{1-x-y} \]
whose diagonal encodes the central binomial coefficients.  Here the critical point equations become
\[ 1-x-y = 0, \qquad -x = -y, \]
so that there is a single critical point $(1/2,1/2)$.  In fact, $(1/2,1/2)$ is a strictly minimal critical point as $|x+y|=1$ on $\mV$ and it is the only point on $\mV$ with coordinates of modulus 1/2.  Thus, Theorem~\ref{thm:smoothAsm} applies and we can determine an asymptotic expansion
\[ \binom{2k}{k} = \frac{4^k}{\sqrt{\pi k}}\left(1-\frac{1}{8k}+\frac{1}{128k^2}+\frac{5}{1024k^3}-\frac{21}{32768k^4} + O\left(\frac{1}{k^5}\right)\right). \]
\end{example}

\begin{example}[Perturbed Central Binomial Coefficients]
Consider now the bivariate rational function
\[ F(x,y) =\frac{1}{(1+2x)(1-x-y)}. \]
The singular variety of $F$ has a single non-smooth point $(x,y) = (-1/2,3/2)$, which is not minimal as it has strictly greater coordinate-wise modulus than $\bp=(1/2,1/2)$. The point $\bp$ is still a minimal critical point, however it is no longer finitely minimal as 
\[ \mV \cap T(\bp) = \{(1/2,1/2)\} \quad \cup \quad \{(-1/2,e^{i\theta}/2) : \theta \in (-\pi,\pi)\}.\]  
Because the singular variety is smooth for all points with $|xy|^{-1} > |(-1/2)(3/2)|^{-1} = 4/3$, the conditions of Corollary~\ref{cor:smoothAsmCP} are met and we obtain the asymptotic expansion
\[ f_{k,k} = \frac{4^k}{\sqrt{\pi k}}\left(\frac{1}{2}-\frac{1}{8k}+\frac{1}{256k^2}+\frac{5}{256k^3}-\frac{819}{65536k^4} + O\left(\frac{1}{k^5}\right)\right). \]
The idea behind Corollary~\ref{cor:smoothAsmCP} in this example is that the domain of integration in the Cauchy integral formula can be deformed around any singularities which are bounded away from critical points without affecting dominant asymptotics (up to an exponentially small error), and near the minimal critical point $(1/2,1/2)$ the singular variety $\mV((1+2x)(1-x-y))$ looks like $\mV(1-x-y)$. 
\end{example}

In the multivariate setting there is an analogue of Pringsheim's Theorem which can greatly help with arguments to determine minimality. We call the rational function $F(\bz) = \frac{G(\bz)}{H(\bz)}$ \emph{combinatorial} if $G$ and $H$ are co-prime, $H(\bzer)\neq0$, and the power series expansion of $1/H(\bz)$ at the origin has only a finite number of negative coefficients. 

\begin{lemma} 
\label{lem:combCase}
Suppose that $F(\bz)$ is combinatorial.  Then $\bz \in \mV$ is a minimal point of $F(\bz)$ if and only if the point $(|z_1|,\dots,|z_d|)$ with non-negative coordinates is a minimal point (i.e., is in $\mV$).
\end{lemma}

This is essentially Theorem 3.16 of Pemantle and Wilson~\cite{PemantleWilson2008}, although our definition of combinatorality is slightly weaker (in that paper they require $F(\bz)$ to have all non-negative coefficients).  It is sufficient to examine the coefficients of $1/H(\bz)$ as their arguments use only properties of the singular set $\mV(H)$.  Furthermore, we can allow a finite number of negative coefficients since one can always add a polynomial to $1/H(\bz)$ and obtain a rational function with non-negative coefficients, the same set of singularities, and all but a finite number of the same coefficients.  Lemma~\ref{lem:combCase} implies that it is easier to prove a point is minimal when $F(\bz)$ is combinatorial.

\begin{proposition}
\label{prop:lineMin}
A point $\bw \in \mV$ is minimal if and only if there does not exist $\bz \in \mV$ with $(|w_1|,\dots,|w_n|) = (t|z_1|,\dots,t|z_n|)$ and $t \in (0,1)$.  If $F(\bz)$ is combinatorial then $\bw \in \mV$ is a minimal point if and only if $(|w_1|,\dots,|w_n|) \in \mV$ and the line segment 
\[ \{(t|w_1|,\dots,t|w_n|) : 0 < t < 1\}\] 
from the origin to $(|w_1|,\dots,|w_n|)$ in $\mathbb{R}^n$ does not contain an element of $\mV$. 
\end{proposition}

\begin{proof}
If $\bw$ is minimal then there cannot exist $t \in (0,1)$ such that $(|w_1|,\dots,|w_n|) = (t|z_1|,\dots,t|z_n|)$.  If $\bw$ is not minimal, then $(|w_1|,\dots,|w_n|)$ lies outside of the closed convex set $\overline{\Relog(\mD)} \subset \mathbb{R}^n$, so any path in $\mathbb{R}^n$ from $(|w_1|,\dots,|w_n|)$ to the open set $\Relog(\mD)$ must pass through $\partial \Relog(\mD)$.  Thus, when $\bw$ is not minimal there exists some $\bz \in \mV$ and $t \in (0,1)$ such that 
\[ (\log|z_1|,\dots,\log|z_n|) = (\log|w_1| + \log t,\dots,\log|w_n| + \log t),\]
since $\Relog(\mD)$ contains all points with negative coordinates of sufficiently large modulus whose ratio approaches 1. Taking the exponential of this equation implies 
\[ (|w_1|,\dots,|w_n|) = (t|z_1|,\dots,t|z_n|).\]
When $F(\bz)$ is combinatorial, Lemma~\ref{lem:combCase} shows that it is sufficient to consider only the points in $\mV \cap \left(\mathbb{R}_{>0}\right)^n$ to determine the minimality of $\bw$.  
\end{proof}

Although Lemma~\ref{lem:combCase} can be seen as a multivariate generalization of Pringsheim's Theorem one must note that it is restrictive as it requires all coefficients of the power series expansion to be non-negative, not just those on the diagonal which are of (combinatorial) interest.  As mentioned in Chapter~\ref{ch:OtherSources}, it is still unknown even in the univariate case how to decide when a rational function is combinatorial.  In practice, then, one usually applies these results when $F(\bz)$ is the multivariate generating function of a combinatorial class with parameters, or when the form of $F(\bz)$ makes combinatorality easy to prove (for instance, when $F(\bz) = \frac{G(\bz)}{1-J(\bz)}$ with $J(\bz)$ a polynomial vanishing at the origin with non-negative coefficients).

\begin{example}[Asymptotics of Simple Walks in a Quarter Plane]
\label{ex:simplesmoothwalk}
From the results of Chapter~\ref{ch:KernelMethod} we know that the diagonal of the rational function 
\[ F(x,y,t) = \frac{(1+x)(1+y)}{1-t(x^2y+xy^2+y+x)}\] 
is the generating function for the number of lattice path walks starting at the origin, taking the steps $(\pm1,0),(0,\pm1)$, and staying in the first quadrant.  The critical point equations imply that there are two critical points,
\[ \bp = (1,1,1/4) \qquad \text{and} \qquad \bs = (-1,-1,-1/4),\] 
and $F(x,y,t)$ is clearly combinatorial by the Binomial Theorem.  Proposition~\ref{prop:lineMin} then implies $\bp$ and $\bs$ are minimal critical points, as if $(x,y,t) \in \mV$ has positive coordinates and $x\leq1,y\leq1$ with one of the inequalities being strict, then 
\[ |t| = \left|\frac{1}{x^2y+xy^2+y+x}\right| > 1/4.\]  
Furthermore, if $|x|=1$ and $|y|=1$ for $x,y \in \mathbb{C}$ then 
\[ |x^2y+xy^2+x+y|=4\] 
only if $x^2$ and $y^2$ are real\footnote{If $|x^2y+xy^2+x+y|=4$ and $|x|,|y| \leq 1$ then $|x|=|y|=1$ and $xy^2$ and $x$ have the same argument (as do $x^2y$ and $y$).} and have modulus 1. It can then be checked that the only other point of $\mV$ with the same coordinate-wise modulus as $\bp$ is $\bs$, so these points are finitely minimal.

Corollary~\ref{cor:smoothAsm} implies that only $\bp$ contributes to the dominant asymptotics of the diagonal sequence, as the numerator $G(x,y) = (1+x)(1+y)$ vanishes (to order 2) when $(x,y,t) = \bs$.  The contributions from each minimal critical point, to order 4, are
\begin{align*} 
\Phi_{\bp} &= 4^k \left( \frac{4}{\pi k} - \frac{6}{\pi k^2}+\frac{19}{2\pi k^3} - \frac{121}{12\pi k^4} + O\left(\frac{1}{k^5}\right)\right) \\
\Phi_{\bs} &= (-4)^k \left( \frac{1}{\pi k^3} - \frac{9}{2\pi k^4} + O\left(\frac{1}{k^5}\right)\right).
\end{align*}
Note that the presence of two minimal critical points leads to periodicity in the higher order asymptotic terms: 
\[ f_{k,k,k} = 4^k \left( \frac{4}{\pi k} - \frac{6}{\pi k^2} + \frac{19 + 2(-1)^k}{2\pi k^3} - \frac{121+54(-1)^k}{12\pi k^4} + O\left(\frac{1}{k^5}\right)  \right). \]
\end{example}

Unfortunately, the minimizer of the upper bound $|z_1\cdots z_n|^{-1}$ on $\partial \mD$ does not need to be a critical point.  In fact, when $F(\bz)$ is not combinatorial it is possible to have no critical points, even if the minimum of $|z_1\cdots z_n|^{-1}$ on $\overline{\mD}$ is achieved and the singular variety $\mV$ is smooth.

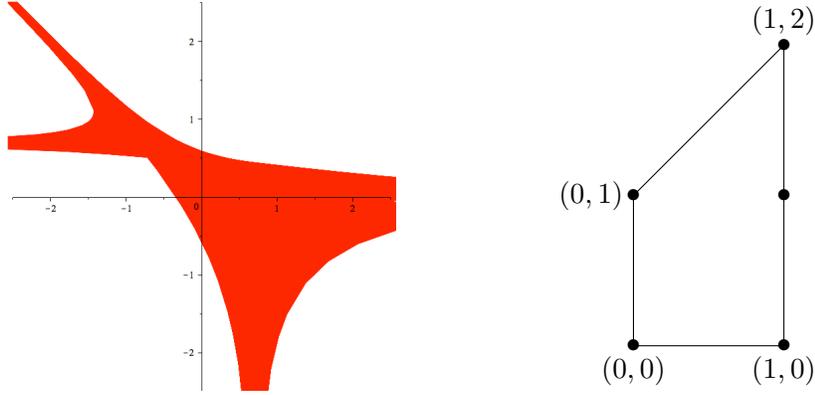
\begin{figure}
\centering
	\begin{minipage}{.4\textwidth}
        \centering
        \includegraphics[width=0.8\linewidth]{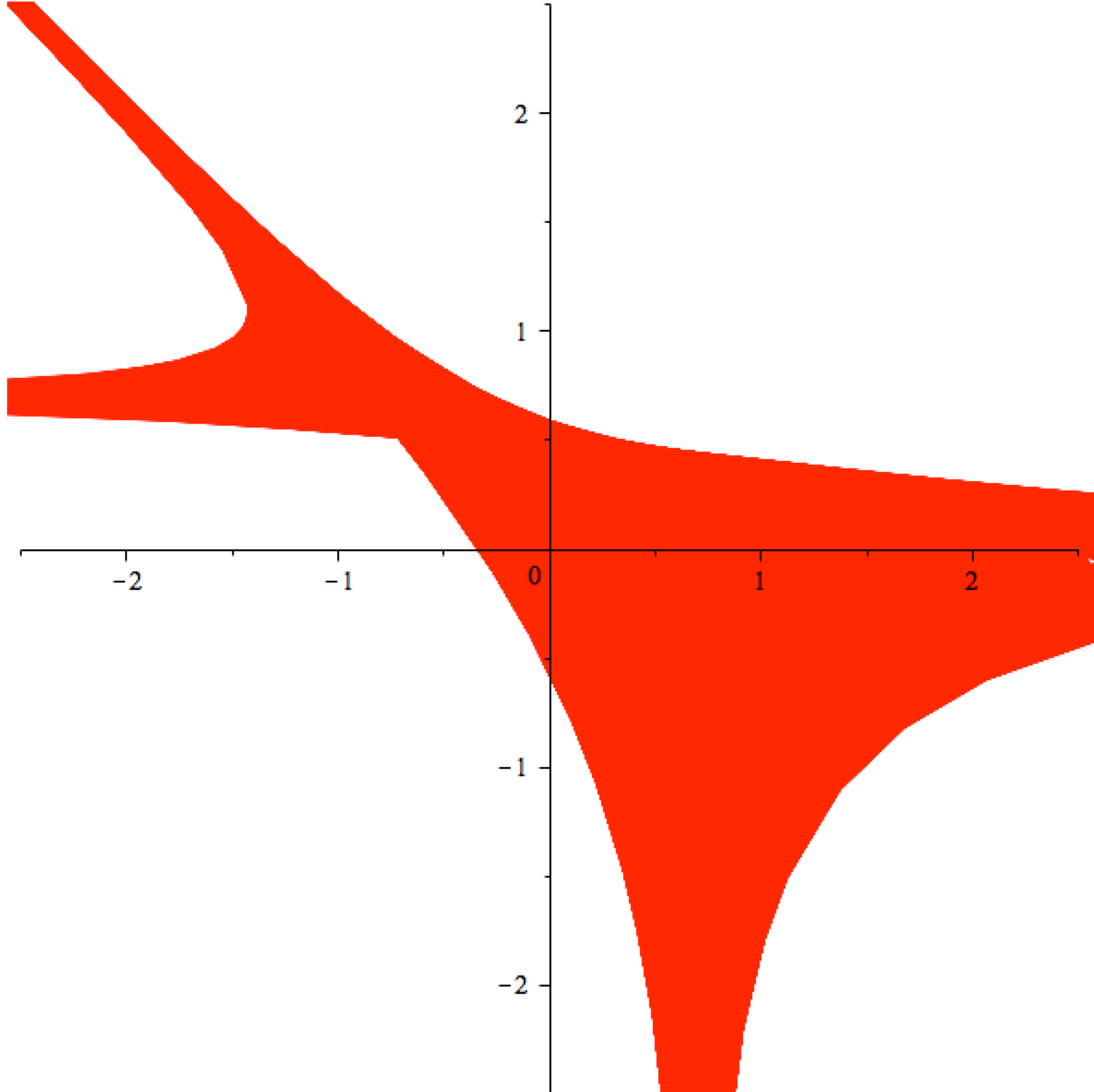}
    \end{minipage}%
    \begin{minipage}{0.4\textwidth}
        \centering
	\begin{tikzpicture}
	\draw (0,0) node[anchor=north]{$(0,0)$}
	-- (2,0) node[anchor=north]{$(1,0)$}
	-- (2,4) node[anchor=south]{$(1,2)$}
	-- (0,2) node[anchor=east]{$(0,1)$}
 	-- cycle;
 	\foreach \Point in {(0,0), (0,2), (2,0), (2,2), (2,4)}{
    	\node at \Point {\textbullet};
	}
	\end{tikzpicture}
    \end{minipage}
    \caption[The amoeba and Newton polygon of $2+y-x(1+y)^2$]{The amoeba (left) and Newton polygon (right) of $2+y-x(1+y)^2$. Note the correspondence between the limit directions of the amoeba and the outward normals to the edges of the Newton polygon.}
    \label{fig:amoeba2xy}
\end{figure}

\begin{example}
\label{ex:smoothnomincrit}
Consider the bivariate function
\[ F(x,y) = \frac{G(x,y)}{H(x,y)} = \frac{1}{2+y-x(1+y)^2}.\]
Computing resultants, or using Gr\"obner Bases, it is easy to show that both systems of polynomials 
\[ H=\partial H/\partial x = \partial H/\partial y=0 \qquad \text{and} \qquad H=x(\partial H/\partial x) - y(\partial H/\partial y)=0\] 
have no solutions, so the singular variety $\mV$ is smooth and there are no critical points.  The limit directions of $\amoeba(H)$ -- that is, the set of vectors $\bv \in \mathbb{R}^2$ such that $\bx + r\bv \in \amoeba(H)$ for some $\bx \in \mathbb{R}^2$ and all $r>0$ -- are given by the outward normal directions to the Newton polytope of $H$ on each of its edges~\cite[Theorem 9.6]{Sturmfels2002}, which is illustrated in Figure~\ref{fig:amoeba2xy}.  Since the Newton polygon of $H$ has edges with outward normals $(-1,0)$ and $(0,-1)$, the set $B = \overline{\Relog(\mD)} \subset \mathbb{R}^2$ is a closed convex set contained in some translation of the third quadrant of the plane.  This implies the linear function $- \bone \cdot \bx$ achieves its minimum on $B$, so $|z_1 \cdots z_n|^{-1}$ achieves its minimum on $\overline{\mD}$.  

The minimizers of $|z_1 \cdots z_n|^{-1}$ on $\overline{\mD}$ are not local minimizers of $|z_1 \cdots z_n|^{-1}$ on $\mV$, they only become minimizers when $\mV$ is mapped into $\mathbb{R}^2$ via the Relog map. One can imagine $\mV$ wrapping over itself in complex space, so that $\mV$ is smooth but the boundary of $\amoeba(H)$ is not.
\end{example}

In Chapter~\ref{ch:EffectiveACSV} we will see that $F(\bz)$ being combinatorial makes a large difference to the complexity of determining dominant asymptotics with the algorithms discussed in this thesis.  

\section{Further Examples}

We now return to several of the examples from Chapter~\ref{ch:OtherSources}.  In addition, a detailed treatment of a class of lattice path models will be given in Chapter~\ref{ch:SymmetricWalks}.

\begin{example}[The Apéry Numbers]
\label{ex:Apery2}
In Example~\ref{ex:Apery} we saw the sequence $(b_k)$ of Apéry numbers, whose generating function could be written as the rational diagonal
\[ \Delta F(x,y,z,t) = \Delta\left(\frac{1}{1-t(1+x)(1+y)(1+z)(1+y+z+yz+xyz)}\right). \]
This rational function is clearly combinatorial, and solving the critical point equations gives two smooth critical points, of which one has positive coordinates:
\[ \bp = \left(1+\sqrt{2}, \frac{\sqrt{2}}{2}, \frac{\sqrt{2}}{2}, -82+58\sqrt{2}\right). \]
Proposition~\ref{prop:lineMin} implies that $\bp$ is a minimal critical point, since 
\[ t = \frac{1}{(1+x)(1+y)(1+z)(1+y+z+yz+xyz)} \]
at any singular point, and when $x,y,$ and $z$ are positive and real then decreasing any of their values causes this expression to increase in value. Using an argument similar to the one in Example~\ref{ex:simplesmoothwalk}, it can be shown that $\bp$ is finitely minimal.  Alternatively, as $F$ is combinatorial, smooth, and admits a finite number of critical points, Lemma~\ref{lemma:comb_two_min_crit} in Chapter~\ref{ch:EffectiveACSV} will imply that all minimizers of $F$ on $\overline{\mD}$ have the same coordinate-wise modulus as $\bp$. 

In any case, we obtain dominant asymptotics:
\[ b_k = \frac{(17+12\sqrt{2})^k}{k^{3/2}} \cdot \frac{\sqrt{48+34\sqrt{2}}}{8\pi^{3/2}}\left(1 + O\left(\frac{1}{k}\right)\right). \]
\smallskip

The generating function of the second sequence of Apéry numbers $(c_n)$ can be written as the diagonal
\[ \Delta F(x,y,z) =  \Delta\left( \frac{1}{1-z(1+x)(1+y)(1+y+xy)}\right). \]
Again $F$ is combinatorial, and an analogous argument shows 
\[ c_k = \frac{\left(\frac{11}{2}+\frac{5\sqrt{5}}{2}\right)^k}{k} \cdot \frac{\sqrt{250+110\sqrt{5}}}{20\pi}\left(1 + O\left(\frac{1}{k}\right)\right). \]
We treat these examples algorithmically in Examples~\ref{ex:Apery3a} and~\ref{ex:Apery3b} of Chapter~\ref{ch:EffectiveACSV}.
\end{example}
\smallskip

\begin{example}[Singular Vector Tuples of Generic Tensors]
\label{ex:SingularTuple2}
In Section~\ref{sec:SingularTuple} we encountered the rational function
\[ F(\bz) = \frac{z_1 \cdots z_n}{(1-z_1) \cdots (1-z_n)\left( 1 - \sum_{i=2}^n (i-1)e_i(\bz) \right)}, \]
where $e_i(\bz)$ is the $i$th elementary symmetric function 
\[ e_i(\bz) = \sum_{1 \leq j_1 < \cdots < j_i \leq n} z_{j_1} \cdots z_{j_i}. \]
This rational function is combinatorial (indeed, it is the multivariate generating function of a combinatorial class with parameters).  Furthermore, it can easily be verified that
\[ \bp = \left(\frac{1}{n-1},\dots,\frac{1}{n-1}\right) \]
is a smooth point (the partial derivatives of the denominator do not simultaneously vanish at this point) and it satisfies the smooth critical point equations\footnote{One can derive $\bp$ by solving the critical point equations, or by using the symmetry of $F(\bz)$ and the structure of the monomials appearing in its denominator to argue that any smooth critical point must have equal coordinates; see Pantone~\cite{Pantone2017} for details.}.  Any minimal point either has a coordinate equal to 1, which does not contradict the minimality of $\bp$, or satisfies $1 - \sum_{i=2}^n (i-1)e_i(\bz)=0$.  As $F$ is combinatorial, Proposition~\ref{prop:lineMin} implies that $\bp$ is minimal as long as
\[ 1 - \sum_{i=2}^n (i-1)e_i(r,\dots,r) = (r+1)^{n-1}(r(1-n)+1) \neq 0  \]
for $r \in (0,1/(n-1))$, which is true.  A few simple computations (contained in Pantone~\cite{Pantone2017}) give the unknown quantities in Theorem~\ref{thm:smoothAsm} and verify that this strictly minimal critical point is nondegenerate, yielding the asymptotic expansion
\[ C_n(k) = \frac{(n-1)^{n-1}}{(2\pi)^{(n-1)/2}n^{(n-2)/2}(n-2)^{(3n-1)/2}} \cdot \left((n-1)^n\right)^k \cdot k^{(1-n)/2}\left( 1 + O\left(\frac{1}{n}\right)\right) \]
for the diagonal sequence.
\end{example}
\smallskip

\begin{example}[Mirror Families of Calabi-Yau Varieties]
\label{ex:CalabiYau2}
Recall the discussion in Example~\ref{ex:CalabiYau} of Chapter~\ref{ch:OtherSources}.  The database of Lairez~\cite{Lairez} gives annihilating differential equations for the principle periods of the varieties determined by Batyrev and Kreuzer~\cite{BatyrevKreuzer2010}.  The first entry, ``polytope v6.1'', has a principal period given by the diagonal of the rational function
\[ F(w,x,y,z,t) = \frac{1}{1-twxyz\left(\frac{1}{wxz} + y + x + w + z + \frac{1}{wxy}\right)} \]
and is annihilated by the differential operator
\[ \mathcal{L} = t^3(32t^2-1)(32t^2+1)\partial_t^4+2t^2(7168t^4-3)\partial_t^3+t(55296t^4-7)\partial_t^2+(61440t^4-1)\partial_t+12288t^3. \]
The differential equation $\mathcal{L} \cdot f=0$ has a basis of 4 solutions, $f_{\pm1},f_{\pm i}$, with $f_{\omega}$ admitting a singularity at $\omega 4\sqrt{2}$ (these singularities are the roots of the leading polynomial factor $(32t^2-1)(32t^2+1)$ of $\mathcal{L}$).  The power series coefficients of $f_{\omega}$ have dominant asymptotics\footnote{This basis of solutions is determined up to a constant scaling which is fixed by these asymptotic expansions.  The basis and their coefficient asymptotics were calculated with the \detokenize{ore_algebra} package of Sage~\cite{KauersJaroschekJohansson2015}. As of February 1, 2017 the version of \detokenize{ore_algebra} bundled with Sage does not run, but an extension developed by Marc Mezzarobba works and is available at \url{http://marc.mezzarobba.net/code/ore_algebra-analytic/}.}
\[ \frac{(\omega 4\sqrt{2})^k}{k^2}\left(1 - \frac{7}{4k} +  O\left(\frac{1}{k^3}\right)\right). \] %\frac{49}{32k^2} +
Here the critical point equations have 4 solutions, all with the same coordinate-wise modulus, and a short argument (analogous to the one presented in Example~\ref{ex:simplesmoothwalk}) shows that they are finitely minimal.  Applying Corollary~\ref{cor:smoothAsm} to these minimal critical points, which are nondegenerate, gives
\[ f_{k,\dots,k} = \sum_{\omega \in \{\pm1,\pm i\}} \frac{(\omega 4\sqrt{2})^k}{k^2}\left(\frac{2\sqrt{2}}{\pi^2} - \frac{7\sqrt{2}}{2k\pi^2} + O\left(\frac{1}{k^3}\right) \right). \]
This implies that the connection constants for the generating function $(\Delta F)(t)$ are equal to $2\sqrt{2}/\pi^2$ for each $f_{\pm1},f_{\pm i}$.  
\end{example}

In Example~\ref{ex:CalabiYau2} one could determine the connection coefficients for the diagonal directly from the dominant (first order) asymptotics of its coefficients, but it can happen that higher order asymptotics are needed.  Examining exactly when and how the connection problem can be solved for rational diagonals using this approach is ongoing work.

\section{Generalizations}

In Chapter~\ref{ch:NonSmoothACSV} we will discuss the theory of ACSV when $\mV$ does not define a smooth manifold, but before moving on we illustrate a few generalizations of the theory in the smooth case.

\subsection*{Expansions in other Directions}
Given the multivariate rational function $F(\bz)$, we constructed a sequence by studying the diagonal coefficients $[z_1^k \cdots z_n^k]F(\bz)$ as $k\rightarrow\infty$.  This diagonal construction is very useful, as it can encode a wide variety of sequences, but it is only one of many possible coefficient subsequences of $F(\bz)$.  For instance, given $r_1,\dots,r_n \in \mathbb{Q}_{>0}$ one can assume, after a possible scaling of the variables of $F(\bz)$, that the $r_j$ are positive integers and determine asymptotics of the sequence
\[ [z_1^{r_1 \cdot k} \cdots z_n^{r_n \cdot k}]F(\bz) = \frac{1}{(2\pi i)^n} \int \frac{F(\bz)}{\left(z_1^{r_1} \cdots z_n^{r_n}\right)^k} \frac{dz_1 \cdots dz_n}{z_1 \cdots z_n} \]
following the analytic framework presented above.  Not only can the above results be re-derived in this context, but the methods of Pemantle and Wilson show how, in the presence of nondegenerate minimal critical points, uniform asymptotic estimates can often be obtained as $k \rightarrow \infty$ and $\br = (r_1,\dots,r_n)$ varies smoothly around some fixed direction\footnote{More generally, given any $n$-tuple of increasing functions $r_1(k),\dots,r_n(k)$ one can ask about asymptotics of the sequence $[z_1^{r_1(k)} \cdots z_n^{r_n(k)}]F(\bz)$.  When any of the $r_j$ are super-linear then the coefficient sequence $[z_1^{r_1(k)} \cdots z_n^{r_n(k)}]F(\bz)$ will typically grow or decay super-exponentially, and the methods of this chapter do not apply.}.  This can lead to powerful statements when $F(\bz)$ is a multivariate generating function, and applications to computing distributions of parameters of combinatorial classes are given in Section 9.6 of Pemantle and Wilson~\cite{PemantleWilson2013}.  

\subsection*{Expansions in Other Domains}
Throughout this chapter we assumed that $F(\bz)=G(\bz)/H(\bz)$ was analytic at the origin, and determined coefficient asymptotics from its power series expansion.  As seen in Section~\ref{sec:LaurentExp}, however, $F(\bz)$ will have several well-defined convergent Laurent expansions, each corresponding to a connected component of $\mathbb{R}^n \setminus \amoeba(H)$.  Given such a component $B \subset\mathbb{R}^n$,  Proposition~\ref{prop:convLaurent} generalizes the Cauchy integral formula to give an analytic expression for the coefficients of the corresponding Laurent series.  If one defines a minimal point for this Laurent expansion to be a point $\bw \in \mV(H$) such that $\Relog(\bw) \in \partial B$ then the results derived above continue to hold\footnote{Our definition of critical points depends only on the singular variety $\mV$ and not on the domain of convergence under consideration.  In fact, one potential source of non-minimal critical points for diagonal coefficients of power series expansions are critical points which determine diagonal asymptotics when $F(\bz)$ is expanded into a Laurent series over another domain.}.

\subsection*{Diagonals of Multivariate Algebraic Functions}
Recent work of Greenwood~\cite{Greenwood2015,Greenwood2016} has shown how to determine diagonal asymptotics of bivariate functions of the form $F(x,y) = \frac{G(x,y)}{H(x,y)^{\beta}}$, where $G$ and $H$ are analytic functions, $\beta \in \mathbb{R} \setminus \mathbb{Z}_{\leq 0}$, and the zero set of $H$ is a smooth manifold.  Thus, one can determine asymptotics in the presence of some algebraic singularities.  Although the diagonal of a bivariate algebraic function can be expressed as the diagonal of a four variable rational function, nice bivariate expressions which arise in applications can become very involved\footnote{Recall Example~\ref{ex:nPartEx}, for instance.}, making it harder (or impossible with the currently developed theory) to find the singularities contributing to dominant asymptotics.  In order to work with algebraic singularities Greenwood constructs explicit contours, which look similar to Hankel contours, to avoid branch cuts instead of using products of circles as done above in the smooth case.  Forthcoming work of Greenwood\footnote{Personal communication from Torin Greenwood.} extends this result from the bivariate case to any number of variables.

%%%%%%%%%%%%%%%%%%%%%%%%%%%
% Chapter 7
%%%%%%%%%%%%%%%%%%%%%%%%%%%
\chapter{Orthant Walks with Highly Symmetric Step Sets}
\label{ch:SymmetricWalks}
\vspace{-0.2in}

\begin{center}This chapter is based on an article of Melczer and Mishna~\cite{MelczerMishna2016}.\end{center}

\setlength{\epigraphwidth}{3.2in}
\epigraph{This symmetrical composition\dots may seem quite ``novelistic'' to you, and I am willing to agree, but only on condition that you refrain from reading such notions as ``fictive,'' ``fabricated,'' and ``untrue to life'' into the word ``novelistic.'' Because human lives are composed in precisely such a fashion.\footnotemark}{Milan Kundera, \emph{The Unbearable Lightness of Being}}
\footnotetext{Translated from the Czech by Michael Henry Heim.}

In this chapter we give an in-depth treatment of a problem from lattice path enumeration, using the techniques of ACSV in the smooth case.  Recall that in Chapter~\ref{ch:KernelMethod} we saw how the kernel method can be used to represent generating functions of two dimensional walks restricted to a quadrant as diagonals of rational functions.  When the group of transformations $\mG$ is finite, and the orbit sum $\sum_{\sigma \in \mG} \sgn(\sigma)\sigma(xy)$ is non-zero, Theorem~\ref{thm:nzOrbDiag} gives a representation for the generating function of the number of walks of length $k$ as the diagonal of an explicit rational function.

Of the 19 models to which Theorem~\ref{thm:nzOrbDiag} applies, only\footnote{Additionally, the diagonal expression arising from the model taking steps $\{(-1,1),(1,-1),(\pm1,0)\}$, known as \emph{Gouyou-Beauchamps'} model, can be converted into a representation of this form.  Its asymptotics are discussed in Chapter~\ref{ch:QuadrantLattice} and weighted generalizations of the model are discussed in Chapter~\ref{ch:WeightedWalks}.} the following 4 have a representation of the form $\Delta(G(x,y,t)/H(x,y,t))$ where $\mV(H)$ is globally smooth and $H(0,0,0) \neq 0$
\[ \diagrF{N,S,E,W} \qquad\qquad \diagrF{NE,NW,SE,SW} \qquad\qquad \diagrF{NE,NW,SE,SW,N,S}  \qquad\qquad \diagrF{NE,NW,SE,SW,N,S,E,W} \]
Note that these models are precisely the ones which are symmetric over both the $x$ and $y$-axes. Because of these symmetries, the group $\mG$, and thus the orbit sum, does not depend on the underlying step set.  For any of these 4 models, Theorem~\ref{thm:nzOrbDiag} implies that the generating function counting the total number of walks of a fixed length has the representation
\[ Q(1,1,t) = \Delta \left( \frac{(1+x)(1+y)}{1-txyS(x,y)}\right), \]
where we recall that for the model defined by the step set $\mS$ one has $S(x,y) = \sum_{(i,j) \in \mS}x^iy^j$.  In this chapter we show that a similar representation exists for higher dimensional models with symmetric step sets restricted to an orthant, and derive asymptotics for the number of walks in such models.  

The link between symmetry in a model's step set and a rational diagonal representation which has a smooth singular variety is the first hint of a principle we will see several times in this thesis: combinatorial models which have ``nice'' properties (like underlying symmetry) often admit rational diagonal representations with ``nice'' properties (like smooth minimal critical points).  In addition to dealing with models in arbitrary dimension, we also allow each step set to have positive real weights.

\subsection*{Setup and Statement of Results}

Fix a dimension $n \geq 1$ and let $\mS \subset \{\pm1,0\}^n \setminus \{\bzer\}$.  The \emph{characteristic polynomial} of $\mS$ is the Laurent polynomial
\[ S(\bz) := \sum_{\bi \in \mS} \bz^{\bi}. \]
We say that $\mS$ is \emph{non-trivial} if for each coordinate there are steps with $-1$ and $1$ in that coordinate, and call $\mS$ \emph{highly symmetric} if 
\[ S(z_1,\dots,z_{j-1},\oz_j,z_{j+1},\dots,z_n) = S(\bz) \] 
for each $j=1,\dots,n$ (equivalently, negating the $j$th coordinate of all steps in $\mS$ fixes $\mS$ for each $j$).  In order to allow for weights, we assign to each $\bss \in \mS$ a positive real number $a_{\bss}>0$ and define
\[ S_{\ba}(\bz) := \sum_{\bss \in \mS} a_{\bss} \bz^{\bss}. \]
An unweighted model can be realized as a weighted one where each weight $a_{\bss}$ equals 1.  A weighted model is called \emph{highly symmetric} if 
\[ S_{\ba}(z_1,\dots,z_{j-1},\oz_j,z_{j+1},\dots,z_n) = S_{\ba}(\bz) \] 
for each $j=1,\dots,n$.  Given step set $\mS$ and weights $\ba$ we form the multivariate generating function
\[ Q_{\ba}(\bz,t)= \sum_{\substack{\bw \textrm{ walk in $\mathbb{N}^n$ starting at }\bzer\\\textrm{ending at }\bi \\ \textrm{of length }k}} \, \prod_{\substack{ \bss \textrm{ step in } \bw \\ \textrm{(with multiplicity)}}} \hspace{-22pt} a_{\bss} \hspace{15pt} \, \bz^{\bi} t^k \]
in $\mathbb{R}[\bz][[t]]$, and note that $Q_{\ba}(\bone,t)$, and $Q_{\ba}(\bzer,t)$, are the univariate generating functions counting weighted walks of length $k$ ending anywhere, and ending at the origin, respectively.  The main theorem of this chapter is the following.

\begin{theorem}[{Melczer and Mishna~\cite[Theorem 3.4]{MelczerMishna2016}}]
\label{thm:symAsm}
Let $\mS \subset \{\pm1,0\}^n\setminus \{\bzer\}$ be a non-trivial highly symmetric step set with positive weights $\ba$.  Then the number of weighted walks of length $k$ beginning at the origin, staying in the non-negative orthant $\mathbb{N}^n$, and ending anywhere has dominant asymptotics
\[ [t^k]Q_{\ba}(\bone,t) = S(\bone)^k \cdot k^{n/2} \cdot S(\bone)^{n/2}\pi^{-n/2}\left(s^{(1)} \cdots s^{(n)}\right)^{-1/2}\left(1 + O\left(\frac{1}{k}\right)\right), \]
where $s^{(j)} = {\Big(}[z_j]S(\bz){\Big)}{\Big|}_{\bz=\bone}$ is the weight of steps which move forward in the $j$th coordinate.
\end{theorem}
\smallskip

\begin{example} 
When $n=2$ there are four non-isomorphic unweighted highly symmetric models in the quarter plane, whose asymptotics are listed in Table~\ref{tab:quartersym}.  This proves the guessed asymptotics of Bostan and Kauers~\cite{BostanKauers2009} for these models.
\end{example}

\begin{table}
\centering
\begin{tabular}{ | c | c @{ \hspace{0.01in} }@{\vrule width 1.2pt }@{ \hspace{0.01in} } c | c | }
  \hline
   $\mS$ & Asymptotics & $\mS$ & Asymptotics \\ \hline
  &&&\\[-5pt] 
  \diag{N,S,E,W}  & $\displaystyle \frac{4}{\pi\sqrt{1 \cdot 1}} \cdot k^{-1} \cdot 4^k = \frac4\pi \cdot \frac{4^k}k$ & 
  \diag{NE,SE,NW,SW} & $\displaystyle \frac{4}{\pi\sqrt{2 \cdot 2}} \cdot k^{-1} \cdot 4^k = \frac2\pi \cdot \frac{4^k}k$ \\[+7mm]
  %%%%%%%%%%%%%
  \diag{N,S,NE,SE,NW,SW} & $\displaystyle \frac{6}{\pi\sqrt{3 \cdot 2}} \cdot k^{-1} \cdot 6^k = \frac{\sqrt{6}}\pi \cdot \frac{6^k}k$ &
  \diag{N,S,E,W,NW,SW,SE,NE} & $\displaystyle \frac{8}{\pi\sqrt{3 \cdot 3}} \cdot k^{-1} \cdot 8^k = \frac{8}{3\pi} \cdot \frac{8^k}k$ \\[5mm]
\hline
\end{tabular}\\[2mm]

\caption{The four highly symmetric models with unit steps in the quarter plane.} \label{tab:quartersym}
\end{table}

\begin{example}
Let $\mS = \{\pm1,0\}^n \setminus \{\mathbf{0}\}$, which is highly symmetric when unweighted. Then $S(\bone) = |\mS|=3^n-1$, and $s^{(j)}=3^{n-1}$ for all $j$, so the total number of walks satisfies
\[ [t^k]Q_{\ba}(\bone,t) = \left( (3^n-1)^k \cdot k^{-n/2} \cdot \frac{(3^n-1)^{n/2}}{3^{n(n-1)/2}\cdot\pi^{n/2}}\right) \left(1 + O\left(\frac{1}{k}\right)\right). \] 
\end{example}

Our arguments allow us to determine higher order terms in the expansion, however such terms will typically have periodic behaviour and are not as simple to state.  We also derive the following result on walks returning to boundary regions of the non-negative orthant.

\begin{theorem}
\label{thm:symAsmEx}
Let $\mS \subset \{\pm1,0\}^n\setminus \{\bzer\}$ be a non-trivial highly symmetric step set with positive weights $\ba$.  Then the number of weighted walks of length $k$ beginning at the origin, staying in the non-negative orthant, and ending on the intersection of $r$ of the boundary hyperplanes $\{z_j = 0\}$ has asymptotic growth of order $O\left(S(\bone)^k k^{-n/2-r}\right)$.  In particular, the number of such walks ending at the origin satisfies
\[ [t^k]Q_{\ba}(\bzer,t) = O\left(\frac{S(\bone)^k}{k^{3n/2}}\right). \]
\end{theorem}

The statement for walks returning to the origin in Theorem~\ref{thm:symAsmEx} first appeared in the work of Melczer and Mishna~\cite{MelczerMishna2016}, and was later re-derived by D'Arco et al.~\cite{DArcoLacivitaMustapha2016} using results from potential theory.

\section{The Kernel Method in Higher Dimensions}

In order to prove Theorem~\ref{thm:symAsm} we generalize the kernel method described in Chapter~\ref{ch:KernelMethod} to models in higher dimensions with highly symmetric step sets.  As for all previously examined cases, the recursive decomposition of a walk of length $k+1$ as a walk of length $k$ plus a valid step gives a functional equation satisfied by the generating function $Q_{\ba}(\bz,t)$. To ensure that walks remain in the non-negative orthant, we must not count walks which add a step with a negative $j$th component to a walk ending on the hyperplane $z_j = 0$. To account for this, it is sufficient to subtract the term $t\oz_jQ_{\ba}(z_1,...,z_{j-1},0,z_{j+1},...,z_n,t)$ from the functional equation for unrestricted walks. However, if a given step has several negative components we must use the principle of inclusion and exclusion to prevent over compensation.  These considerations lead to the functional equation
\begin{align}
(z_1 \cdots z_n) Q_{\ba}(\bz,t)=& (z_1 \cdots z_n) + t(z_1 \cdots z_n)S_{\ba}(\bz)Q_{\ba}(\bz,t) \nonumber \\
&- t\sum_{\varnothing \neq V \subset \{1,\dots,n\}} (-1)^{|V|-1}{\Big(} (z_1\cdots z_n)S_{\ba}(\bz,t)Q_{\ba}(\bz,t) {\Big)} {\Big|}_{\{z_j=0:j\in V\}} \label{eq:higherdimkernel}
\end{align}
as the term
\smallskip
\[ {\Big(} (z_1\cdots z_n)S_{\ba}(\bz,t){\Big)}{\Big|}_{\{z_j=0:j\in V\}}\] 
is the result of extracting the coefficient of $\prod_{j \in V}z_j^{-1}$ in $S_{\ba}(\bz)$.  Rearranging this expression gives the following result.

\begin{lemma}[{Melczer and Mishna~\cite[Lemma 2.1]{MelczerMishna2016}}]
Let $Q_{\ba}(\bz,t)$ be the multivariate generating function described above. Then there exist $A_1,\dots,A_n$, with $A_r \in \mathbb{R}[\bzht{r}][[t]]$ for $r=1,\dots,n$, such that
\begin{equation} (z_1 \cdots z_n)(1-tS_{\ba}(\bz))Q_{\ba}(\bz,t) = (z_1 \cdots z_n) + \sum_{r=1}^n A_r(\bzht{r},t). \label{eq:symker} \end{equation}
\end{lemma}

\subsection*{A Diagonal Representation}

Following the kernel method for walks in the quarter plane, we look for rational transformations of $\mathbb{R}^n$ which fix $S_{\ba}(\bz)$.  Because our model is highly symmetric we can replace any of our variables by their reciprocals and preserve $S_{\ba}(\bz)$. Thus, we define the (abelian) group $\mG$ of $2^n$ rational maps by
\[ \mG = \left\{ (z_1,\dots,z_n)\mapsto (z_1^{i_1},\dots,z_n^{i_n}) : \bi \in \{\pm1\}^n \right\}. \]
Given $\sigma \in \mG$ we consider $\sigma$ as map on $\mathbb{R}[\bz,\overline{\bz}][[t]]$ through the group action $\sigma(A(\bz,t)) := A(\sigma(\bz),t)$ for $A \in \mathbb{R}[\bz,\overline{\bz}][[t]]$.  Furthermore, we define the sign of $\sigma \in \mG$ by
\[ \sgn(\sigma) = (-1)^{\#\{j:\sigma(z_j) = \oz_j \}}, \]
and for $j=1,\dots,n$ we let $\sigma_j$ be the map which sends $z_j$ to $\oz_j$ and fixes all other variables.  These are direct generalizations of the kernel method for walks in the quarter plane, simplified due to the fact that the models we consider are highly symmetric. 
\smallskip

\begin{lemma}[{Melczer and Mishna~\cite[Lemma 2.3]{MelczerMishna2016}}]
\label{lemma:highdimkersum}
Let $Q_{\ba}(\bz,t)$ be the multivariate generating function described above. Then 
\begin{equation} \sum_{\sigma \in \mG} \sgn(\sigma) \sigma(z_1 \cdots z_n) Q_{\ba}(\sigma(\bz),t) = \frac{\sum_{\sigma \in \mG} \sgn(\sigma) \sigma(z_1 \cdots z_n)}{1-tS(\bz)} \label{eq:sympos} \end{equation}
as elements of $\mathbb{R}[\bz,\overline{\bz}][[t]]$.
\end{lemma}
\smallskip

\begin{proof}
As $S(\bz)$ is fixed by the elements of $\mG$, to prove Equation~\eqref{eq:sympos} from Equation~\eqref{eq:symker} it is sufficient to show that for each $r=1,\dots,n$,
\[ \sum_{\sigma \in \mG} \text{sgn}(\sigma) \sigma(A_r(\bzht{r},t)) = 0. \]
Fix $r$ and write $\mG$ as the disjoint union $\mG = \mG_0 \cup \mG_1$, where 
\begin{align*}
\mG_0 &= \left\{ \sigma_1^{j_1} \cdots \sigma_{n-1}^{j_{n-1}} \sigma_n^{j_n} : j_1,\dots,j_n \in \{0,1\}, j_r=0 \right\} \\
\mG_1 &= \left\{ \sigma_1^{j_1} \cdots \sigma_{n-1}^{j_{n-1}} \sigma_n^{j_n} : j_1,\dots,j_n \in \{0,1\}, j_r=1 \right\}.
\end{align*}
For all $\sigma \in \mG_1$, $(\sigma_r \sigma)(A_r(\bzht{r},t)) = \sigma(A_r(\bzht{r}),t))$ while $\sgn(\sigma_r \sigma) = -\sgn(\sigma)$, so
\begin{align*}
\sum_{\sigma \in \mG} \sgn(\sigma) \sigma(A_r(\bzht{r},t)) 
&= \sum_{\sigma \in \mG_0}\sgn(\sigma) \sigma(A_r(\bzht{r},t)) + \sum_{\sigma \in \mG_1} \sgn(\sigma) \sigma(A_r(\bzht{r},t)) \\
&= \sum_{\sigma \in \mG_0}\sgn(\sigma) \sigma(A_r(\bzht{r},t)) + \sum_{\sigma \in \mG_0} \sgn(\sigma_r\sigma) \cdot (\sigma_r\sigma)(A_r(\bzht{r},t)) \\
&= \sum_{\sigma \in \mG_0} \left(\sgn(\sigma) - \sgn(\sigma) \right) \sigma(A_r(\bzht{r},t)) \\
&= 0.
\end{align*}
Applying each $\sigma \in \mG$ to Equation~\eqref{eq:symker} and summing the results weighted by $\sgn(\sigma)$ then cancels each of the terms $A_r$ on the right-hand side, and Equation~\eqref{eq:sympos} follows.
\end{proof}

Since $\sgn(\sigma)$ is the number of variables which $\sigma$ sends to their reciprocals, the orbit sum simplifies to
\[\sum_{\sigma \in \mG} \sgn(\sigma) \sigma(z_1 \cdots z_n) = (z_1-\oz_1) \cdots (z_n-\oz_n). \]
Furthermore, unless $\sigma \in \mG$ is the identity there exists an index $j$ such that $\sigma(z_j)=\oz_j$, meaning $\sigma(z_1 \cdots z_n) Q_{\ba}(\sigma(\bz),t)$ contains only strictly negative powers of $\oz_j$.  Combining these two observations implies
\begin{equation} Q_{\ba}(\bz,t) = [z_1^{\geq}]\cdots[z_n^{\geq}]\left(\frac{(z_1-\oz_1) \cdots (z_n-\oz_n)}{(z_1 \cdots z_n)(1-tS_{\ba}(\bz))}\right),
\label{eq:symNonneg} \end{equation}
and an application of Proposition~\ref{prop:postodiag} yields the following.

\begin{proposition}
\label{prop:symDiag}
Let $\mS \subset \{\pm1,0\}^n \setminus \{\bzer\}$ be a non-trivial highly symmetric step set under the positive weighting $\ba$.  Then the generating function for the number of weighted walks beginning at the origin, using the steps $\mS$ and staying in the non-negative orthant is given by
\[ Q_{\ba}(\bone,t) = \Delta \left( \frac{(1+z_1) \cdots (1+z_n)}{1-t(z_1\cdots z_n)S_{\ba}(\bz)} \right). \]
Furthermore, for any $V \subset \{1,\dots,n\}$ the generating function for the number of such walks which end on the hyperplane intersection $\{z_j=0 : j \in V\}$ is given by
\[ \Delta \left( \prod_{j \in V}(1-z_j) \cdot \frac{(1+z_1) \cdots (1+z_n)}{1-t(z_1\cdots z_n)S_{\ba}(\bz)} \right). \]
\end{proposition}

The second statement follows from the fact that the generating function for the number of walks ending on the hyperplane intersection $\{z_j=0 : j \in V\}$ is given by $Q(\bw,t)$ where $w_j=0$ for $j \in V$ and $w_j=1$ otherwise.

The calculation which derives Proposition~\ref{prop:symDiag} from Equation~\eqref{eq:symNonneg} hints at why the highly symmetric models are those for which the kernel method gives rational diagonal expressions with smooth singular varieties.  In order to move from a non-negative series extraction of a multivariate generating function to a diagonal representation of a univariate generating function, Proposition~\ref{prop:postodiag} introduces the factors $(1-z_1) \cdots (1-z_n)$ into the denominator of the rational function under consideration.  One can generalize the group of a walk from the quadrant case to short step models in arbitrary dimensions which are not necessarily highly symmetric.  When $S_{\ba}(\bz)$ is invariant under the transformation $z_j \mapsto \oz_j$ then the map $\sigma_j$ will be one of the generators of the group $\mG$, and the orbit sum $O(\bz) = \sum_{\sigma \in \mG}\sgn(\sigma)\sigma(z_1\cdots z_n)$ contains a factor\footnote{If $\sigma_j\in\mG$ then $O(\bz)|_{z_j=1} = \sum_{\sigma \in \mG}\sgn(\sigma_j\sigma)(\sigma_j\sigma)(z_1\cdots z_n){\Big|}_{z_j=1} = \sum_{\sigma \in \mG}\sgn(\sigma_j\sigma)\sigma(z_1\cdots z_n){\Big|}_{z_j=1} = -O(\bz)|_{z_j=1}$.} of $1-z_j$.  Although the orbit sum can be divisible by $(1-z_1)\cdots(1-z_n)$ without $\mS$ being highly symmetric, this occurs rarely (for instance, in two dimensions it only holds for 1 of the 15 non-highly symmetric models with finite group and non-zero orbit sum).

\section{An Application of ACSV in the Smooth Case}

The next step is to apply the results of ACSV to the rational diagonals given in Proposition~\ref{prop:symDiag}.  To begin, let
\begin{align*}
G(\bz,t) &= (1+z_1) \cdots (1+z_n)\\
H(\bz,t) &= 1-t(z_1\cdots z_n)S_{\ba}(\bz)
\end{align*}
so that the generating function $Q_{\ba}(\bone,t)$ equals $\Delta(G/H)$.  The singular variety $\mV=\mV(H)$ is a complex manifold as $H$ and $(\partial H/\partial t)$ cannot simultaneously vanish.

\subsubsection*{Determining Minimal Critical Points}

Since we only consider highly symmetric models, for each $r=1,\dots,n$ there exist unique Laurent polynomials $U_r(\bzht{r})$ and $V_r(\bzht{r})$ such that
\[ S(\bz) = (\oz_r + z_r) U_r(\bzht{r}) + V_r(\bzht{r}). \] 
The equation $t(\partial H/\partial t) = z_r (\partial H/\partial z_r)$ states
\[ t(z_1\cdots z_n) S_{\ba}(\bz) = t(z_1 \cdots z_n)S_{\ba}(\bz) + t z_r (z_1\cdots z_n)(\partial S_{\ba}/\partial z_r) \]
which implies
\[ 0 = t z_r (z_1\cdots z_n)(\partial S_{\ba}/\partial z_r) = t(z_r^2-1)(z_1 \cdots z_{r-1}z_{r+1} \cdots z_n)U_r(\bzht{r}). \]
This gives the following characterization of critical points.
\smallskip

\begin{lemma}[{Melczer and Mishna~\cite[Proposition 3.1]{MelczerMishna2016}}]
\label{lemma:symmcrit}
The point $(\bz,t) \in \mV$ is a critical point if and only if for each $1 \leq r \leq n$ either:
\begin{itemize}
	\item $z_r = \pm1$ or,
	\item the polynomial $(y_1 \cdots y_{r-1}y_{r+1} \cdots y_n)U_r(\mathbf{y}_{\hat{r}})$ has a root at $\bzht{r}$.
\end{itemize}
\end{lemma}

Note that it is possible to have an infinite set of critical points due to the second condition (this cannot happen in two dimensions but does occur when $n \geq 3$). 

\begin{example}
Consider the unweighted highly symmetric model in three dimensions restricted to the non-negative octant taking the twelve steps 
\[ \mS = \{(-1,0,\pm1), (1,0,\pm1), (0,1,\pm1), (0,-1,\pm1), (\pm1,1,0),(\pm1,-1,0)\}. \]
Then
\begin{align*} 
H(x,y,z,t) &= 1-t(xyz)\sum_{\bss \in \mS}x^{s_1}y^{s_2}z^{s_3} \\
&= 1-t(z^2+1)(x+y)(xy+1)-tz(y^2+1)(x^2+1)
\end{align*}
and solving the system of smooth critical point equations via a Gröbner basis computation gives the two isolated critical points
\[ \left(1,1,1,\frac{1}{12}\right) \text{ and } \left(-1,-1,-1,\frac{-1}{12}\right) \]
together with a collection of non-isolated critical points
\[
\left(x,1,-1,\frac{1}{4x}\right), \left(x,-1,1,\frac{1}{4x}\right),
\left(1,y,-1,\frac{1}{4y}\right), \left(-1,y,1,\frac{1}{4y}\right),
\left(1,-1,z,\frac{1}{4z}\right), \left(-1,1,z,\frac{1}{4z}\right)
\]
for $x,y,z \in \mathbb{C}$.  
\end{example}

\begin{proposition}
\label{prop:symmCritMin}
The point 
\[ \bp = \left(1,\dots,1,\frac{1}{S_{\ba}(\bone)}\right)\] 
is a smooth finitely minimal critical point and there are at most $2^n$ critical points in $T(\bp) \cap \mV$.  Any critical point $(\bz,t) \in T(\bp) \cap \mV$ satisfies $\bz  \in \{\pm1\}^n$.
\end{proposition}
\begin{proof}
The point $\bp$ is critical by Lemma~\ref{lemma:symmcrit}.  Suppose $(\bw,t_\bw)$ lies in $D(\bp) \cap \mV$, where we note that any choice of $\bw$ uniquely determines $t_\bw$ on $\mV$.  Then
\begin{equation*}
\left| \sum_{\bi \in \mS} a_{\bi} \bw^{\bi + \bone} \right| = \left| (w_1\cdots w_n)\sum_{\bi \in \mS} a_{\bi} \bw^{\bi} \right| = \left|\frac{1}{t_\bw}\right| \geq S_{\ba}(\bone) = \sum_{\bi \in \mS} a_{\bi}.
\end{equation*}
Since $(\bw,t_\bw) \in D(\bp)$ implies $|w_j|\leq1$ for each $1\leq j \leq n$, and each weight $a_{\bi}$ is positive\footnote{This is why we restrict ourselves to positive weights.}, the only way this can hold is if $|w_j|=1$ for each $j=1,\dots,n$, and $\bw^{\bi+\bone}$ has the same complex argument for all $\bi \in \mS$.  

By symmetry, and the assumption that we take a positive step in each direction, the set $\{\bw^{\bi + \bone} : \bi \in \mS\}$ contains two elements of the form 
\[ w_2^{i_2+1}\cdots w_n^{i_n+1} \quad \text{and} \quad w_1^2w_2^{i_2+1}\cdots w_n^{i_n+1}, \] 
so $w_1^2$ must be real in order for them to have the same argument.  Thus, $w_1=\pm1$ and applying the same argument to each coordinate gives the stated result.
\end{proof}

\subsection*{Calculating Asymptotics}

We have determined that the collection of minimal critical points for $G(\bz,t)/H(\bz,t)$ form the finite set
\[ E = \left\{ \left(\bw,\frac{1}{S_{\ba}(\bw)}\right) :\quad \bw \in \{\pm1\}^n, \quad \left| S_{\ba}(\bw) \right| = S_{\ba}(\bone) \right\}. \]
In order to find asymptotics using Corollary~\ref{cor:smoothAsm} it remains only to determine the matrix $\mH$ whose entries are given in Equation~\eqref{eq:Hess}.  A direct calculation shows that 
% for $\bw \in \{\pm1\}$ and $1 \leq i < j \leq n$,
% \[ (\partial S/\partial z_i)(\bw) = (\partial^2 S/\partial z_i \partial z_j)(\bw) = 0 \]
% while
% \[ (\partial^2 S/\partial z_i^2)(\bw) = 2w_i U_i(\bw), \]
% so that
\[ w_iw_j(\partial^2 H/\partial z_i \partial z_j)(\bw) = \begin{cases} 0 &: i \neq j \\ 2 \frac{U_j(\bw)}{S_{\ba}(\bw)} &: i=j \end{cases}  \]
so that $\mH$ is the diagonal matrix
\[ \mH = \frac{2}{S_{\ba}(\bw)} 
\begin{pmatrix} 
U_1(\bw) & 0 & \cdots & 0 \\ 
0 & U_2(\bw) & \ddots & \vdots \\
\vdots & \ddots & \ddots & \vdots \\
0 & \cdots & 0 & U_n(\bw)
 \end{pmatrix}. \]
Since $G(\bz) = (1+z_1) \cdots (1+z_n)$ does not vanish at $\bp$, but vanishes at any other minimal critical point, Corollary~\ref{cor:smoothAsm} implies that it is the only point whose asymptotic contribution affects the dominant asymptotics of the diagonal.  Using the quantities computed above with Corollary~\ref{cor:smoothAsm} gives the asymptotics listed in Theorem~\ref{thm:symAsm}:
\[ [t^k]Q_{\ba}(\bone,t) = S(\bone)^k \cdot k^{n/2} \cdot S(\bone)^{n/2}\pi^{-n/2}\left(s^{(1)} \cdots s^{(n)}\right)^{-1/2}\left(1 + O\left(\frac{1}{k}\right)\right), \]
where $s^{(j)} = U_j(\bone)$.  In order to determine higher order terms in this asymptotic expansion, the contributions from other minimal critical points must also be calculated.  This can be done automatically for any explicit step set.  

\subsection*{Walks Returning to the Boundary}

Proposition~\ref{prop:symDiag} gives a rational diagonal representation for walks ending in the set $\{z_j = 0 : j \in V \}$ for any subset $V \subset \{1,\dots,n\}$.  By possibly reordering coordinates, if $V$ contains $r>0$ elements we may assume that $V = \{1,\dots,r\}$, obtaining the rational diagonal expression
\[ \Delta\left(\frac{G(\bz)}{H(\bz,t)}\right) = \Delta\left(\frac{(1-z_1^2)\cdots(1-z_r^2)(1+z_{r+1})\cdots(1+z_n)}{1-t(z_1\cdots z_n)S_{\ba}(\bz)}\right). \]
The denominator $H(\bz,t)$ is the same as in our analysis above, meaning the set of minimal critical points is unchanged.  Now, however, the numerator $G(\bz)$ vanishes at all minimal critical points. Corollary~\ref{cor:smoothAsm} and Proposition~\ref{prop:HighAsm} shows that high-order asymptotic terms are obtained by applying powers of the differential operator
\[ \mE = - \frac{S_{\ba}(\bw)}{2}\sum_{i=0}^n \partial_{\theta_i}^2 \]
to an analytic function containing
\[ A(\bt) = \left(1-e^{2i\theta_1}\right) \cdots \left(1-e^{2i\theta_r}\right)\left(1+e^{i\theta_{r+1}}\right)\cdots\left(1+e^{i\theta_n}\right) \]
as a factor, and setting $\bt = \bzer$.  The power series expansion of $A(\bt)$ at the origin has lowest order term $2^n (-i)^r \theta_1 \cdots \theta_r$, meaning the lowest power of $\mE$ that can be applied to $A(\bt)$ in order to give a non-zero value when evaluated at the origin is $\mE^r$.  Corollary~\ref{cor:smoothAsm} then gives the order bounds listed in Theorem~\ref{thm:symAsmEx}, and dominant asymptotics can be calculated automatically and explicitly for any given step set.

\begin{example} 
Consider the unweighted two dimensional model with step set 
\[ \diagrF{N,S,NE,SE,NW,SW}\] 
whose generating function (counting walks ending anywhere in the quarter plane) is given by 
\[ Q(1,1,t) = \Delta \left( \frac{(1+x)(1+y)}{1-t(x^2y^2+y^2+x^2+xy^2+x+1)}  \right). \]
To determine the set of minimal critical points we substitute all values of $(x,y) \in \{\pm1\}^2$ into 
\[ H(x,y,t) = 1 - t(1+y^2+x+xy^2+x^2+x^2y^2) = 0, \]
solve the resulting expression for $t$, and check whether the corresponding solution $t_{x,y}$ satisfies $|t_{x,y}| = 1/|S(\bone)| = 1/6$.  Of the four possible points, we get only two minimal critical points: the expected point $\bp = (1,1,1/6)$ along with the point $\bs = (1,-1,1/6)$.

Corollary~\ref{cor:smoothAsm} implies that these two points give the asymptotic contributions
\begin{align*}
\Phi_{\bp} &=  6^k  \left(\frac{\sqrt{6}}{\pi k} - \frac{17 \sqrt{6}}{16\pi k^2} + \frac{605 \sqrt{6}}{512 \pi k^3} + O\left(\frac{1}{k^4}\right)\right) \\
\Phi_{\bs} &=  (-6)^k \left(\frac{\sqrt{6}}{4\pi k^2} - \frac{33 \sqrt{6}}{64 \pi k^3} + O\left(\frac{1}{k^4}\right)\right),
\end{align*}
so that the number of walks of length $k$ admits the asymptotic expansion
\[ [t^k]Q(1,1,t) = 6^k \left( \frac{\sqrt{6}}{\pi k} - \frac{\sqrt{6}(17 - 4(-1)^k)}{16\pi k^2} + \frac{\sqrt{6}(38720 - 16896(-1)^k)}{32768\pi k^3} +  O\left(\frac{1}{k^4}\right) \right). \]
\smallskip

The generating function for the number of walks returning to the origin is given by the rational diagonal
\[ Q(0,0,t) = \Delta \left( \frac{(1-x^2)(1-y^2)}{1-t(x^2y^2+y^2+x^2+xy^2+x+1)}  \right), \]
and the finitely minimal critical points $\bp$ and $\bs$ now have asymptotic contributions
\begin{align*}
\Phi_{\bp}' &=  6^k  \left(\frac{3\sqrt{6}}{2\pi k^3} +  O\left(\frac{1}{k^4}\right)\right) \\
\Phi_{\bs}' &=  (-6)^k \left(\frac{3\sqrt{6}}{2\pi k^3} + O\left(\frac{1}{k^4}\right)\right).
\end{align*}
Thus, the number of walks ending at the origin admits the asymptotic expansion
\[ [t^k]Q(0,0,t) = 6^k \left( \frac{3\sqrt{6}}{2\pi k^3}(1 + (-1)^k) + O\left(\frac{1}{k^4}\right) \right), \]
where we note that there are no excursions of odd length. 
\end{example}

%%%%%%%%%%%%%%%%%%%%%%%%%%%
% Chapter 8
%%%%%%%%%%%%%%%%%%%%%%%%%%%
\chapter{Effective Analytic Combinatorics in Several Variables}
\label{ch:EffectiveACSV}
\vspace{-0.2in}

\begin{center}This chapter is based on an article of Melczer and Salvy~\cite{MelczerSalvy2016}, and a forthcoming extension.\end{center}
\setlength{\epigraphwidth}{4.4in}

\epigraph{In a static universe you cannot imagine algebra, but geometry is essentially static. I can just sit here and see, and nothing may change, but I can still see. Algebra, however, is concerned with time\dots}{Michael Atiyah, \emph{Mathematics in the 20th Century}}
\vspace{-0.3in}

\epigraph{It seems to us obvious\dots to bring out a double set of results, viz.—1st, the numerical magnitudes which are the results of operations performed on numerical data\dots 2ndly, the symbolical results to be attached to those numerical results, which symbolical results are not less the necessary and logical consequences of operations performed upon symbolical data, than are numerical results when the data are numerical.}{Ada Augusta, Countess of Lovelace, \emph{Sketch of the Analytical Engine Invented by Charles Babbage (Notes by the Translator)}}

We now turn to the problem of automatically determining diagonal asymptotics for a rational function $F(\bz)$ which is analytic at the origin.  As the theory of such asymptotics has not been worked out in general, we must place some restrictions on the rational functions we consider.  To begin we will assume that $F(\bz)$ is combinatorial and admits a minimal critical point, together with assumptions, including that the singular variety is everywhere smooth, which hold \emph{generically} (that is, for all rational functions with fixed numerator and denominator degree, except for those whose coefficients satisfy certain fixed algebraic relations).  

Informally, the main result of this chapter is the following theorem, which is stated precisely as Theorem~\ref{thm:EffectiveCombAsm} below.

\begin{theorem*}
Let $F(\bz) \in \mathbb{Z}(z_1,\dots,z_n)$ be a rational function with numerator and denominator of degrees at most $d$ and coefficients of absolute value at most $2^h$.  Assume that $F$ is combinatorial, has a minimal critical point, and satisfies additional restrictions\footnote{See Section~\ref{sec:assumptions_comb}.} which hold generically.  Then there exists a probabilistic algorithm computing dominant asymptotics of the diagonal sequence in $\tilde{O}(hd^{4n+5})$ bit operations\footnote{We write $f = \tilde{O}(g)$ when $f=O(g\log^kg)$ for some $k\ge0$; see Section~\ref{sec:complexity_measure} for more information on our complexity model and notation.}.  The algorithm returns three rational functions $A,B,C \in \mathbb{Z}(u)$, a square-free polynomial $P \in \mathbb{Z}[u]$ and a list $U$ of roots of $P(u)$ (specified by isolating regions) such that
\[ f_{k,\dots,k} = (2\pi)^{(1-n)/2}\left(\sum_{u \in U} A(u)\sqrt{B(u)} \cdot C(u)^k \right)k^{(1-n)/2}\left(1 + O\left(\frac{1}{k}\right) \right). \]
The values of $A(u), B(u),$ and $C(u)$ can be determined to precision $2^{-\kappa}$ at all elements of $U$ in $\tilde{O}(d^{n+1}\kappa + hd^{3n+3})$ bit operations.
\end{theorem*}

Being combinatorial is not a generic property of rational functions and, unlike the other assumptions on $F$ which we require, it is unknown how to decide whether a given rational function is combinatorial.  Verifying minimality of a finite set of critical points is the most expensive operation we must perform, and Proposition~\ref{prop:lineMin} shows this is easier in the combinatorial case. The following result is stated rigorously in Theorem~\ref{thm:GenEffMinCrit}, and discusses the complexity of finding minimal critical points.  

\begin{theorem*} 
Let $F(\bz) \in \mathbb{Z}(z_1,\dots,z_n)$ be a rational function with numerator and denominator of degrees at most $d$ and coefficients of absolute value at most $2^h$.  Assuming that $F$ satisfies certain verifiable assumptions\footnote{See Section~\ref{sec:EffectiveGenAsm}.}, $F$ admits a finite number of minimal critical points which can be determined in $\tilde{O}\left(hd^{9n+4}2^{3n}\right)$ bit operations.
\end{theorem*}

Aside from admitting minimal critical points, we conjecture that the assumptions on $F$ required to apply Theorem~\ref{thm:GenEffMinCrit} hold generically. When $F(\bz)$ admits a minimal critical point $\bw$ which is known to be finitely minimal, or if it is known that all minimizers of $|z_1 \cdots z_n|^{-1}$ on $\overline{\mD}$ lie in $T(\bw)$, then one can additionally obtain diagonal coefficient asymptotics in the same complexity.

In order to find minimal critical points, we work from the algebraic system defined by the smooth critical point equations which, under generic conditions, is zero-dimensional (i.e., has a finite number of solutions).  The probabilistic nature of our results come from algorithms determining a representation of critical points which will allow us to determine minimality.  We use a parametrization of the critical points known as a \emph{Kronecker representation}, which is closely related to the notion of a rational univariate representation (RUR) found in the literature on polynomial system solving.  One can compute a Kronecker representation deterministically using a Gröbner Basis calculation, but the complexity of this step may be larger than what is discussed here.

\subsubsection*{The Kronecker Representation}
The Kronecker representation of a zero-dimensional system dates back to work of Kronecker~\cite{Kronecker1882} and Macaulay~\cite{Macaulay1916} on polynomial system solving\footnote{See Castro et al.~\cite{CastroPardoHageleMorais2001} for a detailed history and account of this approach to solving polynomial systems.}.  The representation uses an integer linear form 
\[ u = \lambda_1 z_1 + \cdots + \lambda_n z_n \in \mathbb{Z}[\bz] \] 
which takes distinct values at the solutions of the zero-dimensional system and encodes these solutions in a new system of equations
\[
  P(u)=0 ,\qquad 
    \left\{\begin{array}{l}
    P'(u)z_1 - Q_1(u) =0 ,\\ 
    \hspace{1.15in}\vdots \\
    P'(u)z_n - Q_n(u) =0,\end{array}\right.
\]
with $P\in\mathbb{Z}[u]$ square-free and $Q_1,\dots,Q_n\in\mathbb{Z}[u]$ of degrees smaller than the degree of $P$.  For our purposes we consider the representation to be computed by a probabilistic algorithm of Safey El Din and Schost~\cite{Safey-El-DinSchost2016} with bounded error probability.

In order to test minimality we must be able to isolate and argue about individual elements of this finite algebraic set. A collection of mostly classical bounds associated to univariate polynomials are exhibited below, allowing one to determine a precision such that questions about elements of the algebraic set can be answered exactly by determining the zeroes of $P(u)$ numerically to such precision.  Combined with bounds of Safey El Din and Schost~\cite{Safey-El-DinSchost2016} (following Schost~\cite{Schost2001}) on the coefficient sizes of the polynomials $P$ and the $Q_j$ appearing in the Kronecker representation, this allows us to determine the complexity of rigorously deciding several properties of the solutions to the original polynomial system, such as which have coordinates that are exactly equal to each other or zero, or deciding which solutions have real coordinates in defined ranges.  We show below how these tests are sufficient to determine diagonal asymptotics under genericity assumptions.

\subsubsection*{Previous work}
Not much previous work has been completed on automating the theory of analytic combinatorics in several variables. De Vries et al.~\cite{DeVriesHoevenPemantle2011} give an algorithm which takes a bivariate rational function $F(x,y)$ with smooth singular variety admitting a (non-zero) finite number of isolated critical points and returns dominant asymptotics of the diagonal sequence.  This algorithm does not require the critical points to be minimal, which is very powerful, however the techniques rely strongly on being in the bivariate case.  A Sage package of Raichev~\cite{Raichev2011} determines asymptotic contributions of nondegenerate minimal critical points which are smooth (or convenient points, to be described in Chapter~\ref{ch:NonSmoothACSV}), however the package cannot determine minimality of these points and thus cannot rigorously determine asymptotics.  Neither of these works has a complexity analysis. 

Another approach to diagonal asymptotics is to use the theory of creative telescoping to compute an annihilating differential equation of the (univariate) diagonal generating function.  The work of Bostan et al.~\cite{BostanLairezSalvy2013} and Lairez~\cite{Lairez} shows that the creative telescoping procedure has a complexity which is essentially polynomial in $d^n$, where $d$ is a bound on the degrees of the numerator and denominator of the rational function under consideration, which is comparable to the complexity of our results. Note that asymptotics often cannot be rigorously computed through this method due to the connection problem, which was described in Section~\ref{sec:DFin}.

Our work on the Kronecker representation follows several articles~\cite{GiustiHeintzMoraisMorgensternPardo1998,GiustiLecerfSalvy2001,Schost2001,KrickPardoSombra2001} on the use of the Kronecker representation in complex or real geometry, which go far beyond the simple systems we consider here. There has also been work on solving polynomial systems using the Kronecker representation under the name rational univariate representation~\cite{Rouillier1999,BasuPollackRoy2003}.   The idea of using a Kronecker representation to reduce numerical computations with elements of zero-dimensional algebraic sets to the manipulation of univariate polynomials is not new. However, to the best of our knowledge, the connection between good properties of the Kronecker representation in terms of bit size of its output and fast and precise algorithms operating on univariate polynomials had not been explored before the proceedings article this chapter is based on, except in the case of bivariate systems~\cite{BouzidiLazardPougetRouillier2015,KobelSagraloff2015}.

%%%%%%%%%%%%%%%%%%%%%%%%%%%%%%%%%%%%%%%%%%%%%%%
% Section: Algorithms for Effective Asymptotics
%%%%%%%%%%%%%%%%%%%%%%%%%%%%%%%%%%%%%%%%%%%%%%%
\section{Main Algorithms and Results}
\label{sec:Algorithms}
This section describes our main algorithmic tools and the assumptions we require.  Correctness of the algorithms, and proofs of their complexity, are described in Section~\ref{sec:AlgorithmProofs}.

\subsection{Complexity Measurements}
\label{sec:complexity_measure}

The \emph{bit complexity} of an algorithm whose input can be encoded by integers (for instance, a multivariate polynomial over the integers) is obtained by considering the base $B$ representations of these integers for some fixed $B$ (usually a power of 2) and counting the number of additions, subtractions, and multiplications modulo $B$ performed by the algorithm.  Our algorithms typically take as input polynomials in $\mathbb{Z}[z_1,\dots,z_n]$, and the algorithms' bit complexity will depend on the number of variables $n$, along with the degrees and heights of these polynomials.  The \emph{height} $h(P)$ of a polynomial $P\in\mathbb{Z}[\bz]$ is the maximum of 0 and the base~2 logarithms\footnote{Note that some works define the height in terms of the maximum of the coefficients' moduli instead of their logarithms.} of the absolute values of the coefficients of $P$. As $h(P)$ gives a bound on the bit size of the coefficients of $P$, it helps give separation bounds on the roots of $P$.  Unless otherwise specified we assume $d$ denotes a quantity of value at least 2 (typically corresponding to polynomial degree) and define $D := d^n$. 

For two functions $f$ and $g$ defined and positive over~$(\mathbb{N}^*)^m$, the notation $f(a_1,\dots,a_m) = O(g(a_1,\dots,a_m))$ states the existence of a constant~$K$ such that $f(a_1,\dots,a_m)\le Kg(a_1,\dots,a_m)$ over~$(\mathbb{N}^*)^m$.  Furthermore, we write $f = \tilde{O}(g)$ when $f=O(g\log^kg)$ for some $k\ge0$; for instance, $O(nD)=\tilde{O}(D)$ since we assume $d \geqslant 2$. The dominant factor in the complexity of most operations we consider grows like $\tilde{O}(D^c)$ for some constant $c$, and our goal is typically to bound the exponent $c$ as tightly as possible.  %We make use of the $\tilde{O}$ notation to avoid technical discussions on subexponential terms.  

It is often convenient to consider a system of polynomials $\bbf = (f_1,\dots,f_n)$ of degree at most $d$ as given by a straight-line program (a program using only assignments, constants, $+,-,$ and $\times$) which evaluates the elements of $\bbf$ simultaneously at any point~$\bz$ using at most $L$ arithmetic operations (see Section 4.1 of Burgisser et al.~\cite{BurgisserClausenShokrollahi1997} for additional details on this complexity model).  For instance, this can allow one to take advantage of sparsity in the polynomial system.  An upper bound on $L$ is obtained by considering $n$ dense polynomials in $n$ variables, leading to $L=\tilde{O}(D)$.

Given a zero-dimensional polynomial system $\bbf$, we will use the results of Safey El Din and Schost~\cite{Safey-El-DinSchost2016} to compute a Kronecker representation, and these results can take into account some underlying structure present in $\bbf$.  More precisely, Safey El Din and Schost derive upper bounds on the degrees and heights of the polynomials appearing in a Kronecker representation which will help us take into account the fact that different blocks of variables occur in disjoint elements of the systems we consider, except for polynomials of degree 3.  We present the complexity results and output bounds of Safey El Din and Schost in the next section, after defining some related quantities, and apply these results in the context of ACSV in Sections~\ref{sec:EffectiveCombAsm} and~\ref{sec:EffectiveGenAsm}. 
\smallskip

\subsubsection*{Quantities for Degree and Height Bounds}
Fix a positive integer $m$ and vector $\bn \in \mathbb{N}^m$.  Given any vector $\bv = (v_1,\dots,v_m) \in \mathbb{N}^m$ and variables $\theta_1,\dots,\theta_m$, we define
\[ \{\bv\} := v_1\theta_1 + \cdots + v_m\theta_m, \]
and given a sequence of vectors $\bd_1,\dots,\bd_r \in \mathbb{N}^m$ we let $\sC_{\bn}(\bd)$ be the sum of the non-zero coefficients of $\theta_1,\dots,\theta_m$ in the expression
\[ \{\bd_1\} \cdots \{\bd_r\} \mod \left(\theta_1^{n_1+1},\dots,\theta_m^{n_m+1}\right). \]
Furthermore, if $\eta$ is a real number and $\zeta$ a variable, we let 
\[ \{\eta,\bv\} := \eta \zeta + v_1\theta_1 + \cdots + v_m\theta_m, \]
and for $\bETA \in \mathbb{R}^r$ we let $\sH_{\bn}(\bETA,\bd)$ be the sum of the non-zero coefficients of $\zeta,\theta_1,\dots,\theta_m$ in the expression
\[ \{\eta_1,\bd_1\} \cdots \{\eta_r,\bd_r\} \mod \left(\zeta^2,\theta_1^{n_1+1},\dots,\theta_m^{n_m+1}\right). \]

Consider a polynomial $f(\bz) \in \mathbb{Z}[\bz]$ and let $\bZ_1,\dots, \bZ_m$ be a partition of the variables $\bz$.  We say that $f$ has \emph{multi-degree at most} $(v_1,\dots,v_m)$ if the total degree $\deg_{\bZ_j}(f)$ of $f$ considered as a polynomial only in the variables of $\bZ_j$ is at most $v_j$, for each $j=1,\dots,m$.  If $\bbf=(f_1,\dots,f_r)$ is a polynomial system where $f_j$ has multi-degree at most $\bd_j\in\mathbb{N}^m$, and the block of variables $\bZ_j$ contains $n_j$ elements, Safey El Din and Schost~\cite{Safey-El-DinSchost2016} prove that the quantity $\sC_{\bn}(\bd)$ will be an upper bound on the degrees of the polynomials appearing in a Kronecker representation of $\bbf$, where $\bd=(\bd_1,\dots,\bd_m)$ is a vector of vectors.

If $h(f)$ denotes the height of a polynomial $f(\bz) \in \mathbb{Z}[\bz]$, we define 
\begin{equation} \eta(f) := h(f) + \sum_{j=1}^m \log(1+n_j)\deg_{\bZ_j}(f). \label{eq:etaHeight} \end{equation}
Given $\bETA \in \mathbb{R}^m$ such that $\eta(f_j) \leq \eta_j$ for each $j=1,\dots,r$, Safey El Din and Schost~\cite{Safey-El-DinSchost2016} prove that a combination of the quantities $\sC_{\bn}(\bd)$ and $\sH_{\bn}(\bETA,\bd)$ yields an upper bound on the heights of the polynomials appearing in a Kronecker representation of $\bbf$.

\begin{example}
If $\bd = (d,\dots,d)$, then $\sC_{n}(\bd) = d^n=D$ as it is the sum of the coefficients in 
\[ (d\theta_1)^n \mod \left(\theta_1^{n+1}\right). \]
Furthermore, if $\eta_j := h + d\log(1+n)$ then $\sH_{n}(\bETA,\bd) = \tilde{O}(hd^{n-1} + D)$ as it is the sum of the coefficients in 
\[ (\eta \zeta + d\theta_1)^n \mod \left(\zeta^2,\theta_1^{n+1}\right). \]
If $\bbf$ is a zero-dimensional polynomial system consisting of polynomials of degrees at most $d$ and heights at most $h$, this calculation will imply that the polynomials appearing in any Kronecker representation of $\bbf$ have heights $\tilde{O}\left(hd^{n-1}+D\right) \in \tilde{O}(hD)$. 
\end{example}

\subsection{Kronecker Representation}

A \emph{Kronecker representation} $\left[P(u),\bQ\right]$ \glsadd{KronRep} of a zero-dimensional algebraic set 
\[ \mV(\bbf) = \{\bz : f_1(\bz) = \cdots = f_n(\bz)=0\} \] 
defined by the polynomial system $\bbf = (f_1,\dots,f_n) \subset \mathbb{Z}[\bz]^n$ consists of an integer linear form 
\[ u = \lambda_1 z_1 + \cdots + \lambda_r z_n \in \mathbb{Z}[\bz] \] 
which takes distinct values at elements of $\mV$, a square-free polynomial $P \in \mathbb{Z}[u]$, and $Q_1,\dots,Q_n \in \mathbb{Z}[u]$ of degrees smaller than the degree of $P$ such that the elements of $\mV(\bbf)$ are given by projecting the solutions of the system
\begin{equation}
\label{eq:KronRep}
	P(u)=0 ,\qquad 
    \left\{\begin{array}{l}
    P'(u)z_1 - Q_1(u) =0 ,\\ 
    \hspace{1.15in}\vdots \\
    P'(u)z_n - Q_n(u) =0,\end{array}\right.
\end{equation}
onto the coordinates $z_1,\dots,z_n$.  The \emph{degree} of a Kronecker representation is the degree of $P$, and the \emph{height} of a Kronecker representation is the maximum height of its polynomials $P,Q_1,\dots,Q_n$.  The following probabilistic result allows one to calculate a Kronecker representation for a zero-dimensional polynomial system, assuming that the Jacobian of the system is invertible at the elements of $\mV(\bbf)$.  

\begin{proposition}[{Safey El Din and Schost~\cite[Proposition 18]{Safey-El-DinSchost2016}}]
\label{prop:kronrep}
Let $\bbf \in \mathbb{Z}[\bz]^n$ be a zero-dimensional polynomial system given by a straight-line program $\Gamma$ of size $L$ that uses integer constants of height at most $b$, and let $Z(\bbf)$ be the solutions of $\bbf$ where the Jacobian matrix of $\bbf$ is invertible.  Then, fixing a partition $\bZ$ of the variables such that $f_j$ has multi-degree at most $\bd_j$ and $\eta(f_j) \leq \eta_j$, there exists an algorithm {\sf KroneckerRep} that takes $\Gamma$ and produces one of the following outputs:
\begin{itemize}
	\item a Kronecker representation of $Z(\bbf)$, 
	\item a Kronecker representation of degree less than that of $Z(\bbf)$,
	\item \textsc{fail}.
\end{itemize}
The first outcome occurs with probability at least 21/32.  In any case, the algorithm has bit complexity
\[ \tilde{O}\left( Lb + \sC_{\bn}(\bd)\sH_{\bn}(\bETA,\bd)\left(L + n\delta + n^2\right) n \left(n + \log h \right) \right), \]
where
\[ \delta = \max_{1 \leq j \leq n} \left[\deg_{\bZ_1}(f_j) + \cdots + \deg_{\bZ_r}(f_j)\right].\]  
The polynomials in the output have degree at most $\sC_{\bn}(\bd)$ and height $\tilde{O}\left( \sH_{\bn}(\bETA,\bd) + n\sC_{\bn}(\bd) \right)$.  When $\bbf$ consists of polynomials of degrees at most $d$ and heights at most $h$ one can put all variables in a single block to obtain an algorithm with bit complexity $\tilde{O}(D^3+hD^2d^{n-1}) \in \tilde{O}(hD^3)$, whose output consists of polynomials of degrees at most $D$ and heights in $\tilde{O}(D+hd^{n-1}) \in \tilde{O}(hD)$.
\end{proposition}

Repeating the algorithm $k$ times, and taking the output with highest degree, allows one to obtain a Kronecker representation of $Z(\bbf)$ with probability $1-\left(\frac{11}{32}\right)^k$.  When the Jacobian of $\bbf$ is invertible at each of its solutions then Proposition~\ref{prop:kronrep} gives a Kronecker representation of all solutions of $\bbf$.

Suppose $\bbf$ is a zero-dimensional polynomial system with polynomials of degree at most $d$ and heights at most $h$.   Determining the inverse of $P'(u)$ modulo the polynomial $P(u)$ allows one to calculate a polynomial 
\[ A_j(u) := Q_j(u) \cdot P'(u)^{-1} \text{ mod } P(u)\] 
of degree at most $D$ such that $z_j = A_j(u)$ at the solutions of $\bbf$ encoded by $P(u)=0$.  Effective versions of the \emph{arithmetic Nullstellensätze}~\cite[Theorem 1]{KrickPardoSombra2001} imply that the maximum height of the numerators and denominators of the (lowest terms) rational coefficients in $P'(u)^{-1}$, and thus $A_j(u)$, is bounded by $\tilde{O}(hD^2)$.  This upper bound on heights is often observed in practice, making computations with the polynomials $Q_j$ of height $\tilde{O}(hD)$ appearing in the Kronecker representation much more efficient than those with the polynomials $A_j$.  This is one reason why the Kronecker representation has become a widely used tool in computer algebra.

Another probabilistic algorithm which computes a Kronecker representation of the solutions of $\bbf$ under similar assumptions, and with a similar complexity to Proposition~\ref{prop:kronrep} when all variables are put into the same block, was given by Giusti et al.~\cite{GiustiLecerfSalvy2001}.  The algorithm of Giusti et al. was used in the article~\cite{MelczerSalvy2016} on which this chapter is based.  %The probablistic aspects are usually harmless in practice: at points in the algorithm one must pick random rational numbers and the algorithm succeeds unless these numbers satisfy fixed algebraic relations depending on the degree $d$ (hence, it could be said that the algorithm succeeds ``generically'').

\subsection{Numerical Kronecker Representation}
A \emph{numerical Kronecker representation} $[P(u),\bQ,\bU]$ \glsadd{NumerKronRep} of a zero-dimensional polynomial system is a Kronecker representation $[P(u),\bQ]$ of the system together with a sequence $\bU$ of isolating intervals for the real roots of the polynomial $P$ and isolating disks for the non-real roots of $P$. We say that the \emph{size of an interval} is its length, while the \emph{size of a disk} is its radius.  In practice, the elements of $\bU$ are stored as floating point approximations whose accuracy is certified to a specified precision.  We use standard results on univariate polynomial root solving and root bounds, described in Sections~\ref{sec:UniPolyBounds} and~\ref{sec:UniPolyAlg}, to obtain the following result.

\begin{proposition}
\label{prop:numKron}
Suppose the zero-dimensional system $\bbf = (f_1,\dots,f_r) \subset \mathbb{Z}[z_1,\dots,z_r]$ is given by a Kronecker representation $[P(u),\bQ]$ of degree $d$ and height $h$.  Then there exists an algorithm {\sf NumericalKroneckerRep} which takes $[P(u),\bQ]$ and $\kappa>0$ and returns a numerical Kronecker representation $[P(u),\bQ,\bU]$, with isolating regions in $\bU$ of size at most $2^{-\kappa}$, in $\tilde{O}(d^3+d^2h+d\kappa)$ bit operations.
\end{proposition}

Once a Kronecker representation is known, several important properties of the underlying zero-dimensional algebraic set can be detected.

\begin{proposition}
\label{prop:NumKronAlgs}
Suppose the zero-dimensional polynomial system $\bbf$ is given by a known numerical Kronecker representation $[P(u),\bQ,\bU]$ of degree $\delta$ and height $\eta$.
\begin{enumerate}
	\item Given a natural number $\kappa$, one can determine approximations to the elements of $\mV(\bbf)$ whose coordinates are at most $2^{-\kappa}$ from their true values in $\tilde{O}(n(\delta^3+\delta^2\eta+\delta\kappa))$ bit operations.
	\item The positivity, negativity and equality to 0 of all real coordinates of all elements of $\mV(\bbf)$ can be determined in $\tilde{O}(n(\delta^3+\delta^2\eta))$ bit operations.
\end{enumerate}
Furthermore, if $\bbf$ contains $n$ polynomials of degrees at most $d$ and heights at most $h$,
\begin{enumerate}
	\setcounter{enumi}{2}
	\item All coordinates of elements of $\mV(\bbf)$ which are equal to each other can be detected in $\tilde{O}(hD^3)$ bit operations.
	\item Given an element of $\mV(\bbf)$ with non-negative coordinates, one can determine all elements of $\mV(\bbf)$ with the same coordinate-wise moduli in $\tilde{O}(hD^4)$ bit operations.
\end{enumerate}
\end{proposition}

The following result discusses incorporating new polynomials into a numerical Kronecker representation.

\begin{proposition}
\label{prop:KronReduce}
Let $\bbf$ be a zero-dimensional polynomial system and $[P(u),\bQ]$ be a Kronecker representation of $\bbf$ calculated using Proposition~\ref{prop:kronrep} with respect to a partition $\bZ=(\bZ_1,\dots,\bZ_m)$ of the variables $z_1,\dots,z_n$.  If $\sC_{\bn}(\bd)$ and $\sH_{\bn}(\bETA,\bd)$ are the quantities appearing in Proposition~\ref{prop:kronrep}, and $q \in \mathbb{Z}[\bz]$ has height $\eta$ and degree $\delta_i$ in the block of variables $\bZ_i$ for each $i \in \{1,\dots,m\}$, then 
\begin{enumerate}
	\item there exists a parameterization $P'(u)-TQ_q(u)$ of the values taken by $q$ on $\mV(\bbf)$ with $Q_q\in\mathbb{Z}[u]$ a polynomial of degree at most $\sC_{\bn}(\bd)$ and height $\tilde{O}\left(\sC_{\bn}(\bd) + n \overline{\sH_{\bn}}(\bETA,\bd)\right)$, where 
\[ \overline{\sH_{\bn}}(\bETA,\bd) := (\delta+1)\sH_{\bn}(\bETA,\bd) + \sC_{\bn}(\bd)\left(\eta + 1 + \sum_{i=1}^m \log(1+|\bZ_i|)\delta_i \right), \]
and $\delta = \delta_1 + \cdots + \delta_m$;
	\item there exists a polynomial $\Phi_q \in \mathbb{Z}[T]$ which vanishes on the values of $q$ at the elements of $\mV(\bbf)$, of degree at most $\sC_{\bn}(\bd)$ and height
    \[ \tilde{O}{\Big(}\left(\sH_{\bn}(\bETA,\bd) + n\sC_{\bn}(\bd) \right)\sC_{\bn}(\bd) + \left(\overline{\sH_{\bn}}(\bETA,\bd) + n\sC_{\bn}(\bd) \right)\sC_{\bn}(\bd) + \log_2\left[\left(2\sC_{\bn}(\bd)\right)!\right]{\Big)}.  \]
\end{enumerate}
The polynomial $Q_q$ can be determined in $\tilde{O}\left(D\sC_{\bn}(\bd)\overline{\sH_{\bn}}(\bETA,\bn)\right)$ bit operations.
%, while $\Phi_q$ can be determined using a probabilistic algorithm in 
%\[ \tilde{O}\left(\left(\overline{\sC_{\bZ}}(\bbf)^3 + n\overline{\sC_{\bZ}}(\bbf)^2\overline{\sH_{\bZ}}(\bbf)\right)^{1+\epsilon}\right)\] 
%bit operations, for any fixed $\epsilon>0$.  
\end{proposition}
If $\bZ$ is composed of a single block of variables, and $q$ and the elements of $\bbf$ have degrees at most $d$ and heights at most $h$, then $Q_q$ has degree at most $D$ and height $\tilde{O}(hD)$, and can be computed in $\tilde{O}(hD^3)$ bit operations.  We will not need to compute $\Phi_q$ for our applications, except in the special case when $q$ is one of the variables $z_j$, which is discussed in Lemma~\ref{lemma:coordpoly}.  In general, assuming one computes a Kronecker representation using a single block of variables and $q$ has degree at most $d$ and height at most $h$, $\Phi_q$ can be determined in $\tilde{O}(hD^4)$ bit operations using fast resultant calculations~\cite[Lemma 3]{BouzidiLazardMorozPougetRouillierSagraloff2015a}.

Proofs of these results are given in Section~\ref{sec:NumKronAlg}.
\vspace{-0.1in}
%%%%%%%%%%%%%%%%%%%
%%%
%%%%%%%%%%%%%%%%%%%

\subsection{Assumptions for Asymptotics}
\label{sec:assumptions_comb}
Our algorithms for asymptotics require the following assumptions:
\begin{itemize} 
	\item[(A0)] $F(\bz)=G(\bz)/H(\bz)$ admits a minimal critical point;
	\item[(A1)] $H$ and its partial derivatives do not have a common solution in $\mathbb{C}^n$;
	\item[(A2)] $G(\bz)$ is non-zero at at least one minimal critical point;
	\item[(A3)] all minimal critical points of $F(\bz)$ are nondegenerate;
	\item[(J1)] the Jacobian matrix of the system 
	\[ \bbf = \left(H, \quad z_1(\partial H/\partial z_1)-\lambda, \quad \dots \quad ,z_n(\partial H/\partial z_n)-\lambda, \quad H(tz_1,\dots,tz_n)\right)\]
	with respect to the variables $\bz,\lambda,$ and $t$, is non-singular at its solutions.
\end{itemize}

Note that assumption (A1) implies the singular variety $\mV(H)$ is everywhere smooth.  Furthermore, the \emph{Jacobian criterion}~\cite[Theorem 16.19]{Eisenbud1995} implies that the polynomial system $\bbf$ is zero-dimensional whenever the square Jacobian matrix of $\bbf$ has full rank (i.e., is non-singular) at all of its solutions, so that (J1) is stronger than requiring that $F(\bz)$ admits a finite number of critical points.  We only require assumption (J1) to compute a Kronecker representation of $\bbf$ using Proposition~\ref{prop:kronrep}.  Another sufficient condition for $F$ to admit a finite number of critical points is that all critical points are nondegenerate, as any nondegenerate critical point is isolated\footnote{A smooth critical point is a critical point of the analytic map $\phi: \mV \rightarrow \mathbb{C}$ defined by $\phi(\bz) = z_1 \cdots z_n$.  The fact that a nondegenerate critical point of an analytic map from a smooth manifold to the complex numbers is isolated follows from a result known as the \emph{Complex Morse Lemma}~\cite[Proposition 3.15 and Corollary 3.3]{Ebeling2007}.  The number of isolated solutions of a system of $n$ degree $d$ polynomials in $n$ variables is at most $d^n$ by Bézout's inequality. }.  

Pemantle and Wilson~\cite{PemantleWilson2013} always assume the existence of at least one critical point, and although they have some results when there are no minimal critical points they do not have explicit asymptotic formulas for such cases.  Their results require isolated critical points, and all  asymptotic results in dimension $n>2$ need nondegenerate critical points.  Chapters 10 and 11 of their text, to be discussed in Chapter~\ref{ch:NonSmoothACSV} of this thesis, generalize the theory of ACSV to several cases when $\mV(H)$ is not smooth. Automating these extensions is a direction for future work. 

Our assumptions often hold in practice because they are satisfied generically\footnote{Although many examples that come from combinatorial problems have non-generic behaviour, as we will see in Part~\ref{part:NonSmoothACSV} of this thesis.}.  Recall that there are $m_d := \binom{d+n}{n}$ monic monomials in $\mathbb{C}[\bz]$ of (total) degree at most $d$.

\begin{definition}
A property $\mathcal{P}$ of polynomials in $\mathbb{C}[\bz]$ is said to hold \emph{generically} if for every positive integer $d$ there exists a proper algebraic subset $\mathcal{C}_d \subsetneq \mathbb{C}^{m_d}$ such that any polynomial of degree $d$ satisfies $\mathcal{P}$ unless its coefficients lie in $\mathcal{C}_d$.  

A property of rational functions holds \emph{generically} if for every pair of positive integers $(d_1,d_2)$ there exists a proper algebraic subset $\mathcal{C}_{d_1,d_2} \subsetneq \mathbb{C}^{m_{d_1}+m_{d_2}}$ such that any rational function of numerator and denominator degrees $d_1$ and $d_2$ satisfies $\mathcal{P}$ unless the coefficients of its numerator and denominator lie in $\mathcal{C}_{d_1,d_2}$.  
\end{definition}

This definition implies that the conjunction of finitely many generic properties is generic.  In Section~\ref{sec:generic} we prove the following result.

\begin{proposition}
\label{prop:generic_assumptions}
The assumptions (A1)--(A3) and (J1) hold generically.
\end{proposition}

Any point where $H$ and its partial derivatives simultaneously vanish will satisfy the smooth critical points equations.  Thus, when $Z(\bbf)=\mV(\bbf)$, there are at most a finite number of such points and assumption (A1) can be verified from a Kronecker representation of $\bbf$.  The main purpose of our algorithms will be to prove the existence of minimal critical points, verifying assumption (A0), and to verify assumption (A2) one can use Proposition~\ref{prop:KronReduce} to determine the values of $G$ at specific critical points.  To verify (A3) we take the matrix $\mH$ defined by Equation~\eqref{eq:Hess}, where the $\zeta_j$ are taken to be variables $z_j$ and $\lambda = z_1(\partial H/\partial z_1)$, and multiply each entry by $\lambda$ to obtain a polynomial matrix $\tilde{\mH}$.  The determinant of $\tilde{\mH}$ is a polynomial, and Proposition~\ref{prop:KronReduce} can be used to check whether or not it vanishes at any minimal critical point.  Unfortunately, the best complexity of which we are aware for verifying that the Jacobian of a system $\bbf$ of $n$ degree $d$ polynomials in $n$ variables is non-singular at its solutions (i.e., that $Z(\bbf)=\mV(\bbf)$) is $d^{\tilde{O}(n)}$, given by Giusti et al.~\cite{GiustiHeintzSabia1993}. 

\subsection{Asymptotics in the Combinatorial Case}
\label{sec:EffectiveCombAsm}
We can now rigorously state our main result on asymptotics when $F(\bz)$ is combinatorial.  

\begin{theorem}
\label{thm:EffectiveCombAsm}
Assume (A0)---(A3), (J1), and that $F(\bz)$ is combinatorial.  Then there exists a probabilistic algorithm computing dominant asymptotics of the diagonal sequence in $\tilde{O}(hd^5D^4)$ bit operations.  The algorithm returns three rational functions $A,B,C \in \mathbb{Z}(u)$, a square-free polynomial $P \in \mathbb{Z}[u]$ and a list $U$ of roots of $P(u)$ (specified by isolating region) such that
\begin{equation} f_{k,\dots,k} = (2\pi)^{(1-n)/2}\left(\sum_{u \in U} A(u)\sqrt{B(u)} \cdot C(u)^k \right)k^{(1-n)/2}\left(1 + O\left(\frac{1}{k}\right) \right). \label{eq:EffectiveCombAsm} \end{equation}
The values of $A(u),B(u)$ and $C(u)$ can be refined to precision $2^{-\kappa}$ at all elements of $U$ in $\tilde{O}(dD\kappa + hd^3D^3)$ bit operations.
\end{theorem}

Theorem~\ref{thm:EffectiveCombAsm} follows directly from the smooth point asymptotics given in Proposition~\ref{prop:smoothAsmCP}, together with Proposition~\ref{prop:lineMin}, once the algorithms for the numerical Kronecker representation stated in Propositions~\ref{prop:NumKronAlgs} and~\ref{prop:KronReduce} are established in Section~\ref{sec:AlgorithmProofs}.  A high level description of the algorithm is given in Algorithm~\ref{alg:EffectiveCombAsm}.

% MOVE IF NECESSARY
\begin{algorithm}
\DontPrintSemicolon
\KwIn{Rational function $F(\bz) = G(\bz)/H(\bz)$ which satisfies the hypotheses of Theorem~\ref{thm:EffectiveCombAsm}}
\KwOut{$A,B,C \in \mathbb{Z}(u)$, $P \in \mathbb{Z}[u]$ and finite list $U$ such that Equation~\eqref{eq:EffectiveCombAsm} is satisfied}
\;
$\bbf \gets \left[H, \quad z_1(\partial H/\partial z_1)-\lambda, \quad \dots \quad ,z_n(\partial H/\partial z_n)-\lambda, \quad H(tz_1,\dots,tz_n)\right]$ \\
$[P,\bQ] \gets {\sf KroneckerRep}(\bbf)$ \\
$[P,\bQ,\mathbf{V}] \gets {\sf NumericalKroneckerRep}(P,\bQ)$ \\
Group elements of $\mV(\bbf)$ with the same $z_1,\dots,z_n$ coordinates \\
\If{$\lambda=0$ and $z_1,\dots,z_n \neq 0$ at any solution}{
	return \textsc{fail}, ``$H$ and its partial derivatives share root''
}
Remove points where $\lambda=0$ \\
\For{each element of $\mV(\bbf)$ with positive real values of $z_1,\dots,z_n$}{
	Check whether there is a corresponding element of $\mV(\bbf)$ with $0 < t < 1$ \\
	If not, then $\bp = (z_1,\dots,z_n)$ is a minimal critical point
} 
\If{No such element $\bp$}{
	return \textsc{fail}, ``no minimal-critical points''
}
$\bz^{(1)},\dots,\bz^{(k)} \gets$ elements of $\mV(\bbf)$ with same coordinate-wise modulus as $\bp$ \\
$\bU \gets $ elements of $\mathbf{V}$ corresponding to $\bz^{(1)},\dots,\bz^{(k)}$ \\
$\tilde{\mH} \gets$ determinant of the matrix defined by Equation~\eqref{eq:Hess} after each entry multiplied by $\lambda$ \\
$Q_{\tilde{\mH}}(u) \gets$ polynomial from Prop.~\ref{prop:KronReduce} parameterizing $\tilde{\mH}$ on $\mV(\bbf)$ using $[P,\bQ]$ \\
$Q_{T}(u) \gets$ polynomial from Prop.~\ref{prop:KronReduce} parameterizing $T = z_1 \cdots z_n$ on $\mV(\bbf)$ using $[P,\bQ]$ \\
$Q_{-G}(u) \gets$ polynomial from Prop.~\ref{prop:KronReduce} parameterizing $-G(\bz)$ on $\mV(\bbf)$ using $[P,\bQ]$ \\
\If{$Q_{\tilde{\mH}}(u)=0$ at any element of $\bU$}{
	return \textsc{fail}, ``degenerate minimal-critical point''
}
Return $Q_{-G}/Q_{\lambda}, \quad Q_{\lambda}^{n-1}/Q_{\tilde{\mH}}\cdot (P')^{2-n}, \quad P'/Q_T,\quad P, \quad \bU$
\caption{{\sf CombinatorialAsymptotics}}
\label{alg:EffectiveCombAsm}
\end{algorithm}

\begin{example}
\label{ex:Apery3a}
Consider Apéry's sequence $(b_k)$ from Example~\ref{ex:Apery}, whose generating function was given by a rational diagonal of the form $F(w,x,y,z) = 1/H(w,x,y,z)$. The polynomial
\[H(w,x,y,z) = 1-w(1+x)(1+y)(1+z)(1+y+z+yz+xyz)\]
of degree~7 defines a smooth algebraic set $\mV(H)$, and $F=1/H$ is combinatorial. Taking the system 
\[ H(w,x,y,z), \quad w(\partial H/\partial w) - \lambda, \quad \dots \quad , z(\partial H/\partial z) - \lambda, \quad  H(tw,tx,ty,tz), \]
we try the linear form $u=w+x+y+z+t$ and (using a Gröbner basis calculation, not the algorithm of Safey El Din and Schost~\cite{Safey-El-DinSchost2016}) we find that it is separating and a Kronecker representation is given by
\begin{itemize}
	\item a polynomial $P(u)$ of degree 14 and coefficients of absolute value less than $2^{65}$;
	\item polynomials $Q_w,Q_x,Q_y,Q_z,Q_{\lambda},Q_t$ of degrees at most 13 and coefficients of absolute value less than $2^{68}$.  
\end{itemize}
Note that the elements of a reduced Gröbner Basis of this system have coefficients of absolute value up to $2^{344}$, which illustrates the benefit of a Kronecker representation (and hints at why a specialized algorithm should be used instead of a Gröbner basis calculation).

The critical points of $F$ are determined by the roots of 
\[ \tilde{P}(u) = \gcd(P,P'-Q_t) = u^2+160u-800,\] 
as these are the solutions of the polynomial system where $t=1$.  Substituting the roots 
\[ u_1 = -80+60\sqrt{2}, \qquad u_2 =  -80-60\sqrt{2} \]
of $\tilde{P}$ (which can be solved exactly since $\tilde{P}$ is quadratic) into the Kronecker representation determines the two critical points
\begin{align*}
\bp = \left(\frac{Q_w(u_1)}{P'(u_1)},\frac{Q_x(u_1)}{P'(u_1)},\frac{Q_y(u_1)}{P'(u_1)},\frac{Q_z(u_1)}{P'(u_1)}\right) &= \left( -82+58\sqrt{2}, 1+\sqrt{2}, \frac{\sqrt{2}}{2},\frac{\sqrt{2}}{2} \right)\\
\bs = \left(\frac{Q_w(u_2)}{P'(u_2)},\frac{Q_x(u_2)}{P'(u_2)},\frac{Q_y(u_2)}{P'(u_2)},\frac{Q_z(u_2)}{P'(u_2)}\right) &= \left( -82-58\sqrt{2}, 1-\sqrt{2}, \frac{-\sqrt{2}}{2},\frac{-\sqrt{2}}{2} \right)
\end{align*}
of which only $\bp$ has non-negative coordinates and thus could be minimal.  Determining the roots of $P(u)=0$ to sufficient precision shows that there are 6 real values of $t$, and none lie in $(0,1)$.  Thus, $\bp$ is a smooth \emph{minimal} critical point, and there are no other critical points with the same coordinate-wise modulus.  

Once minimality has been determined, the Kronecker representation of this system can be reduced to a Kronecker representation which encodes only critical points.  This is done using $\tilde{P}$ by determining the inverse $P'(u)^{-1}$ of $P'$ modulo $\tilde{P}$ (which exists as $\tilde{P}$ is a factor of $P$, and $P$ and $P'$ are co-prime) and setting 
\[ \tilde{Q}_v(u) := Q_v(u) \tilde{P}'(u)P'(u)(u)^{-1} \text{ mod } \tilde{P}(u) \] 
for each variable $v \in \{w,x,y,z,\lambda\}$.  In this case we obtain a Kronecker representation of the critical point equations given by
\[ \tilde{P}(u) = u^2+160u-800=0 \]
and
\[ w=\frac{-164u+800}{2u+160}, \quad x=\frac{2u+400}{2u+160}, \quad y=z=\frac{120}{2u+160}, \quad \lambda = \frac{-2u-160}{2u+160}. \]
Computing the determinant of the polynomial matrix $\tilde{\mH}$ obtained from multiplying each row of the matrix in Equation~\eqref{eq:Hess} by $\lambda$ shows that the values of this determinant, together with the polynomial $T=wxyz$, can be represented at solutions of the Kronecker representation by
\[ \frac{Q_{\tilde{\mH}}}{\tilde{P}'} = \frac{96u-480}{2u+160}, \qquad \frac{Q_T}{\tilde{P}'} = \frac{34u-160}{2u+160}. \]
Ultimately, noting that $-G=-1$ for this example, we obtain diagonal asymptotics
\[ f_{k,k,k,k} = \left(\frac{u+80}{17u-80}\right)^k \cdot k^{-3/2} \cdot \frac{\sqrt{6u+480}}{48\pi^{3/2} \sqrt{5-u}}\left(1+O\left(\frac{1}{k}\right)\right), \quad u \in \bU \]
where $\bU = \{u_1\} = \{-80+60\sqrt{2}\}$. In general, when $\tilde{P}$ is not quadratic, $\bU$ contains isolating intervals of roots of $\tilde{P}$.  Since we have $u$ exactly here we can determine the leading asymptotic term exactly,
\[ \frac{(17+12\sqrt{2})^k}{k^{3/2}} \cdot \frac{\sqrt{48+34\sqrt{2}}}{8\pi^{3/2}}\left(1+O\left(\frac{1}{k}\right)\right)=\frac{(33.97056\ldots)^k}{k^{3/2}}\left(.22004\ldots+O\left(\frac{1}{k}\right)\right).\]
\end{example}

The calculations for this example can be found in an accompanying Maple worksheet\footnote{Available at~\websiteurl, together with a preliminary package implementing our algorithm in the combinatorial case.}, and additional examples are given in Section~\ref{sec:KronExamples} below.

\subsection{Asymptotics in the General Case}
\label{sec:EffectiveGenAsm}

The general case is trickier, as there may no longer be minimal critical points with non-negative coordinates and we can no longer simply test the line segment between the origin and a finite set of points to determine minimality.  Although we could naively use algorithms on the emptiness of semi-algebraic sets from real algebraic geometry to test whether each critical point is minimal, these algorithms are singly-exponential in the degree of the polynomials encoding critical point coordinates, which are themselves singly exponential.  Instead we adapt similar techniques for our purposes. Our results fall into a category of algorithms known as \emph{critical point methods}, an influential approach to real polynomial system solving popularized by Grigor'ev and Vorobjov~\cite{GrigorevVorobjov1988} and Renegar~\cite{Renegar1992} as the first technique to give a singly-exponential time algorithm sampling a point in each connected component of a real algebraic set. 

Given a polynomial $f(\bz) \in \mathbb{C}[\bz]$ we define $f(\bx + i\by) := f(x_1 + iy_1, \dots, x_n + iy_n)$, and note the unique decomposition
\[ f(\bx+i\by) = f^{(R)}(\bx,\by) + if^{(I)}(\bx,\by),\] 
for polynomials $f^{(R)}(\bx,\by),f^{(I)}(\bx,\by)$ in $\mathbb{R}[\bx,\by]$.  The Cauchy-Riemann equations imply
\[ \frac{\partial f}{\partial z_j}(\bx + i \by) = \frac{1}{2} \cdot \frac{\partial}{\partial x_j}\left(f^{(R)}(\bx,\by) + if^{(I)}(\bx,\by)\right) 
- \frac{i}{2} \cdot \frac{\partial}{\partial y_j}\left(f^{(R)}(\bx,\by) + if^{(I)}(\bx,\by)\right), \]
and it follows that the set of real solutions of the system
\begin{align}
H^{(R)}(\ba,\bb) = H^{(I)}(\ba,\bb) &= 0  \label{eq:GenSys1} \\
a_j \left(\partial H^{(R)}/\partial x_j\right)(\ba,\bb) + b_j \left(\partial H^{(R)}/\partial y_j\right)(\ba,\bb) - \lambda_R&=0, \qquad j = 1,\dots, n  \label{eq:GenSys2} \\
a_j \left(\partial H^{(I)}/\partial x_j\right)(\ba,\bb) + b_j \left(\partial H^{(I)}/\partial y_j\right)(\ba,\bb) - \lambda_I&=0, \qquad j = 1,\dots, n  \label{eq:GenSys3}
%a_j \left(\partial H^{(R)}/\partial x_j\right)(\ba,\bb) + b_j \left(\partial H^{(R)}/\partial y_j\right)(\ba,\bb) - \lambda_R&=0, \qquad j = 1,\dots, n  \label{eq:GenSys2} \\
%a_j \left(\partial H^{(I)}/\partial x_j\right)(\ba,\bb) + b_j \left(\partial H^{(I)}/\partial y_j\right)(\ba,\bb) - \lambda_I &=0, \qquad j = 1,\dots, n \label{eq:GenSys3}
\end{align}
in the variables $\ba,\bb,\lambda_R,\lambda_I$ correspond exactly to all complex solutions of the critical point equations 
\[ H(\bz)=0, \qquad \lambda=z_1(\partial H / \partial z_1) = \cdots = z_n(\partial H/\partial z_n)\] 
with $\bz = \ba+i\bb$ and $\lambda = \lambda_R + i \lambda_I$. Furthermore, if we consider the equations
\begin{align}
H^{(R)}(\bx,\by) = H^{(I)}(\bx,\by) &= 0 \label{eq:GenSys4} \\
x_j^2 + y_j^2 - t(a_j^2+b_j^2) &= 0, \qquad j=1,\dots,n \label{eq:GenSys5}
\end{align}
then, by Proposition~\ref{prop:lineMin}, any $\bz = \ba + i\bb \in \mV$ is minimal if and only if there is no solution to Equations~\eqref{eq:GenSys4}--\eqref{eq:GenSys5} with $\bx,\by,t$ real and $0\leq t < 1$.  

Let 
\begin{itemize}
	\item $\mW$ denote all complex solutions of the system of equations~\eqref{eq:GenSys1}--\eqref{eq:GenSys5}
	\item $\mWR := \mW \cap \mathbb{R}^{4n+3}$ be the real part of $\mW$
	\item $\mWRs := \mW \cap \left(\mathbb{R}^*\right)^{4n+3}$ be the points in $\mWR$ with non-zero coordinates
	\item $\pi_t:\mWR \rightarrow \mathbb{C}$ be the projection map $\pi_t(\ba,\bb,\bx,\by,\lambda_R,\lambda_I,t)=t$.
\end{itemize}
Then we have the following result.

\begin{proposition}
\label{prop:genmincrit}
Let $H\in\mathbb{Q}[\bz]$ be a polynomial which does not vanish at the origin. Suppose that the Jacobian matrix of the polynomials in~\eqref{eq:GenSys1}--\eqref{eq:GenSys5} has full rank at any point in $\mW$ (so $\mW$ is a manifold) and that Equations~\eqref{eq:GenSys1}--\eqref{eq:GenSys3} admit a finite number of complex solutions.  Then:
\begin{enumerate}
	\item The point $(\ba,\bb,\bx,\by,\lambda_R,\lambda_I,t) \in \mWRs$ is a critical point (in the differential geometry sense) of $\pi_t$ if and only if there exists $\nu \in \mathbb{R}$ such that
	\begin{equation}
		(y_j-\nu x_j)\left(\partial H^{(R)}/\partial x_j\right)(\bx,\by) - (x_j+\nu y_j)\left(\partial H^{(R)}/\partial y_j\right)(\bx,\by) =0, \quad j = 1,\dots, n \label{eq:GenSys6}
	\end{equation}
	\item The point $\bz = \ba + i\bb \in \left(\mathbb{C}^*\right)^n$ with $\ba,\bb \in \mathbb{R}^n$ is a minimal critical point of $\mV(H)$ if and only if $(\ba,\bb)$ satisfies Equations~\eqref{eq:GenSys1}--\eqref{eq:GenSys3} and there does not exist $(\bx,\by,\nu,t) \in \mathbb{R}^{2n+2}$ with $0 < t < 1$ satisfying Equations~\eqref{eq:GenSys4}--\eqref{eq:GenSys6}.
\end{enumerate}  
\end{proposition}

\begin{proof}
\emph{(i)} First, we note that $\mWRs$ is a real smooth manifold whenever $\mW$ is a complex analytic manifold.  Furthermore, as Equations~\eqref{eq:GenSys1}--\eqref{eq:GenSys3} admit a finite number of solutions, each connected component of $\mWRs$ (and $\mW$) corresponds to one of the values of $(\ba,\bb)$ satisfying these equations.  A point $(\ba,\bb,\bx,\by,\lambda_R,\lambda_I,t)$ in the connected component of $\mWRs$ defined by $(\ba,\bb)$ is then a critical point of $\pi_t$ if and only if the matrix
\begin{align*}
J &=
\begin{pmatrix}
\nabla H^{(R)}(\bx,\by) \\
\nabla H^{(I)}(\bx,\by) \\
\nabla (x_1^2 + y_1^2 - t(a_1^2+b_1^2)) \\
\vdots \\
\nabla (x_n^2 + y_1^2 - t(a_n^2+b_n^2)) \\
\nabla (t)
\end{pmatrix}
\\[+2mm] &=
\begin{pmatrix}
(\partial H^{(R)}/\partial x_1) & \cdots & (\partial H^{(R)}/\partial x_n) & (\partial H^{(R)}/\partial y_1) & \cdots & (\partial H^{(R)}/\partial y_n) & 0  \\
(\partial H^{(I)}/\partial x_1) & \cdots & (\partial H^{(I)}/\partial x_n) & (\partial H^{(I)}/\partial y_1) & \cdots & (\partial H^{(I)}/\partial y_n) & 0 \\
2x_1 & \bzer & 0 & 2y_1 & \bzer & 0 & -(a_1^2+b_1^2) \\
\bzer & \ddots & \bzer & \bzer & \ddots & \bzer & \vdots \\
0 & \bzer & 2x_n & 0 & \bzer & 2y_n & -(a_n^2+b_n^2) \\
0 & \cdots & 0 & 0 & \cdots & 0 & 1
\end{pmatrix}
\end{align*}
is rank deficient, since a critical point of $\pi_t$ is precisely one where the gradient of $t$ is perpendicular to the tangent plane of $\mWR$. 
\smallskip

Using the Cauchy-Riemann equations to write 
\[ (\partial H^{(I)}/\partial x_j) = -(\partial H^{(R)}/\partial y_j) \quad \text{and} \quad (\partial H^{(I)}/\partial y_j) = (\partial H^{(R)}/\partial x_j)\] 
implies that $(\ba,\bb,\bx,\by,\lambda_R,\lambda_I,t) \in \mWRs$ is a critical point if and only if there exists $\nu,\lambda_1,\dots,\lambda_n$ such that
\begin{align*} 
(\partial H^{(R)}/\partial x_j) - \nu(\partial H^{(R)}/\partial y_j) + \lambda_j x_j &= 0\\
(\partial H^{(R)}/\partial y_j) + \nu(\partial H^{(R)}/\partial x_j) + \lambda_j y_j &= 0
\end{align*}
for each $j=1,\dots,n$.  This system of equations simplifies to Equations~\eqref{eq:GenSys6}.
\bigskip

\noindent
\emph{(ii)} When $\mWRs$ is a smooth manifold, any local minimum of the function $\pi_t$ must occur at a critical point of the function.  For each of the finite real values of $(\ba,\bb)$ satisfying Equations~\eqref{eq:GenSys1}--\eqref{eq:GenSys3}, the set
\[ S = \left\{ (\bx,\by,t) \in \mathbb{R}^{2n+1} : t \in [0,1], \quad (\ba,\bb,\bx,\by,t) \text{ satisfy Equations~\eqref{eq:GenSys4} and \eqref{eq:GenSys5}}\right\} \]
is compact, as $t \in [0,1]$ implies $x_j^2+y_j^2 \leq a_j^2+b_j^2$ for each $j=1,\dots,n$.  Furthermore, $S$ is non-empty because it contains $(\ba,\bb,\ba,\bb,1)$.  Thus, the continuous function $\pi_t$ achieves its minimum on the compact set $S$, and such a minimizer must be a critical point of $\pi_t$ or have $t=1$.  Any solution $(\bx,\by,t) \in S$ with $t<1$ gives a point $\bx+i\by$ that has smaller coordinate-wise modulus than $\ba+i\bb$, meaning $\bz$ is not minimal. Likewise, if $\pi_t$ has no critical points with $t<1$ then Equations~\eqref{eq:GenSys4} and~\eqref{eq:GenSys5} have no solution with $t<1$, meaning $\bz$ is minimal.  

Finally, if $\bz=\ba+i\bb$ is a minimal critical point of $\mV(H)$ then $(\ba,\bb)$ satisfies Equations~\eqref{eq:GenSys1}--\eqref{eq:GenSys3} and there are no solutions of Equations~\eqref{eq:GenSys4} and~\eqref{eq:GenSys5} with $t<1$.
\end{proof}

Our strategy will be to use Proposition~\ref{prop:genmincrit} to prove minimality in the non-combinatorial case, and we define the following assumptions:
\begin{itemize} 
\item[(A4)] the system of equations~\eqref{eq:GenSys1}--\eqref{eq:GenSys3} has a finite number of complex solutions;
\item[(A5)] the Jacobian matrix of the system of equations~\eqref{eq:GenSys1}--\eqref{eq:GenSys5} has full rank at its solutions;
%\item[(A6)] the system of equations~\eqref{eq:GenSys1}--\eqref{eq:GenSys6} has a finite number of complex solutions;
\item[(J2)] the Jacobian matrix of the system of equations~\eqref{eq:GenSys1}--\eqref{eq:GenSys6} is non-singular at its solutions.
\end{itemize}

Note that assumption (A5) implies that $\mW$ is a manifold, and that (J2) implies the system of equations~\eqref{eq:GenSys1}--\eqref{eq:GenSys6} has a finite number of complex solutions.  The following result is proven in Section~\ref{sec:generic}.

\begin{proposition}
\label{prop:genericEffectiveAssumption}
Assumption (A4) holds generically.
\end{proposition}

The structure of the polynomials appearing in the system of equations~\eqref{eq:GenSys1}--\eqref{eq:GenSys6} makes it more difficult to prove assumptions (A5) and (J2) hold generically, and we state the following conjecture.

\begin{conjecture}
\label{conj:genericEffectiveAssumptions}
Assumptions (A5) and (J2) hold generically.
\end{conjecture}

Work in progress on proving Conjecture~\ref{conj:genericEffectiveAssumptions} is discussed at the end of Section~\ref{sec:generic}. We also note that assumptions (A4), (A5), and (J2) can be rigorously verified.

Since equations~\eqref{eq:GenSys1}--\eqref{eq:GenSys6} have a multi-homogeneous structure (the $\ba,\bb$ variables and the $\bx,\by$ variables do not appear with high degree in the same equation), we will apply the results of Safey El Din and Schost~\cite{Safey-El-DinSchost2016} with the variables partitioned into the blocks 
\[ \left[\bZ_1,\dots,\bZ_6\right] = \left[ (\ba,\bb),(\bx,\by),(\lambda_R),(\lambda_I),(t),(\nu) \right]. \]
This will ultimately lead to the following result.

\begin{algorithm}
\DontPrintSemicolon
\KwIn{Rational function $F(\bz) = G(\bz)/H(\bz)$ which satisfies the hypotheses of Theorem~\ref{thm:GenEffMinCrit}}
\KwOut{Numerical Kronecker representation $[P,\bQ,\bU]$ for the minimal critical points of $F$}
\;
$\bbf' \gets $ Polynomials in Equations~\eqref{eq:GenSys1} -- \eqref{eq:GenSys6} \\
$[P',\bQ'] \gets {\sf KroneckerRep}(\bbf')$ \\
$[P',\bQ',\mathbf{V}'] \gets {\sf NumericalKroneckerRep}(P,\bQ)$ \\
Group elements of $\mV(\bbf')$ with the same value of $\bz := \ba + i \bb$ \\
\For{each $\bz$ with non-zero coordinates}{
	Check whether there is a corresponding element of $\mV(\bbf')$ with $0 < t < 1$ \\
	If not, then $\bz$ is a minimal critical point
} 
$\bbf \gets \left[H, \quad z_1(\partial H/\partial z_1)-\lambda, \quad \dots \quad ,z_n(\partial H/\partial z_n)-\lambda \right]$ \\
$[P,\bQ] \gets {\sf KroneckerRep}(\bbf)$ \\
$[P,\bQ,\mathbf{V}] \gets {\sf NumericalKroneckerRep}(P,\bQ)$ \\
$\bU \gets$ solutions of $P$ corresponding to minimal critical points \\
Return $[P,\bQ,\bU]$
\caption{{\sf MinimalCritical}}
\label{alg:MinimalCritical}
\end{algorithm}

\begin{theorem}
\label{thm:GenEffMinCrit}
Assume (A0)---(A5), (J1), and (J2).  Then there exists a probabilistic algorithm which determines the minimal critical points of $F(\bz)$ in $\tilde{O}\left(hd^42^{3n}D^9\right)$ bit operations.  The algorithm returns polynomials $Q_1(u),\dots,Q_n(u),P(u) \in \mathbb{Z}[u]$, with $P(u)$ square-free, and a set $U$ of isolating intervals of roots of $P$ such that the set of minimal critical points of $F$ is
\[ E := \left\{ \left( \frac{Q_1(u)}{P'(u)}, \dots, \frac{Q_n(u)}{P'(u)} \right) : u \in U \right\}. \]
The coordinates of all minimal critical points of $F(\bz)$ can be determined to precision $2^{-\kappa}$ in $\tilde{O}(dD\kappa + hd^3D^3)$ bit operations.
\end{theorem} 

Algorithm~\ref{alg:MinimalCritical} gives a high level implementation of Theorem~\ref{thm:GenEffMinCrit}.  If the minimal critical points of $F(\bz)$ are finitely minimal,  Corollary~\ref{cor:smoothAsm} implies that diagonal asymptotics can be determined from the numerical Kronecker representation $[P,\bQ,\bU]$ returned by Algorithm~\ref{alg:MinimalCritical}.  This is achieved by following the steps of Algorithm~\ref{alg:EffectiveCombAsm} after it has determined the minimal critical points $\bz^{(1)},\dots,\bz^{(k)}$.  Similarly, if all minimizers of $|z_1 \cdots z_n|^{-1}$ on the boundary of the boundary of convergence of $F(\bz)$ have the same coordinate-wise modulus and contain the critical points in $E$, then Proposition~\ref{prop:smoothAsmCP} implies that diagonal asymptotics can be determined from the output of Algorithm~\ref{alg:MinimalCritical}.  Under our assumptions, to show that all minimizers of $|z_1 \cdots z_n|^{-1}$ on the boundary of the domain of convergence of $F(\bz)$ have the same coordinate-wise modulus it is sufficient to prove that 
\[ \overline{\Relog(\mD)} \cap \amoeba(H) \subset \partial\amoeba(H) \subset \mathbb{R}^n\] 
contains no line segment lying in a hyperplane normal to $\bone$.

Furthermore, assuming that the boundary of $\amoeba(H)$ contains no such line segment (which is a standing assumption of Pemantle and Wilson~\cite{PemantleWilson2013}), one can detect when minimal critical points are finitely minimal. Bannwarth and Safey El Din~\cite{BannwarthSafey-El-Din2015} describe a probabilistic algorithm which can determine when the real solution set of a polynomial $A(\bz)$ of degree $\delta$ is finite in $\tilde{O}(n^8\delta^{3n-1})$ arithmetic operations over the rational numbers\footnote{This is achieved by running the algorithm {\sf HasEmptyInterior}$(A,i)$ of Bannwarth and Safey El Din~\cite{BannwarthSafey-El-Din2015} with $i=1$ to check if the projection of the real solution set of $A$ to a generic one dimensional space has empty interior. This generic projection is achieved by making a random invertible linear change of variables, which is where the probabilistic nature of the result arises.  Because we only care about determining when the set of real solutions is zero dimensional, not what its dimension is, we can improve on the complexity stated in that article for the general case.  In particular, the proof of Theorem 7 in that paper gives the complexity of {\sf HasEmptyInterior} for any input $i \in\{1,\dots,n\}$.  When $i=1$, the set of points $L$ determined in Step 2 of their analysis can be taken to be any $\delta^n$ points between each of the solutions of the univariate polynomial $g=g_1$ which they construct in Step 1.  Using this improvement in the rest of their complexity analysis gives the complexity listed here.}.  To verify that every critical point of $F(\bz)$ is finitely minimal it is sufficient\footnote{When $\mV$ is smooth and $\amoeba(H)$ does not contain such a line segment on its boundary then $D(\bz) \cap \mV \subset T(\bz) \cap \mV$ at any critical point $\bz$ by the convexity of $\Relog(\mD)$.} to take $A$ to be the sum of squares of the polynomials appearing in Equations~\eqref{eq:GenSys1}--\eqref{eq:GenSys5} after setting $t=1$ and eliminating the variables $\lambda_r$ and $\lambda_I$, giving a polynomial $A$ of degree $2d$ in $4n$ variables.  Thus, when every critical point of $F(\bz)$ is finitely minimal, this can be verified in $\tilde{O}((2d)^{16n})$ arithmetic operations.

Without using Proposition~\ref{prop:kronrep} to take advantage of the homogeneous structure of the system of equations~\eqref{eq:GenSys1} -- \eqref{eq:GenSys6}, the bit complexity of Theorem~\ref{thm:GenEffMinCrit} would contain a factor of $D^{12}$ instead of $2^{3n}D^9$. 

\begin{example}
Section 3.1 of Adamczewski and Bell~\cite{AdamczewskiBell2013} shows that the algebraic function 
\[ A(z) = z\sqrt{1-z} = z - \frac{z^2}{2} - \frac{z^3}{8} - \frac{z^4}{16} - \cdots \]
is the diagonal of the rational function
\[ F(w,z) = \frac{G(w,z)}{H(w,z)} = \frac{wz(16w^2z^4-12w^2z^3+4w^2z^2-w^2z+24wz^2-16wz+2w+16)}{8+w(2z-1)^2}, \]
which is not combinatorial.
%(note also that the minimal polynomial $P(z,y) = y^2-z^2(1-z)$ of $A$ has $(\partial P/\partial y)(0,0) = 0$, so that Theorem~\ref{thm:algToDiag} does not apply).  
It can be checked that $\mV(H)$ is smooth, and that there is a single critical point $\bp = (-1/2,-2)$.  In order to prove that $\bp$ is minimal, we substitute $(w,z) = (x_1 + iy_1, x_2 + iy_2)$ and form the system
\begin{align}
H^{(R)}(x_1,x_2,y_1,y_2) = 4x_1x_2^2-4x_1y_2^2-8x_2y_1y_2-4x_1x_2+4y_1y_2+x1+8 &= 0 \label{eq:genex1} \\
H^{(I)}(x_1,x_2,y_1,y_2) = 8x_1x_2y_2+4x_2^2y_1-4y_1y_2^2-4x_1y_2-4x_2y_1+y_1 &= 0 \\
x_1^2 + y_1^2 - 4t &= 0 \\
x_2^2 + y_2^2 - t/4 &= 0, \label{eq:genex4}
\end{align}
where we use the moduli of the coordinates of $\bp$ directly, instead of encoding them with the critical point equations, since they are rational (the system of equations~\eqref{eq:GenSys1} -- \eqref{eq:GenSys3} has only one complex solution, corresponding to $\bp$).  Our goal is to determine whether there exists a real solution of these equations with $t \in (0,1)$, but the fact that the system (with 4 equations in 5 variables) has an infinite number of complex solutions makes this difficult.  By examining the maximal minors of the Jacobian of this system at its solutions we can verify that the algebraic set formed by these equations satisfies our required assumptions, so to determine minimality we can add Equations~\eqref{eq:GenSys6}, which in this example are
\begin{align}
(-\nu x_1+y_1)(4x_2^2-4y_2^2-4x_2+1)-(\nu y_1+x_1)(-8x_2y_2+4y_2) &= 0 \\
(-\nu x_2+y_2)(8x_1x_2-8y_1y_2-4x_1)-(\nu y_2+x_2)(-8x_1y_2-8x_2y_1+4y_1) &= 0, \label{eq:genex6}
\end{align}
and use Proposition~\ref{prop:genmincrit}.  Computing a Gröbner Basis of the polynomials in Equations~\eqref{eq:genex1}--\eqref{eq:genex6} shows that this system has a finite number of complex solutions, and that at any solution
\[ (t-1)(t^2-t+16)(t^3-2t^2-15t-16) = 0. \]
The solution $t=1$ corresponds to $\bp$, and the only other real solution is an algebraic number of degree 3, which is approximately $5.35$. Thus, $\bp$ is minimal.  

There are several ways to show that the boundary of $\amoeba(H)$ doesn't contain a line segment, and we exhibit a method which can be used on other bivariate systems (generalizations also exist for higher dimensional cases).  As discussed in Theobald~\cite[Theorem 5.1]{Theobald2002}, every point $(w,z) \in \mV$ which maps to the boundary of $\amoeba(H)$ satisfies
\[ H(w,z) = 0, \qquad w(\partial H/\partial w) - \lambda z(\partial H/\partial z) = 0 \]
for some real $\lambda$.  Solving this system for $w$ and $z$ in terms of $\lambda$ gives a real parametrization of a set containing the boundary of the amoeba.  In this example, we have that any point in the amoeba boundary lies in the set
\[ \left\{ {\Big(}\log2 + 2\log|2\lambda-1|-2\log|\lambda|,\,\, -\log2-\log|2\lambda-1|{\Big)}: \lambda \in \mathbb{R} \right\}, \]
which does not contain a line segment. Setting $t=1$ in Equations~\eqref{eq:genex1}--\eqref{eq:genex4}, a Gröbner Basis computation shows that this system has a finite number of complex solutions\footnote{Note that this argument can only work for bivariate systems, as in general one obtains $2n+2$ equations in $n+2$ variables and must use the techniques from real algebraic geometry discussed above to prove finite minimality.}, so $\bp$ is finitely minimal.  Theorem~\ref{thm:smoothAsm} then implies
\[ f_{k,k} = -\frac{1}{2\sqrt{\pi}}k^{-3/2} - \frac{15}{16\sqrt{\pi}}k^{-5/2} + O\left(\frac{1}{k^{7/2}}\right), \]
which matches what one obtains from analyzing the univariate generating function $A(z) = z\sqrt{1-z}$.  The calculations for this example can be found in an accompanying Maple worksheet\footnote{Available at~\websiteurl.}.
\end{example}

%%%%%%%%%%%%%%%%%%%%%%%%%%%%%%%%%%%%%%%%%%%%%%%
% Section: Algorithm Correctness and Complexity
%%%%%%%%%%%%%%%%%%%%%%%%%%%%%%%%%%%%%%%%%%%%%%%
\section{Algorithm Correctness and Complexity}
\label{sec:AlgorithmProofs}
We now prove the results of the previous section by building up a symbolic-numeric toolkit describing the algorithms we will require.

\subsection{Univariate Polynomial Bounds}
\label{sec:UniPolyBounds}

\begin{lemma}
\label{lemma:height} 
For univariate polynomials $P_1,\dots,P_{k},P,Q \in \mathbb{Z}[z]$,
\begin{align*}
  h(P_1 + \cdots + P_k) &\leqslant \max_i  h(P_i) + \log_2 k,\\
  h(P_1 \dotsm P_k) & \leqslant \sum_{i=1}^k {h(P_i)} + \sum_{i=1}^{k-1} \log_2(\deg P_i + 1),\\
  h(P)&\leqslant \deg P + h(PQ) + \log_2\sqrt{\deg(PQ)+1}.
\end{align*}
\end{lemma}

The first two results follow directly from the definition of polynomial height, and the final one---sometimes referred to as `Mignotte's bound on factors'---follows from Theorem 4 in Chapter 4.4 of Mignotte~\cite{Mignotte1992}.

%\subsection*{Polynomial Root Separation Bounds}
\begin{lemma}[Mignotte~\cite{Mignotte1992}]
\label{lemma:roots} 
Let $A \in \mathbb{Z}[z]$ be a polynomial of degree $d\geqslant 2$ and height $h$. If $A(\alpha)=0$ then
\begin{enumerate}[(i)]
  \item if $\alpha\neq0$, then $1/(2^h+1)\le |\alpha|\le 2^h+1$; \label{item:roots1}
  \item if $A(\beta)=0$ and $\alpha\neq\beta$, then $|\alpha-\beta|\ge d^{-(d+2)/2} \cdot \|A\|_2^{1-d}$; \label{item:roots2}
  \item if $Q(\alpha)\neq0$ for $Q\in\mathbb{Z}[T]$, then $|Q(\alpha)|\ge ((\deg Q+1)2^{h(Q)})^{1-d} \cdot \|A\|_2^{-\deg Q}$; \label{item:roots3}
  \item if $A$ is square-free then $|A'(\alpha)| \ge 2^{-2dh \, + \, 2h \, + \, 2(1-d)\log d \, + \, (1-d) \log \sqrt{d+1}}$, \label{item:roots4}
\end{enumerate}
where $\|A\|_2$ is the 2-norm of the vector of coefficients, bounded by $2^h\sqrt{d+1}$.
\end{lemma}

The upper bound of statement~\eqref{item:roots1} comes from Theorem 4.2(ii) in Chapter 4 of Mignotte~\cite{Mignotte1992} (note that Theorem numbers reset between chapters), and the lower bound is a consequence of applying the upper bound to the reciprocal polynomial $z^dA(1/z)$.  Statement~\eqref{item:roots2} comes from Theorem 4.6 in Section 4.6 of that text. 

The proof of~\eqref{item:roots3} uses the \emph{Mahler measure} $M(P)$ of $P$.  If the leading coefficient of $P$ is $c \in \mathbb{Z}$, and the roots of $P$ are $\alpha_1,\dots,\alpha_d$, then $M(P) = c \prod_{j=1}^d \max\{1,|\alpha_j|\}$.  Classical results on the Mahler measure, which are proven in Mignotte~\cite{Mignotte1992}, include that $M(P) \leq \|P\|_2$ and that if $\tilde{P}$ is a factor of $P$ then $M(\tilde{P}) \leq M(P)$. If $Q(\alpha) \neq 0$ and $P$ and $Q$ are polynomials with integer coefficients, then there exists a factor $\tilde{P}$ of $P$ containing $\alpha$ as a root such that $\tilde{P}$ and $Q$ share no roots.  Since the resultant of $\tilde{P}$ and $Q$ is a non-zero integer, it follows that
\[ |Q(\alpha)| \geq ((\deg Q+1)2^{h(Q)})^{1-d} \cdot M(\tilde{P})^{-\deg Q}, \]
which implies~\eqref{item:roots3}.  Item~\eqref{item:roots4} is a special case of~\eqref{item:roots3}.

\paragraph{Univariate Resultant and GCD Bounds}
A height bound on the greatest common divisor of two univariate polynomials is given by Lemma~\ref{lemma:height}, and the complexity of computing gcds is well known~\cite[Corollary 11.14]{GathenGerhard2013}. 

\begin{lemma}
\label{lemma:gcd}
For $P$ and $Q$ in ${\mathbb Z}[U]$ of height at most $h$ and degree at most $d$, $\gcd(P,Q)$ has height $\tilde{O}(d+h)$ and can be computed in $\tilde{O}(d^2+hd)$ bit operations. 
\end{lemma}

Similarly, a degree bound for the resultant of two polynomials follows from a direct expansion of the determinant of the Sylvester matrix, and Lemma~\ref{lemma:height} combined with this expansion gives a bound on the resultant height. 

\begin{lemma}
\label{lemma:resultant}
For $P$ and $Q$ in $\mathbb{Z}[T,U]$ let $R = \Res_T(P,Q)$ and
\begin{align*}
\delta &:= \deg_TP\deg_UQ+\deg_TQ\deg_UP \\
\eta &:= h(P)\deg_TQ+ h(Q)\deg_TP+\log_2((\deg_TP+\deg_TQ)!) + \log_2(\deg_UP+1) \deg_TQ\\
&\hspace{3.9in} + \log_2(\deg_UQ+1) \deg_T P.
\end{align*}
Then $\deg R \leq \delta$ and $h(R) \leq \eta$. Furthermore, if all coefficients of $P$ and $Q$ as polynomials in $T$ are monomials in $U$ then $h(R) \leq h(P)\deg_TQ+ h(Q)\deg_TP+\log_2((\deg_TP+\deg_TQ)!)$.
\end{lemma}

\subsection{Univariate Polynomial Algorithms}
\label{sec:UniPolyAlg}

Our algorithms for computing a numerical Kronecker representation rely on the following results for numerically evaluating and finding the roots of polynomials.

\begin{lemma}[Sagraloff and Mehlhorn~\cite{SagraloffMehlhorn2016}, Mehlhorn et al.~\cite{MehlhornSagraloffWang2015}]
\label{lemma:fsolve}
Let $A \in \mathbb{Z}[T]$ be a square-free polynomial of degree $d$ and height $h$.  Then for any positive integer $\kappa$ 
\begin{itemize}
\item isolating disks of radius less than $2^{-\kappa}$ can be computed for all roots of $A(T)$ in $\tilde{O}(d^3+d^2h+d\kappa)$ bit operations;
\item isolating intervals of length less than $2^{-\kappa}$ can be computed for all real roots of $A(T)$ in $\tilde{O}(d^3+d^2h+d\kappa)$ bit operations.
\end{itemize}
\end{lemma}

The statement for real roots is Theorem 3 of Sagraloff and Mehlhorn~\cite{SagraloffMehlhorn2016}, and an implementation is discussed in Kobel, Rouillier, and Sagraloff~\cite{KobelRouillierSagraloff2016}.  The second statement follows from Theorem 5 of Mehlhorn, Sagraloff, and Wang~\cite{MehlhornSagraloffWang2015}.

\begin{lemma}[Kobel and Sagraloff~\cite{KobelSagraloff2015}]
\label{lemma:feval} 
Let $A \in \mathbb{Z}[T]$ be a square-free polynomial of degree $d$ and height $h$, and $t_1,\dots,t_m \in \mathbb{C}$ be a sequence of length $m=O(d)$.  Then for any positive integer $\kappa$, approximations $a_1,\dots,a_m \in \mathbb{C}$ such that $|A(t_j)-a_j|<2^{-\kappa}$ for all $1 \leq j \leq m$ can be computed in $\tilde{O}(d (h+\kappa+d\log\max_j|t_j| ))$ bit operations, given $\tilde{O}(h+\kappa+d\log\max_j |t_j|)$ bits of $t_1,\dots,t_m$.  If all $t_j$ are real, the approximations $a_j$ are also real.
\end{lemma}

Lemma~\ref{lemma:feval} follows from Theorem 10 of Kobel and Sagraloff~\cite{KobelSagraloff2015} (the statement about real roots follows from the proof given in Appendix B of that paper). These results immediately imply the following.

\begin{corollary}
\label{cor:refineroots}
Given a square-free polynomial $A\in\mathbb{Z}[T]$ of degree $d$ and height $h$, isolating regions of the roots of $A$ and a factor $P\in\mathbb{Z}[T]$ of $A$, selecting which roots of $A$ are roots of $P$ can be achieved in $\tilde{O}(d^3+d^2h)$ bit operations.
\end{corollary}
\begin{proof}
Lemma~\ref{lemma:roots} shows that regions of size $2^{-\tilde{O}(hd)}$ are sufficient to separate the roots of $A$. By Lemma~\ref{lemma:height}, the height of $P$ is at most $\deg P+h+\log_2\sqrt{d+1} =\tilde{O}(d+h)$, so Lemma~\ref{lemma:fsolve} implies that isolating regions for the roots of $P$ of sufficient size can be computed in $\tilde{O}(d^3+d^2h)$ bit operations.
\end{proof}

\subsection{Basic Algorithms for the Numerical Kronecker Representation}
\label{sec:NumKronAlg}

In this section and Section~\ref{sec:groupmod} we prove the existence of the algorithms detailed in Proposition~\ref{prop:NumKronAlgs}.  To that end, let $[P(u),\bQ]$ be a Kronecker representation in dimension $n$, of degree $d$ and height~$h$.  

\begin{lemma}
\label{lemma:EvalCoords}
There exists an algorithm, {\sf EvaluateCoordinates}, which takes $[P(u),\bQ]$ and a natural number $\kappa$ and returns approximations to the solutions of the Kronecker representation with isolating regions for each coordinate of size $2^{-\kappa}$ in $\tilde{O}(n(d^3+d^2h+d\kappa))$ bit operations.
\end{lemma}
\begin{proof}
Fix a coordinate $z_j$ and root $v \in \mathbb{C}$ of $P(u)=0$.  Our aim is to evaluate $z_j = Q_j(v)/P'(v)$ to an accuracy of $2^{-\kappa}$. Assume that we have approximations $q$ and $p$ to $Q_j(v)$ and $P'(v)$ such that $|Q_j(v)-q|$ and $|P'(v)-p|$ are both less than $2^{-a}$ for some natural number $a$.  It follows that
\[ \left|\frac{Q_j(v)}{P'(v)} - \frac{q}{p} \right| = \left|\frac{Q_j(v)p - qp + qp - qP'(v)}{P'(v)p}\right| \leq 2^{-a} \left(\frac{1}{|P'(v)|}+\left|\frac{q}{P'(v)p}\right|\right). \]
When 
\[ a \geq (2d-2)\log_2 d + 2dh - 2h + (d-1)\log\sqrt{d+1} + 1 = \tilde{O}(hd)\] 
then the triangle inequality implies 
\begin{align*}
|p| &\geq |P'(v)| - 2^{-a} \geq 2^{-a} = 2^{-\tilde{O}(hd)} \\
|q| & \leq |Q_j(v)| + 2^{-a} \leq 2^h(1+|v|+\cdots+|v|^d) +2^{-a}  \leq 2^h(d+1)(2^h+1)^d+2^{-a}  = 2^{\tilde{O}(hd)},
\end{align*} 
where Lemma~\ref{lemma:roots} gives a lower bound on $|P'(v)|$ and an upper bound on $|v|$.  Thus,
\[ \left|\frac{Q_j(v)}{P'(v)} - \frac{q}{p} \right| = 2^{-a + \tilde{O}(hd)}. \]
To determine $z_j$ to precision $\kappa$ it is therefore sufficient to determine $Q_j(v)$ and $P'(v)$ to $\kappa+\tilde{O}(hd)$ bits.  Lemmas~\ref{lemma:fsolve} and~\ref{lemma:feval} imply that this can be done at all roots of $P$ in $\tilde{O}(d^3+d^2h+d\kappa)$ bit operations, and doing this for each coordinate $z_j$ gives the stated complexity.
\end{proof}

As $u$ is a linear form with integer coefficients in the coordinates $z_j$, and the polynomials in the Kronecker representation have integer coefficients, a root of $P(u)$ is real if and only if every coordinate $z_j$ in the corresponding solution is real.  This allows us to prove the following result.

\begin{lemma}
\label{lemma:DetSign}
Let $[P(u),\bQ,\bU]$ be a numerical Kronecker representation of degree $d$ and height $h$.  There exists an algorithm, {\sf DetermineSign}, which takes $[P(u),\bQ,\bU]$ and determines for every real solution to the underlying system whether each coordinate is positive, negative or exactly 0, in $\tilde{O}(n(d^3+d^2h))$ bit operations.
\end{lemma}

\begin{proof}
Fix a coordinate $z_j$.  The roots of $P(u)$ that correspond to solutions with $z_j=0$ are exactly those cancelling the polynomial $G_j(u) := \gcd(P,Q_j)$.  By Lemma~\ref{lemma:gcd}, $G_j$ has height $\tilde{O}(d+h)$ and can be computed in $\tilde{O}(d^2+hd)$ bit operations.  Corollary~\ref{cor:refineroots} then shows that the roots of $P(u)$ which also cancel $G_j(u)$ can be determined in $\tilde{O}(d^3+d^2h)$ bit operations by computing isolating regions of the roots of size $2^{-\tilde{O}(d^2+hd)}$. Lemma~\ref{lemma:roots} shows that knowing approximations of $Q_j(u)$ and $P'(u)$ within $2^{-\tilde{O}(h)}$ allows one to determine their signs when they are real, and Lemma~\ref{lemma:feval} shows that such approximations can be determined in $\tilde{O}(hd^2)$ bit operations knowing only $\tilde{O}(hd)$ bits of the solutions to $P(u)=0$.
\end{proof}

In order to detect when the coordinates of two solutions are exactly equal we must first determine separation bounds on the coordinates.  Thus, we give a result bounding the minimal polynomials of the variables $z_1,\dots,z_n$ at points specified by a Kronecker representation.

\begin{lemma}
\label{lemma:coordpoly}
Let $\bbf$ be a zero-dimensional polynomial system containing $n$ polynomials of degree at most $d$ and heights at most $h$.  For each coordinate $z_j$ there exists a polynomial~$\Phi_j\in\mathbb{Z}[T]$ of degree at most~$D$ and height less than 
\[ \left(2 + nh + 2n\log(n+1) + 3 \log(2n)\right)D \]
that vanishes exactly at the values taken by $z_j$ on the solutions of the system.  Given a Kronecker representation $[P,\bQ]$ of the solutions of $\bbf$, the polynomial $\Phi_j$ can be determined in $\tilde{O}(hD^3)$ bit operations.
\end{lemma}

\begin{proof}
The bound on the height of $\Phi_j$ follows from arithmetic arguments concerned with affine heights of algebraic sets; we do not define such concepts here, but refer the reader to Krick et al~\cite{KrickPardoSombra2001} or Schost~\cite{Schost2001}.  The polynomial system $\bbf$ defines an algebraic set $V(\bbf)$ of (algebraic set) height\footnote{See Krick et al.~\cite[Corollary 2.10]{KrickPardoSombra2001} or Schost~\cite[Proposition 14]{Schost2001}.} 
\[ \mathfrak{h}(V(\bbf)) \leq \left(nh + 2n\log(n+1)\right)D.\] 
Lemma 2.6 of Krick et al.~\cite{KrickPardoSombra2001} then implies that the algebraic set obtained by projecting $V(\bbf)$ onto the coordinate $z_j$ has height at most 
\[ \mathfrak{h}(V(\bbf)) + 3D\log(2n) \leq \left(nh + 2n\log(n+1)+3\log(2n)\right)D, \]
and any algebraic set of that height can be defined~\cite[Proposition 13]{Schost2001} as the zero set of an irreducible polynomial of height at most 
\[ \left(2 + nh + 2n\log(n+1) + 3 \log(2n)\right)D = \tilde{O}(hD). \]

The degree bound on $\Phi_j$ comes from the fact that it divides the resultant of the polynomials $P(u)$ and $P'(u)T - Q_j(u)$ in the Kronecker representation, which is a polynomial of degree at most $D$ by Lemma~\ref{lemma:resultant}.  Finally, the arithmetic complexity~\cite[Corollary 11.21]{GathenGerhard2013} of this bivariate resultant is $\tilde{O}(D^2)$ and, as the resultant has coefficients of height $\tilde{O}(hD)$, one can use modular techniques\footnote{See the proof of Lemma 6 in Kobel and Sagraloff~\cite{KobelSagraloff2015} for details on the modular algorithm (our complexity is lower than the result stated there because our polynomials are linear with respect to $T$, and we have the bound $\tilde{O}(hD)$ on the height of our resultant, but the same proof holds).} to obtain a bit complexity of $\tilde{O}(hD^3)$.
\end{proof}

The root separation bound in Lemma~\ref{lemma:roots} then implies that to separate the coordinates of $V(\bbf)$ it is sufficient to know each coordinate of each solution to precision $2^{-\tilde{O}(hD^2)}$, and Lemma~\ref{lemma:EvalCoords} shows that all coordinates can be determined to this precision in $\tilde{O}(hD^3)$ bit operations.

\begin{corollary}
Given a numerical Kronecker representation $[P(u),\bQ,\bU]$ corresponding to a zero-dimensional system of $n$ polynomials of degrees at most $d$ and heights at most $h$, there exists an algorithm, {\sf EqualCoordinates}, which takes $[P(u),\bQ,\bU]$ and determines which coordinates of its solutions are equal in $\tilde{O}(hD^3)$ bit operations.
\end{corollary}

We now prove Proposition~\ref{prop:KronReduce}, which describes how to parametrize the values of a new polynomial in terms of an existing Kronecker representation. 

\begin{proof}[Proof of Proposition~\ref{prop:KronReduce}]
    \begin{itemize}
    	\item[(i)] Adding the polynomial $T-q$ to a polynomial system $\bbf$ gives a new polynomial system $\bbf'$ with the same number of solutions as $\bbf$, and any separating linear form $u$ for the solutions of $\bbf$ is a separating linear form for the solutions of $\bbf'$.  Thus, the degree of a Kronecker representation of $\bbf'$ is at most the degree of a Kronecker representation of $\bbf$, which is bounded by $\sC_{\bn}(\bd)$.  Furthermore, one can construct a partition of the variables $z_1,\dots,z_n,T$ by partitioning $z_1,\dots,z_n$ according to $\bZ$ and taking $T$ by itself.  Working through the bounds of Safey El Din and Schost~\cite[Proposition 12]{Safey-El-DinSchost2016} using this new partition of variables shows that the polynomials in any Kronecker representation of $\bbf'$ using the separating linear form $u$ contains polynomials with the stated height bound, and
    	Proposition~\ref{prop:kronrep} shows that $Q_q$ can be determined in the stated complexity.
    	\item[(ii)] The minimal polynomial $\Phi_q$ divides the resultant of the polynomials $P'(u)-TQ_q$ and $P(u)$, so the stated height and degree bounds on $\Phi_q$ follow from Lemma~\ref{lemma:resultant}. 
    \end{itemize} 
\end{proof}

\subsection{Grouping Roots by Modulus}
\label{sec:groupmod}
The most costly operation we will make use of is grouping roots with the same modulus.  Unlike the separation bound given in Lemma~\ref{lemma:roots} between distinct complex roots of a polynomial, which has order~$2^{-\tilde{O}(hd)}$, the best separation bound for the \emph{moduli} of roots that we know of~\cite[Theorem 1]{GourdonSalvy1996} has order $2^{-\tilde{O}(hd^3)}$, and computing the coordinates of a Kronecker representation to this accuracy would be costly.  Luckily, for the cases in which we need to group roots of a polynomial by modulus it will always be the case that the modulus itself is a root of $P$.  In this situation, we have a better bound.

\begin{lemma}
\label{lemma:modsep}
For square-free $A\in\mathbb{Z}[T]$ of degree $d\ge2$ and height $h$ let $G(T)$ be the \emph{Graeffe polynomial} $G(T):=A\left(\sqrt{T}\right)A\left(-\sqrt{T}\right)$.  If $A(\alpha)=0$ and $A(\pm|\alpha|)\neq 0$, then
\[\left|G\left(|\alpha|^2\right)\right| \ge((d+1)^2 \, 2^{2h})^{1-d^2}2^{-2d^2(h+\log_2d)-d\log_2\sqrt{d+1}} = 2^{-\tilde{O}(hd^2)}.\]
\end{lemma}
\begin{proof}
By Lemma~\ref{lemma:resultant}, the resultant $R(u)=\Res_T(A(T),T^dA(u/T))$ has degree at most $d^2$ and height at most $2hd+\log((2d)!) \leq 2hd + 2d\log d$.  This resultant vanishes at the products $\alpha\beta$ of roots of $A$, and in particular at the square $|\alpha|^2=\alpha\overline\alpha$.  The Graeffe polynomial has degree $d$ and height at most $2h+\log(d+1)$, so the conclusion follows directly from Lemma~\ref{lemma:roots}(iii).
\end{proof}

\begin{corollary}
\label{cor:moduli}
With the same notation as Lemma~\ref{lemma:modsep}, given $A(T)$, the real positive roots $0<r_1\le\dots\le r_k$ of $A(T)$ and all roots of modulus exactly $r_1,\dots,r_k$ can be computed, with isolating regions of size~$2^{-\tilde{O}(hd^2)}$, in $\tilde{O}(hd^3)$ bit operations.
\end{corollary}
\begin{proof}
Let 
\[ b := 2(d^2-1)\log_2(d+1) + 2h(d^2-1) + 2d^2(h+\log_2d) + d\log_2\sqrt{d+1} = \tilde{O}(hd^2)\] 
and $\alpha$ be a root of $A$, so that Lemma~\ref{lemma:modsep} implies at least one of $\pm|\alpha|$ is a root of $A$ if and only if $\left|G\left(|\alpha|^2\right)\right| \leq 2^{-b}$.  If we know an approximation $a$ to $\alpha$ such that $|\alpha-a| < 2^{-(b+h+2)}$, then $|\overline{\alpha} - \overline{a}| < 2^{-(b+h+2)}$ and 
\[ \left| |\alpha|^2 - a \overline{a} \right| 
= \left| \alpha \overline{\alpha} - \alpha \overline{a} + \alpha \overline{a} - a \overline{a} \right| 
\leq |\alpha| 2^{-b-h-2} + |\overline{\alpha}|2^{-b-h-2} + 2^{-2b-2h-4} \leq 2^{-b},\]
as $|\alpha| \leq 2^h + 1$ by Lemma~\ref{lemma:roots}.  Lemma~\ref{lemma:feval} implies that knowing $|\alpha|^2$ to $\tilde{O}(hd^2)$ bits is sufficient to compute $G\left(|\alpha|^2\right)$ to accuracy $2^{-b}$.  Thus, knowing an approximation to $\alpha$ of accuracy $2^{-\tilde{O}(hd^2)}$ is sufficient to decide whether or not at least one of $\pm|\alpha|$ is a root of $A$, and to decide which real $\alpha$ are positive.  Furthermore, knowing $|\alpha|$ and $-|\alpha|$ to a precision higher than the root separation bound of Lemma~\ref{lemma:roots} allows one to decide when $|\alpha|$ is a root of $A$.  It is therefore sufficient to compute the roots of $A(T)$ to precision $2^{-(b+h+2)}$ in $\tilde{O}(hd^3)$ bit operations using Lemma~\ref{lemma:fsolve}, and evaluate $G(T)$ at the squares of the moduli of the roots in $\tilde{O}(hd^3)$ bit operations.
\end{proof}

In practice, one would first compute roots only at precision $\tilde{O}(hd)$, in $\tilde{O}(hd^2)$ bit operations, and then check whether any of the non-real roots has a modulus that could equal one of the real positive roots in view of its isolating interval. Only those roots need to be refined to higher precision before invoking Lemma~\ref{lemma:modsep}. 

\begin{corollary}
\label{cor:equalmoduli}
Given a numerical Kronecker representation $[P(u),\bQ,\bU]$ corresponding to a zero-dimensional system of $n$ polynomials of degrees at most $d$ and heights at most $h$, there exists an algorithm, {\sf EqualModuli}, which takes $[P(u),\bQ,\bU]$ and determines the real solutions $\br_1,\dots,\br_k$ of the system and a list of the solutions with the same coordinate-wise moduli in the variables $z_1,\dots,z_n$ in $\tilde{O}(hD^4)$ bit operations.  
\end{corollary}

\begin{proof}
By Lemma~\ref{lemma:coordpoly} the minimal polynomials $\Phi_1,\dots,\Phi_n$ of the $z_j$ coordinates have degree at most $d$, height $\tilde{O}(hD)$, and can all be computed in $\tilde{O}(hD^3)$ bit operations. The result then follows from Corollary~\ref{cor:moduli}.
\end{proof}

\subsection{Correctness and Complexity of the Main Algorithms}

With the algorithms developed above, we are almost ready to prove Theorems~\ref{thm:EffectiveCombAsm} and~\ref{thm:GenEffMinCrit}.  Before justifying our algorithm in the combinatorial case, we need one final result.

\begin{lemma}
\label{lemma:comb_two_min_crit}
If $F(\bz)$ is combinatorial, has a smooth singular variety, and admits a minimal point $\bw$ with non-negative coordinates which minimizes $|z_1 \cdots z_n|^{-1}$ on $\overline{\mD}$ then $\bw$ is a critical point.  If $F(\bz)$ is combinatorial, has a smooth singular variety, and admits two distinct minimal critical points $\ba$ and $\bb$ with non-negative coordinates then $F(\bz)$ admits an infinite number of minimal critical points.
\end{lemma}

\begin{proof}
Suppose $\bw$ is a minimal point with non-negative coordinates which minimizes $|z_1 \cdots z_n|^{-1}$ on $\overline{\mD}$.  By Lemma~\ref{lem:combCase}, when $F(\bz)$ is combinatorial $\bw$ is a local extremum of the smooth map $h(\bz) = -\log(z_1 \cdots z_n)$ from the manifold $\mV \cap \left(\mathbb{R}_{>0}\right)^n \subset \mathbb{R}^n$ to $\mathbb{R}$, meaning $\bw$ satisfies the smooth critical point equations.  Furthermore, as $\bw$ is a smooth minimal critical point the hyperplane with normal $\bone$ containing the point $\Relog(\bw)$ is a support hyperplane~\cite[Proposition 3.12]{PemantleWilson2008} to the convex set $\Relog(\mD)$. 

Thus, if $\ba$ and $\bb$ are two distinct minimal critical points with non-negative coordinates every point on the line segment $\{r\Relog(\ba) + (1-r)\Relog(\bb):r \in [0,1]\}$ between $\Relog(\ba)$ and $\Relog(\bb)$ is on $\partial \Relog(\mD)$ and is a minimum of the function $\bx \mapsto - \bone \cdot \bx$ on $\Relog(\overline{\mD})$. This implies every point $\bw = \left(a_1^rb_1^{1-r},\dots,a_n^rb_n^{1-r}\right)$ with $r \in [0,1]$ is a minimizer of $|z_1 \cdots z_n|^{-1}$ on $\overline{\mD}$, and thus a minimal critical point of $F(\bz)$.
\end{proof}

\begin{proof}[Proof of Theorem~\ref{thm:EffectiveCombAsm}]
Under our assumptions, Proposition~\ref{prop:smoothmincrit}, Proposition~\ref{prop:lineMin}, and Lemma~\ref{lemma:comb_two_min_crit} imply that the hypotheses of Proposition~\ref{prop:smoothAsmCP} are satisfied, meaning that Algorithm~\ref{alg:EffectiveCombAsm} correctly outputs rational functions $A(u),B(u),C(u)$ such that Equation~\eqref{eq:EffectiveCombAsm} is satisfied.  Note that the matrix $\tilde{\mH}$ is obtained from the matrix $\mH$ in Equation~\eqref{eq:Hess} by multiplying every element of $\mH$ by $\lambda$ (to obtain a polynomial matrix), so the determinant of $\mH$ is the determinant of $\tilde{\mH}$ divided by $\lambda^{n-1}$.

Given polynomials $H$ and $G$ of degrees at most $d$ and heights at most $h$, the Kronecker representation $[P,\bQ]$ in Algorithm~\ref{alg:EffectiveCombAsm} computed by partitioning the variables as $\bZ = [(\bz),(t),(\lambda)]$ contains polynomials of degrees in $\tilde{O}(dD)$ and heights in $\tilde{O}(hdD)$, and can be calculated in $\tilde{O}(hd^3D^3)$ bit operations using Proposition~\ref{prop:kronrep}.  Propositions~\ref{prop:numKron} and~\ref{prop:NumKronAlgs}, using the algorithms of Section~\ref{sec:NumKronAlg}, show how to compute a numerical Kronecker representation $[P,\bQ,\mathbf{V}]$ of $[P,\bQ]$, determine the elements of this representation with the same values of the variables $z_1,\dots,z_n$, and decide which have real and positive coordinates in $\tilde{O}(hd^3D^3)$ bit operations.  

For each solution of the numerical Kronecker representation, the root separation bounds in Lemma~\ref{lemma:roots} show that it is sufficient to know its value of $t$ to precision $2^{-\tilde{O}(hd^2D^2)}$ to determine when it is strictly between $0$ and $1$, and this can be accomplished for every solution in $\tilde{O}(hd^3D^3)$ bit operations using Lemma~\ref{lemma:EvalCoords}.  Once a minimal critical point is identified, all other points with the same modulus can be identified in $\tilde{O}(hd^5D^4)$ bit operations using\footnote{The minimal polynomials $\Phi_j$ will have degrees in $\tilde{O}(dD)$ and Lemma~\ref{lemma:coordpoly} implies that they will have heights in $\tilde{O}(hd^2D)$.} Corollary~\ref{cor:moduli}.  This is the most computationally expensive step we perform.

The entries of the matrix $\tilde{\mH}$ are polynomials of degree at most $d$ and heights at most $h + \log_2(d^2-d) + 2$, so a cofactor expansion shows that its determinant has degree at most $nd$ and height in $\tilde{O}(d^2hn)$.  Proposition~\ref{prop:KronReduce} then implies that the polynomial $Q_{\tilde{\mH}}$ has degree in $\tilde{O}(dD)$, height in $\tilde{O}(hd^3D)$, and can be determined in $\tilde{O}(hd^4D^3)$ bit operations.  By assumption the polynomial $G(\bz)$ has degree at most $d$ and height at most $h$, and the polynomial $T = z_1 \cdots z_n$ has degree $n$ and height 1.  Thus, Proposition~\ref{prop:KronReduce} also implies that the polynomials $Q_T$ and $Q_{-G}$ have degrees in $\tilde{O}(dD)$, heights in $\tilde{O}(hd^2D)$, and can be determined in $\tilde{O}(hd^3D^3)$ bit operations.  

With these degree and height bounds on the polynomials $P(u), P'(u), Q_{\tilde{\mH}}(u), Q_T(u),Q_{\lambda}(u)$ and $Q_{-G}(u)$, and the knowledge that $Q_{\tilde{\mH}},Q_T(u),$ and $Q_{\lambda}(u)$ are non-zero at the roots of $P(u)$, an argument analogous to the one presented in the proof of Lemma~\ref{lemma:EvalCoords} shows that to determine
\[ A(u) = \frac{Q_{-G}(u)}{Q_{\lambda}(u)^{n-1}}, \qquad B(u) = \frac{Q_{\lambda}(u)^{n-1}}{Q_{\tilde{\mH}}(u)} \cdot P'(u)^{2-n}, \qquad C(u) = \frac{P'(u)}{Q_T(u)} \] 
at all roots of $P(u)=0$ to $\kappa$ bits of precision requires $\tilde{O}(dD\kappa + hd^3D^3)$ bit operations (at least $\kappa=\tilde{O}(hd^2D^2)$ bit of precision are needed to isolate the values of these polynomials).
\end{proof}

The analysis for the general case is similar, except that the algebraic system under consideration is more complicated due to the replacement of the smooth critical point equations with Equations~\eqref{eq:GenSys1} -- \eqref{eq:GenSys6}.  As stated above, we partition the variables appearing in these equations as
\[ \bZ = \left[\bZ_1,\dots,\bZ_6\right] = \left[ (\ba,\bb),(\bx,\by),(\lambda_R),(\lambda_I),(t),(\nu) \right]. \]
Bounds on the multi-degree of each polynomial in this system, together with the quantities $\eta(f_j)$ appearing in Equation~\eqref{eq:etaHeight}, are given in Table~\ref{tab:genmultdeg}.  A tedious calculation\footnote{See the corresponding Maple worksheet on~\websiteurl.} shows that with the degrees and values of $\eta(f_j)$ given there one has
\[ \sC_{\bn}(\bd) = 2^{n-1}D^3dn^4, \qquad \sH_{\bn}(\bETA,\bd) = \tilde{O}\left(2^nD^3d(h+d)\right) \in \tilde{O}\left(hd^22^nD^3\right). \]

\begin{table}
\centering
\label{my-label}
\begin{tabular}{c|c|c}
Equation & Multi-degree & $\eta$ \\ \hline
\eqref{eq:GenSys1}	& $(d,0,0,0,0,0)$	&	$h + d + d\log(2n + 1)$   \\
\eqref{eq:GenSys2}	& $(d,0,1,0,0,0)$	&   $h + d + \log d + d\log(2n + 1) + 2$ \\
\eqref{eq:GenSys3}	& $(d,0,0,1,0,0)$	&   $h + d + \log d + d\log(2n + 1) + 2$ \\
\eqref{eq:GenSys4}	& $(0,d,0,0,0,0)$	&   $h + d + d\log(2n + 1)$ \\
\eqref{eq:GenSys5}	& $(2,2,0,0,1,0)$	&   $4\log(2n + 1) + 1$ \\
\eqref{eq:GenSys6}	& $(0,d,0,0,0,1)$	&   $h + d + \log d + d\log(2n + 1) + 2$
\end{tabular}
\caption[Values for the complexity bound of our minimal critical point algorithm]{The multi-degrees and values of $\eta(f_j)$ for Equations~\eqref{eq:GenSys1} -- \eqref{eq:GenSys6}, under the partition of variables $\bZ$.}
\label{tab:genmultdeg}
\end{table}

\begin{proof}[Proof of Theorem~\ref{thm:GenEffMinCrit}]
Proposition~\ref{prop:genmincrit} implies that Algorithm~\ref{alg:MinimalCritical} correctly identifies all minimal critical points. Furthermore, Proposition~\ref{prop:kronrep} implies that a Kronecker representation of Equations~\eqref{eq:GenSys1} -- \eqref{eq:GenSys6} can be determined in $\tilde{O}\left(d^32^{2n}D^7\right)$ bit operations, and contains polynomials of degrees at most $2^{n-1}D^3dn^4$ and heights in $\tilde{O}\left(hd^22^nD^3\right)$.

For each solution of the numerical Kronecker representation, the root separation bounds in Lemma~\ref{lemma:roots} show that it is sufficient to know its value of $t$ to precision $\tilde{O}\left(hd^32^{2n}D^6\right)$ bits to determine when it is strictly between $0$ and $1$, and this can be accomplished for every solution in $\tilde{O}\left(hd^42^{3n}D^9\right)$ bit operations using Lemma~\ref{lemma:EvalCoords}. As the coordinates $(\ba,\bb,\lambda_R,\lambda_I)$ satisfy the zero-dimensional polynomial system defined by Equations~\eqref{eq:GenSys1} -- \eqref{eq:GenSys3}, Lemma~\ref{lemma:coordpoly} implies that the elements of the numerical Kronecker representation which define equal critical points $\bz = \ba + i\bb$ can be determined by finding the values of these coordinates at all solutions to precision $2^{-\tilde{O}(hD^4)}$.
\end{proof}

\section{Additional Examples}
\label{sec:KronExamples}

We now discuss a few additional examples highlighting the above techniques. The calculations for these examples, together with a preliminary Maple implementation of our algorithms for the combinatorial case and automated examples using that implementation, can be found in accompanying Maple worksheets\footnote{The code for these examples is available at~\websiteurl.}. This preliminary implementation computes the Kronecker representation through Gröbner bases computations, meaning it does not run in the complexity stated above, and does not use certified numerical computations.

\begin{example}
\label{ex:Apery3b}
Consider the second Apéry sequence $(c_k)$ from Example~\ref{ex:Apery}, using (for varieties sake) the representation for the generating function $C(z)$ as the diagonal of the combinatorial rational function
\[ \frac{1}{1-x-y-z(1-x)(1-y)} = \frac{1}{1-x-y} \cdot \frac{1}{1-z} \cdot \frac{1}{1-\frac{xyz}{(1-x-y)(1-z)}}. \]  
An argument analogous to the one in Example~\ref{ex:Apery3a}, detailed in the accompanying Maple worksheet, shows that there are two critical points
\[ \bp = \left(\frac{3-\sqrt{5}}{2}, \frac{3-\sqrt{5}}{2}, \frac{-1+\sqrt{5}}{2} \right) \quad \text{and} \quad
\bs = \left(\frac{3+\sqrt{5}}{2}, \frac{3+\sqrt{5}}{2}, \frac{-1-\sqrt{5}}{2} \right), \]
of which $\bp$ is minimal.  Ultimately, one obtains 
\[ c_k = \left(\frac{2(5-u)}{11u-30}\right)^k \cdot k^{-1} \cdot \frac{(10-2u)(2u-5)}{\pi(4u-10)\sqrt{10(5u-14)(u-5)}}\left(1+O\left(\frac{1}{k}\right)\right), \]
where $u = 5-\sqrt{5}$ is a root of the polynomial $P(u) = u^2-7u+12$ which can be determined explicitly as $P$ is a quadratic, so 
\[ c_k = \frac{\left(\frac{11}{2}+\frac{5\sqrt{5}}{2}\right)^k}{k} \cdot \frac{\sqrt{250+110\sqrt{5}}}{20\pi}\left(1 + O\left(\frac{1}{k}\right)\right). \]
When combined with the BinomSums Maple package of Lairez\footnote{Available at \url{https://github.com/lairez/binomsums}.}, our preliminary implementation of the results in this chapter allows one to automatically go from the specification 
\[ c_k := \sum_{j=0}^k \binom{k}{j}^2\binom{k+j}{j} \]
to asymptotics of $c_k$, proving the main result of Hirschhorn~\cite{Hirschhorn2015}.  

By Proposition~\ref{prop:lineMin} there exists a singularity $\bw \in \mV$ and $t \in (0,1)$ such that the modulus of a coordinate of $\bw$ is $t$ times the modulus of the corresponding coordinate in $\bs$.  To determine such a point, or to prove minimality of $\bp$ when it is not known a priori that the rational function under consideration is combinatorial, one can compute a Kronecker representation of the system of equations~\eqref{eq:GenSys1}--\eqref{eq:GenSys6}.  In this case we can find several points $\bw$, including one with algebraic coordinates of degree 4 which is approximately $(x,y,z) \approx (.535, .535, -0.331)$, when $t \approx 0.194$.
\end{example}

\begin{example}
The rational function
\[F(x,y) = \frac{1}{(1-x-y)(20-x-40y)-1},\]
has a smooth denominator and is combinatorial as it can be written
\[ F(x,y) = \frac{1}{1-x-y} \cdot \frac{1}{20 - x - 4y - \frac{1}{1-x-y}}. \]
A Kronecker representation of the system 
\[ H(x,y), \quad x(\partial H/\partial x) - \lambda, \quad y(\partial H/\partial y) - \lambda, \quad  H(tx,ty), \]
using the linear form $u=x+y$ (which a Gröbner basis computation verifies separates the solutions of the system) shows that the system has 8 solutions, of which 4 have $t=1$ and correspond to critical points.  There are two critical points with positive coordinates: 
\[ (x_1,y_1)\approx (0.548, 0.309) \qquad \text{and} \qquad (x_2,y_2)\approx (9.997, 0.252).\] 
Since $x_1<x_2$ and $y_1>y_2$, it is not immediately clear which (if any) should be a minimal critical point. However, examining the full set of solutions, not just those where $t=1$, shows there is a point with approximate coordinates $(0.092x_2,0.092y_2)$ in $\mV$, so that $x_1$ is the minimal critical point.  To three decimal places the diagonal asymptotics have the form 
\[f_{k,k} = (5.884\ldots)^k k^{-1/2} \left(0.054\ldots + O\left(\frac{1}{k}\right)\right).  \]
\end{example}

%%%%%%%%%%%%%%%%%%%%%%%%%%%%%%%%%%%%%%%%%%%%%%%
% Section: Genericity Results
%%%%%%%%%%%%%%%%%%%%%%%%%%%%%%%%%%%%%%%%%%%%%%%
\section{Genericity Results}
\label{sec:generic}

In this section we show that the assumptions required for the algorithms in Section~\ref{sec:Algorithms} hold generically.  Given a collection of polynomials $f_1(\bz), \dots, f_r(\bz)$ of degrees at most $d_1,\dots,d_r$, respectively, we can write 
\[ f_j(\bz) = \sum_{|\bi| \leq d_j} c_{j,\bi}\bz^{\bi}\]  
for all $j=1,\dots,r$, where $|\bi| = i_1+\cdots+i_n$ for indices $\bi \in \mathbb{N}^n$.  Given a polynomial $P$ in the set of variables $\{ u_{j,\bi} : |\bi| \leq d, 1 \leq j \leq r\}$ we let $P(f_1,\dots,f_r)$ denote the evaluation of $P$ obtained by setting the variable $u_{j,\bi}$ equal to the coefficient $c_{j,\bi}$.

Our results on genericity will make use of multivariate resultants and discriminants, for which we refer to Cox, Little, and O'Shea~\cite{CoxLittleOShea2005} and Jouanolou~\cite{Jouanolou1991}.  For all positive integers $d_0,\dots,d_n$ the resultant defines an explicit polynomial $\Res = \Res_{d_0,\dots,d_n}\in \mathbb{Z}[u_{j,\bi}]$ such that $n+1$ homogeneous polynomials $f_0,\dots,f_n \in \mathbb{C}[z_0,\dots,z_n]$ of degrees $d_0,\dots,d_n$ share a non-zero solution in $\mathbb{C}^n$ if and only if $\Res(f_0,\dots,f_n)=0$.  %

In order to prove that there are generically a finite number of critical points, we will make use of the following result.

\begin{lemma}
\label{lem:hom_res}
Given polynomials $f_1,\dots,f_n \in \mathbb{C}[\bz]$ of degrees $d_1,\dots,d_n$, let $\of_1,\dots,\of_n$ be their homogeneous parts of degrees $d_1,\dots,d_n$.  Then $\Res(\of_1,\dots,\of_n) \neq 0$ if and only if the only common root of $\of_1,\dots,\of_n$ is zero, and when these conditions hold the set of common roots of $f_1,\dots,f_n$ is finite.
\end{lemma}

The idea behind this lemma is that homogenizing the system of equations $f_1=\cdots=f_n$ yields a projective variety, and the condition $\Res(\of_1,\dots,\of_n) \neq 0$ implies that this system has no solutions ``at infinity''.  But then this projective variety is isomorphic to the affine variety defined by $f_1 =\cdots = f_n$, and the only way for this to occur is for the variety to be a finite set of points. See Theorem 3.4 of Cox, Little, and O'Shea~\cite[Chapter 3]{CoxLittleOShea2005} and the following discussion for more information.

We now prove Propositions~\ref{prop:generic_assumptions} and~\ref{prop:genericEffectiveAssumption} on the genericity of our assumptions, and discuss Conjecture~\ref{conj:genericEffectiveAssumptions}

\subsection*{Assumptions in the Combinatorial Case}

\paragraph{(A1) $H$ and its partial derivatives do not simultaneously vanish}
Suppose $H$ has degree $d$ and let 
\[ E(z_0,z_1,\dots,z_n) = z_0^d H(z_1/z_0,\dots,z_n/z_0) \]
be the homogenization of $H$.  As
\[ \partial E/\partial z_j = z_0^{d-1} (\partial H/\partial z_j)(z_1/z_0,\dots,z_n/z_0)  \]
for $j=1,\dots,n$, and Euler's relationship for homogeneous polynomials states
\[ d \cdot E(z_0,\dots,z_n) = \sum_{j=0}^n (\partial E/\partial z_j)(z_0,\dots,z_n), \]
$H$ and its partial derivatives vanish at some point $(p_1,\dots,p_n)$ only if the system 
\begin{equation} \partial E/\partial z_0 = \cdots = \partial E/\partial z_n=0 \label{eq:Eres} \end{equation}
admits the non-zero solution $(1,p_1,\dots,p_n)$.  Thus, assumption (A1) holds unless the multivariate resultant $P_d$ of the polynomials in Equation~\eqref{eq:Eres}, which depends only on the degree $d$, is zero when evaluated at the coefficients of $H$.  

It remains to show that $P_d$ is not identically zero for any $d$. If $H_d(\bz) = 1-z_1^d - \cdots -z_n^d$ then Equation~\eqref{eq:Eres} becomes
\[ z_0^{d-1} = -z_1^{d-1} = \cdots = -z_n^{d-1} = 0, \]
which has only the zero solution.  This implies the multivariate resultant $P_d$ is non-zero when evaluated at the coefficients of $H_d$, so it is a non-zero polynomial.

\paragraph{(A2) $G(\bz)$ is generically non-zero at any minimal critical point}
We prove the stronger statement that $G(\bz)$ is generically non-zero at any critical point.  Homogenizing the polynomials
\[ H(\bz), \quad G(\bz), \quad z_1(\partial H/\partial z_1) - z_2(\partial H/\partial z_2), \dots, 
z_1(\partial H/\partial z_1) - z_n(\partial H/\partial z_n) \]
gives a system of $n+1$ homogeneous polynomials in $n+1$ variables\footnote{As in all arguments using the multivariate resultant in this section, $G$ and $H$ are considered as dense polynomials of the specified degrees whose coefficients are indeterminates.}.  Applying the multivariate resultant to this system gives a polynomial $P_{d_1,d_2}(G,H)$ in the coefficients of $H$ and $G$, depending only on the degrees $d_1$ and $d_2$ of $G$ and $H$, which must be zero whenever $G(\bz)$ vanishes at a critical point.  

It remains to show that $P_{d_1,d_2}$ is non-zero for all $d_1,d_2 \in \mathbb{N}^*$, which we do by showing it is non-zero for an explicit family of polynomials of all degrees.  If 
\[ G(\bz) = z_1^{d_1} \quad \text{ and } \quad H(\bz) = 1-z_1^{d_2} - \cdots - z_n^{d_2} \]
then the system of homogeneous polynomial equations
\begin{align*}
u^{d_2}H\left(z_1/u^{d_2},\dots,z_n/u^{d_2}\right) = u^{d_2}-z_1^{d_2} - \cdots - z_n^{d_2} &= 0 \\
G = z_1^{d_1} &= 0\\
-d_2z_1^{d_2} - d_2z_j^{d_2} &= 0, \qquad j=2,\dots,n
\end{align*}
has only the trivial solution $(u,z_1,\dots,z_n) = \bzer$. This implies the multivariate resultant $P_{d_1,d_2}$ is non-zero when evaluated on the coefficients of the polynomials $G$ and $H$ given here, so it is a non-zero polynomial.

\paragraph{(A3) Generically, all minimal critical points are nondegenerate}
We prove the stronger statement that the matrix $\mH$ in Equation~\eqref{eq:Hess} is generically non-singular at every critical point (here we let $\zeta_j$ in Equation~\eqref{eq:Hess} be the variable $z_j$, which will be eliminated from the critical point equations).  After multiplying every entry of $\mH$ by $\lambda = z_1(\partial H/\partial z_1)$, which is non-zero at any minimal critical point, we obtain a polynomial matrix $\tilde{\mH}$ whose determinant vanishes if and only if an explicit polynomial $D$ in the variables $\bz$ and the coefficients of $H$ vanishes. After homogenizing the system of $n+1$ equations consisting of $D=0$ and the critical point equations~\eqref{eq:critpt} we can apply the multivariate resultant to determine an integer polynomial $P_d$ in the coefficients of $H$, depending only on the degree $d$ of $H$, which must be zero at any degenerate critical point.

It remains to show that the polynomial $P_d$ is non-zero for all $d \in \mathbb{N}^*$. Fix a non-negative integer $d$ and consider the polynomial $H(\bz) = 1 - z_1^d - \cdots - z_n^d$.  Calculating the quantities in Equation~\eqref{eq:Hess}, and substituting $z_j^d=z_1^d$ for each $j=2,\dots,n$, shows that $\tilde{\mH}$ is the polynomial matrix with entries of value $a:=-d^2z_1^d$ on its main diagonal and entries of value $b := -2d^2z_1^d$ off the main diagonal.  Such a matrix has determinant 
\[ D  = a^{n-1}(a+(n-1)b) = (-z_1^dd^2)^n(1-2n) \]
so the only solution to the homogenized smooth critical point equations and $D$ is the trivial zero solution.  This implies that the polynomial $P_d$ is non-zero when given $H$, and thus it is a non-zero polynomial.
\vspace{-0.1in}

\paragraph{(J1) The Jacobian of the smooth critical point equations is generically non-singular at the critical points}

The Jacobian of the system
\[ \bbf := \left( H, \quad z_1(\partial H/\partial z_1)-\lambda, \quad \dots \quad ,z_n(\partial H/\partial z_n)-\lambda, \quad H(tz_1,\dots,tz_n)\right) \]
with respect to the variables $\bz,\lambda,$ and $t$ is a square matrix which is non-singular at its solutions if and only if its determinant $D(\bz,t)$ (which is independent of $\lambda$) is non-zero at its solutions.  Any solution of $\bbf$ has $t \neq 0$, so the existence of a solution to $\bbf=D=0$ gives the existence of a non-zero solution to the system obtained by homogenizing the polynomials $H, z_1(\partial H/\partial z_1) - z_j(\partial H/\partial z_j),t^dH(\bz/t),$ and $t^{d-1}D(\bz,1/t),$ where $d$ is the degree of $H$ (note $D$ has degree $d-1$ in $t$).  The multivariate resultant of this system is an integer polynomial $P_d$, depending only on the degree $d$ of $H$, which must vanish if the Jacobian is singular at at least one of its solutions.
\vspace{0.1in}

It remains to show that the polynomial $P_d$ is non-zero for all $d \in \mathbb{N}^*$. Fix a non-negative integer $d$ and consider the polynomial $H(\bz) = 1 - z_1^d - \cdots - z_n^d$.  The Jacobian of $\bbf$ is the matrix
{\small\[
J := 
\begin{pmatrix}
-dz_1^{d-1} & \cdots & -dz_n^{d-1} & 0 & 0 \\
-d^2z_1^{d-1} & \bzer & 0 & 1 & 0 \\
0 & \ddots & \bzer & \vdots & \vdots \\ 
0 & \bzer & -d^2 z_n^{d-1} & 1 & 0 \\ 
-dt^dz_1^{d-1} & \cdots & -dt^d z_n^{d-1} & 0 & -dt^{d-1}(z_1^d + \cdots + z_n^d)
\end{pmatrix},
\]}

\noindent
and a short calculation shows
%factoring out common terms from each column, together with a co-factor expansion along the final column, implies
$D = \det J = (z_1 \cdots z_n t)^{d-1}(z_1^d + \cdots + z_n^d)(-d)^{n+1} \cdot \det M,$
where $M$ is the $(n+1)\times(n+1)$ matrix 
{\small \[
M := 
\begin{pmatrix}
1 & \cdots & 1 & 0 \\
d & \bzer & 0 & 1\\
0 & \bzer & 0 & \vdots\\
0 & \cdots & d & 1
\end{pmatrix}.
\]}

\noindent
The matrix $M$ is invertible, so $\det M$ is a non-zero constant. The system of homogeneous equations under consideration thus simplifies to
\[ u^d-z_1^d-\cdots-z_n^d = -(z_1^d-z_j^d) = t^d - z_1^d-\cdots-z_n^d = (z_1 \cdots z_n)^{d-1}(z_1^d + \cdots + z_n^d) = 0, \]
which has only the trivial zero solution.  This implies that the polynomial $P_d$ is non-zero when given $H$, and is thus a non-zero polynomial.

\subsection*{Assumptions in the General Case}
The arguments in the general case are similar, except the systems are more unwieldy.

\paragraph{(A4) The system of equations~\eqref{eq:GenSys1}--\eqref{eq:GenSys3} generically has a finite number of complex solutions}
By Lemma~\ref{lem:hom_res}, Equations~\eqref{eq:GenSys1}--\eqref{eq:GenSys3} have a finite number of solutions unless the multivariate resultant of the leading homogeneous terms of the polynomials in this system is zero.  This multivariate resultant $P_d$ is an integer polynomial in the coefficients of $H$ depending only on the degree $d$ of $H$.

It remains to show that $P_d$ is a non-zero polynomial for all $d \in \mathbb{N}^*$.   Fix a non-negative integer $d$ and consider the polynomial $H(\bz) = 1 - z_1^d - \cdots - z_n^d$.  Then
\begin{align}
H^{(R)}(\ba,\bb) &= 1 - \sum_{l \geq 0}\binom{d}{2l}(-1)^l\left[a_1^{d-2l}b_1^{2l} + \cdots + a_n^{d-2l}b_n^{2l} \right] \label{eq:gengen1} \\
H^{(I)}(\ba,\bb) &=  -\sum_{l \geq 0}\binom{d}{2l+1}(-1)^l\left[a_1^{d-2l-1}b_1^{2l+1} + \cdots + a_n^{d-2l-1}b_n^{2l+1} \right] \label{eq:gengen2}
\end{align}
and
\begin{align}
a_j \left(\partial H^{(R)}/\partial x_j\right)(\ba,\bb) + b_j \left(\partial H^{(R)}/\partial y_j\right)(\ba,\bb) -\lambda_R &=  
- d \sum_{l \geq 0}\binom{d}{2l}(-1)^l a_j^{d-2l}b_j^{2l} -\lambda_R \label{eq:gengen3} \\
a_j \left(\partial H^{(I)}/\partial x_j\right)(\ba,\bb) + b_j \left(\partial H^{(I)}/\partial y_j\right)(\ba,\bb) -\lambda_I &= 
- d \sum_{l \geq 0}\binom{d}{2l+1}(-1)^l a_j^{d-2l-1}b_j^{2l+1} -\lambda_I, \label{eq:gengen4}
\end{align}
where the binomial coefficient $\binom{p}{q}$ is zero when $q>p$.  As
\[ \left(\sum_{l \geq 0}\binom{d}{2l}(-1)^l a_j^{d-2l}b_j^{2l}\right) + i \left(\sum_{l \geq 0}\binom{d}{2l+1}(-1)^l a_j^{d-2l-1}b_j^{2l+1}\right) = (a_j+ib_j)^d \]
for all $a_j,b_j \in \mathbb{C}$, when the leading homogeneous terms of the polynomials in Equations~\eqref{eq:gengen3} and~\eqref{eq:gengen4} are zero for all $j=1,\dots,n$ then
\[ (a_1+ib_1)^d = \cdots = (a_n+ib_n)^d = 0, \]
so that $a_j+ib_j=0$.  When the leading homogeneous terms of the right-hand side of Equation~\eqref{eq:gengen3} vanishes, substituting $b_j = ia_j$ gives
\[ 0 = - d \sum_{l \geq 0}\binom{d}{2l}a_j^d = -d2^{d-1}a_j^d. \]
Thus, the only solution of the leading homogeneous terms of the polynomials in Equations~\eqref{eq:GenSys1}--\eqref{eq:GenSys3} is the trivial zero solution.  This implies that the polynomial $P_d$ is non-zero when evaluated at the coefficients of $H$, so it is a non-zero polynomial.

\paragraph{Assumptions (A5) and (J2)}

We now discuss work in progress on proving Conjecture~\ref{conj:genericEffectiveAssumptions}.  The difficulty, compared to our other assumptions, is that Equation~\eqref{eq:GenSys5} reads
\[ x_j^2+y_j^2 - t(a_j^2+b_j^2) = 0 \]
for all $j=1,\dots,n$, which is not ``generic enough'' for the multivariate resultant used above.  In particular, homogenizing these equations gives equations of the form
\[ u(x_j^2+y_j^2) - t(a_j^2+b_j^2) = 0 \]
which have non-zero solutions $(u,x_j,y_j,a_j,b_j,t)=(0,0,0,0,0,t)$ with $t$ free.  
%It should be possible to work around this difficulty by using a more nuanced version of the multivariate resultant known as the \emph{mixed sparse resultant}, which is discussed in Cox, Little, and O'Shea~\cite[Chapter 6]{CoxLittleOShea2005}. 

Suppose we want to prove that assumption (J2) holds generically (in fact, assumption (A5) follows from (J2)).  Fix a dimension $d$ and let $H(\bz)$ be a dense degree $d$ polynomial whose coefficients are parametrized by the variables $c_{\bi}$.  We can consider the polynomials in Equations~\eqref{eq:GenSys1}--\eqref{eq:GenSys6}, together with the determinant of the Jacobian matrix of this system, as polynomials in both the $c_{\bi}$ and $\ba,\bb,\bx,\by,\lambda_R,\lambda_I,t,\nu$.  Let $Z$ be the projection of the algebraic set defined by these equations onto the coefficient variables $c_{\bi}$.  To prove that (J2) holds generically, we want to show that the Zariski closure $\overline{Z}$ of $Z$ is a proper subset of $\mathbb{C}^{m_d}$, where $m_d$ is the number of monomials in $\mathbb{C}[\bz]$ of degree at most $d$.
\smallskip

Consider again the polynomial $H(\bz) = 1 - z_1^d - \cdots - z_n^d$.  Let $R(p,q)$ and $I(p,q)$ to be the real and imaginary parts of the expression $(p+iq)^{d-1}$ when $p$ and $q$ are treated as real variables.  For $j=1,\dots,n$ we define
\[ R_j := R(a_j,b_j), \quad I_j := I(a_j,b_j), \quad R_j' := R(x_j,y_j), \quad I_j' := I(x_j,y_j).\]  
For this choice of $H$, basic computations show that equations~\eqref{eq:GenSys1}--\eqref{eq:GenSys6} become
\begin{align}
1 - \sum_{j=1}^n\left( a_jR_j-b_jI_j\right) = -\sum_{j=1}^n\left(a_jI_j + b_jR_j\right) &= 0 \label{eq:bigex1} \\
-d(a_jR_j-b_jI_j) - \lambda_R = -d(a_jI_j+b_jR_j) - \lambda_I &= 0  \label{eq:bigex2}\\
1 - \sum_{j=1}^n\left( x_jR_j'-y_jI_j'\right) = -\sum_{j=1}^n\left(x_jI_j' + y_jR_j'\right) &= 0\label{eq:bigex3} \\
x_j^2+y_j^2 - t(a_j^2+b_j^2) &= 0 \label{eq:bigex4} \\
\nu\left(x_jR_j'-y_jI_j'\right) - \left(y_jR_j' + x_jI_j'\right) &= 0, \label{eq:bigex5}
\end{align}
and straightforward algebraic manipulations show that the Jacobian matrix of Equations~\eqref{eq:bigex1}--\eqref{eq:bigex5} has non-zero determinant at all solutions of the system.  

Thus, for every natural number $d$ the point in $\mathbb{C}^{m_d}$ corresponding to the coefficients of our choice of $H$ does not lie in the projection $Z$. Unfortunately, it is not clear whether or not it lies in the Zariski closure $\overline{Z}$.  For previous assumptions, to prove that $\overline{Z}$ was proper it was sufficient to exhibit for each degree a polynomial such that the systems under consideration had no solutions in projective space after they were homogenized.  Although justified above by the multivariate resultant, this is related to the fact that a projective variety is a proper variety (so the image of a projection onto the coordinate variables, after homogenizing and considering solutions in projective space, is an algebraic variety).  

For the polynomial systems we now consider, this is not possible, so we must search for different ways of ``homogenizing'' Equations~\eqref{eq:bigex1}--\eqref{eq:bigex5} and look for solutions over products of projective varieties.  This is ongoing work.  Computationally, this is related to a more nuanced version of the multivariate resultant known as the \emph{mixed sparse resultant}, which is also discussed in Cox, Little, and O'Shea~\cite[Chapter 6]{CoxLittleOShea2005}.

\part[Non-Smooth ACSV and Applications to Lattice Paths]{Non-Smooth Analytic Combinatorics in Several Variables and Applications to Lattice Paths}
\label{part:NonSmoothACSV}

%%%%%%%%%%%%%%%%%%%%%%%%%%%
% Chapter 9
%%%%%%%%%%%%%%%%%%%%%%%%%%%
\chapter[The Theory of ACSV for Multiple Points]{The Theory of Analytic Combinatorics in Several Variables for Multiple Points}
\label{ch:NonSmoothACSV}

\setlength{\epigraphwidth}{3.2in}
\epigraph{But whatever happens to you, you have to keep a slightly comic attitude. In the final analysis, you have got not to forget to laugh.}{Katharine Hepburn}

We now return to the theory of analytic combinatorics in several variables, detailing a larger class of functions for which diagonal asymptotics can be determined.  In particular, we relax our previous assumption that the singular variety $\mV$ of the rational function $F(\bz)=G(\bz)/H(\bz)$ is a manifold (near its minimal critical points).  For any collection of complex-valued functions $P_1(\bz),\dots,P_r(\bz)$, let $\mV(P_1,\dots,P_r) \subset \mathbb{C}^n$ denote their common set of complex solutions.  \glsadd{V}

We begin, as in Chapter~\ref{ch:SmoothACSV}, with an extended example which will illustrate how the general theory proceeds.

\section{A Non-Smooth Rational Diagonal}

Consider the rational function
\[  F(x,y,z) = \frac{1}{(1-3x-y-z)(1-x-3y-z)}. \]
If we define  
\[ H_1(x,y,z) = 1-3x-y-z \qquad \text{and} \qquad H_2(x,y,z) = 1-x-3y-z \]
then the singular variety $\mV=\mV(H)$ is the union $\mV=\mV(H_1) \cup \mV(H_2)$.  The points in the constructible\footnote{A constructible set is one of the form $\cup_{i=1}^m (W_i \setminus Z_i)$, where each $W_i$ and $Z_i$ are algebraic varieties and $m$ is a natural number.} sets
\[ \mV_1 := \mV(H_1) \setminus \mV(H_2) \qquad \text{and} \qquad \mV_2 := \mV(H_2) \setminus \mV(H_1) \]
are those where the singular variety is smooth\footnote{The singular variety $\mV$ is locally smooth at points of $\mV_1$ and $\mV_2$ as these sets are planes with a line removed.  Smoothness of the points is also verified by the fact that $H_1$ and its partial derivatives don't simultaneously vanish, and the same holds for $H_2$ and its partial derivatives.}, while the points in the algebraic set
\[ \mV_{1,2} = \mV(H_1,H_2) \]
are the singular points of $\mV$.  Since the Jacobian matrix
\[ \begin{pmatrix} \nabla H_1 \\  \nabla H_2 \end{pmatrix} = \begin{pmatrix} -3 & -1 & -1 \\ -1 & -3 & -1 \end{pmatrix} \]
is not rank deficient at any point in $\mV_{1,2}$, the set $\mV_{1,2}$ forms a complex manifold and we have partitioned $\mV = \mV_1 \cup \mV_2 \cup \mV_{1,2}$ into a disjoint union of (constructible) complex manifolds.

\subsection*{Step 1: Determine Minimal Critical Points}
Regardless of the geometry of $\mV$, the Cauchy integral representation 
\[ f_{k,k,k} = \frac{1}{(2\pi i)^3} \int_{T(p,q,r)} F(x,y,z) \frac{dx\, dy\, dz}{x^{k+1}y^{k+1}z^{k+1}} \]
for $(p,q,r)$ in the domain of convergence $\mD$ implies that any minimal point $(a,b,c) \in \mV \cap \partial \mD$ gives an upper bound $\rho \leq |abc|^{-1}$ on the exponential growth $\rho$ of the diagonal sequence.  Thus, as in the smooth case, we search for local extrema of the function $g(x,y,z) = |xyz|^{-1}$ on $\mV^* = \mV \cap \left(\mathbb{C}^*\right)^n$, after which we will determine whether any are minimal points.  If $(a,b,c) \in \mV_1$ is a local minimizer of $g$, then it must satisfy the smooth critical point equations
\[ H_1(a,b,c) = 0, \qquad a \cdot (\partial H_1/\partial x)(a,b,c) = b \cdot (\partial H_1/\partial y)(a,b,c)  =  c \cdot (\partial H_1/\partial z)(a,b,c),  \]
since it would be a critical point of the restricted map $\phi|_{\mV_1}:\mV_1 \rightarrow \mathbb{C}$, where we recall that $\phi(x,y,z) = xyz$.  Here we obtain one solution, $\bs_1 = (1/9,1/3,1/3)$, but it is not minimal as each coordinate of 
\[ \left(\frac{1}{13},\frac{3}{13},\frac{3}{13}\right) \in \mV_2 \subset \mV\] 
has smaller modulus than the corresponding coordinate of $\bs_1$.  Similarly, $\mV_2$ contains a single smooth critical point $\bs_2 = (1/3,1/9,1/3)$ but it is not minimal as 
\[ \left(\frac{3}{13},\frac{1}{13},\frac{3}{13}\right) \in \mV_1 \subset \mV.\] 

Thus, any local minimizer of $|xyz|^{-1}$ on $\mV^* \cap \overline{\mD}$ must be an element of $\mV_{1,2}$.  Using an argument analogous to the one in Example~\ref{ex:binDomainConverge}, one can determine the set of minimal points and, since $\mV_{1,2}$ can be parametrized as 
\[ \mV_{1,2} = \left\{ \left(\frac{1-z}{4}, \frac{1-z}{4}, z\right) : z \in \mathbb{C} \right\},\] 
show that the only local minimizer of $|xyz|^{-1}$ on $\mV^* \cap \overline{\mD}$ occurs at the point 
\[ \bw = \left(\frac{1}{6},\frac{1}{6},\frac{1}{3}\right),\] 
when $z=1/3$.  

Alternatively, since $\mV_{1,2}$ is itself a complex manifold we can examine the critical points of the restricted map $\phi|_{\mV_{1,2}}:\mV_{1,2} \rightarrow \mathbb{C}$.  Analogously to the smooth case, this set of points contains all local extrema of $|xyz|^{-1}$ on $\mV_{1,2}$.  As $\mV_{1,2}$ is defined by the vanishing of the irreducible polynomials $H_1$ and $H_2$, the critical points of $\phi$ are precisely the points of $\mV_{1,2}$ where the matrix
\[ M = \begin{pmatrix} 
\nabla H_1  \\ 
\nabla H_2 \\
\nabla \phi 
\end{pmatrix}
= 
\begin{pmatrix}
-3 & -1 & -1 \\ -1 & -3 & -1 \\ yz & xz & yz
\end{pmatrix}
\]
is rank deficient\footnote{A critical point of $\phi|_{\mV_{1,2}}$ is a point where its differential is zero.  As $\mV_{1,2}$ is defined by $H_1=H_2=0$, this is equivalent to the gradient $\nabla \phi$ being in the span of the gradients $\nabla H_1$ and $\nabla H_2$.  Since $\nabla H_1$ and $\nabla H_2$ are linearly independent, this occurs if and only if $M$ is rank deficient.}.  This occurs when $\det(M) = 8xy-2xz-2yz$ vanishes, and solving the system $H_1=H_2=\det(M)=0$ gives the point $\bw$.

Not only is $\bw$ minimal, it is strictly minimal.  If $H_1(x,y,z) = 0$ and $(x,y,z) \in D(\bw)$, then 
\[ 1/3 \geq |z| = |1-3x-y| \geq 1 - |3x+y|, \]
while $|x|,|y| \leq 1/6$.  Since $|3x+y| \geq 2/3$, the complex triangle inequality implies that $x$ and $y$ have the same argument, and the condition $1/3 \geq |1-3x-y|$ forces $(x,y) = (1/6,1/6)$.  Solving $H_1(1/6,1/6,z)=0$ then gives $(x,y,z)=\bp$.  A similar argument applies to points satisfying $H_1(x,y,z) = 0$ and $(x,y,z) \in D(\bw)$, proving strict minimality.  As already seen in the smooth case, the triangle inequality is a useful tool for minimality arguments.

\subsection*{Step 2: Compute a Residue}

We have found the strictly minimal point $\bw$ giving an upper bound of $6^2\cdot 3=108$ on the exponential growth. Next, we attempt to asymptotically approximate the sequence $(f_{k,k,k})_{k \geq 0}$ by an integral whose domain lies near $\bw$. For a sufficiently small neighbourhood $\mN \subset \{|x|=1/6,|y|=1/6\}$ of $(1/6,1/6)$ and sufficiently small $\epsilon>0$, the integral
\begin{equation} 
\chi := \frac{-1}{(2\pi i)^3} \int_\mN \left(\int_{|z|=1/3+\epsilon} F(\bz) \cdot \frac{dz}{z^{k+1}} - \int_{|z|=1/3-\epsilon} F(\bz) \cdot \frac{dz}{z^{k+1}} \right) \frac{dx \, dy}{x^{k+1}y^{k+1}} 
\label{eq:chimpex} \end{equation}
exists as its integrand is bounded and analytic.  Arguing analogously to the smooth case presented in Chapter~\ref{ch:SmoothACSV}, it can be shown that for $\mN$ and $\epsilon$ small enough
\begin{equation} |f_{k,k,k} - \chi| =  O\left( \delta^k \right), \label{eq:mpexbound} \end{equation}
for some $\delta \in [0,108)$. Furthermore, if $\mN$ is sufficiently small then for each $(x,y) \in \mN$ there exist two poles of $F(x,y,z)$, when $z=1-3x-y$ and when $z=1-x-3y$.  If $x \neq y$, one can compute the residues
\begin{align*} 
\text{Res}\left(\frac{z^{-(k+1)}}{(1-3x-y-z)(1-x-3y-z)}; z=1-3x-y \right) &= \frac{-(1-3x-y)^{-(k+1)}}{2(x-y)} \\
\text{Res}\left(\frac{z^{-(k+1)}}{(1-3x-y-z)(1-x-3y-z)}; z=1-x-3y \right) &= \frac{(1-x-3y)^{-(k+1)}}{2(x-y)},
\end{align*}
so that Cauchy's residue theorem applied to the inner difference of integrals in Equation~\eqref{eq:chimpex} gives
\begin{equation} \chi = \frac{1}{(2\pi i)^2} \int_\mN \left(\frac{(1-3x-y)^{-(k+1)}}{2(x-y)} - \frac{(1-x-3y)^{-(k+1)}}{2(x-y)} \right) \frac{dx \, dy}{x^{k+1}y^{k+1}}. \label{eq:chimpexsum} \end{equation}
Although the integrand of this expression appears to be undefined when $x=y$, which occurs for points in any neighbourhood of $(1/6,1/6)$, the integral is well defined as 
\[  \frac{(1-3x-y)^{-(k+1)}}{2(x-y)} - \frac{(1-x-3y)^{-(k+1)}}{2(x-y)}  = \sum_{j=0}^k (1-x-3y)^{-j-1}(1-3x-y)^{j-k-1}. \]
Even if one could not compute such a representation of the integrand, the integral in Equation~\eqref{eq:chimpexsum} must be bounded for all $k$ because of the relationship in Equation~\eqref{eq:mpexbound}.  

Our aim is to convert the expression in Equation~\eqref{eq:chimpexsum} into a finite sum of Fourier-Laplace integrals which can be asymptotically approximated as $k\rightarrow\infty$.  Note that we cannot simply distribute the integral expression in Equation~\eqref{eq:chimpexsum} over the two summands of the integrand, as the integrals
\[ \int_\mN \frac{(1-x-3y)^{-(k+1)}}{x-y}\frac{dx \, dy}{x^{k+1}y^{k+1}} \qquad \text{and} \qquad \int_{\mN} \frac{(1-3x-y)^{-(k+1)}}{x-y} \frac{dx \, dy}{x^{k+1}y^{k+1}}\]
are not well defined for any neighbourhood $\mN$ of $(1/6,1/6)$.  Thus, following the work of Pemantle and Wilson~\cite{PemantleWilson2004}, we take a less direct approach.

\subsection*{Step 3: Introduce a New Variable and Obtain a Fourier-Laplace Integral}

Suppose $x$ and $y$ are fixed such that $x \neq y$ and $(1-3x-y)(1-x-3y) \neq 0$.  A direct computation verifies\footnote{This integral representation is a special case of DeVore and Lorentz~\cite[Equation 7.12]{DeVoreLorentz1993}.} that
{\small
\[ \frac{(1-3x-y)^{-(k+1)}}{2(x-y)} - \frac{(1-x-3y)^{-(k+1)}}{2(x-y)} 
= \frac{k+1}{(1-3x-y)(1-x-3y)}\int_0^1\left(\frac{t}{1-3x-y} + \frac{1-t}{1-x-3y}\right)^k dt \] }
for all natural numbers $k$, and substitution into Equation~\eqref{eq:chimpexsum} yields
\[ \chi = \frac{k+1}{(2\pi i)^2} \int_{\mN \times [0,1]} \frac{1}{(1-3x-y)(1-x-3y)}\left(\frac{t}{1-3x-y} + \frac{1-t}{1-x-3y}\right)^k \frac{dx \, dy \, dt}{x^{k+1}y^{k+1}}. \]
Making the substitutions $x=(1/6)e^{i\theta_1}$,  $y=(1/6)e^{i\theta_2}$, and $\tau=t-1/2$ then gives
\begin{equation} \chi = 108^k \cdot \frac{k+1}{(2\pi)^2} \int_{\mN' \times [-1/2,1/2]} A(\theta_1,\theta_2)e^{-k\phi(\theta_1,\theta_2,\tau)} d\theta_1\,d\theta_2\,d\tau, \label{eq:mpexFL} \end{equation}
where
\begin{align*}
\phi(\theta_1,\theta_2,\tau) &= i(\theta_1 + \theta_2) - \log\left[\frac{1}{3}\left(\frac{1/2+\tau}{1-e^{i\theta_1}/2-e^{i\theta_2}/6} + \frac{1/2-\tau}{1-e^{i\theta_1}/6-e^{i\theta_2}/2} \right)\right] \\[+2mm]
A(\theta_1,\theta_2) & = \frac{1}{\left(1-e^{i\theta_1}/2-e^{i\theta_2}/6\right)\left(1-e^{i\theta_1}/6-e^{i\theta_2}/2\right)}.
\end{align*}
The domain of integration $\mN'$ can be taken to be any sufficiently small neighbourhood of the origin without affecting the asymptotic statement in Equation~\eqref{eq:mpexbound}.

\subsection*{Step 4: Determine Asymptotics}

The final step is to use the Fourier-Laplace expression in Equation~\eqref{eq:mpexFL} to determine asymptotics of the diagonal sequence.  Note that for $\mN'$ sufficiently small
\begin{itemize}
	\item $\phi(0,0,0)=0$ and $(\nabla \phi)(0,0,0)=(0,0,0)$;
	\item the origin is the only point of $\mN' \times [-1/2,1/2]$ where $\nabla \phi$ is 0;
	\item the Hessian matrix 
		\[ \mH = \begin{pmatrix} 5/2 & 1/2 & -i \\  1/2 & 5/2 & i \\ -i & i & 0 \end{pmatrix} \] 
		  of $\phi$ at the origin is non-singular (it has determinant 6);
	\item the real part of $\phi(\bt)$ is non-negative on $\mN' \times [-1/2,1/2]$.
\end{itemize}

These conditions imply that Proposition~\ref{prop:HighAsm} can be used to determine asymptotics of $\chi$ from Equation~\eqref{eq:mpexFL}.  Computing second order asymptotics in such a manner, combined with the relationship between the diagonal and $\chi$ in Equation~\eqref{eq:mpexbound}, gives
\[ f_{k,k,k} = 108^k \cdot k^{-1/2} \cdot \left( \frac{3\sqrt{3}}{2\sqrt{\pi}} - \frac{19\sqrt{3}}{48\sqrt{\pi}k} + O\left(\frac{1}{k^2}\right) \right). \]

We now show that much of this analysis can be applied to a large class of rational functions, generalizing the results in Chapter~\ref{ch:SmoothACSV} on rational functions with smooth minimal critical points.

\section{The Transverse Multiple Point Case} 

Let $F(\bz)=G(\bz)/H(\bz)$ be any rational function which is analytic at the origin, and $\mV$ be its singular variety. Given a point $\bw \in \mV$, we say that $\bw$ is a \emph{multiple point} of $\mV$ if any sufficiently small neighbourhood of $\bw$ in $\mV$ can be written as the union of a finite collection of manifolds $\mV_1,\dots,\mV_r$, each containing $\bw$.  We call $\bw \in \mV$ a \emph{transverse multiple point} if it is a multiple point and the tangent planes of $\mV_1,\dots,\mV_r$ at $\bw$ are linearly independent (equivalently, the normals to these tangent planes are linearly independent).  A collection of manifolds $\mV_1,\dots,\mV_r$ is said to \emph{intersect transversely} at a point $\bw \in \mV_1 \cap \cdots \cap \mV_r$ if the tangent planes of the $\mV_j$ at $\bw$ are linearly independent.  

When $\bw$ is a transverse multiple point and each $\mV_i$ is a hypersurface, the intersection $\mV_1 \cap \cdots \cap \mV_r$ defines a manifold of complex dimension $n-r$ in a neighbourhood of $\bw$, as in the above example.  Any point where $\mV$ is locally smooth is, by definition, a transverse multiple point.

Multiple points often arise when the denominator $H(\bz)$ factors in $\mathbb{C}[\bz]$, but also appear when $H(\bz)$ factors \emph{locally} at points of the singular variety. The next result follows from the Weierstrass preparation theorem\footnote{A clear presentation of the results on complex analytic geometry which we use can be found in Chapter 2 of Ebeling~\cite{Ebeling2007}.}.

\begin{proposition}[{Pemantle and Wilson~\cite[Proposition 10.1.9]{PemantleWilson2013}}]
\label{prop:localring}
Suppose $\bw \in \mV$ and let $\mO_{\bw}$ \glsadd{localring} denote the ring of germs of analytic functions at $\bw$ (which is isomorphic to the ring of convergent power series centered at $\bw$).  Then $\bw$ is a multiple point if and only if $H$ factors in $\mO_{\bw}$ as 
\begin{equation}
\label{eq:localfactor}
H = U \cdot H_1^{m_1} \cdots H_r^{m_r}, \qquad U,H_j \in \mO_{\bw}, \quad H_j(\bw) = 0, \quad U(\bw) \neq 0,
\end{equation} 
where each $m_j$ is a positive integer, the $H_j$ are distinct, and the gradients of the $H_j$ are non-zero at $\bz=\bw$.  Furthermore, $\bw$ is a transverse multiple point whenever the gradients of the $H_j$ are linearly independent at $\bz=\bw$.  
\end{proposition}

The zero sets $\mV(H_1),\dots,\mV(H_r)$, restricted to sufficiently small neighbourhoods of $\bw$, give the manifolds $\mV_1,\dots,\mV_r$ in the definition of a multiple point.  We call the factorization in Equation~\eqref{eq:localfactor} a \emph{square-free factorization of $H$ in $\mO_{\bw}$}, and when $m_1=\cdots=m_r=1$ we say that $H$ is \emph{square-free} at $\bw$. 

\begin{example}
Suppose $H(\bz)$ has the square-free factorization $H(\bz)=H_1(\bz)^{m_1}\cdots H_r(\bz)^{m_r}$ over $\mathbb{C}[\bz]$ (so that the $H_j$ are distinct and square-free polynomials, and the $m_j$ are positive integers).  Further assume that 
\begin{enumerate} 
	\item for each $j=1,\dots,r$, the polynomial $H_j$ and its partial derivatives do not simultaneously vanish (i.e., $\mV(H_j)$ is a manifold), \label{enum:multp1}
	\item if $\bw \in \mV(H_{i_1},\dots,H_{i_p})$ then the vectors $(\nabla H_{i_1})(\bw), \dots, (\nabla H_{i_p})(\bw)$ are linearly independent.
\end{enumerate}
Then every point in the singular variety is a transverse multiple point. If only condition~\ref{enum:multp1} holds then every point in the singular variety is a multiple point, but some are not transverse multiple points.  If $H_{i_1},\dots,H_{i_q}$ are the irreducible factors of $H$ which vanish at $\bw \in \mV$ then the factorization
\[ H = \underbrace{\left(\prod_{j \notin \{i_1,\dots,i_q\}} H_j(\bz)^{m_j}\right)}_{U} H_{i_1}^{m_{i_1}} \cdots H_{i_q}^{m_{i_q}} \]
gives a square-free factorization of $H$ in $\mO_{\bw}$.
\end{example}

\begin{figure}
\centering
\includegraphics[width=0.5\linewidth]{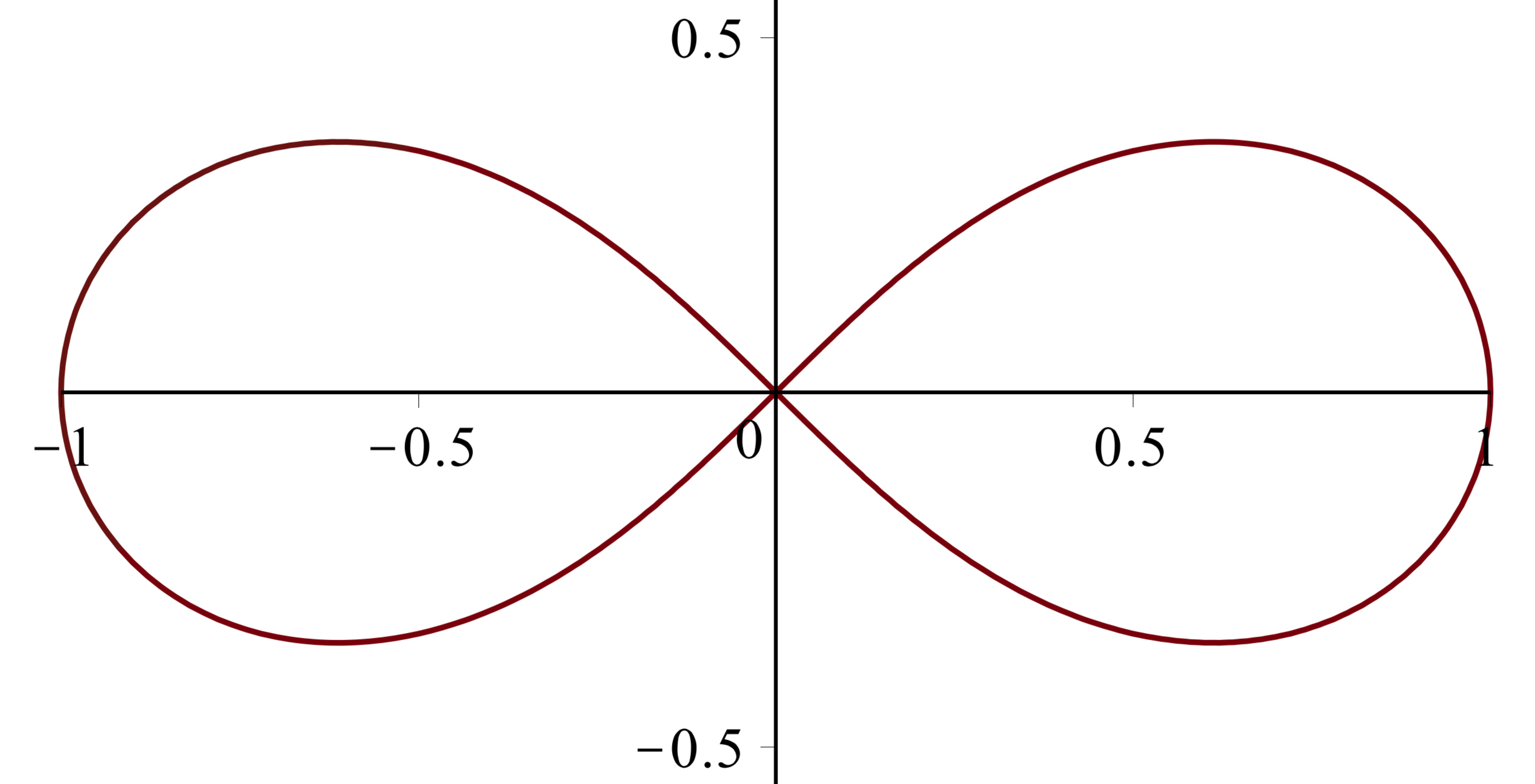}
\caption{The curve defined by the real solutions of $(x^2+y^2)^2-(x^2-y^2)=0$. }
\label{fig:Lem}
\end{figure}

\begin{example}[Lemniscate of Bernoulli]
\label{ex:FigEight}
The real solutions of the bivariate polynomial 
\[ H(x,y) = (x^2+y^2)^2-(x^2-y^2) \]
define the \emph{lemniscate of Bernoulli}\footnote{An article published by Jacob Bernoulli in 1694 studied a family of curves which includes this lemniscate.  Bernoulli used the arc length of such curves to compute certain integrals~\cite[Section 7.5]{OstermannWanner2012}.}, pictured in Figure~\ref{fig:Lem}.  Solving the system of equations
\[ H=(\partial H/\partial x)=(\partial H/\partial y)=0 \]
shows that the origin is a singularity of $\mV = \mV(H)$, and that it is the only singularity.  Solving $H(x,y)=0$ for $y$ while $x$ is in a neighbourhood of the origin gives four solutions $y_1(x),\dots,y_4(x)$ with power series expansions
\begin{align*}
y_1(x) = x-2x^3 + 6x^5 + O(x^6) \qquad & \qquad y_2(x) =  -x+2x^3 - 6x^5 + O(x^6) \\
y_3(x) = i+(3i/2)x^2 - (25i/8)x^4 + O(x^6) \qquad & \qquad y_4(x) = -i-(3i/2)x^2+(25i/8)x^4 + O(x^6).
\end{align*}
Since only two of these solutions have $y$ in a neighbourhood of the origin when $x$ is in a neighbourhood of the origin, there exists an open ball $O \subset \mathbb{C}^2$ around the origin and $\epsilon>0$ such that
\[ \mV \cap O = \underbrace{\{(x,y_1(x)) : |x| \leq \epsilon\}}_{\mV_1} \quad \cup \quad \underbrace{\{(x,y_2(x)) : |x| \leq \epsilon\}}_{\mV_2}. \]
Thus, the origin is a multiple point and, as $y_1'(0) = 1$ while $y_2'(0)=-1$, it is a transverse multiple point\footnote{The tangent planes to $\mV_1$ and $\mV_2$ have normals $(1,1)$ and $(1,-1)$, which are linear independent.}.  Asymptotics for a curve whose real zeroes form a similar ``figure eight'' shape are derived in Example~10.3.9 of Pemantle and Wilson~\cite{PemantleWilson2013}.
\end{example}
\smallskip

\begin{example}[NE, NW, S Lattice Paths in a Quadrant]
\label{ex:NENWS}
Consider the two dimensional lattice path model in the quarter plane defined by the step set $\mS = \{(-1,1), (1,1), (0,-1)\}$.  Theorem~\ref{thm:nzOrbDiag} implies that the generating function for the number of walks beginning at the origin and ending anywhere is the diagonal of the rational function
\[ F(x,y,t) := \frac{(1+x)(1-xy^2+x^2)}{(1+x^2)(1-y)(1-t(1+x^2+xy^2))}.  \]
The denominator $H(x,y,t)$ of $F(x,y,t)$ can be factored as $H=H_1 \cdot H_2 \cdot H_3$, with 
\[ H_1 = 1-t(1+x^2+xy^2), \qquad H_2 = 1-y, \qquad H_3 = 1+x^2. \]
The singular variety $\mV$ is the union of the varieties $\mV(H_1)$, $\mV(H_2)$, and $\mV(H_3)$ which intersect transversely since the gradients $\nabla H_1, \nabla H_2, $ and $\nabla H_3$ are all linearly independent (note $\partial H_1/\partial t = 1+x^2+xy^2 \neq 0$ when $H_1 = 0$).
\end{example}

This lattice path model will serve as a running example in this section.  We now show how to determine (in many cases) diagonal asymptotics when $\mV$ admits minimal critical points which are transverse multiple points.  Much of the analysis in the early steps applies to any singular variety, not just those with multiple point singularities, so we do not make any assumptions on the geometry of $\mV$ unless explicitly stated.

\subsubsection*{Step 1: Stratify the Singular Variety}

As always, the Cauchy integral formula implies that every minimal point $\bw \in \mV \cap \partial \mD$ gives an upper bound of $|w_1 \cdots w_n|^{-1}$ on the exponential growth of the diagonal coefficient sequence of $F(\bz)$.  In order to determine which minimal points could give a tight upper bound on exponential growth, we will relax our definition of critical points.  The overall strategy will be to decompose the singular variety into a collection of manifolds and then compute critical points for the map $\phi(\bz) = z_1\cdots z_n$ restricted to each of the manifolds.  

First, we note that it is effective to partition a given algebraic variety $\mV(H_1,\dots,H_r)$ into complex manifolds which are constructible sets.  By the Jacobian Criterion~\cite[Corollary 16.20]{Eisenbud1995} for algebraic varieties, given a prime ideal $P = (f_1,\dots,f_s)$ of dimension $n-c$ its set of singularities is an algebraic set defined by the vanishing of all $c \times c$ minors of the Jacobian matrix Jac$(f_1,\dots,f_s)$.  The set of singularities of any ideal can then be determined by 
\begin{enumerate}
	\item computing a prime decomposition of the ideal,
	\item computing the singularities of each prime appearing in the decomposition,
	\item determining the points where the prime components intersect.  
\end{enumerate}
To partition the variety $\mV(H_1,\dots,H_r)$ into complex manifolds one determines the ideal $J$ corresponding to its singular points, takes $\mV(H_1,\dots,H_r) \setminus J$ to be one element in the partition, and then repeats this process with $J$.  Algorithms PRIMDEC and DIMENSION of Becker and Weispfenning~\cite[pages 396 and 449]{BeckerWeispfenning1993} describe how to compute prime decompositions and dimensions of ideals, and an implementation of these algorithms is given in the PolynomialIdeals package of Maple.

When dealing with singular varieties having transverse multiple points this decomposition often allows for an asymptotic analysis (in a manner made precise below).  In general, however, the results of Pemantle and Wilson require the singular variety to be partitioned into smooth manifolds which define a \emph{Whitney stratification}\footnote{Pemantle and Wilson make use of techniques similar to those from the study of stratified Morse theory, which require Whitney stratifications (see Goresky and MacPherson~\cite{GoreskyMacPherson1988}).}.  Such a stratification imposes additional restrictions on how the tangent planes of each smooth manifold `fit together', in order to allow for necessary integral computations to be performed (see Appendix C of Pemantle and Wilson~\cite{PemantleWilson2013} for a full definition and discussion).  It was shown by Whitney~\cite[Theorem 18.11]{Whitney1965} that every (real or complex) algebraic variety admits a Whitney stratification.  Rannou~\cite[Theorem 12]{Rannou1998} sketched an algorithm for computing Whitney stratifications of (real or complex) algebraic varieties using quantifier elimination algorithms, but to the best of our knowledge no such algorithm has been implemented.

\begin{example}[continues=ex:NENWS]
\label{ex:NENWS2}
For indices $i_1,\dots,i_p$ let 
\[ \mV_{i_1,\dots,i_p} := \mV(H_{i_1},\dots,H_{i_p}) \setminus \bigcup_{j \notin \{i_1,\dots, i_p\}} \mV(H_j). \]
Then for this lattice path example, the decomposition of $\mV$ into the sets $\mV_1,\mV_2,\mV_3,\mV_{1,2},\mV_{1,3},\mV_{2,3},$ and $\mV_{1,2,3}$ is a partition of $\mV$ into complex manifolds (which is also a Whitney stratification).
\end{example}

\subsection*{Step 2: Determine Minimal Critical Points}
Pemantle and Wilson~\cite[Section 8.3]{PemantleWilson2013} show that any singular variety $\mV$ admits a Whitney stratification whose elements, called \emph{strata}, are algebraic sets (possibly) minus an algebraic set of lower dimension.  Furthermore, they show that if $S$ is a stratum of dimension $n-r$ then there exist irreducible polynomials $f_1,\dots,f_r$ such that $S$ equals $\mV(f_1,\dots,f_r)$ minus an algebraic set of lower dimension and the sets $\mV(f_1),\dots,\mV(f_r)$ intersect transversely.  
\smallskip

For a complex differentiable function $f(\bz)$, define the \emph{logarithmic gradient} to be
\[ \nabla_{\log}f := \left(z_1(\partial f/\partial z_1), \dots, z_n(\partial f/\partial z_n) \right).\]
Since we search for minimizers of $|z_1\cdots z_n|^{-1}$, it is sufficient to consider the subset $S^*$ of $S$ consisting of points with non-zero coordinates.  The argument which gave Lemma~\ref{lem:smoothCPphi} in the smooth case shows that any local minimizer of $|z_1\cdots z_n|^{-1}$ on $S^*$ is a critical point of the restricted function $\phi|_{S^*}:S^*\rightarrow \mathbb{C}$, where $\phi(\bz) = z_1 \cdots z_n$.  The following proposition characterizes the critical points of $\phi|_{S^*}$.

\begin{proposition}[{Pemantle and Wilson~\cite[Section 8.3]{PemantleWilson2013}}]
\label{prop:gencrit}
Given $S$ and $f_1,\dots,f_r$ as above, define the $(r+1) \times n$ matrix 
\[ M := 
\begin{pmatrix} 
\nabla_{\log} f_1 \\ 
\vdots \\
\nabla_{\log} f_r \\ 
\bone
\end{pmatrix}.
\]
When $r<n$ then $\bw \in S^*$ is a critical point of $\phi|_{S^*}$ if and only if it satisfies
\begin{align}
\begin{split}
f_j(\bw) &= 0, \qquad j=1,\dots,r \\[+2mm]
\det(N)(\bw) &= 0, \qquad N \text{ is a maximal minor of } M.
\end{split} 
\label{eq:gencrit}
\end{align}
When $r=n$ then $S^*$ contains a finite set of points, all of which are critical points.
\end{proposition}

\begin{proof}
The critical points of the polynomial map $\phi$ are those where the differential of $\phi|_{S^*}$ is zero.  Since $S$ is defined by the points where $f_1 = \cdots = f_r=0$, and the gradients $\nabla f_j$ are linearly independent, the differential of $\phi|_{S^*}$ is zero if and only if the matrix
\[ 
\begin{pmatrix} 
\nabla f_1 \\ 
\vdots \\
\nabla f_r \\ 
\nabla (z_1 \cdots z_n)
\end{pmatrix}
\]
has rank $r$.  This implies any point in $S^*$ is a critical point when $r=n$.  

Multiplying the $j$th column of this matrix by $z_j$ and dividing the final row by $z_1 \cdots z_n$ gives the matrix $M$ and does not change its rank as each variable is non-zero on $S^*$.  When $r<n$ then $M$ has rank $r$ if and only if all $(r+1) \times (r+1)$ minors simultaneously vanish.
\end{proof}

Thus, each stratum $S$ defines a system of critical point equations~\eqref{eq:gencrit}.  Any point $\bw \in \mV^*$ lies in some stratum $S$, and we call $\bw$ a \emph{critical point} if it satisfies the critical point equations corresponding to $S$ (or, equivalently, if it lies in $S$ and is a critical point of $\phi|_{S^*}$).  The smooth points of $\mV$ lie in a stratum defined by the vanishing of the denominator $H(\bz)$, where the equations~\eqref{eq:gencrit} become the smooth critical points equations~\eqref{eq:critpt}.

\begin{example}[continues=ex:NENWS2]
\label{ex:NENWS3}
The polynomials $H_2 = 1-y$ and $H_3 = 1+x^2$ are independent of the variable $t$, meaning the strata $\mV_2, \mV_3, $ and $\mV_{2,3}$ cannot contain any critical points (this can be verified by constructing the matrix $M$ in Proposition~\ref{prop:gencrit}).  To determine critical points
\begin{itemize}
\item on the stratum $\mV_1$, we solve the smooth critical point equations 
\[ H_1 = 0, \qquad x (\partial H_1/\partial x) = y (\partial H_1/\partial y) = t (\partial H_1/\partial t)  \]
subject to the condition $(1+x^2)(1-y) \neq 0$, giving 4 smooth critical points $\left(\omega^2, \omega \sqrt{2}, \frac{1}{4}\right)$ where $\omega \in \{\pm1,\pm i\}$.  None of these critical points are minimal, as they have $y$-coordinate of modulus $\sqrt{2}$ and the denominator of $F(x,y,t)$ contains $H_2 = 1-y$ as a factor.

\item on the stratum $\mV_{1,3}$, we compute the matrix
\[ M = \begin{pmatrix} 
-tx(y^2 + 2x) & -2txy^2 & -t(1+x^2+xy^2) \\
2x & 0 & 0 \\
1 & 1 & 1
\end{pmatrix} \]
from Proposition~\ref{prop:gencrit} and solve $H_1 = H_3 = \det M = 0$.  This system of polynomial equations has no solutions, so $\mV_{1,3}$ contains no critical points. 

\item on the stratum $\mV_{1,2}$, we compute the matrix
\[ M = \begin{pmatrix} 
-tx(y^2 + 2x) & -2txy^2 & -t(1+x^2+xy^2) \\
0 & -y & 0 \\
1 & 1 & 1
\end{pmatrix} \]
from Proposition~\ref{prop:gencrit} and solve $H_1 = H_2 = \det M = 0$.  This gives two critical points $\bp = (1,1,1/3)$ and $(-1,1,1)$, of which the second is not minimal (since it has larger coordinate-wise modulus than $\bp$).

\item on $\mV_{1,2,3}$, we note that the two points $(i,i,-i)$ and $(-i,-i,i)$ on this stratum are critical points, but they are not minimal.
\end{itemize} 

As $H_2$ and $H_3$ contain only points where $|x|=1$ or $|y|=1$, any point $(x,y,t) \in \mV$ with $|x|<1$ or $|y|<1$ must lie in $\mV(H_1)$.  But at any point on $\mV_1$, $t = \frac{1}{1 + x^2+xy^2}$, and if $|x|,|y| \leq 1$ and one of the inequalities is strict then $|t| > 1/3$.  Thus, $\bp$ is a minimal point.  

Note, however, that $\bp$ is not finitely minimal, as all points $\left(e^{i\theta_1},1,e^{i\theta_2}/3 \right)$ with $\theta_1,\theta_2 \in (-\pi,\pi)$ lie in $T(\bp) \cap \mV$.  On the other hand, $\bp$ is the only critical point of $F(x,y,t)$ in $T(\bp)$ and we will see that this is sufficient to determine asymptotics.
\end{example}
\smallskip

We are now able to characterize minimal critical points for any rational function $F(\bz)$, but in order to calculate asymptotics we will further restrict the types of singularities we consider.  A minimal point $\bw \in \mV$ is called \emph{convenient} if 
\begin{itemize}
	\item $\bw$ is a transverse multiple point, so that $H(\bz)$ has a square-free factorization 
	\[ H = U \cdot H_1^{m_1} \cdots H_r^{m_r}\] 
	in $\mO_{\bw}$;  
	\item there exists an index $i$ such that $(\partial H_j/\partial z_i)(\bw) \neq 0$ for all $j=1,\dots,r$
	(note that $i$ is independent of $j$);  
	\item there exist positive constants $s_1,\dots,s_r$ such that 
	\[ \bone = s_1 \bg_1 + \cdots + s_r \bg_r, \]
	where $\bg_j$ is the vector
	\[ \bg_j = \frac{(\nabla_{\log} H_j)(\bw)}{w_i(\partial H_j/\partial z_i)(\bw)}. \] 
\end{itemize}
The final condition implies that any convenient point is a critical point, as $\bone$ lies in the span of the logarithmic gradients $\nabla_{\log}H_j$,  and any smooth minimal critical point is convenient.  As the $j$th coordinate of each $\bg_j$ is 1, it follows that $s_1+\cdots+s_r=1$.  Furthermore, linear independence of the vectors $\nabla H_1(\bw),\dots,\nabla H_r(\bw)$ implies\footnote{The vector $(\nabla_{\log}H_j)(\bw)$ is obtained from $\nabla H_j(\bw)$ by multiplying its entries by the non-zero components of $\bw \in \left(\mathbb{C}^*\right)^n$.} linear independence of the vectors $(\nabla_{\log} H_1)(\bw),\dots,(\nabla_{\log} H_r)(\bw)$, so the coefficients $s_1,\dots,s_r$ are unique.

\subsection*{Step 3: Compute a Residue}

Suppose that $\bw \in \mV$ is a strictly minimal convenient point and $H(\bz)$ has the square-free factorization $H = U \cdot H_1 \cdots H_r$ in $\mO_{\bw}$  (in particular, $H$ is square-free in $\mO_{\bw}$).  Without loss of generality, we may assume that $(\partial H/\partial z_n)(\bw) \neq 0$ for all $j=1,\dots,r$.  The Weierstrass preparation theorem then implies that, possibly by modifying the factor $U$ in the square-free factorization of $H$, we may assume
\[ H_j(\bz) = z_n - \frac{1}{\nu_j(\bzhtn)} \qquad \text{for } j=1,\dots,r, \]
where the $\nu_j$ are analytic functions defined in a neighbourhood of $\bwhtn$.  Note that
\begin{equation} F(\bz) = \frac{G(\bz)}{U(\bz)\left(z_n - \nu_1(\bzhtn)^{-1}\right)\cdots\left(z_n - \nu_r(\bzhtn)^{-1}\right)}  \label{eq:mpnu} \end{equation}
for $\bz$ in a neighbourhood of $\bw$.  Let $\rho = |w_n|$ and $\mT = T(\bwhtn)$.  Following the same arguments as in the smooth case, it can be shown that for any sufficiently small neighbourhood $\mN$ of $\bwhtn$ in $\mT$ and sufficiently small $\epsilon>0$, the integral
\[ \chi = \frac{-1}{(2\pi i)^n} \int_\mN \left(\int_{|z_n|=\rho+\epsilon} F(\bz) \cdot \frac{dz_n}{z_n^{k+1}} - \int_{|z_n|=\rho-\epsilon} F(\bz) \cdot \frac{dz_n}{z_n^{k+1}} \right) \frac{dz_1 \cdots dz_{n-1}}{z_1^{k+1}\cdots z_{n-1}^{k+1}} \]
satisfies
\begin{equation} |f_{k,\dots,k} - \chi| =  O\left(\left(|w_1\cdots w_n| + \delta\right)^{-k} \right) \label{eq:chibdmp} \end{equation}
for some $\delta>0$.  The inner difference of integrals can be computed using Cauchy's residue theorem, where the appropriate poles can be determined using Equation~\eqref{eq:mpnu}. Ultimately, one obtains 
\begin{equation}
\chi = \frac{1}{(2\pi i)^{n-1}} \int_\mN \left( \sum_{j=1}^r \frac{-\nu_j(\bzhtn)^{k-1} \cdot \check{G}\left(\bzhtn,\nu_j(\bzhtn)^{-1}\right)}{\prod_{i \neq j}\left(\nu_j(\bzhtn)-\nu_i(\bzhtn)\right)} \right) \frac{dz_1 \cdots dz_{n-1}}{z_1^{k+1}\cdots z_{n-1}^{k+1}}, \label{eq:chiresmp2}
\end{equation}
where
\begin{equation} \check{G}(\bzhtn,y) := \frac{G\left(\bzhtn,y\right) \cdot \prod_{i=1}^r(-\nu_i(\bzhtn))}{y^r \cdot U\left(\bzhtn,y\right)} = \frac{F(\bzhtn,y)(1-y\nu_1(\bzhtn))\cdots(1-y\nu_r(\bzhtn))}{y^r}. 
\label{eq:Gcheck} \end{equation}

\subsection*{Step 4: Introduce New Variables and Obtain a Sum of Fourier-Laplace Integrals}

Using a result of DeVore and Lorentz~\cite[Equation 7.12]{DeVoreLorentz1993}, Pemantle and Wilson converted the integral expression in Equation~\eqref{eq:chiresmp2} into a sum of Fourier-Laplace integrals. Given a natural number $r$ let
\[ \Delta_{r-1} := \left\{ \bx \in \left(\mathbb{R}_{>0}\right)^{r-1} : x_1 + \cdots + x_{r-1} \leq 1 \right\}.  \]

\begin{theorem}[{Pemantle and Wilson~\cite[Lemma 10.4.5]{PemantleWilson2013}}]
\label{thm:mpFL}
Suppose $\bw$ is a strictly minimal convenient point, with $\check{G},\nu_1,\dots,\nu_r$ defined as above.  For fixed $\bt \in \mathbb{R}^{n-1}$ and $\btt \in \mathbb{R}^{r-1}$ define
\[ \iota(\bt,\btt) :=  t_1\nu_1\left(\bwhtn e^{i\bt}\right) + \cdots + t_{r-1}\nu_{r-1}\left(\bwhtn e^{i\bt}\right) + (1-t_1-\cdots-t_{r-1})\nu_r\left(\bwhtn e^{i\bt}\right), \]
where
\[ \nu_j\left(\bwhtn e^{i\bt}\right) = \nu_j\left(w_1e^{i\theta_1},\dots,w_{n-1}e^{i\theta_{n-1}}\right).\]  
Then for any sufficiently small neighbourhood $\mN' \subset \mathbb{R}^{n-1}$ of the origin there exists an $\epsilon>0$ such that
\[ |f_{k,\dots,k} - \chi| = O\left(\left(|w_1\cdots w_n| + \epsilon\right)^{-k} \right), \]
where
\begin{equation} \chi := \frac{(w_1 \cdots w_n)^{-k}}{(2\pi)^{n-1}} \cdot
\sum_{j=0}^{r-1} \binom{r-1}{j}\frac{(k-1)!}{(k+j-r)!} 
\int_{\mN' \times \Delta_{r-1}} A_j(\bt,\btt) e^{-k\phi(\bt,\btt)} d\bt \, d\btt \label{eq:FLmp}
\end{equation}
and 
\begin{equation}
\begin{split}
\phi(\bt,\btt) &= i(\theta_1 + \cdots + \theta_{n-1}) - \log\left(\frac{\iota(\bt,\btt)}{\iota(\bzer,\btt)} \right) \\[+2mm]
A_j(\bt,\btt) &= \left. (-1)^{r-1} y^{j-r}\left(\frac{d}{dy}\right)^j \check{G}\left(\bwhtn e^{i\bt},y^{-1}\right)\right|_{y = \iota(\bt,\btt)}
\end{split}
\label{eq:mpAphi}
\end{equation}
\end{theorem}

In general, it may not be possible to explicitly determine $\phi$ and the $A_j$, however asymptotics of the diagonal sequence depends only on the evaluations of their partial derivatives at the origin, which can be calculated implicitly.

\subsection*{Step 5: Determine Asymptotics}

Theorem~\ref{thm:mpFL} gives an expression for $\chi$ in terms of a finite sum of Fourier-Laplace integrals.  If $\bw \in \mV$ is a minimal convenient point then there exists a unique vector $(s_1,\dots,s_r)$ with positive entries summing to 1 such that
\[ \bone = s_1 \bg_1 + \cdots s_r \bg_r, \]
and we let $\bss := (s_1,\dots,s_{r-1})$, which lies in the interior of $\Delta_{r-1}$. We call a minimal convenient point $\bw$ \emph{nondegenerate} if the Hessian matrix $\mH$ of $\phi(\bt,\btt)$ at $(\bzer,\bss)$ is nonsingular.

A set of routine calculations, performed by Raichev and Wilson~\cite{RaichevWilson2011}, shows that the conditions of Proposition~\ref{prop:HighAsm} are satisfied by the Fourier-Laplace integrals in Theorem~\ref{thm:mpFL} when $\bw$ is a finitely minimal nondegenerate convenient point.

As any smooth point is a convenient point, the following result is a generalization of Theorem~\ref{thm:smoothAsm}.

\begin{theorem}[{Raichev and Wilson~\cite[Theorem 3.4]{RaichevWilson2011}}]
\label{thm:mpAsm}
Let $F(\bz)=G(\bz)/H(\bz)$ be a rational function with a nondegenerate strictly minimal convenient point $\bw$ such that $H(\bz)$ has a square-free factorization $H = U \cdot H_1 \cdots H_r$ in $\mO_{\bw}$.  Then for any nonnegative integer $M$ there exist effective constants $C_0,\dots,C_M$ such that
\begin{equation} 
f_{k,\dots,k} = \frac{(w_1\cdots w_n)^{-k}}{k^{(n-r)/2}} \cdot (2\pi)^{(r-n)/2} (\det(\mH))^{-1/2} \left(\sum_{q=0}^M C_q k^{-j} + O\left(k^{-M-1}\right)\right)
\label{eq:mpAsm}
\end{equation}
as $k\rightarrow\infty$, where $\mH$ is the Hessian matrix of $\phi(\bt,\btt)$ in Equation~\eqref{eq:mpAphi} at $(\bzer,\bss)$ and the square root of the determinant is the product of the principal square roots of the eigenvalues of $\mH$.  The leading constant $C_0$ in this series has the value
\[ C_0 = \frac{-G(\bw)}{U(\bw)\prod_{j=1}^r\left(w_n (\partial H_j/\partial z_n)(\bw)\right)}. \]
\end{theorem}
\smallskip

Explicit formulas for the higher order constants $C_j$ are also given by Raichev and Wilson, along with a procedure for determining asymptotics in some cases when $H(\bz)$ is not square-free at $\bw$.  When $F(\bz)$ has a finitely minimal point $\bp$ such that all points in $T(\bp) \cap \mV$ are convenient points satisfying the conditions of Theorem~\ref{thm:mpAsm}, then one can sum the right hand side of Equation~\eqref{eq:mpAsm} determined by each element of $T(\bp) \cap \mV$ to calculate dominant asymptotics.  A Sage package of Raichev~\cite{Raichev2011} finds the asymptotic contributions of nondegenerate convenient points when it is independently known that the points are finitely minimal.

%%%%%%%%%%%%%%%%%%%%%%%%%%%%%%%%%%%%%%%%%%%%%%%%%%%%%%
% Residue Approach
%%%%%%%%%%%%%%%%%%%%%%%%%%%%%%%%%%%%%%%%%%%%%%%%%%%%%%
\section{A Multivariate Residue Approach}
\label{sec:MPres}

All of the lattice path asymptotics we calculate in the next chapters will be determined by minimal convenient points, however (as already seen in Example~\ref{ex:NENWS3}), there will be cases where these points are not finitely minimal.  This can be worked around using more complicated deformations of the domain of integration in the Cauchy residue integral representation of diagonal coefficients.  Such deformations were first used in the context of ACSV by Baryshnikov and Pemantle~\cite{BaryshnikovPemantle2011}. Furthermore, when combined with the theory of multivariate complex residues, which was briefly discussed in Chapter~\ref{ch:SmoothACSV}, this approach will allow us to relax some of our assumptions on convenient points.  We will obtain asymptotic expansions of a similar form to Equation~\eqref{eq:mpAsm} for the diagonal coefficient sequence, however only the leading asymptotic term $C_0$ will be explicitly determined in general.

As in the smooth case, we simply give an overview of the results we need and refer the reader to Pemantle and Wilson~\cite{PemantleWilson2013} for details.

\subsection*{Asymptotic Results}

Suppose that we have a minimal critical point $\bw$ which is a transverse multiple point on a stratum $S$ of dimension $r$, where $H$ has a square-free factorization
\[ H = U \cdot H_1 \cdots H_r \]
in $\mO_{\bw}$ (in particular, $H$ is square-free at $\bw)$.  For $j=1,\dots,r$, Proposition 11.1.14 of Pemantle and Wilson~\cite{PemantleWilson2013} shows\footnote{Proposition 11.1.14 of Pemantle and Wilson~\cite{PemantleWilson2013} shows that the leading homogeneous part $\overline{H}$ of $H(e^{z_1},\dots,e^{z_n})$ is a complex scalar multiple of a real polynomial, and when $\bw$ is a transverse multiple point then $\overline{H}$ is a product of linear polynomials whose coefficients are the elements of the vectors $(\nabla_{\log} H_j)(\bw)$.  See Examples 11.1.11 and 11.1.16 of that text for details about the cone $N(\bw)$ at multiple points.} that the vector $(\nabla_{\log} H_j)(\bw)$ is a complex multiple of a real vector $\bv_j$, and we pick $\mathbf{v}_j$ so that it has non-negative dot product with the all-ones vector $\bone$.  Let $N(\bw)$ denote the intersection of the half-spaces defined by the $\bv_j$:
\[ N(\bw) = \left\{ \bz : \bz \cdot \bv_j \geq 0 \text{ for each } j=1,\dots,r \right\}.\]
Sections 8.5 and 10.2.1 of Pemantle and Wilson~\cite{PemantleWilson2013} show that the domain of integration in the Cauchy residue integral can be deformed, under certain restrictions, to a domain of integration obtained from an $(n-r)$-dimensional chain $\sigma$ lying in the stratum $S$ and staying sufficiently close to $\bw$ and an $r$-dimensional chain $T$ lying outside of the singular variety which can be made arbitrarily close to $\bw$ except at points which do not affect dominant asymptotics; such domains of integration are called \emph{quasi-local cycles}.  To state these definitions rigorously requires the language of relative homology, and we refer the reader to Appendix C of Pemantle and Wilson~\cite{PemantleWilson2013}.

As $\bw$ lies on a stratum $S$ of dimension $r$, there exist $n-r$ distinct coordinates $z_{\pi(1)}, \dots, z_{\pi(n-r)}$ which analytically parametrize the remaining $r$ coordinates in a neighbourhood of $\bw$ in $S$.  Define the matrix
\[ \Gamma_{\Psi} := \begin{pmatrix} \nabla_{\log} H_1 \\ \vdots \\ \nabla_{\log} H_r \\ z_{\pi(1)} \mathbf{e}_{\pi(1)} \\ \vdots \\ z_{\pi(n-r)} \mathbf{e}_{\pi(n-r)} \end{pmatrix}, \]
where $\mathbf{e}_j$ is the $j$th elementary basis vector, with a 1 in its $j$th position and 0 in its other positions.  When $\bw$ lies on a stratum $S$ of dimension $r=n$ then $S$ consists of a finite number of points, and we may take $S$ to be the set containing the single point $\bw$.  In this case we say that $\mV$ has a \emph{complete intersection} at $\bw$, and $\Gamma_{\Psi}$ equals the matrix formed by the logarithmic gradients $\nabla_{\log} H_j$.
\smallskip

After reducing the domain of integration in the Cauchy residue to a quasi-local cycle (when possible), asymptotics are derived by computing a multidimensional residue over the $r$-chain $T$ followed by a saddle-point integral over the $(n-r)$-chain $\sigma$.  The easiest case is when $\mV$ has a complete intersection at $\bw$: here there is no inner integral over the chain $\sigma$ and, when $G(\bw)\neq0$, asymptotics are determined \emph{up to an exponentially small error} by computing a multidimensional residue.  

\begin{theorem}[{Pemantle and Wilson~\cite[Theorem 10.3.3 and Proposition 10.3.6]{PemantleWilson2013}}]
\label{thm:compintasm}
Let $F(\bz)$ be a rational function with square-free denominator which is analytic at the origin.  Suppose $\bx \in \partial \mD$ minimizes $|z_1 \cdots z_n|^{-1}$ on $\overline{\mD}$, and all minimizers of $|z_1 \cdots z_n|^{-1}$ on $\overline{\mD}$ lie in $T(\bx)$.  Assume that each critical point $\bz$ of $F$ in $T(\bx)$ is a transverse multiple point such that $\bone \notin \partial N(\bz)$.  If the set
\[ E := \{ \bz \in T(\bx) : \bz \text{ is a critical point and } \bone \in N(\bz) \} \]
contains a single point $\bw$ where $\mV$ forms a complete intersection, and $G(\bz) \neq 0$, then 
\begin{equation}
\label{eq:compintasm}
f_{k,\dots,k} = (w_1 \cdots w_n)^{-k} \cdot \frac{G(\bw)}{\det \Gamma_{\Psi}(\bw)} + O\left(\left(|w_1 \cdots w_n| + \epsilon\right)^{-k} \right) 
\end{equation}
as $k\rightarrow\infty$, for some $\epsilon>0$.
\end{theorem}

\begin{example}
Consider the rational function
\[ F(x,y) = \frac{1}{(1-x-2y)(1-2x-y)}. \]
Here $F$ admits two smooth critical points $(x,y) = (2/3,1/3)$ and $(1/3,2/3)$ on the zero sets $\mV(1-x-2y)$ and $\mV(1-2x-y)$, and one transverse multiple point $(x,y) = (1/3,1/3)$ on $\mV(1-x-2y,1-2x-y)$.  Because $F(x,y)$ is combinatorial, and a product of linear factors, it is not difficult to show that $(1/3,1/3)$ is a strictly minimal critical point which minimizes $|xy|^{-1}$ on $\overline{\mD}$.  Since
\[ \Gamma_{\Psi}(1/3,1/3) = \begin{pmatrix} -1/3 & -2/3 \\ -2/3 & -1/3 \end{pmatrix}, \]
Theorem~\ref{thm:compintasm} implies
\[ f_{k,k} = 3 \cdot 9^k + O\left(\delta^k\right) \]
for some $\delta \in (0,9)$.
\end{example}

In the case when $r<n$, a saddle-point integral over the chain $\sigma$ must be dealt with after the residue calculation.  Let $P$ be the set of variables which locally parametrize $S$,
\[ P = \{\pi(1),\dots,\pi(n-r)\}. \]  
For $j \notin P$ there exists an analytic function $\zeta_j(z_{\pi(1)},\dots,z_{\pi(n-r)})$ parameterizing $z_j$ on a neighbourhood of $\bw$ in $S$, and we define
\[ g(\theta_1,\dots,\theta_{n-r}) := \sum_{j \notin P} \log\left[ \zeta_j\left(w_{\pi(1)}e^{i\theta_1}, \dots, w_{\pi(n-r)}e^{i\theta_{n-r}}\right)\right]. \]
Let $Q$ be the $(n-r)\times(n-r)$ matrix whose $(i,j)$th entry is $\frac{\partial^2 g}{\partial \theta_i\partial \theta_j}(\bzer)$; we say that $\bw$ is \emph{nondegenerate} if the determinant of $Q$ is non-zero.  In this case, asymptotics are determined by the following result.

\begin{theorem}[{Pemantle and Wilson~\cite[Theorem 10.3.4 and Proposition 10.3.6]{PemantleWilson2013}}]
\label{thm:resasm}
Let $F(\bz)$ be a rational function with square-free denominator which is analytic at the origin.  Suppose $\bx \in \partial \mD$ minimizes $|z_1 \cdots z_n|^{-1}$ on $\overline{\mD}$, and all minimizers of $|z_1 \cdots z_n|^{-1}$ on $\overline{\mD}$ lie in $T(\bx)$.  Assume that each critical point $\bz$ of $F$ in $T(\bx)$ is a nondegenerate transverse multiple point such that $\bone \notin \partial N(\bz)$.  If the set
\[ E := \{ \bz \in T(\bx) : \bz \text{ is a critical point and } \bone \in N(\bz) \}, \]
contains a single point $\bw$ and $H$ has the square-free factorization $H = H_1 \cdots H_r$ in $\mO_{\bw}$ then
\begin{equation}
\label{eq:resasm}
f_{k,\dots,k} = (w_1 \cdots w_n)^{-k} \cdot k^{(r-n)/2} \cdot \frac{(-1)^{n-r}(2\pi)^{(r-n)/2}}{\sqrt{ \det Q} \cdot \det \Gamma_{\Psi}(\bw)} \left(G(\bw) + O\left(\frac{1}{k}\right)\right)
\end{equation}
as $k\rightarrow\infty$.
\end{theorem}

If the set $E$ described in these results contains a finite set of points, one can simply sum the contributions of each given by Theorems~\ref{thm:compintasm} and~\ref{thm:resasm}, when they apply.  Pemantle and Wilson~\cite{PemantleWilson2013} also give formulae for higher order poles (i.e., cases when $H$ is not square-free at $\bw$).  

\begin{example}[continues=ex:NENWS3]
\label{ex:NENWS4}
We have the rational function
\[ F(x,y,t) := \frac{(1+x)(1-xy^2+x^2)}{(1+x^2)(1-y)(1-t(1+x^2+xy^2))} \]
and minimal critical point $\bp = (1,1,1/3)$.  We claim that $\bp$ minimizes $|xyt|^{-1}$ on $\overline{\mD}$. If $(x,y,t)$ and $(a,b,c)$ are positive real solutions to $H_1$ and $a \geq x,b \geq y$ then 
\[ c = \frac{1}{1+a^2+ab^2} \leq \frac{1}{1+x^2+xy^2} = t.\] 
Thus, any solution to $H_1(x,y,t)$ with positive coordinates lies on the boundary of the domain of convergence of $1/H_1$.  Let $\mD_{1,2}$ be the domain of convergence of $1/(H_1H_2)$.  Since $1/(H_1H_2)$ is combinatorial, Lemma~\ref{lem:combCase} implies that the every point in $\partial \mD_{1,2}$ has the same coordinate-wise modulus as a point in $\mV \cap \partial \mD_{1,2}$ with non-negative coordinates. Furthermore, any minimizer of $|xyt|^{-1}$ on $\mV \cap \partial \mD_{1,2}$ must satisfy $H_1(x,y,t)=0$ as $H_2$ is independent of $x$ and $t$. Thus, to minimize $|xyt|^{-1}$ on $\overline{\mD_{1,2}}$ it is sufficient to minimize the function 
\[ \psi(x,y) = \left. (xyt)^{-1} \right|_{t = 1/(1+x^2+xy^2)} = \ox\,\oy + x\oy + y \]
on the domain $(x,y) \in (0,\infty) \times (0,1]$ (or determine that such a minimum does not exist).  The function $\psi$ approaches infinity as $x$ approaches 0 or infinity, or as $y$ approaches 0, so the minimum occurs either at a critical point of $\psi$, where
\[ (\partial \psi/\partial x)(x,y) = (\partial \psi/\partial x)(x,y) = 0, \]
or when $y=1$ and $(\partial \psi/\partial x)(x,1)=0$.  In fact, solving these equations gives the $(x,y)-$coordinates of the critical points of $F(\bz)$ on the strata $\mV_1$ and $\mV_{1,2}$ (which is not surprising, because the critical points of $F$ give the local minimizers of $|xyt|^{-1}$).  The only solution of these equations with $(x,y) \in (0,\infty) \times (0,1]$ is $(x,y)=(1,1)$, which corresponds to $\bp$.  Thus, $\bp$ is the unique minimizer of $|xyt|^{-1}$ on $\overline{\mD_{1,2}}$ with positive coordinates, and every minimizer on $\overline{\mD_{1,2}}$ has the same coordinate-wise modulus as $\bp$.  Since $\bp$ has an $x$-coordinate of modulus 1, it lies in $\overline{\mD}$ and every minimizer of $|xyt|^{-1}$ on $\overline{\mD}$ lies in $T(\bp)$.
\smallskip

Here we have the logarithmic gradients
\[ (\nabla H_1)(\bp) = (0,-1,0) \qquad \text{and} \qquad (\nabla H_2)(\bp) = (-1, -2/3, -1) \]
so that 
\[ N(\bp) = \{ (p,q,r) \in \mathbb{R}^3 : q \geq 0 \text{ and } p+(2/3)q + r \geq 0 \},  \]
and the vector $\bone$ thus lies in the interior of $N(\bp)$. 

On the stratum $\mV_{1,2} = \mV(1-y,1-t(1+x^2+xy^2))$ containing $\bp$ we can parametrize $y$ and $t$ by their $x$-coordinates:
\[ y = 1 \qquad\text{and}\qquad t = \frac{1}{1+x+x^2}, \]
giving 
\[ g(\theta) = \log\left(\frac{1}{1+e^{i\theta}+e^{2i\theta}}\right)\] 
and $Q=g''(0) = 2/3$. Furthermore,
\[ \Gamma_{\Psi}(\bp) = 
\begin{pmatrix}
&(\nabla_{\log}H_1)(\bp)& \\
&(\nabla_{\log}H_2)(\bp)& \\
1 & 0 & 0
\end{pmatrix} 
= 
\begin{pmatrix}
-1 & -2/3 & -1 \\
0 & -1 & 0 \\
1 & 0 & 0
\end{pmatrix}.
\]
Putting everything together, Theorem~\ref{thm:resasm} implies
\[ f_{k,k,k} = 3^k \cdot k^{-1/2} \cdot \frac{\sqrt{3}}{2\sqrt{\pi}}\left(1 + O\left(\frac{1}{k}\right)\right). \]
This proves one of the conjectures of Bostan and Kauers listed in Table~\ref{tab:shortQPasm}.
\end{example}

Theorem~\ref{thm:resasm} is derived by writing the diagonal coefficient sequence as a sum of Fourier-Laplace integrals of the form
\[ \frac{1}{(2\pi i)^{n-r}} \int_{\sigma(\bw)} R\left(z_{\pi(1)},\dots,z_{\pi(n-r)}\right) \,\, dz_{\pi(1)} \wedge \cdots \wedge z_{\pi(n-r)},  \]
where
\[ \left. R\left(z_{\pi(1)},\dots,z_{\pi(n-r)}\right) = \frac{G(\bz)}{(z_1 \cdots z_n)^k \cdot \det \Gamma_{\Psi}} \right|_{z_j = \zeta_j, \,\, j \notin P} \]
and $\sigma(\bw) \subset S$ is a chain of integration arbitrarily close to each $\bw \in E$.  Existence of quasi-local cycles containing the $\sigma(\bw)$ is shown in Proposition 10.3.6 of Pemantle and Wilson~\cite{PemantleWilson2013}, and when explicit representations of the $\sigma(\bw)$ are known higher order asymptotic terms of $f_{k,\dots,k}$ can often be derived.

%As mentioned above, the $(n-r)$-chain of integration $\sigma$ is determined through an intricate construction,  meaning higher order constants (which are the leading non-zero constants when the numerator vanishes) are harder to determine.

\section{Further Generalizations}
We end by briefly describing some further generalizations of the theory of ACSV.

\subsection*{More Complicated Singular Structure}
The methods determining diagonal asymptotics presented here fall into two parts: first determine a finite set of singularities contributing to dominant asymptotics, and then determine the asymptotic contribution of each.  In addition to dealing with transverse multiple points, Chapter 10 of Pemantle and Wilson~\cite{PemantleWilson2013} shows how to determine when (not necessarily transverse) multiple points yield dominant asymptotics\footnote{This is achieved by a result similar to Theorems~\ref{thm:compintasm} and~\ref{thm:resasm}, after the cone $N(\bw)$, described here for transverse multiple points, is generalized.}.  Furthermore, they show how to determine asymptotic contributions of singularities belonging to a super-class of transverse multiple points called \emph{arrangement points}.  Chapter 11 of that text, describing material from Baryshnikov and Pemantle~\cite{BaryshnikovPemantle2011}, shows how to determine the contributions of \emph{cone point} singularities, where the singular variety is defined by the vanishing of an analytic function whose lowest order non-zero Taylor coefficients have degree two and satisfy certain conditions (for instance, $\mV(xy + xz + yz)$ has a cone point at the origin, where the real part of the singular variety looks like two cones meeting at their tips). 

\subsection*{Diagonals of Meromorphic Functions}
Although the results of this chapter and Chapter~\ref{ch:SmoothACSV} were stated for diagonals of rational functions, they hold more generally for diagonals of meromorphic functions.  In particular, although it may not be possible to write a meromorphic function $F(\bz)$ as the ratio of analytic functions $G(\bz)/H(\bz)$ over its domain of definition $\Omega$, at each point $\bw \in \Omega$ there exists a neighbourhood $U$ of $\bw$ in $\Omega$ and analytic functions $G_{\bw},H_{\bw}:U \rightarrow \mathbb{C}$ such that $F(\bz) = G_{\bw}(\bz)/H_{\bw}(\bz)$ on $U$.  \L{}ojasiewicz~\cite{Lojasiewicz1965} was the first to show that semianalytic sets (including the singular sets of meromorphic functions) admit Whitney stratifications, and the characterization of critical points given in Proposition~\ref{prop:gencrit} is a local characterization, meaning one can determine when the point $\bw \in \mathbb{C}^n$ is a critical point by replacing $G$ and $H$ in the statement of the proposition by $G_{\bw}$ and $H_{\bw}$.  Note, however, that while there are effective elimination tools such as resultants and Gröbner Bases for polynomial systems it is much harder to work with systems of equations involving general analytic functions.

%%%%%%%%%%%%%%%%%%%%%%%%%%%
% Chapter 10
%%%%%%%%%%%%%%%%%%%%%%%%%%%
\chapter{Lattice Walks in A Quadrant}
\label{ch:QuadrantLattice}
\vspace{-0.2in}

\begin{center}This chapter is based on an article of Melczer and Wilson~\cite{MelczerWilson2016}.\end{center}

\setlength{\epigraphwidth}{2.3in}
\epigraph{We shall not cease from exploration \\
And the end of all our exploring \\
Will be to arrive where we started \\
And know the place for the first time.}{T. S. Eliot, \emph{Little Gidding}}

Combining the ACSV results of Chapter~\ref{ch:NonSmoothACSV} with the rational function expressions for lattice path generating functions given in Chapter~\ref{ch:KernelMethod}, we will prove the conjectured asymptotics of Bostan and Kauers on lattice paths with short steps in a quadrant (shown in Table~\ref{tab:shortQPasm} of Chapter~\ref{ch:KernelMethod}).  Furthermore, we are able to determine asymptotics for excursions (walks ending at the origin), and derive some results for asymptotics of walks returning to their boundary axes.

The analysis splits into several cases.  First, we recall that the trivariate generating functions $Q(x,y,t)$ marking endpoint and length for the models
\[\diagrF{S,W,NE} \qquad\qquad \diagrF{N,E,SW} \qquad\qquad \diagrF{N,E,S,W,NE,SW} \qquad\qquad \diagrF{W,E,SW,NE}\]
are algebraic, and their minimal polynomials were given by Bousquet-Mélou and Mishna~\cite{Bousquet-MelouMishna2010} and Bostan and Kauers~\cite{BostanKauers2010}.  This means any desired asymptotic information about these models can be rigorously determined through a univariate analysis, and we do not consider them for the rest of this chapter\footnote{As algebraic functions, the univariate generating functions of these models can be represented as diagonals of bivariate rational functions, but the expressions obtained through this connection are usually large and hard to deal with using the theory of ACSV.  Thus, we focus on the remaining 19 models where the orbit sum method gives a ``nice'' diagonal representation coming from a combinatorial argument.}. 

The models 
\[ \diagrF{N,S,E,W} \qquad\qquad \diagrF{NE,NW,SE,SW} \qquad\qquad \diagrF{NE,NW,SE,SW,N,S}  \qquad\qquad \diagrF{NE,NW,SE,SW,N,S,E,W} \]
are symmetric over every axis, have smooth singular varieties, and were analyzed in Chapter~\ref{ch:SymmetricWalks}.  Note that although Theorem~\ref{thm:symAsmEx} gives a bound on the number of walks returning to the origin and each axis, it is automatic to determine the actual asymptotics for each of these four models using Corollary~\ref{cor:smoothAsm}.
\smallskip

The main part of this chapter examines models which are symmetric over one axis.  The generating functions of such models have a uniform expression as rational diagonals, however the location of minimal critical points will depend on the model.  The models defined by the step sets
\[ \diagrF{NE,NW,S} \qquad \diagrF{N,NW,NE,S} \qquad \diagrF{N,NE,NW,SE,SW} \qquad \diagrF{NE,NW,E,W,S} \qquad \diagrF{N,NW,NE,E,W,S} \qquad \diagrF{N,E,W,NE,NW,SE,SW} \] 
containing more steps with positive $y$-coordinate than negative $y$-coordinate have a rational diagonal representation with minimal critical points where the singular variety is non-smooth.  The model defined by the step set $\mS$ admits $(1,1,1/|\mS|)$ as a minimal critical point, which will imply that the exponential growth of its counting sequence is the same as the exponential growth of the number of walks using the steps in $\mS$ with no restriction on where they can go.  These models are said to have \emph{positive drift}. 
\smallskip

In contrast, the models defined by the step sets
\[ \diagrF{N,SE,SW} \qquad \diagrF{N,S,SE,SW} \qquad \diagrF{NE,NW,SE,SW,S} \qquad \diagrF{N,E,W,SE,SW} \qquad \diagrF{N,E,W,S,SW,SE}  \qquad \diagrF{NE,NW,E,W,SE,SW,S} \]
containing more steps with negative $y$-coordinate than positive $y$-coordinate, have rational diagonal representations with minimal critical points where the singular variety is smooth.  In this case, the exponential growth of the number of walks in the quadrant is smaller than the exponential growth of the number of unrestricted walks using the same steps. These models are said to have \emph{negative drift}.  
\smallskip

Finally, the three models
\[ \diagrF{N,W,SE} \qquad\qquad \diagrF{NW,SE,N,S,E,W} \qquad\qquad \diagrF{E,SE,W,NW} \]
do not fit into the above families. Asymptotics for these models were given by Bousquet-Mélou and Mishna~\cite{Bousquet-MelouMishna2010}, and we refer to that source for asymptotics on the first two of these models.  The final model, with four steps, is known as the \emph{Gouyou-Beauchamps} model and is studied in detail in Chapter~\ref{ch:WeightedWalks}.

\section{Models with One Symmetry}

Suppose that $\mS \subset \{\pm1,0\}^2$ is symmetric over one axis.  If $\mS'$ is the step set obtained by rotating $\mS$ over the line $y=x$, the lattice path model in the quarter plane defined by $\mS'$ is isomorphic to the one defined by $\mS$, so we are free to assume that $\mS$ is symmetric over the $y$-axis.  This implies the existence of Laurent polynomials $A_{\pm1}(x),A_0(x),B_1(y),B_0(y)$ such that 
\begin{equation} S(x,y) = \sum_{(i,j) \in \mS}x^iy^j = A_{-1}(x)\oy + A_0(x) + A_1(x)y = B_0(y) + B_1(y)(\ox+x), \label{eq:ABqp} \end{equation}
as the models we consider have short steps.  The group of transformations $\mG$ corresponding to one of these models is the group of order 4 generated by the maps $(x,y)\mapsto(\ox,y)$ and $(x,y)\mapsto(x,\oy A_{-1}(x)/A_1(x))$, and one can calculate the orbit sum
\[ \sum_{\sigma \in \mG} \sgn(\sigma)\sigma(xy) = (x-\ox)\left(y-\oy\frac{A_{-1}(x)}{A_1(x)}\right).\]  
Theorem~\ref{thm:nzOrbDiag} then gives the generating function counting the number of walks using the steps in $\mS$ which stay in the quarter plane and end anywhere as the diagonal
\begin{equation} Q(1,1,t) = \Delta F(x,y,t) = \Delta\left(\frac{(1+x)\left(A_1(x)-y^2A_{-1}(x)\right)}{A_1(x)(1-y)(1-txyS(x,\oy))} \right). \label{eq:diagAlmostSim} \end{equation}
Note that this rational function may be singular at the origin (if $A_{-1}=\ox+x$ and $A_1=1$, for instance) but if it is not analytic it is of the form $R(x,y,t)/x$ where $R$ is analytic at the origin.  Thus, one can use the identity $[t^n]\Delta(R/x) = [t^{n+1}]\Delta(ytR)$ to determine the asymptotics of the diagonal sequence by analyzing the function $ytR(x,y,t)$ which is analytic at the origin.  Alternatively, one can consider the expansion of $F$ in the ring $\mathbb{Q}[x,\ox][[y,t]]$ and use the theory of ACSV for convergent Laurent expansions.  

To perform the analysis, we define
\[ H_1(x,y,t) = 1-txyS(x,\oy), \quad H_2(x,y,t) =1-y, \quad H_3(x,y,t) =A_1(x).\]

The arguments for this family of models are often analogous to the specific model (from the family) studied in the running Example~\ref{ex:NENWS} of Chapter~\ref{ch:NonSmoothACSV}.

\subsubsection*{Determining Minimal Points}
As the denominator of $F(x,y,t)$ takes a simple form, it is easy to determine its set of minimal points.  Note that for the models we consider here the Laurent polynomial $A_1(x)$ is $1,x+\ox,$ or $\ox+1+x$.

\begin{proposition}
\label{prop:almostsymmin}
If $\mS$ is a model such that $A_1(x) = 1$ then $(x,y,t) \in \mV$ is minimal if and only if 
\[ |y| \leq 1, \qquad |t| \leq \frac{1}{|xy|S(|x|,|\oy|)}  \]
and both inequalities do not hold strictly. Furthermore, if $\mS$ is a model such that $A_1(x) = x+\ox$ or $A_1(x) = \ox+1+x$ then $(x,y,t) \in \mV$ is minimal if and only if 
\[ |x|\leq 1, \qquad |y| \leq 1, \qquad |t| \leq \frac{1}{|xy|S(|x|,|\oy|)}  \]
and all three inequalities do not hold strictly.
\end{proposition}

\begin{proof}
Let $\mD_1$ denote the domain of convergence of $1/H_1$.  As $S(x,y)$ has non-negative coefficients, all coefficients in the power series expansion of $1/H_1(x,y,t)$ are non-negative and Lemma~\ref{lem:combCase} implies that $(x,y,t) \in \partial \mD_1$ if and only if $(|x|,|y|,|t|) \in \partial \mD_1$.  Furthermore, if $(|a|,|b|,|c|),(|x|,|y|,|t|) \in \mV(H_1)$, and $|a| \leq |x|$ and $|b| \leq |y|$, then
\[ |t| = \frac{1}{|xy|S(|x|,|\oy|)} \leq \frac{1}{|ab|S(|a|,|\overline{b}|)} = |c| \]
as $xyS(x,\oy)$ is a polynomial in $x$ and $y$ with non-negative coefficients.  Thus, every solution of $H_1(x,y,t)=0$ with non-negative coefficients is minimal.  In other words,  $(x,y,t) \in \partial \mD_1$ if and only if $H_1(|x|,|y|,|t|)=0$. The domains of convergence $\mD_2$ and $\mD_3$ of $1/(1-y)$ and $1/A_1(x)$ are easy to determine, and the domain of convergence of $F(x,y,t)$ is the intersection $\mD_1 \cap \mD_2 \cap \mD_3$.  Note that the solutions of $\ox + 1 + x=0$ and $x + \ox=0$ are roots of unity.
\end{proof}

\subsubsection*{Determining Critical Points}
If $A_1(x)$ is not the constant polynomial 1, the algebraic varieties $\mV(H_1)$, $\mV(H_2)$, and $\mV(H_3)$ are smooth manifolds which intersect transversely, and $\mV$ can be partitioned into the disjoint collection of manifolds $\mV_1,\mV_2,\mV_3,\mV_{1,2},\mV_{1,3},\mV_{2,3},$ and $\mV_{1,2,3}$, where 
\[ \mV_{i_1,\dots,i_p} := \mV(H_{i_1},\dots,H_{i_p}) \setminus \bigcup_{j \notin \{i_1,\dots, i_p\}} \mV(H_j). \]
If $A_1(x)=1$ then we can partition $\mV$ into the disjoint manifolds $\mV_1,\mV_2,\mV_{1,2}$.  
\smallskip

In either case, as $H_2$ and $H_3$ are independent of the variable $t$ the only strata which can contain critical points are $\mV_1,\mV_{1,2},\mV_{1,3}$, and $\mV_{1,2,3}$.  We now examine each of these strata separately:
\begin{enumerate}
	\itemsep=2em
	\item Critical points on $\mV_1$ are characterized by the system of smooth critical point equations
		\[ H_1 = 0, \quad (\partial H_1/\partial x) = (\partial H_1/\partial y) = (\partial H_1/\partial t), \]
		which simplifies to 
		\[ (\partial S/\partial x)(x,\oy) = (\partial S/\partial y)(x,\oy) =  0, \qquad t = \frac{xy}{S(x,\oy)},\]
		together with the condition that $A_1(x)(1-y) \neq 0$.  Substituting the expressions in Equation~\eqref{eq:ABqp} then implies
		\begin{equation} B_1(\oy)(1-\ox^2) = A_1(x)-A_{-1}(x)y^2 = 0 \label{eq:almostsimAB} \end{equation}
		whenever $(x,y,t) \in \mV_1$ is a critical point.  Using Proposition~\ref{prop:almostsymmin}, to search for minimal critical points we examine non-negative solutions of these equations, of which there is one
		\[ \bp =  \left(1,\sqrt{A_1(1)/A_{-1}(1)}, \frac{\sqrt{A_{-1}(1)/A_1(1)}}{S(1,\sqrt{A_{-1}(1)/A_1(1)})} \right). \]
		Furthermore, there are at most four critical points in $\mV_1$ with the same coordinate-wise modulus as $\bp$, those in the set
		\[ E = \left\{ \left(x,y, \ox \,\oy S(x,\oy)\right) : x=\pm1, \quad y=\pm\sqrt{A_1(x)/A_{-1}(x)}, \quad |S(x,\oy)| = S(|x|,|\oy|) \right\}. \]
		Note that these points are not minimal when $A_1(1)>A_{-1}(1)$.
	\item Proposition~\ref{prop:gencrit} implies that the critical points on the stratum $\mV_{1,2}$ satisfy $H_1=H_2=\det(M)=0$, where $M$ is the matrix
		{\small \[ \begin{pmatrix} \nabla_{\log} H_1 \\ \nabla_{\log} H_2 \\ \bone \end{pmatrix} 
		 = \begin{pmatrix} -txyS(x,\oy) + tx^2y(\partial S/\partial x)(x,\oy) & -txyS(x,\oy) - tx^2(\partial S/\partial y)(x,\oy) & -txyS(x,\oy) \\ 0 & -y & 0 \\ 1 & 1 & 1 \end{pmatrix}. \] }
		 This system of equations simplifies to
		 \[ (\partial S/\partial x)(x,1) = 0, \qquad y=1, \qquad t = \frac{1}{xyS(x,\oy)}. \]
		 The point
		 \[ \bs = (1,1,1/S(1,1)) \]
		 is the only solution to these equations with non-negative coordinates.  The point $(-1,1,-1/S(-1,1))$ is also a critical point, but 
		 \[ \frac{1}{S(-1,1)} = \frac{1}{B_0(1)} > \frac{1}{B_0(1) +2B_1(1)} = \frac{1}{S(1,1)}, \] 
		 so $(-1,1,-1/S(-1,1))$ is not minimal.
	\item If $A_1(x) =\ox + x$ or $A_1(x) = \ox+1+x$, Proposition~\ref{prop:gencrit} implies that the critical points on the stratum $\mV_{1,3}$ satisfy $H_1=H_3=\det(M)=0$, where $M$ is the matrix
		{\small \[ \begin{pmatrix} \nabla_{\log} H_1 \\ \nabla_{\log} H_3 \\ \bone \end{pmatrix} 
		 = \begin{pmatrix} -txyS(x,\oy) + tx^2y(\partial S/\partial x)(x,\oy) & -txyS(x,\oy) - tx^2(\partial S/\partial y)(x,\oy) & -txyS(x,\oy) \\ x-\ox & 0 & 0 \\ 1 & 1 & 1 \end{pmatrix}. \] }
		 This system of equations implies $A_1(x) = A_{-1}(x) = 0$, which has no solution for the models we consider here as $A_1(x) \neq A_{-1}(x)$. Thus, there are no critical points on the stratum $\mV_{1,3}$.
	\item If $A_1(x) =\ox + x$ or $A_1(x) = \ox+1+x$ then any point on $\mV_{1,2,3}$ is critical, but it can easily be checked that neither of the two resulting points are minimal.
\end{enumerate}

Thus, we have determined the minimal critical point $\bs$ where the singular variety is locally the union of the smooth manifolds $\mV(H_1)$ and $\mV(H_2)$, together with a finite set of critical points $E$ where the singular variety is smooth and which may be minimal.  The analysis now splits into two cases.

\subsubsection*{Positive Drift Models}
When $\mS$ is symmetric over the $y$-axis, and has more steps with positive $y$-coordinate than negative $y$-coordinate, then $A_{-1}(1) < A_1(1)$ and the only minimal critical point of the singular variety is the point
\[ \bs=\left(1,1,\frac{1}{S(1,1)}\right) = \left(1,1,\frac{1}{|\mS|}\right). \]
Note that $\bs$ is not finitely minimal, but it is the only critical point in $\mV \cap T(\bs)$.

Proposition~\ref{prop:almostsymmin} shows that any minimal point minimizing $|xyt|^{-1}$ satisfies
\[ |xyt|^{-1} = S(|x|,|\oy|). \]
For each of the 6 models in this class, it can be directly verified that $S(a,\overline{b})$ approaches infinity when $a$ and $b$ are positive real numbers and $a$ or $b$ approaches 0 or infinity.  Thus, the minimum of $S(a,\overline{b})$ is attained for $(a,b) \in (0,\infty) \times (0,1]$ and (as in Example~114 of Chapter~\ref{ch:NonSmoothACSV}) 
by taking derivatives of $S(a,\overline{b})$ it can be shown that this minimum occurs uniquely\footnote{It is well known that when $\mS$ is any lattice path model whose steps are not contained in a half-plane then $(\partial S/\partial x)(x,y)=(\partial S/\partial y)(x,y)=0$ has a unique solution $(x,y) \in \left(\mathbb{R}_{>0}\right)^2$ (see, for instance, Bostan et al.~\cite[Theorem 4]{BostanRaschelSalvy2014} or Denisov and Wachtel~\cite[Section 1.5]{DenisovWachtel2015}).  For all of the short step models we consider in this thesis, $(\partial S/\partial x)(x,1) = C - D/x^2$ for positive constants $C$ and $D$, meaning $(\partial S/\partial x)(x,1)=0$ also has a unique positive real solution.} when $x=y=1$.  This implies the minimum of $|xyt|^{-1}$ occurs at $(x,y,t) = \bs$.  Furthermore, our argument shows that any point in $\mV_{1,2}$ achieving this minimum must lie in $T(\bs)$. Since $\bs \in \overline{\mD}$, this implies $\bs$ minimizes $|xyt|^{-1}$ on $\overline{\mD}$ and every such minimizer lies in $T(\bs)$.
\smallskip

A square-free factorization of the denominator of this rational function in $\mO_{\bs}$ is given by 
\[ H = U \cdot H_1 \cdot H_2, \]
where $U=H_3=A_1(x)$, and 
\[ (\nabla_{\log})(H_1)(\bs) = \left(-1,-1 + \frac{a_1-a_{-1}}{|\mS|},-1\right) \qquad (\nabla_{\log})(H_2)(\bs) = (0,-1,0) \]
with $a_j=A_j(1)$. This implies, using the notation of Section~\ref{sec:MPres}, that $\bone \in N(\bs)$ and $\bone \notin \partial N(\bs)$ (in fact, $\bs$ is a minimal convenient point).

Since the numerator of the rational function under consideration does not vanish at $\bs$, Theorem~\ref{thm:resasm} allows us to determine dominant asymptotics of the diagonal sequence.  On the stratum $\mV_{1,2}$ containing $\bs$ we can parametrize $y$ and $t$ by their $x$-coordinate:
\[ y = 1, \qquad t = \frac{1}{xS(x,1)}, \]
giving 
\[ g(\theta) = \log(1) + \log\left(\frac{1}{e^{i\theta}S(e^{i\theta},1)}\right)\] 
and $Q=g''(0) = \frac{2B_1(1)}{|\mS|}$. Since
\[ \Gamma_{\Psi}(\bs) = 
\begin{pmatrix}
-1 & -1 + \frac{a_1-a_{-1}}{|\mS|} & -1 \\
0 & -1 & 0 \\
1 & 0 & 0
\end{pmatrix},
\]
we have $\det \Gamma_{\Psi} = -1$, and Theorem~\ref{thm:resasm} implies the number of walks of length $k$ ending anywhere has the asymptotic expansion
 \[ [t^k]Q(1,1,t) = \frac{|\mS|^k}{\sqrt{k}}\left(\frac{A_1(1)-A_{-1}(1)}{A_1(1)}\cdot \sqrt{\frac{|\mS|}{\pi B_1(1)}} + O\left(\frac{1}{k}\right)\right) \]
 as $k \rightarrow \infty$.

\begin{remark}
Recent work of Bostan et al.~\cite{BostanChyzakHoeijKauersPech2017} proves the guessed differential equations of Bostan and Kauers, and expresses the generating functions of these lattice path models in terms of explicit hypergeometric functions, however even with these representations they are not able to prove all conjectured asymptotics.  For instance, the authors of that paper show~\cite[Conjecture 2]{BostanChyzakHoeijKauersPech2017} that the counting sequence of the positive drift model with step set $\mS = \{(0,-1),(-1,1),(1,1)\}$ has dominant asymptotics of the form $\frac{\sqrt{3}}{2\sqrt{\pi}} 3^k k^{-1/2}$ if and only if the integral
\begin{align*} 
I := \int_0^{1/3} & \left\{  \frac{(1-3v)^{1/2}}{v^3(1+v^2)^{1/2}}\left[ 1 
+ (1-10v^3) \cdot {}_2F_1\left(\left. \genfrac{}{}{0pt}{}{3/4, 5/4}{1} \right| 64v^4 \right)  \right. \right. \\
& \hspace{1.2in} \left.\left. + 6v^3(3-8v+14v^2) \cdot {}_2F_1\left(\left. \genfrac{}{}{0pt}{}{5/4, 7/4}{2} \right| 64v^4 \right)\right]   
- \frac{2}{v^3} + \frac{4}{v^2} \right\} dv 
\end{align*}
has the value $I=1$ (see that paper for definitions of the notation used).  As an indirect corollary of our asymptotic results, we thus determine the values of certain complicated integral expressions involving hypergeometric functions.
\smallskip

For the model defined by $\mS = \{(0,-1),(-1,1),(1,1)\}$, Bostan et al.~use the value of $I$ only to determine the value of the leading constant $\frac{\sqrt{3}}{2\sqrt{\pi}}$, so bounding $I$ away from 0 determines asymptotics up to a constant factor.  The situation is worse for some negative drift models, such as the one with step set $-\mS = \{(0,1),(1,-1),(-1,-1)\}$, where the values of similar integrals are needed exactly to show cancellation in apparently dominant asymptotic terms and determine the correct exponential growth.
\end{remark}

\subsubsection*{Negative Drift Models}
Now let $\mS$ be a step set which is symmetric over the $y$-axis and has more steps with negative $y$-coordinate than positive $y$-coordinate.  Then $A_{-1}(1) > A_1(1)$ and the point
\[ \bp = \left(1,\sqrt{A_1(1)/A_{-1}(1)}, \frac{\sqrt{A_{-1}(1)/A_1(1)}}{S(1,\sqrt{A_{-1}(1)/A_1(1)})} \right) \]
is minimal by Proposition~\ref{prop:almostsymmin}.  As shown earlier in Proposition~\ref{prop:smoothmincrit}, any minimal critical smooth point determines the minimum of $|xyt|^{-1}$ on the closure of the domain of convergence of $F(x,y,t)$.  Furthermore, as $1/H_1$ is combinatorial all minimizers of $|xyt|^{-1}$ on the closure of the domain of convergence of $1/H_1$ (and thus also on $\overline{\mD}$) have the same coordinate-wise modulus as $\bp$ by Lemma~\ref{lemma:comb_two_min_crit}. Note that for these models, $A_1(x) = 1$ or $A_1(x) = x + \ox$.  

\paragraph{Suppose first that $A_1(x) = 1$.}  
Then any point $(a,b,c) \in \mV \cap D(\bp)$ must satisfy $|a| = 1$, $|b| = \bp_y$, and
\[ |abS(a,\overline{b})| = |\bp_y S(1,1/\bp_y)|, \]
where $\bp_y$ is the second coordinate of $\bp$.  Since $xyS(x,\oy)$ is a polynomial with positive real coefficients, and $\bp_y$ is positive and real,
it must be the case that every term in $xyS(x,\oy)$ (when expanded as a polynomial) has the same complex argument when the substitution $x=a$ and $y=b$ is made (this follows from the complex triangle inequality, as in the proof of Proposition~\ref{prop:symmCritMin} in Chapter~\ref{ch:SymmetricWalks}).  Since $\mS$ is symmetric over the $x$-axis, $xyS(x,\oy)$ contains a term of the form $y^j$ and $x^2y^j$ for some integer $j$, meaning that $a \in \{\pm1\}$.  Once it is known that $a$ is real, this argument shows that at least one of $b$ or $b^2$ must be real.  In any case, since the modulus of $b$ is fixed, $\bp$ is finitely minimal.

Corollary~\ref{cor:smoothAsm} then implies that dominant asymptotics of the number of walks in our model ending anywhere is given by computing the smooth point asymptotic contributions (written explicitly in Equation~\eqref{eq:smoothAsm}) of $\bp$ and any other minimal critical points in the set
\[ E = \left\{ \left(x,y, \ox \,\oy S(x,\oy)^{-1} \right) : x=\pm1, \quad y=\pm\sqrt{A_1(x)/A_{-1}(x)}, \quad |S(x,\oy)| = S(|x|,|\oy|) \right\}. \]
Note that the numerator of $F(x,y,t)$ vanishes at each of these points, so one must compute the higher order terms described in Equation~\eqref{eq:smoothAsm}.  Computing these asymptotic expansions proves the asymptotics conjectured by Bostan and Kauers.

\begin{example} 
Consider the model defined by step set $\mS = \{ (0,1),(-1,-1),(1,-1) \}$.  Here we have 
\[ [t^k]Q(1,1,t) = [t^k]\Delta \left( \frac{(1+x) \left(1 - y^2 (\ox + x)\right)}{(1-y)(1-t(x + y^2+x^2y^2))} \right) = 
[t^{k+1}]\Delta \left( \frac{yt(1+x) \left(x - y^2 (x^2 + 1)\right)}{(1-y)(1-t(x + y^2+x^2y^2))} \right), \]
and the set $E$ contains four minimal critical points:
\[ \bp_1 = \left(1,\frac{1}{\sqrt{2}},\frac{1}{2}\right), \quad \bp_2 = \left(1,\frac{-1}{\sqrt{2}},\frac{1}{2}\right), \quad \bp_3 = \left(-1,\frac{i}{\sqrt{2}},\frac{-1}{2}\right), \quad \bp_4 = \left(-1,\frac{-i}{\sqrt{2}},\frac{-1}{2}\right).\]
Using Maple\footnote{Worksheet available at~\websiteurl .} to determine the terms in Equation~\eqref{eq:smoothAsm}, one can calculate the asymptotic contributions
\[ \Phi_{\bp_1} = \frac{4(3\sqrt{2}+4)}{\pi} \cdot \frac{(2\sqrt{2})^k}{k^2}\left(1 + O\left(\frac{1}{k}\right)\right)\qquad
 \Phi_{\bp_2} = \frac{4(3\sqrt{2}-4)}{\pi} \cdot \frac{(-2\sqrt{2})^k}{k^2}\left(1 + O\left(\frac{1}{k}\right)\right)\]
\[\Phi_{\bp_3}, \Phi_{\bp_4} = O\left(\frac{(2\sqrt{2})^k}{k^3}\right) \]
after shifting the index $k$, so that the number of walks of length $k$ has asymptotics
\begin{align*} 
\frac{(2\sqrt{2})^k}{k^2} \cdot \frac{4}{\pi}\left(4(1-(-1)^k) + 3\sqrt{2}(1+(-1)^k)  + O\left(\frac{1}{k}\right)\right) 
= 
\begin{cases}  
\frac{(2\sqrt{2})^k}{k^2} \cdot \left( \frac{24\sqrt{2}}{\pi} + O\left(\frac{1}{k}\right)\right) &: k \text{ even} \\ 
\frac{(2\sqrt{2})^k}{k^2} \cdot \left( \frac{32}{\pi} + O\left(\frac{1}{k}\right)\right) &: k \text{ odd}
\end{cases}
\end{align*}
Note that the original table of Bostan and Kauers~\cite{BostanKauers2009} only had the value of the leading term when $k$ was even.
\end{example}

\paragraph{Now suppose that $A_1(x) = x + \ox$,} 
so that $\mS$ is one of the models
\[ \diagrF{NE,NW,SE,SW,S}  \qquad\qquad \diagrF{NE,NW,E,W,SE,SW,S} \]
When $(a,b,c) \in \mV$ and $b=1$ or $A_1(a)=0$ then $|abc|^{-1} = |S(a,\overline{b})|$ is larger the value of $|xyt|^{-1}$ at $\bp$. Furthermore, although $\bp$ is no longer finitely minimal, if $(x,y,t) \in T(\bp)$ and $(x,y,t)$ is not in the finite set of minimal critical points $E$ then $A_1(x) =0$ so $x$ is bounded away from 1.  Thus, each smooth minimal critical point is isolated in $T(\bp) \cap \mV$, and Corollary~\ref{cor:smoothAsmCP} implies asymptotics can still be determined using Equation~\eqref{eq:smoothAsm}.

% \diagrF{N,W,SE} \qquad \diagrF{E,SE,W,NW} \qquad \diagrF{NW,SE,N,S,E,W}
\section*{Sporadic Examples}
\label{sec:Sporadic}

\subsection*{Models $\mS = \{(0,1),(-1,0),(1,-1)\}$ \\ and $\mS = \{(-1,1), (1,-1), (0,1),(0,-1),(1,0),(-1,0)\}$}

The kernel method gives the generating functions of these models as diagonals of the rational functions
\[ F_1(x,y,t) = \frac{(x^2-y)(x-y^2)(1-\ox\,\oy)}{(1-x)(1-y)(1-txy(\oy+y\ox+x))} \]
and 
\[ F_2(x,y,t) =  \frac{(x^2-y)(x-y^2)(1-\ox\,\oy)}{(1-x)(1-y)(1-txy(x+y+x\oy+y\ox+\ox+\oy))}, \]
respectively.  These rational functions admit $\bp = (1,1,1/3)$ and $\bs = (1,1,1/6)$ as minimal critical points where the singular variety forms a complete intersection, however the numerators of $F_1$ and $F_2$ vanish at these points so Theorem~\ref{thm:compintasm} cannot be applied.

Writing 
\[ x^2-y = (x-1)(x+1)-(y-1), \]
the diagonals $\Delta F_1$ and $\Delta F_2$ can each be turned into a sum of two rational diagonals with simpler singular varieties.  Dominant asymptotics should still be determined by the critical points $\bp$ and $\bs$, however these points are not finitely minimal, the numerators of these rational functions vanish at the points, and the vector $\bone$ lies on the boundary $\partial N(\bz)$.  Thus, Theorem~\ref{thm:resasm} does not apply, and a more detailed analysis is needed to determine asymptotics using the methods of ACSV.  Since asymptotics of these models were given by Bousquet-Mélou and Mishna~\cite{Bousquet-MelouMishna2010}, we do not pursue this here.

\subsection*{Model $\mS = \{(1,0),(1,-1),(-1,0),(-1,1)\}$}
For this model we obtain the diagonal expression
\[ Q(1,1,t) = \Delta \left( \frac{(x+1)(\ox^2-\oy)(x-y)(x+y)}{1-xyt(x+x\oy+y\ox+\ox)} \right) \]
which is easy to analyze, since the denominator is smooth.  There are two points which satisfy the smooth critical point equations, $\bp_1 = (1,1,1/4)$ and $\bp_2 = (-1,1,1/4)$, both of which are finitely minimal (using the same argument as for negative drift models).  In fact, only $\bp_1$ determines the dominant asymptotic term, which can be anticipated by noting that the numerator vanishes to order 2 at $\bp_1$ but order 3 at $\bp_2$. Ultimately, we obtain dominant asymptotics
\[ [t^k]Q(1,1,t) = \frac{4^k}{k^2} \cdot \frac{8}{\pi}\left(1 + O\left(\frac{1}{k}\right)\right). \] 
Weighted versions of this model are studied in detail in Chapter~\ref{ch:WeightedWalks}.  As mentioned previously, for this model Bousquet-Mélou and Mishna~\cite[Proposition 11]{Bousquet-MelouMishna2010} gave asymptotics for the number of walks ending anywhere, those ending on each boundary axis, and those ending at the origin.

\begin{table}
\centering
{ \small
\begin{tabular}{ | c | c | c | c @{ \hspace{0.01in} }@{\vrule width 1.2pt }@{ \hspace{0.01in} }  c | c | c | c |  }
  \hline
   $\mS$ & $Q(0,1,t)$ & $Q(1,0,t)$ & $Q(0,0,t)$ & $\mS$ & $Q(0,1,t)$ & $Q(1,0,t)$ & $Q(0,0,t)$ \\ \hline
  &&&&&&& \\[-5pt] 
  \diag{N,S,E,W}  & $\frac{8}{\pi} \cdot \frac{4^k}{k^2}$ & 
  $\frac{8}{\pi} \cdot \frac{4^k}{k^2}$ &  
  $\delta_k \frac{32}{\pi} \cdot \frac{4^k}{k^3}$ &
  \diag{NE,SE,NW,SW}  & $\delta_k \frac{4}{\pi} \cdot \frac{4^k}{k^2}$ & 
  $\delta_k \frac{4}{\pi} \cdot \frac{4^k}{k^2}$ &  
  $\delta_k \frac{8}{\pi} \cdot \frac{4^k}{k^3}$ \\
  \diag{N,S,NE,SE,NW,SW} &$\frac{3\sqrt{6}}{2\pi} \cdot \frac{6^k}{k^2}$ & 
  $\delta_k \frac{2\sqrt{6}}{\pi} \cdot \frac{6^k}{k^2}$ &  
  $\delta_k \frac{3\sqrt{6}}{\pi} \cdot \frac{6^k}{k^3}$ &
  \diag{N,S,E,W,NW,SW,SE,NE}  & $\frac{32}{9\pi} \cdot \frac{8^k}{k^2}$ & 
  $\frac{32}{9\pi} \cdot \frac{8^k}{k^2}$ &  
  $\frac{128}{27\pi} \cdot \frac{8^k}{k^3}$  \\
  \diag{NE,NW,S}  & $\frac{3\sqrt{3}}{4\sqrt{\pi}}  \frac{3^k}{k^{3/2}}$ & 
  $\delta_k \frac{4\sqrt{2}}{\pi}  \frac{(2\sqrt{2})^k}{k^2}$ & 
  $\epsilon_k \frac{16\sqrt{2}}{\pi} \frac{(2\sqrt{2})^k}{k^3}$ &
  \diag{N,NW,NE,S}  & $\frac{8}{3\sqrt{\pi}} \frac{4^k}{k^{3/2}}$ & 
  $\delta_k \frac{4\sqrt{3}}{\pi}  \frac{(2\sqrt{3})^k}{k^2}$ & 
  $\delta_k \frac{12\sqrt{3}}{\pi}  \frac{(2\sqrt{3})^k}{k^3}$ \\
  % %%%%%%%%%%%%%
  \diag{NE,NW,E,W,S} & $\frac{5\sqrt{10}}{16\sqrt{\pi}}  \frac{5^k}{k^{3/2}}$  & 
  $\frac{\sqrt{2}A^{3/2}}{\pi}  \frac{(2A)^k}{k^2}$ & 
  $\frac{2A^{3/2}}{\pi}  \frac{(2A)^k}{k^3}$ &
  \diag{N,NE,NW,SE,SW}  & $\frac{5\sqrt{10}}{24\sqrt{\pi}}  \frac{5^k}{k^{3/2}}$ & 
  $\delta_k \frac{4\sqrt{30}}{5\pi}  \frac{(2\sqrt{6})^k}{k^2}$ & 
  $\delta_k \frac{24\sqrt{30}}{25\pi}  \frac{(2\sqrt{6})^k}{k^3}$ \\
  % %%%%%%%%%%%%%
  \diag{N,NW,NE,E,W,S} & $\frac{\sqrt{3}}{\sqrt{\pi}} \frac{6^k}{k^{3/2}}$ &
  $\frac{2\sqrt{3}B^{3/2}}{3\pi} \frac{(2B)^k}{k^2}$ & 
  $\frac{2B^{3/2}}{\pi} \frac{(2B)^k}{k^3}$ &
  \diag{N,E,W,NE,NW,SE,SW} & 
  $\frac{7\sqrt{21}}{54\sqrt{\pi}} \frac{7^k}{k^{3/2}}$ & 
  $ \frac{D}{285\pi} \frac{(2K)^k}{k^2}$ & 
  $\frac{2E}{1805\pi} \frac{(2K)^k}{k^3}$ \\[+2mm] 
\hline
\end{tabular}
}
{\footnotesize
\[ A = 1+\sqrt{2}, \qquad B = 1+\sqrt{3}, \qquad K = 1+\sqrt{6}, \qquad D = (156+41\sqrt{6})\sqrt{23-3\sqrt{6}}, \qquad E = (583+138\sqrt{6})\sqrt{23-3\sqrt{6}}\]}
\vspace{-0.45in}

\[ \delta_k = \begin{cases} 1 &: k \equiv 0 \mod 2 \\ 0 &: otherwise \end{cases} 
\qquad \epsilon_k = \begin{cases} 1 &: k \equiv 0 \mod 4 \\ 0 &: otherwise \end{cases} 
\]
\vspace{-0.2in}

\caption[Asymptotics of boundary returns for highly symmetric and positive drift models]{Asymptotics of boundary returns for the highly symmetric and positive drift cases.} \label{tab:boundary1}
\end{table}

\section{Boundary Returns and Excursions}

In addition to giving diagonal expressions for the number of walks of a fixed length, Theorem~\ref{thm:nzOrbDiag} gives diagonal expressions for the number of walks in a model ending on the $x$-axis (with generating function $Q(0,1,t)$), on the $y$-axis (with generating function $Q(1,0,t)$), and at the origin (with generating function $Q(0,0,t)$).  Not only does this allow one to determine the asymptotics of such sequences, given in Table~\ref{tab:boundary1} and Table~\ref{tab:boundary2}, but the close connection between the diagonal expressions for $Q(1,1,t)$, $Q(1,0,t)$, $Q(0,1,t)$, and $Q(0,0,t)$ helps analytically explain some of the observed connections in the asymptotics.
\smallskip

For example, if $\mS$ is a highly symmetric model then there are rational diagonal representations
\[ Q(1,1,t) = \Delta\left(\frac{(1+x)(1+y)}{1-txyS(x,y)}\right) \qquad\text{and}\qquad Q(0,0,t) = \Delta\left(\frac{(1-x^2)(1-y^2)}{1-txyS(x,y)}\right). \]
Both of the rational functions listed here admit the same finitely minimal smooth critical point $(1,1,1/|\mS|)$, meaning asymptotics of their diagonal sequences will have the same exponential growth, but the numerator in the expression for $Q(0,0,t)$ vanishes to second order at this point while the numerator in the expression for $Q(1,1,t)$ does not.  This matches the observation that the sub-exponential growth for the number of walks returning to the origin has the form $Ck^{-3}$ (up to possible periodicity) while for the number of walks ending anywhere the sub-exponential growth has the form $C'k^{-1}$, for constants $C,C'>0$.
\smallskip

\begin{table}
\centering
{\small
\begin{tabular}{ | c | c @{ \hspace{0.01in} }@{\vrule width 1.2pt }@{ \hspace{0.01in} } c | c |  }
  \hline
   $\mS$ & $Q(0,1,t)$ & $\mS$ & $Q(0,1,t)$ \\ \hline
  &&& \\[-5pt] 
  \diag{N,SE,SW}  & $
    \left(\epsilon_k \frac{448\sqrt{2}}{9\pi}  +   
    \epsilon_{k-1} \frac{640}{9\pi}  +
    \epsilon_{k-2} \frac{416\sqrt{2}}{9\pi} + 
    \epsilon_{k-3} \frac{512}{9\pi}  
    \right)\cdot \frac{(2\sqrt{2})^k}{k^3} 
   $ & 
  \diag{S, SW, SE, N}  & $\left(\delta_k \frac{36\sqrt{3}}{\pi} + \delta_{k-1} \frac{54}{\pi} \right)\cdot \frac{(2\sqrt{3})^k}{k^3}$ \\
  % %%%%%%%%%%%%%
  \diag{SE, SW, E, W, N} & $\frac{4A^{7/2}}{\pi} \cdot \frac{(2A)^k}{k^3}$  &
  \diag{S, SE, SW, NE, NW}  & $\left( \delta_k \frac{72\sqrt{30}}{5\pi} + \delta_{k-1}\frac{864\sqrt{5}}{25\pi} \right) \cdot \frac{(2\sqrt{6})^k}{k^3}$\\
  % %%%%%%%%%%%%%  
  \diag{S, SW, SE, E, W, N} & $\frac{3B^{7/2}}{2\pi} \cdot \frac{(2B)^k}{k^3}$  &
  \diag{NE, NW, E, W, SE, SW, S} & 
  $\frac{6(4571+1856\sqrt{6})\sqrt{23-3\sqrt{6}}}{1805\pi} \cdot \frac{(2K)^k}{k^3}$ \\[+1mm]
  % %%%%%%%%%%%%%%%
\hline
\end{tabular}
}
\caption[Asymptotics of boundary returns for negative drift models]{Asymptotics of $Q(0,1,t)$ for the negative drift cases.  The asymptotics of $Q(1,0,t)$ and $Q(0,0,t)$ for each model $\mS$ are the same as the corresponding asymptotics for $-\mS$ listed in Table~\ref{tab:boundary1}.  Definitions of the constants given here are listed underneath Table~\ref{tab:boundary1}.} \label{tab:boundary2}
\end{table}

Similarly, if $\mS$ is symmetric over one axis then 
\[ Q(1,1,t) = \Delta\underbrace{\left(\frac{(1+x)\left(A_1(x)-y^2A_{-1}(x)\right)}{A_1(x)(1-y)(1-txyS(x,\oy))} \right)}_{F(x,y,t)}. \]
When $\mS$ is a negative drift model, asymptotics of this diagonal sequence are determined by a smooth minimal critical point $\bp_{\mS}$ whose $y$-coordinate is smaller than 1.  The rational functions $(1-x)F(x,y,t),$ $(1-y)F(x,y,t)$, and $(1-x)(1-y)F(x,y,t)$ all admit the same smooth minimal critical point, meaning the number of walks returning to either axis or the origin will have the same exponential growth and asymptotics are given by the same argument as above.

In contrast, when $\mS$ is a positive drift model $\bp_{\mS}$ is not minimal since this point has $y$-coordinate larger than 1 and the denominator of $F$ contains $1-y$ as a factor.  As we saw, asymptotics are thus determined by minimal critical multiple points of the singular variety.  The rational functions $(1-y)F(x,y,t)$ and $(1-x)(1-y)F(x,y,t)$, however, admit $\bp_{\mS}$ as a smooth minimal critical point since the factor of $1-y$ in the denominator is canceled, meaning the exponential growth for the number of walks which end at the origin or on the $y$-axis is smaller than the exponential growth for those ending anywhere.  Asymptotics for these cases are proven using an argument analogous to that of the negative drift models above.

The analysis for the number of walks returning to the $x$-axis (or origin) for the positive drift models is more difficult.  For each of these models, the generating function counting walks returning to the $x$-axis is the diagonal of $(1-x)F(x,y,t)$.  The numerator of this rational function vanishes at the minimal critical point $\bs$, where $x=1$, so Theorem~\ref{thm:resasm} gives only an upper bound on the asymptotic growth of the diagonal sequence.  A more refined argument taking advantage of the form of the rational functions we consider can be used to give asymptotics of boundary returns for these models (see the analysis of weighted directed lattice path models in Section~\ref{sec:DirAsm} of Chapter~\ref{ch:WeightedWalks} for details on the computation in a similar case).

%%%%%%%%%%%%%%%%%%%%%%%%%%%
% Chapter 11
%%%%%%%%%%%%%%%%%%%%%%%%%%%
\chapter[Centrally Weighted Lattice Path Models]{Centrally Weighted Lattice Path Models}
\label{ch:WeightedWalks}
\vspace{-0.2in}

\begin{center}Text in this chapter is adapted from an article of Courtiel, Melczer, Mishna, and Raschel~\cite{CourtielMelczerMishnaRaschel2016}.\end{center}

\setlength{\epigraphwidth}{3.5in}
\epigraph{She had not known the weight until she felt the freedom.}{Nathaniel Hawthorne, \emph{The Scarlet Letter}}

In Chapters~\ref{ch:SymmetricWalks} and~\ref{ch:QuadrantLattice} we saw how combinatorial properties of a lattice path model, such as symmetries of its step set and the number of steps moving towards or away from the boundary of its restricting region, affect asymptotics of its counting sequence.  These results were obtained by deriving uniform diagonal expressions which hold for models with certain properties, and then using tools from analytic combinatorics in several variables.  Furthermore, this approach allowed us to analytically understand certain observed behaviour in the asymptotics of lattice path models, such as why walks returning to the origin or boundary regions have the same or different exponential growth as walks ending anywhere.

In this chapter we study a natural generalization of the lattice path models previously considered: a variety of models with weighted step sets. Given a finite step set $\mS \subset \mathbb{Z}^n$ we can associate a positive real weight $a_{\bss}$ to each step $\bss \in \mS$ and consider the weighted multivariate generating function
\[Q_{\ba}(\bz,t)= \sum_{\substack{\bw \textrm{ walk in $\mathbb{N}^n$ starting at }\bzer,\\\textrm{ending at }(i,j), \textrm{ of length }k}} \, \prod_{\substack{ \bss \textrm{ step in } \bw \\ \textrm{(with multiplicity)}}} \hspace{-22pt} a_{\bss} \hspace{15pt} \, \bz^{\bi} t^k. \]
We will also be interested in the weighted generating function encoding walks beginning at a fixed point $\bj \in \mathbb{N}^n$ other than the origin:
\[Q^{\bj}_{\ba}(\bz,t)= \sum_{\substack{\bw \textrm{ walk in $\mathbb{N}^n$ starting at }\bj,\\\textrm{ending at }\bi, \textrm{ of length }k}} \, \prod_{\substack{ \bss \textrm{ step in } \bw \\ \textrm{(with multiplicity)}}} \hspace{-22pt} a_{\bss} \hspace{15pt} \, \bz^{\bi} t^k. \]
Studying weighted step sets allows for a deep understanding of model behaviour, for example illustrating sharp transitions (or \emph{phase changes}) in asymptotic behaviour as the weights are varied continuously.  Weighted models also capture probabilistic results when the weights are interpreted as transition probabilities related to each step.  One interesting feature of a model which captures a large amount of information is the weighted vector sum of its steps
\[ \mathbf{d}_{\mS} = \sum_{\bss \in \mS} a_{\bss}\bss, \]
known as the \emph{drift} of the model.  To prevent degeneracy, we always assume that there are steps moving forwards and backwards in each coordinate.

In order to obtain uniform diagonal expressions for weighted families of lattice path models, it is necessary to place restrictions on the types of weightings which are allowed.  The focus of this chapter is on \emph{central weightings}, which are weightings of a step set such that the weight of a path on those steps depends only on its length, start, and end points\footnote{A trivial example of a central weighting is one in which every step is given the same weight. We will see several more interesting examples in this chapter.}.  We will describe several nice properties of these weightings, including the ability to express the generating function of a parametrized weighted model in terms of the generating function of the underlying unweighted model with weighted variables\footnote{Central weightings are also related to the probabilistic notion of a Cramér transform.}.  Kauers and Yatchak~\cite{KauersYatchak2015} computationally investigated weighted lattice path models with short steps in the quarter plane, and found what they conjectured to be a finite list of families containing all models with (weighted) D-finite generating functions.  All but one of these families correspond to equivalence classes of short step models in the quarter plane under central weightings, which will be discussed in more detail below.

\begin{figure}\center
\includegraphics[width=.12\textwidth]{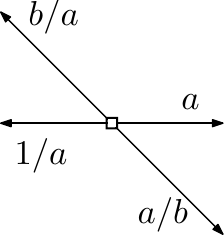}
\caption{A parametrized central weighting of the Gouyou-Beauchamps step set.}
\label{fig:GB_weights}
\end{figure} 

 The first two sections of this chapter focus on the \emph{Gouyou-Beauchamps} step set 
 \[\mS = \{ (-1,0), (1,0), (-1, 1), (1,-1)\}\] 
 when its steps are given weights $1/a$, $a$, $b/a$, $a/b$, as illustrated in Figure~\ref{fig:GB_weights}.  For any $a,b>0$, this defines a central weighting of $\mS$ and, in Section~\ref{sec:central}, we show that every central weighting on this set of steps has this form up to a uniform scaling of weights.  Asymptotics of the unweighted model, corresponding to $a=b=1$, were given at the end of Section~\ref{sec:Sporadic} in Chapter~\ref{ch:QuadrantLattice}. We describe our main asymptotic result on the Gouyou-Beauchamps model, together with several applications, in Section~\ref{sec:GBresults}.  These results are obtained by deriving a diagonal expression for the generating function $Q^{i,j}_{\ba}(1,1,t)$ enumerating the number of walks in a model and using the theory of ACSV; the analysis is detailed in Section~\ref{sec:GBproofs}. 

Section~\ref{sec:central} contains our results on general central weightings.  For any fixed step set $\mS$ we characterize the weightings of $\mS$ which are central, and show that the number of parameters defining a central weighting is always equal to the dimension plus one\footnote{One of these parameters will correspond to a uniform scaling of the weights, so we can specialize this parameter to 1 and still recover asymptotics in the general case (see below).}.  After characterizing the centrally weighted models corresponding to a fixed unweighted model, we show how to partition this collection of models into \emph{universality classes} based on a model's sub-exponential growth. Studying these universality classes illustrates a connection between the ACSV enumerative approach and related probabilistic approaches.  Furthermore, we show how to express (exactly and asymptotically) the numbers of centrally weighted walks returning to the origin in terms of the number of unweighted ones. These relations have strong consequences at the generating function level, and connect the (relatively unstudied) weighted and (well studied) unweighted generating functions.

\section{Results on Centrally Weighted Gouyou-Beauchamps Models}
\label{sec:GBresults}
The main result we derive on the asymptotics of centrally weighted Gouyou-Beauchamps walks is the following.

\begin{theorem}
\label{thm:GB_main_asymptotic_result}
Fix constants $a,b>0$ and consider the Gouyou-Beauchamps model with weights defined by Figure~\ref{fig:GB_weights}.  As $k\to\infty$ the number of weighted walks in this model of length $k$, starting from $(i,j)$ and ending anywhere while staying in the non-negative quadrant, satisfies
\begin{equation} \label{eq:GB_main_asymptotic_result}
[t^k]Q_{a,b}^{i,j}(1,1;t) = V^{[k]}(i,j)\cdot \rho^k\cdot  k^{-\alpha}\cdot\left(1+O\left(\frac{1}{k}\right)\right),
\end{equation}
where the exponential growth $\rho$ and the critical exponent $\alpha$ are given in Table~\ref{tab:cases}. The function $V^{[k]}(i,j)$ is given in Appendix~\ref{appendix:GB} and depends on $k$ only through its parity $[k]$ (if at all).
\end{theorem}

\begin{table}\center
\begin{tabular}{llcl}
Class & Condition & $\rho$ & $\alpha$ \\ \hline 
{\bf Balanced} & $a=b=1$ & 4 & 2\\[+1mm]
{\bf Free} & $\sqrt{b}<a<b$ &  $\frac{(1+b)(a^2+b)}{ab}$ & 0 \\[+1mm]
{\bf Reluctant }&$a<1$ and $b<1$ & $4$ & 5\\[+1mm]
{\bf Axial} &$b=a^2>1$ &  $\frac{2(b+1)}{\sqrt{b}}$ & 1/2 \\[+1mm] 
		& $a=b>1$   & $\frac{(1+a)^2}{a}$ & 1/2\\[+1mm]
{\bf Transitional} & $a=1,b<1$ or $b=1,a<1$ & $4$ & 3 \\[+1mm]
{\bf Directed} & $b>1$ and $\sqrt{b}>a$ & $\frac{2(b+1)}{\sqrt{b}}$& 3/2 \\[+1mm]
               & $a>1$ and  $a>b$ & $\frac{(1+a)^2}{a}$ & 3/2 
\end{tabular}
\caption[Universality classes for weighted Gouyou-Beauchamps models]{The six different universality classes for weighted Gouyou-Beauchamps walks, under the weighting in Figure~\ref{fig:GB_weights}, together with asymptotic information from Theorem~\ref{thm:GB_main_asymptotic_result}.}
\label{tab:cases}
\end{table}

 \begin{figure}\center
\includegraphics[width=0.5\textwidth]{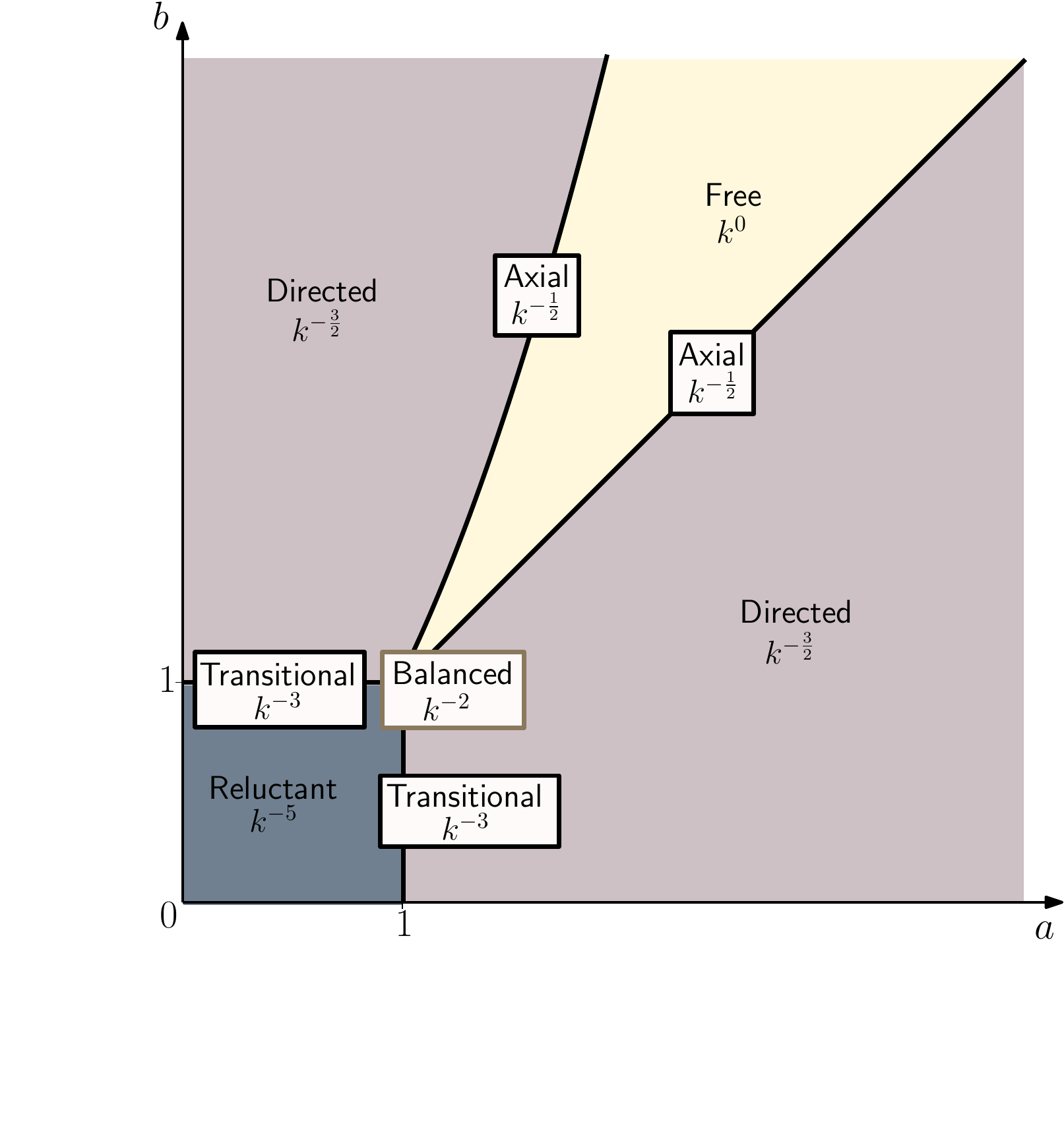}
   \vspace{-5mm}
	\caption[Universality classes for the weighted Gouyou-Beauchamps model]{The polynomial growth ($k^\alpha$ in Theorem~\ref{thm:GB_main_asymptotic_result}) for the weighted Gouyou-Beauchamps model given as a function of the  weight parameters $a$ and $b$. Each region corresponds to a universality class.}
	\label{fig:GB-ab-diagram}
\end{figure}

Although the exponential growth $\rho$ varies continuously with the weights $a$ and $b$, the critical exponent $\alpha$ only takes 6 values, and thus undergoes sharp transitions as $a$ and $b$ vary. We name the six regions in which $\alpha$ is constant, inducing six different \textit{universality classes} for centrally weighted Gouyou-Beauchamps walks:
\begin{itemize}
	\item the \textit{free} class corresponds to walks where the drift is in the interior of the first quadrant, with the name reflecting the fact that the models behave similar to models unrestricted to the first quadrant;
	\item the \emph{reluctant} class is defined by $a < 1$ and $b < 1$, and corresponds to walks which will have the smallest exponential growth;
	\item the \emph{transitional} class is the boundary of the reluctant region;
	\item the \textit{axial} class is the boundary of the free region;
	\item the \textit{balanced} class is defined by the drift being $(0,0)$; 
	\item all other models belong to the \textit{directed} class.
\end{itemize}
The weights $a$ and $b$ belonging to each class are shown in Figure~\ref{fig:GB-ab-diagram}, and these classes are defined more generally for other models in Section~\ref{sec:central}.  Note that the drift for a centrally weighted Gouyou-Beauchamps model under our parametrization is
\begin{equation} 
\mathbf{d}=(d_x,d_y) = \left(\frac{(1+b)(a^2-b)}{ab},\frac{(a+b)(b-a)}{ab}\right). \label{eq:drift} 
\end{equation} 
A general central weighting of the Gouyou-Beauchamps model is given by uniformly scaling the weight assigned to each step by a positive constant $\beta>0$.  Asymptotics of the central weighting corresponding to the scaling $\beta$ is then obtained by multiplying the right-hand side of Equation~\eqref{eq:GB_main_asymptotic_result} by $\beta^k$.

\subsection{Applications of Theorem~\ref*{thm:GB_main_asymptotic_result}}
Theorem~\ref{thm:GB_main_asymptotic_result} is a combinatorial result giving asymptotics of the total number of walks confined to the quarter plane for centrally weighted Gouyou-Beauchamps models. These models are so-named because Gouyou-Beauchamps~\cite{Gouyou-Beauchamps1986} discovered a simple hypergeometric formula for the (unweighted) walks that return to an axis. As shown by Gouyou-Beauchamps~\cite{Gouyou-Beauchamps1989}, they encode several combinatorial classes: the set of walks ending anywhere are in bijection with pairs of non-intersecting prefixes of Dyck paths, and the walks ending on the axis are in bijection with Young tableaux of height at most~$4$. This quarter plane lattice path model is also in bijection with the lattice path model having step set $\mS = \{(\pm1,0),(0,\pm1)\}$ which is restricted to the region $\{(x,y) : 0 \leq x \leq y \}$.  Thus, the Gouyou-Beauchamps model can be considered as a walk in a \emph{Weyl chamber} (see Gessel and Zeilberger~\cite{GesselZeilberger1992} for information on Weyl chamber walks).

We now illustrate several applications of this result.

\subsubsection*{Complexity of Random Generation}
Asymptotic results play a key role in the random generation of structures. Lumbroso et al.~\cite{LumbrosoMishnaPonty2016} describe an algorithm for randomly sampling lattice path models in the quadrant taking positive integer weights, with the expected running time of their algorithm depending only on the length of a walk and the critical exponent $\alpha$. This implies that models in the same universality class whose steps have (up to a uniform scaling) integer weights have the same cost of random generation\footnote{Here the probability of drawing a walk among all those of length $k$ is directly proportional to the weight of the walk (normalized to give a probability distribution).} under their algorithm.  Similarly, Bacher and Sportiello~\cite[Section 4.5]{BacherSportiello2016} describe an anticipated rejection random sampling algorithm with linear average complexity (and give the corresponding complexity distribution) when the number of weighted walks restricted to the quadrant grows like $C k^{-\alpha}$ times the number of unrestricted walks with the same weights, for $\alpha \in [0,1)$. Theorem~\ref{thm:GB_main_asymptotic_result} shows that this occurs in our context for the axial and free universality classes.  Uniformly sampled walks give a picture of how the weights impact the shape of the walk, and two examples of randomly generated walks are given in Figure~\ref{fig:GB_example}.

\begin{figure}\center
\hfill\includegraphics[width=.3\textwidth]{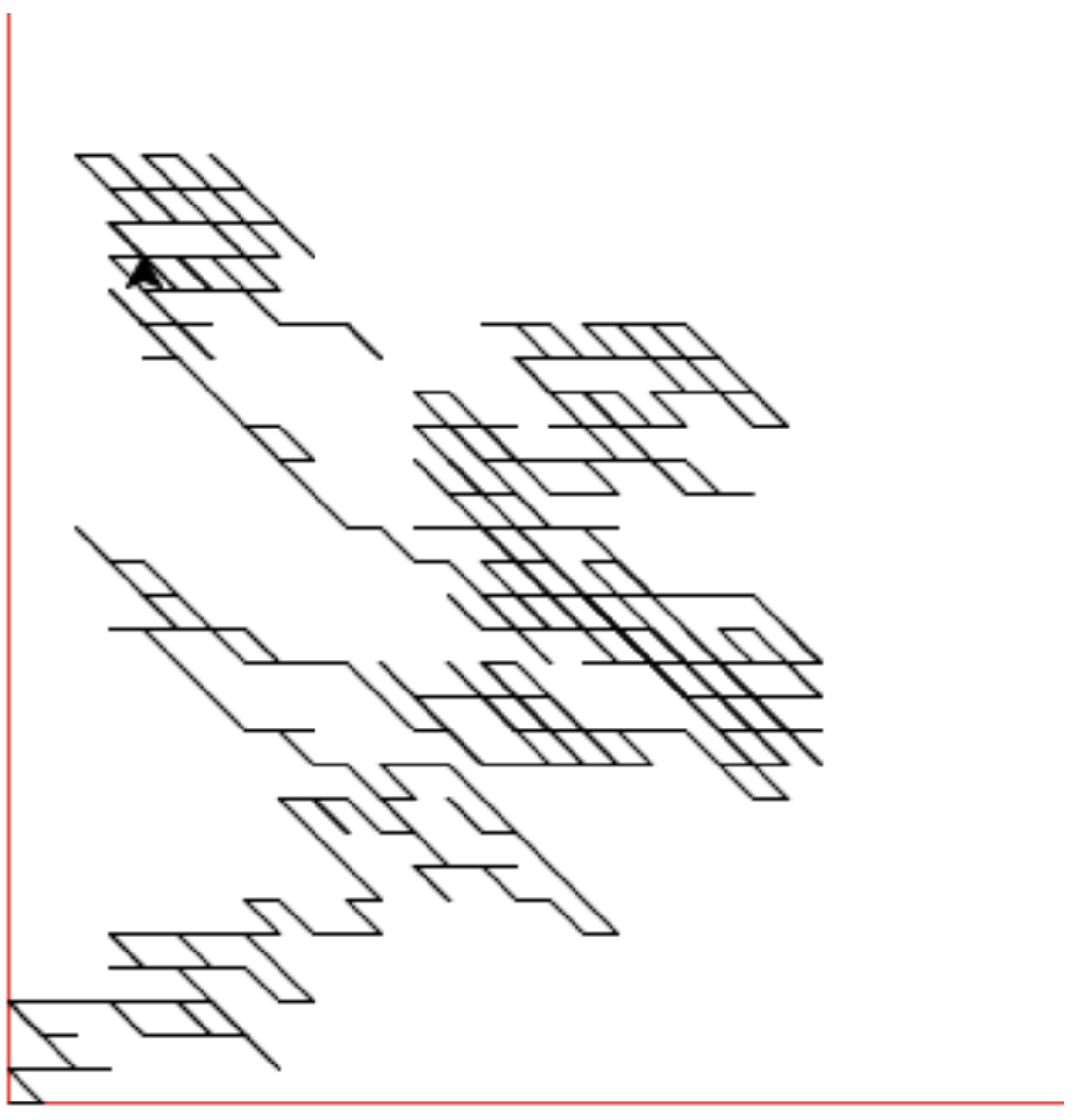}\hfill
\includegraphics[width=.3\textwidth]{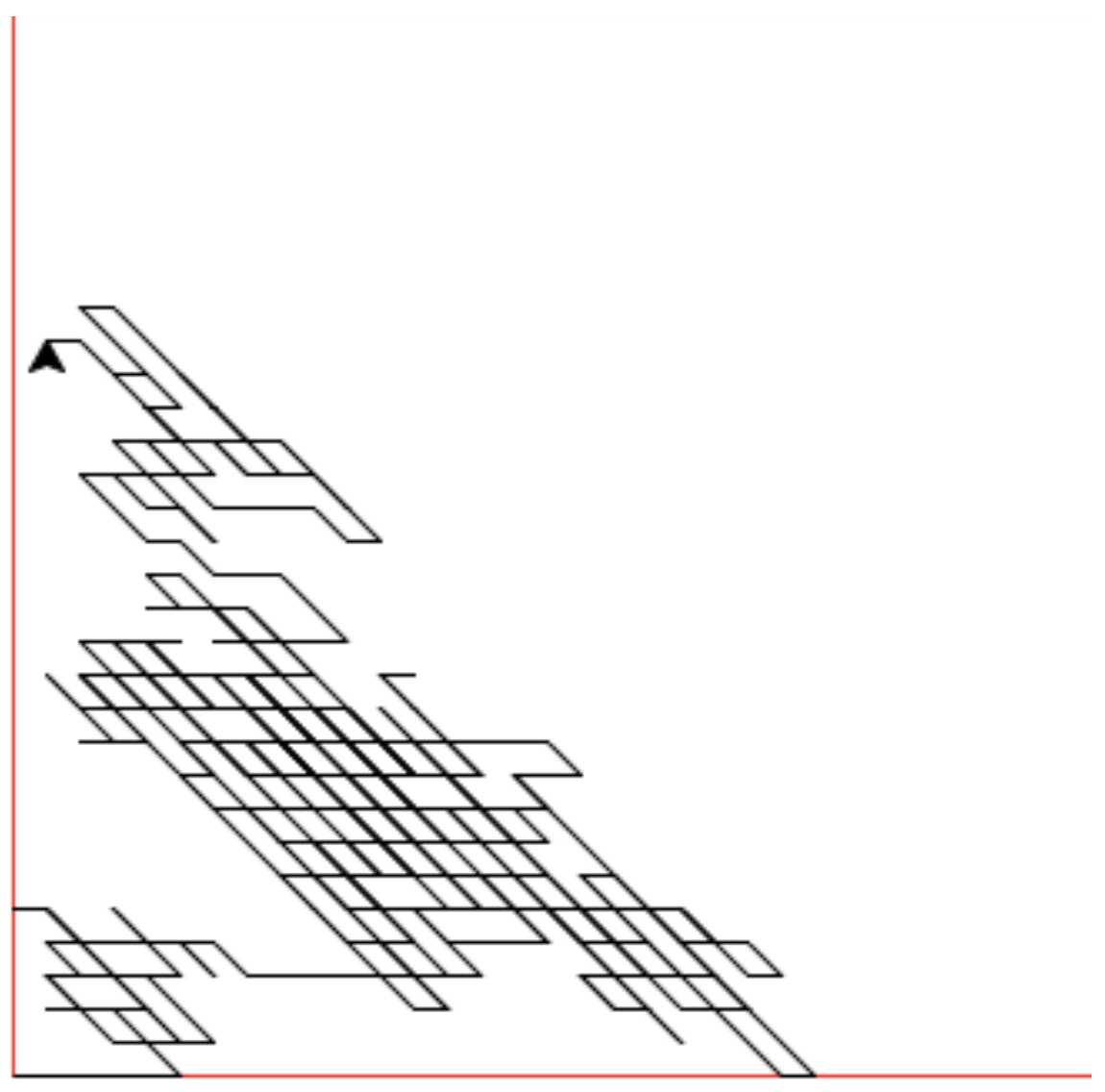}\hfill\mbox{}
\caption[Two uniformly sampled Gouyou-Beauchamps walks on 800 steps]{Two uniformly sampled Gouyou-Beauchamps walks on 800 steps: an unweighted (balanced) model (left) and a weighted (reluctant) model with parameters $a=\frac{1}{\sqrt{2}}$ and $b=1$ (right), which can be scaled to a model with step weights of 1 and 2 by uniformly multiplying each weight by $\sqrt{2}$.} 
\label{fig:GB_example}
\end{figure}

\subsubsection*{Applications to Probability Theory}
Suppose $\mS \subset \mathbb{Z}^2$ is a finite step set with positive weights $a_{\bss}$ which sum to one.  Then there exist identically distributed random variables $X_k$ for all non-negative integers $k$ which take the value $\bss \in \mS$ with probability $a_{\bss}$.  A \emph{random walk of length $k$} defined by $\mS$ is the random variable $S_k := X_1 + X_2 + \cdots + X_k$.  For any $i,j,p,q,k \in \mathbb{N}$, let 
\[ e_{(i,j) \rightarrow (p,q)}(k) := \text{ number of weighted walks on $\mS$ of length $k$ from $(i,j)$ to $(p,q)$.}  \]
An asymptotic result on the number of walks beginning and ending at two fixed points can be turned into a \emph{local limit theorem} by the formula
\begin{equation}
\label{eq:link_excursions}
     {\mathbb P[(i,j)+S_k=(p,q),\quad \tau>k]}=\frac{e_{(i,j) \rightarrow (p,q)}(k)}{|\mS|^k},
\end{equation}
where $\tau$ denotes the first exit time of the random walk $(S_k)$ from the quadrant: 
\begin{equation}
\label{eq:def_exit_time}
     \tau=\inf\{k\geq0:S_k\notin\mathbb{N}^2\}. 
\end{equation}     
In the same way, results on walks with prescribed length but no fixed endpoint can be written as
\begin{equation}
\label{eq:link_non-exit_probability}
     \mathbb P_{(i,j)}[\tau>k]=\frac{e_{(i,j) \rightarrow \mathbb N^2}(k)}{|\mS|^k}=\frac{\sum_{(p,q)\in\mathbb N^2}e_{(i,j) \rightarrow (p,q)}(k)}{|\mS|^k}.
\end{equation} 

In a recent influential paper, Denisov and Wachtel~\cite{DenisovWachtel2015} proved a local limit theorem for random walks in cones\footnote{The work of Denisov and Wachtel applies very generally; in the case of a general cone (not necessarily the two-dimensional quadrant) the exit time $\tau$ becomes the smallest value of $k$ such that a random walk $S_k$ leaves the cone under consideration.} --- characterizing asymptotics of the quantities in Equation~\eqref{eq:link_excursions} --- and obtained a precise estimate of the non-exit probability~\eqref{eq:link_non-exit_probability} for models with zero drift. Duraj~\cite{Duraj2014} established an asymptotic estimate of the non-exit probability for the walks in the \textit{reluctant} universality class (roughly speaking for the general case, the ones with negative drift). The exponential growth $\rho$ in Equation~\eqref{eq:GB_main_asymptotic_result} is computed by Garbit and Raschel~\cite{GarbitRaschel2016}: for a large class of cones and dimensions it is equal to the minimum of the Laplace transform of the increments of the random walk on the dual cone (see the end of this Chapter for details). Theorem~\ref{thm:GB_main_asymptotic_result} allows us to determine dominant asymptotics for a continuous family of processes containing all possible drift vectors\footnote{For any fixed drift, it is easy to find a set of weights realizing the drift by solving Equation~\eqref{eq:drift} for positive solutions $a$ and $b$ (possibly after scaling the drift by a positive constant, as we can uniformly scale the weights of our walks).  The fact that there will always be positive solutions for $a$ and $b$ solving Equation~\eqref{eq:drift} for all values of the drift follows from a simple algebraic argument.} (see Figure~\ref{fig:drift}).

\begin{figure}\center
\includegraphics[width=.4\textwidth]{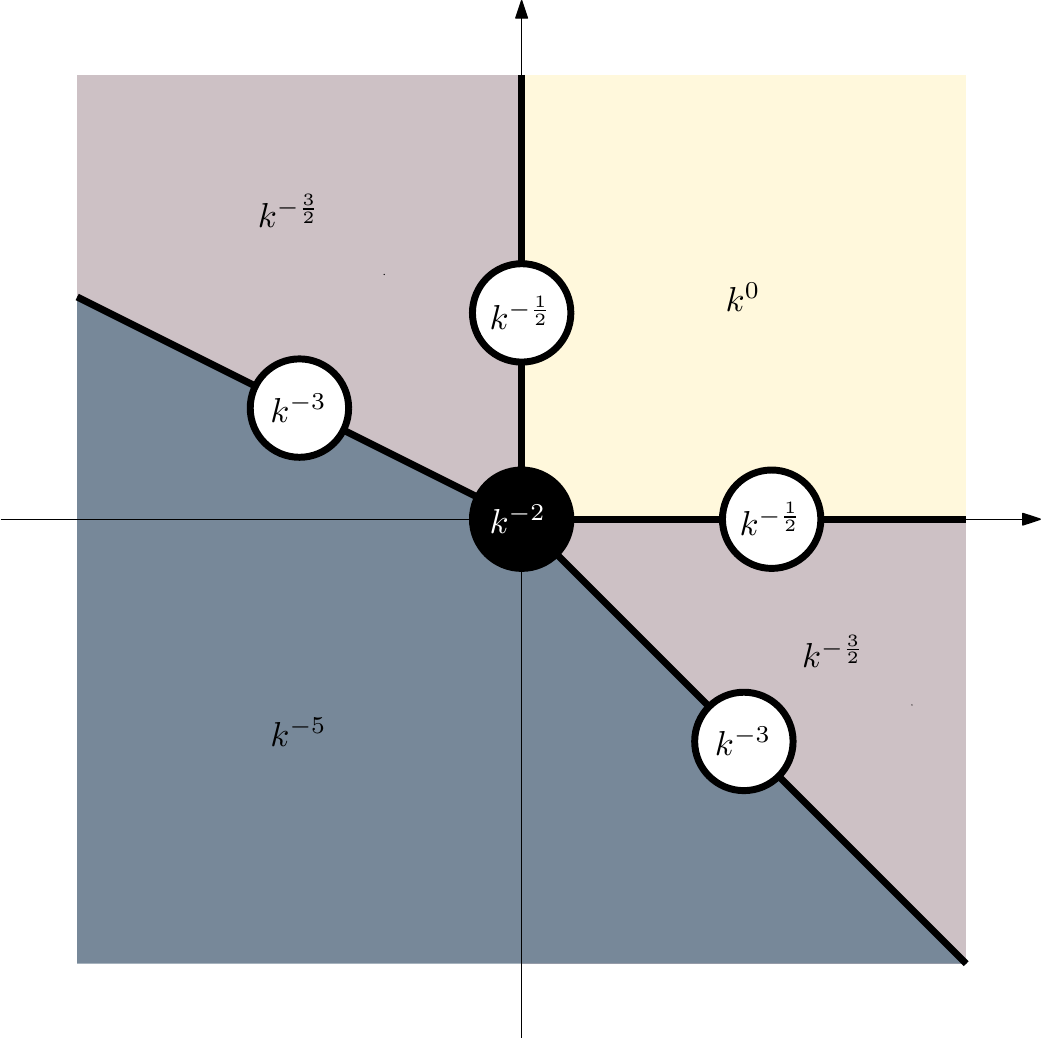}
\caption[Universality classes of the centrally weighted Gouyou-Beauchamps model as a function of drift]{The universality classes as a function of drift for the centrally weighted Gouyou-Beauchamps model prescribed by the weights in Figure~\ref{fig:GB_weights}. Because the regions are cones the diagram does not change when each weight is multiplied by a constant, meaning this diagram holds for every central weighting of the Gouyou-Beauchamps model.}
\label{fig:drift}
\end{figure}

\subsubsection{Construction of Discrete Harmonic Functions} 
The quantity $V^{[k]}(i,j)$ in Equation~\eqref{eq:GB_main_asymptotic_result} satisfies
\begin{equation}
\label{eq:harmonicity_V}
     \begin{split} 
     	\rho \cdot V^{[k+1]}(i,j)=(1/a)V^{[k]}(i-1,j)&+(b/a)V^{[k]}(i-1,j+1)\\&+aV^{[k]}(i+1,j)+(a/b)V^{[k]}(i+1,j-1),
     \end{split}
\end{equation}
which follows from the recurrence relations satisfied by the sequences counting walks beginning at the point $(i,j)$ and ending anywhere. When $V^{[k]}$ does not depend on $k$, which happens in most of our cases, then $V=V^{[k]}$ is called a \emph{discrete $\rho$-harmonic function}. Such discrete harmonic functions are of great interest to probability theorists.  Via a procedure known as a \emph{Doob transform}, they are used to construct processes conditioned to never leave cones~\cite{DenisovWachtel2015,Despax2016,Duraj2014,LecouveyRaschel2016}, which appear in several areas of probability including the study of non-colliding processes and the study of eigenvalues of certain random matrices. When $V^{[k]}$ depends on $k$, it is instead called a $\rho$-\emph{caloric} function~\cite{BouazizMustaphaSifi2015}.
%See Section~\ref{subsubsec:conditioning_random_walks} below for more details. 

Theorem~\ref{thm:GB_main_asymptotic_result} gives an explicit family of simply expressed discrete harmonic and caloric functions, listed in Appendix~\ref{appendix:GB}.  For example, in the zero drift case $a=b=1$ we obtain (up to a constant scaling)
\begin{equation}
\label{eq:universal_GB_harmonic function}
V(i,j)=(i+1)(j+1)(i+j+2)(i+2j+3).
\end{equation}
Discrete harmonic functions also seem promising for the random generation of random walks confined to cones\footnote{Private communication from Eric Fusy to the authors of Courtiel et al.~\cite{CourtielMelczerMishnaRaschel2016}.}. 

\section{Determination of Gouyou-Beauchamps Asymptotics}
\label{sec:GBproofs}

We begin, as always, by using the kernel method to express the generating function $Q_{\ba}(1,1,t)$ as the diagonal of an explicit rational function.  As we add only positive weights to the steps and do not change the underlying step set, the kernel equation~\eqref{eq:QPkernel} for unweighted Gouyou-Beauchamps walks continues to hold when its multivariate generating function $Q(x,y,t)$ is replaced by the weighted generating function $Q_{a,b}(x,y,t)$ and the characteristic polynomial $S(x,y)$ is replaced by its weighted version.  This gives the functional equation
\begin{equation}
\label{eq:GBkernel}
xy(1 - tS_{a,b}(x,y))Q_{a,b}(x,y,t) = xy - t I_{a,b}(y) -  t J_{a,b}(x),
\end{equation} 
where
\[ I_{a,b}(y) = y\left([x^{-1}]S_{a,b}(x,y)\right)Q_{a,b}(0,y,t), \qquad J_{a,b}(x) = x\left([y^{-1}]S_{a,b}(x,y)\right)Q_{a,b}(x,0,t), \]
and
\[ S_{a,b}(x,y) = \frac{1}{ax} + ax + \frac{ax}{by} + \frac{by}{ax}. \]
Furthermore, to count walks beginning at the point $(i,j) \in \mathbb{N}^2$ instead of the origin it is sufficient to replace the term $xy$ on the right-hand side of Equation~\eqref{eq:GBkernel} with $x^{i+1}y^{j+1}$ to obtain an equation of the form 
\begin{equation}
\label{eq:GBkernelij}
xy(1 - tS_{a,b}(x,y))Q^{i,j}_{a,b}(x,y,t) = x^{i+1}y^{j+1} - t I_{a,b}^{i,j}(y) -  t J_{a,b}^{i,j}(x). 
\end{equation}
This follows from the same argument given in Chapter~\ref{ch:KernelMethod} to derive the usual kernel equation, when the condition $[t^0]Q(x,y,t)=1$ denoting an empty walk beginning at the origin is replaced by the condition $[t^0]Q(x,y,t)=x^iy^j$ denoting an empty walk beginning at $(i,j)$.  

\subsection{A Uniform Diagonal Expression}

Our analysis for the unweighted Gouyou-Beauchamps model used the group of bi-rational transformations of the plane generated by the involutions
\[ \Psi:(x,y) \mapsto \left(\frac{y}{x}, y\right), \quad \Phi:(x,y) \mapsto \left(x, \frac{x^2}{y}\right), \]
which fix the unweighted characteristic polynomial $S(x,y)$.  In the weighted case we consider the group $\mG_{a,b}$ of transformations generated by the involutions
{\small 
\[
\Psi_{a,b}:(x,y) \mapsto \left(\frac{[x^{-1}]S_{a,b}(x,y)}{x[x^1]S_{a,b}(x,y)},y\right) = \left( \frac{by}{a^2x},y \right), \qquad
\Phi_{a,b}:(x,y) \mapsto \left(x,\frac{[y^{-1}]S_{a,b}(x,y)}{y[y^1]S_{a,b}(x,y)}\right) = \left( x, \frac{a^2x^2}{yb^2} \right),
\]
}

\noindent
which fix the weighted characteristic polynomial $S_{a,b}(x,y)$.  Just as in the unweighted case, we can apply the eight elements of $\mG_{a,b}$ to the weighted kernel equation~\eqref{eq:GBkernelij}, take an alternating sum to cancel all unknown functions on the right-hand side, and take a non-negative series extraction to obtain
\[ Q^{i,j}_{a,b}(x,y,t) = [x^\geq y^\geq]\frac{O^{i,j}_{a,b}(x,y)}{1-tS_{a,b}(x,y)}, \]
where $O^{i,j}_{a,b}(x,y)$ is the \emph{weighted orbit sum}
\[ O^{i,j}_{a,b}(x,y) = \sum_{\sigma \in \mG_{a,b}} \sigma(x^{i+1}y^{j+1}) \]
which can be determined explicitly after a messy calculation.  This can be translated into an explicit diagonal expression for $Q^{i,j}_{a,b}(1,1,t)$ using Proposition~\ref{prop:postodiag}, however the resulting rational function is quite large so we focus here on the case $i=j=0$ and refer the reader to an accompanying Maple worksheet\footnote{Available for download at~\websiteurl .} for the expressions with general $i$ and $j$. In particular, the parameters $i$ and $j$ do not appear in the denominator of the rational function we obtain, meaning they do not affect the singular variety under consideration or the minimal critical points where asymptotic contributions are calculated.  

In fact, the parameters $i$ and $j$ do not even affect the order of vanishing of the numerator, which is why they appear only in the constant $V^{[k]}(i,j)$ of the dominant asymptotic term.  Short proofs that the $V^{[k]}(i,j)$ do not vanish for any choice of $i,j,a,$ and $b$ in each universality class are given in the accompanying Maple worksheet.  These proofs typically follow from an application of Descartes' rule of signs to give an upper bound $N$ on the number of zeros $V(i,j)$ can have in $a$ when $b,i,$ and $j$ are fixed, followed by an explicit determination of $N$ positive zeroes in $a$ which lie outside the weight restrictions given for each class.
\smallskip

For the case $i=j=0$ we obtain the diagonal expression
\[ Q_{a,b}(1,1;t) = \Delta \left( \frac{(y-b)(a-x)(a+x)(a^2y-bx^2)(ay-bx)(ay+bx)}{a^4b^3x^2y(1-txyS_{a,b}(\ox,\oy))(1-x)(1-y)} \right), \]
which, as the monomial $x^2y$ appears in the denominator, we re-write as 
\begin{equation}
\label{eqn:weighted_extraction}
Q_{a,b}(1,1;t) = \frac{1}{a^4b^3t^2} \Delta \underbrace{\left( \frac{yt^2(y-b)(a-x)(a+x)(a^2y-bx^2)(ay-bx)(ay+bx)}{(1-x)(1-y)(1-txyS(\ox,\oy))} \right)}_{F(x,y,t)}
\end{equation}
so that $F(x,y,t)$ is analytic at the origin\footnote{As mentioned previously, this translation to the diagonal of an analytic function can be avoided by using the theory of diagonals of convergent Laurent expansions.  Ultimately, one may work with the original expression as if it were analytic at the origin and derive the same asymptotic results.}.

\subsection{Minimal Critical Points}

As in previous cases, we let $G(x,y)$ and $H(x,y,t)$ be the numerator and denominator of $F(x,y,t)$, and define the polynomials
\[ H_1(x,y,t) = 1-txyS_{a,b}(\ox,\oy), \qquad H_2(x,y,t) = 1-x, \qquad H_3(x,y,t) = 1-y. \]
The singular variety can be stratified into the sets $\mV_1,\mV_2,\mV_3,\mV_{1,2},\mV_{1,3},\mV_{2,3}$, and $\mV_{1,2,3}$, where $\mV_{i_1,\dots,i_r}$ is defined by the vanishing of $H_{i_1},\dots,H_{i_r}$ and the non-vanishing of the other factors of $H$.  Since $H_2$ and $H_3$ are independent of the $t$ variable, only the strata $\mV_1,\mV_{1,2},\mV_{1,3}$, and $\mV_{1,2,3}$ can contain critical points.  Solving the critical point equations~\eqref{eq:gencrit} on each stratum give the critical points described in Table~\ref{tab:GBcritpt}.

\begin{table} \centering
\renewcommand{\arraystretch}{1.5}
\begin{tabular}{ccc}
Stratum & Critical Points & Exponential Growth \\\hline
$\mV_1$ & $c_1^{\pm} = (\pm a,b)$ & $e_1 = 4$ \\
$\mV_{12}$ & $c_{12} = \left(1,\frac{b}{a}\right)$ & $e_{12} = \frac{(a+1)^2}{a}$ \\
$\mV_{13}$ & $c_{13}^{\pm} = \left(\pm\frac{a}{\sqrt{b}},1\right)$ & $e_{13} = \frac{2(b+1)}{\sqrt{b}}$ \\
$\mV_{123}$ & $c_{123} = (1,1)$ & $e_{123} = \frac{(b+1)(a^2+b)}{ab}$
\end{tabular}

\caption[Critical points for weighted Gouyou-Beauchamps models]{The $(x,y)$-coordinates of the critical points; for each, $t=\ox\oy S(\ox,\oy)^{-1}$.}
\label{tab:GBcritpt}
\end{table}

As we once again deal with a simple denominator, which is combinatorial, it is easy to characterize minimal points.  Note that for some values of the parameters $a$ and $b$ the numerator of $F(x,y,t)$ may contain a factor of $1-x$ or $1-y$, so $G$ and $H$ may not be co-prime.

\begin{lemma}
\label{lem:GBminpts}
When $G(x,y)$ and $H(x,y,t)$ are co-prime, the point $(x,y,t) \in \mV$ is minimal if and only if 
\[|x|\leq  1,\quad |y|\leq 1, \quad |t| \leq \frac{1}{|xy|S(|\ox|,|\oy|)},\]
where the three inequalities are not simultaneously strict inequalities.
\end{lemma}
\begin{proof}
When the polynomials $G$ and $H$ are co-prime the set of minimal points coincides with the minimal points of the rational function $1/H(x,y,t)$, which is the product of three geometric series.  The domain of convergence $\mD$ is then obtained by intersecting the domains of convergence of the rational functions $1/(1-x)$, $1/(1-y)$, and $1/(1-txyS(\ox,\oy))$. Following the same argument as the proof of Proposition~\ref{prop:almostsymmin} in Chapter~\ref{ch:QuadrantLattice}, it can be shown that the domain of convergence of $1/(1-txyS(\ox,\oy))$ is 
\[ \{(x,y,t): |t| < |xy|^{-1}\,S(|\ox|,|\oy|)^{-1}\}\] 
since the polynomial $xyS(\ox,\oy)$ has non-negative coefficients.
\end{proof}

Similar arguments show that when the weights $a$ and $b$ are such that $G(x,y)$ contains a $1-x$ as a factor, but not $1-y$, then $(x,y,t) \in \mV$ is minimal if and only if 
\[ |y| \leq 1 \quad \text{and} \quad |t| \leq \frac{1}{|xy|S(|\ox|,|\oy|)}, \]
and both inequalities are not strict.  The cases when $G(x,y)$ contains only $1-y$ as a factor, or both $1-x$ and $1-y$ as factors, are analogous.

Using arguments familiar from previous chapters, it is also easy to show that the minimum of the bound $|xyt|^{-1}$ on the exponential growth of the diagonal sequence is achieved in $\overline{\mD}$ at a minimal critical point.

\begin{lemma}
\label{lem:GBexpmin}
Every minimizer of $|xyt|^{-1}$ in $\overline{\mD}$ has the same coordinate-wise modulus as a minimal critical point.
\end{lemma}

\begin{proof}
Assume first that the weights $a$ and $b$ are such that the factors $1-x$ and $1-y$ of $H(x,y,t)$ are not factors of $G(x,y)$. 

Since $F(x,y,t)$ is combinatorial, $(x,y,t) \in \partial \mD$ if and only if $(|x|,|y|,|t|) \in \partial \mD$.  Furthermore, $|xyt|^{-1}$ decreases as $|t|$ grows, hence by Lemma~\ref{lem:GBminpts} $|xyt|^{-1}$  is minimized on $\overline \mD \cap \mathbb R_{>0}^3$ at points of the form $(x,y,\ox\oy S(\ox,\oy)^{-1})$ with $0<x,y \leq 1$.  Thus, it is sufficient to show that the minimizer of 
\[ S(\ox,\oy) = \frac{a}{x}+\frac{ay}{bx}+\frac{bx}{ay}+\frac{x}{a}\] 
for $(x,y) \in (0,1]^2$ occurs at the $(x,y)$-coordinates of a minimal critical point. The minimum is achieved because $S(\ox,\oy)$ tends to infinity as either $x$ or $y$ (or both) stay positive and tend to 0: the form of $S(\ox,\oy)$ implies this holds as $x\rightarrow0$ but then it also holds as $y\rightarrow0$ because $x$ bounded away from 0 implies $x/y \rightarrow \infty$.  Thus, the minimum either occurs in the interior, when
\[ (\partial S / \partial x)(\ox,\oy) = (\partial S / \partial y)(\ox,\oy) = 0, \]
or when $x=1$ and $(\partial S / \partial y)(\ox,1) = 0$, or when $y=1$ and $(\partial S / \partial y)(x,\oy) = 0$, or when $x=y=1$.  These sets of equations are exactly equal to the critical point equations on the different strata $\mV_1,\mV_{1,2},\mV_{1,3}$, and $\mV_{1,2,3}$.  

Similar arguments show that when one or both of the factors $1-x$ and $1-y$ in the denominator are canceled then the result holds as long as the minimum of $S(\ox,\oy)$ is achieved on $(0,1] \times (0,\infty)$, $(0,\infty)\times(0,1]$ or $(0,\infty)^2$, depending on which factors have canceled.  Again the minimum is achieved, as $S(\ox,\oy)$ approaches infinity when either $x$ or $y$ (or both) approach infinity, and the rest of the argument is analogous. 
\end{proof}

It is now possible to characterize the minimal critical points which will determine diagonal asymptotics using our explicit characterization of minimal points in Lemma~\ref{lem:GBminpts}.  Recall the critical points described in Table~\ref{tab:GBcritpt}.

\begin{proposition}
\label{prop:GBcontrib}
For fixed weights $a,b>0$, the set of minimal critical points which minimize $|xyt|^{-1}$ on $\overline{\mD}$ consists of the unique points $(x,y,t) \in \mV$ whose $(x,y)$-coordinates are 
\begin{itemize}
\item $c_1^{\pm}$ when $a \leq1$ and $b \leq 1$;
\item $c_{12}$ when $a>1$ and $a\geq b$;
\item $c_{13}^{\pm}$ when $b>1$ and $b\geq a^2$;
\item $c_{123}$ when $b>a>\sqrt{b}>1$.
\end{itemize}
In the first case the singular variety admits smooth minimal critical points, while in the second and third the points are minimal convenient points.  The final case corresponds to a transverse multiple point where the singular variety forms a complete intersection. 
\end{proposition}

Note that on the boundaries of these case distinctions the points with positive coordinates coincide, which is why the exponential growth $\rho$ varies smoothly with $a$ and $b$.

\begin{proof}
The values of the exponential growth $|xyt|^{-1}$ for each set of critical points are listed in the final column of Table~\ref{tab:GBcritpt}.  The AM-GM inequality implies that $e_1 \leq e_{12},e_{13} \leq e_{123}$, so that the set of minimal critical multiple points minimizing $|xyt|^{-1}$ consists of those defined by $(x,y)=c_1^{\pm}$ as long as these points are minimal. Similarly, the points defined by $(x,y)=c_{12}$ or $(x,y)=c_{13}^{\pm}$ minimize $|xyt|^{-1}$ as long as they are minimal and those with $(x,y)=c_1^{\pm}$ are not\footnote{Note that if the point with $(x,y)=c_1^{\pm}$ is not minimal then it can't happen that both those with $(x,y)=c_{12}$ and $(x,y)=c_{13}^{\pm}$ are minimal.}.  Finally, the conditions listed above come from the characterization of minimal points in Lemma~\ref{lem:GBminpts}.  

The factors $1-x$ and $1-y$ in the denominator cancel only when $a=1$ or $b=1$, respectively.  Such models are either Transitional or Directed, and the conclusion can be verified separately for each case.
\end{proof}

Again following arguments similar to those in previous chapters, it is easy to show the following result (which implies that all smooth minimal critical points in Proposition~\ref{prop:GBcontrib} will be finitely minimal).

\begin{lemma}
\label{lemma:weighted-fin-min}
Suppose $(x,y,t) \in \mV(H_1)$ has positive real coordinates, and $(p,q,r) \in T(x,y,t)$ with $H_1(p,q,r)=0$.  Then $(p,q,r) = (x,y,t)$ or $(-x,y,t)$.
\end{lemma}

\begin{proof}
Under these conditions 
\[ \left| (1/a) p^2q + a q + (b/a) p^2 + (a/b) q^2 \right| = (1/a) x^2y + a y + (b/a) x^2 + (a/b) y^2. \]
Since $a,b,x,y>0$, and $|p|=x$ while $|q|=y$, the complex triangle inequality implies that $q$ and $q^2$ have the same complex argument, meaning $q$ is real and positive, and thus equal to $y$.  This, in turn, implies that $p^2$ is real and positive, so $p$ equals $x$ or $-x$.  Finally, $r$ is determined by solving the equation $H_1(p,q,r)=0$, which is linear in $r$.
\end{proof}

The minimal critical points in different strata correspond to different exponential growths of the diagonal sequence, but do not completely determine universality classes: the critical exponents depend on several factors, including the degree of vanishing of the numerator of $F(x,y,t)$ at its minimal critical points.  We now complete the proof of Theorem~\ref{thm:GB_main_asymptotic_result} by showing how to compute the asymptotic contributions of the critical points.  The formulas derived here have been heuristically checked by numerically computing asymptotics for examples in each universality class.

\subsection{Determining Asymptotic Contributions}

\subsubsection*{The Balanced Case $(a=b=1)$}
This is the unweighted Gouyou-Beauchamps model, enumerated in Chapter~\ref{ch:QuadrantLattice} using the smooth point asymptotic result in Corollary~\ref{cor:smoothAsm}:
\[ [t^k]Q_{1,1}(1,1,t) = \frac{8}{\pi} \cdot \frac{4^k}{k^2} \left(1 + O\left(\frac{1}{k}\right)\right). \]
When dealing with walks beginning from the start point $(i,j)$ one obtains a Fourier-Laplace integral expression of the form $\int A_{i,j}(\bt)e^{-k\phi(\bt)} d\bt$ where $\phi(\bt)$ is independent of $i$ and $j$, and the order of vanishing of $A_{i,j}(\bt)$ at the origin is independent of $i$ and $j$.  As Corollary~\ref{cor:smoothAsm} requires only evaluations of derivatives to determine asymptotics corresponding to a smooth minimal critical point, the result can still be computed when $i$ and $j$ are indeterminate parameters.

\subsubsection*{The Reluctant Case $\left(a<1,b<1\right)$}

The analysis here is the same as in the balanced case, except that now the factors $1-x$ and $1-y$, which canceled with factors in the numerator when $a=b=1$, appear in the denominator of $F(x,y,t)$.  In particular, the smooth critical points defined by $(x,y) = c_1^{\pm}$ are still finitely minimal by Lemma~\ref{lemma:weighted-fin-min} and thus determine asymptotics. 

\subsubsection{The Transitional Cases $\left(a=1,b<1\right)$ and $\left(b=1,a<1\right)$}
The transitional cases are on the boundary between being reluctant and directed.  When $a=1$ and $b<1$, the critical points with $(x,y)=c_1^{\pm}$ have an $x$-coordinate of modulus $1$, however the factor of $1-x$ in the denominator cancels with a factor of $1-x$ which becomes present in the numerator when specializing $a$ to $1$ (meaning the critical points with $(x,y)=c_1^{\pm}$ are still finitely minimal smooth points).  When $b=1$ and $a<1$, the factor of $1-y$ in the denominator cancels with a factor of $1-y$ which becomes present in the numerator when specializing $b$ to $1$.  After this simplification, the same argument as in the balanced and reluctant cases applies.

Note also that this cancellation hints as to why the balanced, transitional and reluctant cases have the same exponential growth but different critical exponents $\alpha$.  The smooth minimal critical points are the same for each, but the order of vanishing of the numerator at the critical points is $2$, $3$ and $4$ for balanced, transitional and reluctant models, respectively.  Theorem~\ref{thm:smoothAsm} shows that when the order of vanishing of the numerators increases one expects\footnote{As seen in previous examples, the order of vanishing of the numerator gives a bound on, but does not completely determine, the critical exponent.  For instance, when the numerator of $F(x,y,t)$ vanishes to order 4 at a minimal critical smooth point the critical exponent could be as low as $\alpha=3$, but in the case of balanced walks it is $\alpha=5$, which does not have an immediate explanation.} the critical exponent $\alpha$ to increase. 

\subsubsection{The Free Case $\left(\sqrt{b}<a<b\right)$}

In the free case there is exactly one minimal critical point, $\bp = (1,1,1/S_{a,b}(1,1))$.  This point lies on the stratum $\mV_{1,2,3}$ determined by the intersection of the three varieties $\mV_1,\mV_2,$ and $\mV_3$.  Because of the restrictions on the weights, the numerator $G$ does not vanish here and since $\bp$ lies on a complete intersection we can directly apply Theorem~\ref{thm:compintasm} from Chapter~\ref{ch:NonSmoothACSV} to obtain
\[ [t^k]Q_{1,1}(1,1,t) = \frac{(b-1)(a^2-1)(b-a^2)(a^2-b^2)}{a^4 b^3} \left(\frac{(b+1)(a^2+b)}{ab}\right)^k + O\left(\tau^k\right), \]
with $\tau \in \left(0,\frac{(b+1)(a^2+b)}{ab}\right)$.  Note that we determine dominant asymptotics up to an \emph{exponentially} smaller error term, instead of the typical case of a polynomially smaller error term.

\subsubsection{The Axial Cases $\left(a=b>1\right)$ and $\left(b=a^2>1\right)$}
The axial cases are on the boundary of the directed cases and the free case.  Unfortunately, we cannot use Theorem~\ref{thm:compintasm} on asymptotics in the case of a complete intersection as the numerator of $F$ vanishes at the minimal critical points under consideration. Luckily, we can decompose $F$ under these weight restrictions into two simpler rational functions and analyze each of them.

When $a=b>1$,
\[ F(x,y,t) = \frac{yt^2(y-a)(a-x)(a+x)(ay-x^2)(y-x)(y+x)}{(1-txyS(\ox,\oy))(1-x)(1-y)} \]
and $F$ admits the minimal critical point $(1,1,1/S(1,1))$.  Because we cannot analyze this directly we write $y-x = (1-x)-(1-y)$ and see that
\[ F(x,y,t) =  \underbrace{\frac{yt^2(y-a)(a-x)(a+x)(ay-x^2)(y+x)}{(1-txyS(\ox,\oy))(1-y)}}_{F_1(x,y,t)} - \underbrace{\frac{yt^2(y-a)(a-x)(a+x)(ay-x^2)(y+x)}{(1-txyS(\ox,\oy))(1-x)}}_{F_2(x,y,t)}. \]
As the diagonal operator is linear, we can obtain the desired asymptotics by studying $\Delta F_1$ and $\Delta F_2$.  Following the same argument as above, and using the critical point equations, shows that $F_1$ admits the minimal critical points with $(x,y)$-coordinates $c_{13}^{\pm}=(\pm\sqrt{a},1)$ while $F_2$ admits the minimal critical point $\bp = (1,1,a/(1+a)^2)$ with $(x,y)$-coordinates $c_{12}=c_{123}$.  Thus, the diagonal of $F_2$ will have larger exponential growth than the diagonal of $F_1$, so the diagonal of $F_2$ determines dominant asymptotics of the original diagonal sequence.  At the minimal critical point $\bp$,
\[ \nabla_{\log}(H_1)(\bp) = \left(-\frac{2}{1+b},-1,-1\right) \quad \text{and} \quad \nabla_{\log}(H_2)(\bp) = (-1,0,0),  \]
so that $\bone \in N(\bp)$ and, as the numerator of $F_2$ does not vanish at $\bp$ under these weight restrictions, asymptotics can be determined using Theorem~\ref{thm:resasm}.

When $b=a^2$ the argument is analogous except that the numerator contains $y-x^2 = (1-x)(1+x)-(1-y)$ as a factor, and this is used to decompose the rational diagonal under consideration into a sum of two simpler rational diagonals which are then analyzed.
\smallskip

When the numerator is parametrized by $i$ and $j$, it is still true that there exist polynomials $G_1^{i,j}(x,y)$ and $G_2^{i,j}(x,y)$ such that the numerator $G^{i,j}(x,y)$ of the rational function whose diagonal encodes $Q^{i,j}_{a,b}(x,y)$ can be written
\[ G^{i,j}(x,y) = (1-x)G_1^{i,j}(x,y) + (1-y)G_2^{i,j}(x,y). \]
It is hard to determine these polynomials explicitly, but dominant asymptotics of the diagonal sequence depends only on their evaluations at $x=y=1$.  By L'Hôpital's rule, these evaluations are given by
\[ G_1^{i,j}(1,1) = \lim_{x\rightarrow1}\frac{G^{i,j}(x,1)}{1-x} = -(\partial G^{i,j}/\partial x)(1,1), \qquad G_2^{i,j}(1,1) = \lim_{y\rightarrow1}\frac{G^{i,j}(1,y)}{1-y} = -(\partial G^{i,j}/\partial y)(1,1). \]

\subsubsection{The Directed Cases $\left(a>1, a>b\right)$ and $\left(b>1, \sqrt{b}>a\right)$}
\label{sec:DirAsm}

In the directed cases, the minimal critical points lie on strata defined by the intersection of two smooth varieties.  These points are not finitely minimal, and the numerator of $F(x,y,t)$ vanishes, meaning Theorem~\ref{thm:resasm} only allows for a bound on dominant asymptotics.  Luckily, we can exploit the form of the diagonal under consideration to determine dominant asymptotics exactly.

Suppose first that we are in the case when $a>1$ and $a>b$, so that the unique contributing point has $(x,y)$-coordinates $c_{12} = (1,b/a)$ and lies on $\mV_{12}$.  If we define
\[ \tilde{G}(x,y) := \frac{G(x,y,t)}{t^2} = y(y-b)(a-x)(a+x)(a^2y-bx^2)(ay-bx)(ay+bx),  \]
then the Cauchy Integral Formula implies that we want asymptotics of
{\small
\begin{align*} 
[t^{n+2}][x^{n+2}][y^{n+2}] \left( \frac{t^2\tilde{G}(x,y)}{a^4b^3(1-txyS(\ox,\oy))(1-x)(1-y)} \right)  
&= [x^2][y^2] \left( \frac{\tilde{G}(x,y)S(\ox,\oy)^n}{a^4b^3(1-x)(1-y)} \right) \\[+2mm]
&= \underbrace{\frac{1}{a^4b^3(2 \pi i)^2} \int_{|y|=b/a} \int_{|x|=1-\epsilon} \frac{\tilde{G}(x,y) S(\ox,\oy)^n}{(1-x)(1-y)} \frac{dxdy}{x^3y^3}}_I,
\end{align*}
}

\noindent
for any fixed $0 < \epsilon < 1$.  The standard integral bounds discussed in Chapter~\ref{ch:SmoothACSV} imply the existence of a constant $C>0$ such that 
\begin{align*} 
\underbrace{\left|  \frac{1}{a^4b^3(2 \pi i)^2} \int_{|y|=b/a} \int_{|x|=1+\epsilon} \frac{\tilde{G}(x,y) S(\ox,\oy)^n}{(1-x)(1-y)} \frac{dxdy}{x^3y^3} \right|}_{M}
&\leq C \cdot  \max_{|x|=1+\epsilon, |y|=b/a}\left| S(\ox,\oy) \right|^n \\
&=  C \cdot \max_{|x|=1+\epsilon, |y|=b/a}\left|\frac{a}{x} + \frac{x}{a} + \frac{ay}{bx} + \frac{bx}{ay} \right|^n \\
&= C \cdot \left[\frac{a}{1+\epsilon} + \frac{1+\epsilon}{a} + \frac{1}{1+\epsilon} + (1+\epsilon) \right]^n.
\end{align*}

A basic calculus argument then shows\footnote{This turns out to be a consequence of the fact that $c_{12}$ is a minimal convenient point.} that $M = O(r^N)$ for some $r \in (0,2+a+1/a)$, when $\epsilon>0$ is sufficiently small.  Thus, as $n \rightarrow \infty$,
\begin{align*} 
I = \frac{1}{a^4b^3(2 \pi i)^2} \int_{|y|=b/a} & \left(\int_{|x|=1-\epsilon} \frac{\tilde{G}(x,y) S(\ox,\oy)^n}{(1-x)(1-y)} \frac{dxdy}{x^3y^3} \right. \\
&\left. \qquad\qquad\qquad - \int_{|x|=1+\epsilon} \frac{\tilde{G}(x,y) S(\ox,\oy)^n}{(1-x)(1-y)} \frac{dxdy}{x^3y^3} \right)dy + O(r^n),
\end{align*}
and Cauchy's residue theorem implies
\begin{align} 
I &= \frac{1}{a^4b^3(2 \pi i)} \int_{|y|=b/a} \frac{\tilde{G}(1,y) S(1,\oy)^n}{1-y} \frac{dy}{y^3} + O(r^n)\notag \\[+2mm] 
&= \frac{(2+a+1/a)^n}{a^4b^3(2 \pi)} \int_{-\pi}^\pi A(\theta) e^{-k \phi(\theta)} d\theta + O(r^n), \label{eq:DirFL}
\end{align}
where 
\[ A(\theta) = \frac{\tilde{G}(1,(b/a)e^{i\theta})}{(b/a)^2e^{2i\theta}(1-(b/a)e^{i\theta})}, \qquad\qquad \phi(x,y) = \log S(1,a/b) -\log S\left(1,(a/b)e^{-i\theta}\right). \]
Since $c_{12}$ is a contributing singularity, it can be shown that the integral in Equation~\eqref{eq:DirFL} has a single critical point at $\theta =0$, to which Proposition~\ref{prop:HighAsm} (which is classical in the univariate case) can be applied.  

Note that the multivariate residue approach to ACSV, as described at the end of Chapter~\ref{ch:NonSmoothACSV}, implies that diagonal asymptotics are determined by the integral 
\[ \frac{1}{2\pi i} \int_{\sigma} \frac{\tilde{G}(1,y) S(1,\oy)^n}{1-y} \frac{dy}{y^3}, \]
where $\sigma$ is a one-dimensional curve of integration in $\mV_{12}$ containing $c_{12}$.  As mentioned previously, it can be difficult in general to obtain an explicit description of the possible domains of integration~$\sigma$.  

In the other directed case, when $b>1$ and $\sqrt{b}>a$, the analysis is the same except that the existence of two singularities, with $(x,y)$-coordinates $c_{13}^{\pm}$, contributing to dominant asymptotics points implies that asymptotics are ultimately obtained from a sum of two univariate Fourier-Laplace integrals.

\subsection*{Excursion Asymptotics}
As usual, we also obtain a rational diagonal expression for the generating function $Q^{i,j}_{a,b}(0,0;t)$ counting walks beginning and ending at the origin, which has the smooth singular variety $\mV(H_1)$.  There are always two minimal critical points, characterized by $(x,y) = c^{\pm}_1 = (\pm a,b)$, which are finitely minimal by Lemma~\ref{lemma:weighted-fin-min}.  Thus, Corollary~\ref{cor:smoothAsm} gives the following result.

\begin{theorem}
\label{thm:GB_excursion_asympt}
For any non-negative weights $a, b>0$, the number of excursions of length $k$ has dominant asymptotics
\[ e_{(i,j)\rightarrow(0,0)}(k) =
\begin{cases}
\frac{4^k}{k^5} \left(\frac{128(j+1)(1+i)(3+i+2j)(2+i+j)}{a^ib^j\pi} + O\left(\frac{1}{k}\right)\right) &\text{if } k+i \equiv 0 \pmod{2}, \\
0 &\text{if } k+i \equiv 1 \pmod{2}.
\end{cases}\]
\end{theorem}

\section{General Central Weightings}
\label{sec:central}

We now turn our considerations to general centrally weighted models. Let $\mS\subset\mathbb{Z}^n$ be a finite step set and for $\bss \in \mS$ let $\pi_j(\bss) = s_j$ be its $j$th coordinate.   Given an assignment of positive weights $\ba = \left(a_{\bss}\right)_{\bss \in \mS}$ to each step in $\mS$ we define the following.

\begin{definition}
The weighting $\ba$ is \emph{central} for the non-negative orthant $\mathbb{N}^n$ if the weight of any (weighted) path in $\mathbb{N}^n$ using the steps in $\mS$ depends only on the length, start and end points of the path. 
\label{def:central}
\end{definition}

A step set $\mS \subset\mathbb{Z}^n$ is called \emph{singular}\footnote{For example, the two-dimensional step sets 
\[ \diagrFb{NW,SE,N} \qquad \diagrFb{NW,SE,NE} \qquad \diagrFb{NW,SE,N,E} \qquad \diagrFb{NW,SE,N,NW} \qquad \diagrFb{NW,SE,N,NE,E} \] 
are singular.} if its steps lie in a half-space of $\mathbb{Z}^n$.  We will prove several equivalent characterizations of central weightings on non-singular step sets.

\begin{theorem}
\label{thm:centralweighting}
Let $\mS\subset\mathbb{Z}^n$ be a finite non-singular step set. A weighting $\ba$ of $\mS$ is central if and only if either of the following equivalent statements holds:
\begin{enumerate}
\item\label{(i)}For every point $\bi \in \mathbb{N}^n$ and $k \in \mathbb{N}$, each walk of length $k$ in $\mathbb{N}^n$ starting at the origin and ending at $\bi$ has the same weight;
\item\label{(ii)}There exist constants $\alpha_1,\dots,\alpha_n,$ and $\beta$ such that the weight assigned to any step $\bss \in \mS$ has the form $a_{\bss} = \beta \prod_{j=1}^n \alpha_j^{\pi_j(\bss)}$.
\end{enumerate}
\end{theorem}

The hardest implication to prove in Theorem~\ref{thm:centralweighting} is $\ref{(i)} \Rightarrow \ref{(ii)}$. This will require introducing another condition which takes some setup to state (see Proposition~\ref{prop:productweights}) so the proof will be given in Section~\ref{sec:proofCWthm}.  In probability theory, the weights $a_{\bss}$ given in \ref{(ii)} constitute an exponential change (sometimes called a \emph{Cramér transform}) of the uniform weights on $\mS$.

\begin{example} 
The weight assignment given in Figure~\ref{fig:GB_weights} for the Gouyou-Beauchamps model defines a central weighting, which corresponds to setting $\alpha_1=a$, $\alpha_2 = b,$ and $\beta = 1$ in Theorem~\ref{thm:centralweighting}. We have lost a degree of freedom by setting $\beta = 1$ but, as previously mentioned, to determine asymptotics with arbitrary $\beta$ using Theorem~\ref{thm:GB_main_asymptotic_result} it is sufficient to multiply the right-hand side of Equation~\eqref{eq:GB_main_asymptotic_result} by $\beta^k$. 
\end{example}

Note that Condition~\ref{(ii)} can also be expressed in a matrix form. Let $\mS = \{\bss_1,\dots,\bss_m\}$ and define the matrix
\begin{equation}
M_{\mS} :=
\begin{pmatrix}
\pi_1(\bss_1) & \pi_2(\bss_1) & \cdots & \pi_n(\bss_1) & 1 \\
\pi_1(\bss_2) & \pi_2(\bss_2) & \cdots & \pi_n(\bss_2) & 1 \\
\vdots & \vdots & \ddots & \vdots & \vdots\\
\pi_1(\bss_m) & \pi_2(\bss_m) & \cdots & \pi_n(\bss_m) & 1
\end{pmatrix}.
\label{eq:MatrixS}
\end{equation}
Then the weighting $\ba$ satisfies Condition~\ref{(ii)} in Theorem~\ref{thm:centralweighting} if and only if there exist constants $\alpha_1,\dots,\alpha_n,$ and $\beta$ such that
\begin{equation*}
\begin{pmatrix} \log(a_{\bss_1}) \\ \log(a_{\bss_2}) \\ \vdots \\ \log(a_{\bss_m}) \end{pmatrix}  =  M_{\mS}  \begin{pmatrix} \log(\alpha_1) \\ \vdots \\ \log(\alpha_n) \\ \log(\beta) \end{pmatrix}. 
\end{equation*}

\begin{lemma} 
\label{lemma:rank}
If $\mS \subset \mathbb{Z}^n$ is a finite non-singular step set, then the rank of the matrix $M_{\mS}$ defined by Equation~\eqref{eq:MatrixS} is $n+1$.
\end{lemma}

\begin{proof} 
As $M_{\mS}$ has $n+1$ columns, its rank is at most $n+1$.  Suppose that the rank of $M_{\mS}$ is at most $n$, and let $l$ be the minimum of the rank of $M_{\mS}$ and $m = |\mS|$. 

Then $l \leq n$ and there exist $l$ steps $\btt_1,\dots,\btt_l$ such that the span of the vectors  
\[ (\btt_j,1) = (\pi_1(\btt_j),\dots,\pi_n(\btt_j),1), \qquad j = 1,\dots, l, \]  
contains $(\bss,1)$ for all $\bss \in \mS$. In other words, every step $\bss \in \mS$ belongs to the set
\[ A = \left\{\left. \sum_{j=1}^l q_j\btt_j  \quad \right| \quad  (q_1,\dots,q_l) \in \mathbb{R}^l \textrm{ with } \sum_{j=1}^n q_l = 1 \right\}. \]
The set $A$ is an affine hyperplane contained in the linear span of $\left\{\btt_j\right\}_{j \in \{1,\dots,l\}}$, so it is an affine subspace of $\mathbb R^n$ of dimension at most $n - 1$. Therefore $\mS \subset A \subset \mathbb{R}^n$ is contained in a half-space, contradicting the fact that $\mS$ is non-singular.
\end{proof}

Note, in particular, that any non-singular step set contains at least $n+1$ steps. 

\subsection{Another Definition of Central Weightings}
\label{sec:paths}
Lemma~\ref{lemma:rank} has a direct combinatorial interpretation in terms of lattice paths.

\begin{proposition} 
\label{prop:p1p2}
Given a non-singular step set~$\mS \subset \mathbb{Z}^n$, there exists a set $\mT \subset \mS$ of $n+1$ steps such that for every~$\bss \in \mS \setminus \mathcal T$ there exist two paths $p_{\bss}$ and~$p'_{\bss}$ in~$\mathbb{Z}^n$ (not necessarily $\mathbb{N}^n$) where:
\begin{itemize}
\item $p_{\bss}$ and~$p'_{\bss}$ begin at the origin, and have the same length and endpoint;
\item $p_{\bss}$ contains $\bss$ as a step, with all its other steps belonging to $\mathcal T$;
\item $p'_{\bss}$ uses only steps in $\mathcal T$.
\end{itemize}
\end{proposition}

See Figure~\ref{fig:paths} below for a pictorial example.

\begin{example} 
\label{ex:GBexT}
Consider the Gouyou-Beauchamps model with the set $\mathcal T = \{(1,0),(-1,0),(1,-1)\}$. For $\bss = (-1,1)$ we can choose $p_{\bss}$ to be the concatenation of $(-1,1)$ and $(1,-1)$, and $p'_{\bss}$ to be the concatenation of $(1,0)$ and $(-1,0)$. 
\end{example}

\begin{proof} 
By Lemma~\ref{lemma:rank}, the rank of $M_{\mS}$ is $n+1$, hence we can find a set of steps $\mT = \{\btt_1,\dots,\btt_{n+1}\}$ such that the span of the vectors $(\btt_j,1)$ contains $(\bss,1)$ for every $\bss \in \mS$. Furthermore, as each vector $(\bss,1)$ has integer coefficients it can be written as a linear combination of the $(\btt,1)$ vectors with rational coefficients. Clearing denominators and reorganizing terms according to their signs gives an equation of the form
\begin{equation}
r_{\bss}\cdot(\bss,1) + \sum_{j=1}^{n+1} r_{\bss,\btt_j}\cdot(\btt_j,1) = \sum_{j=1}^{n+1} r'_{\bss,\btt_j}\cdot(\btt_j,1),
\label{eq:linearcombi}
\end{equation} 
where $r_{\bss}$ and the $r_{\bss,\btt_j},r'_{\bss,\btt_j}$ are non-negative integers and $r_{\bss}>0$. We can take $p_{\bss}$ to be any path formed by $r_{\bss}$ copies of the step $\bss$ and $r_{\bss,\btt_j}$ copies of each step $\btt_j$, and $p'_{\bss}$ to be any path formed by $r'_{\bss,\btt_j}$ copies of each step $\btt'_j$. Examining the last coordinate of Equation~\eqref{eq:linearcombi} shows that 
\[ r_{\bss} + r_{\bss,\btt_1} + \cdots r_{\bss,\btt_{n+1}} = r'_{\bss,\btt_1} + \cdots + r'_{\bss,\btt_{n+1}}, \]
meaning these paths have the same length.
\end{proof}

We can finally give our last characterization of a central weighting.

\begin{proposition} 
\label{prop:productweights}
Consider a finite non-singular step set $\mS \subset\mathbb{Z}^n$, along with a set $\mathcal T$ and $|\mS| -|\mathcal T| = |\mS|-n-1$ corresponding pairs of paths $(p_{\bss},p'_{\bss})_{\bss \in \mS \setminus \mT}$ described by Proposition~\ref{prop:p1p2}.

A weighting $\ba=(a_{\br})_{\br \in \mS}$ is central if and only if 
\begin{enumerate}
\setcounter{enumi}{2}
\item\label{(iii)} for every $\bss \in \mS \setminus \mT$, the weights satisfy
\begin{equation}
 \prod_{\br \in p_{\bss}} a_{\br} = \prod_{\br' \in p'_{\bss}} a_{\br'}, \label{eq:prodas}
\end{equation}
where the steps are considered with multiplicity inside each product.
\end{enumerate}
\end{proposition}

\begin{example}
Taking the set $\mT$, step $\bss$, and paths $p_{\bss},p'_{\bss}$ given in Example~\ref{ex:GBexT} for the Gouyou-Beauchamps model, Proposition~\ref{prop:productweights} gives
\[ a_{-1,1} \times a_{1,-1} = a_{1,0} \times a_{-1,0}. \]
\vspace{-0.4in}

\end{example}

\subsection{Proof of Theorem~\ref{thm:centralweighting} and Proposition~\ref{prop:productweights}}
\label{sec:proofCWthm}
\begin{proof}
\begin{description}

\item{Definition~\ref{def:central} $ \boldsymbol \Rightarrow \ref{(i)}$.}
This implication is trivial as Condition~\ref{(i)} is a restriction of Definition~\ref{def:central} to the case where walks begin at the origin.

\item{$\ref{(i)} \boldsymbol \Rightarrow$ \ref{(iii)}.}
Since $\mS$ is non-singular, we can find a path starting at the origin and ending at a point arbitrarily far from both the $x$- and $y$-axes. Thus, we can pick some path $p$ using the steps in $\mS$ such that for every $\bss \in \mS \setminus \mathcal T$ the concatenation of $p$ and $p_{\bss}$ and the concatenation of $p$ and $p'_{\bss}$ stay in $\mathbb{N}^n$. Condition~\ref{(i)} implies that the weights of these two walks are equal, which implies that the products of the weights in $p_{\bss}$ and $p'_{\bss}$ are equal, giving Equation~\eqref{eq:prodas}.

\item{$\ref{(iii)} \boldsymbol \Rightarrow \ref{(ii)}$.}
Assume that Equation~\eqref{eq:prodas} holds. We will prove that the image of the matrix $M_{\mS}$ defined by Equation~\eqref{eq:MatrixS} is equal to the set
\begin{equation*}
E = \left\{ (y_{\bss})_{\bss \in \mS} \ \left| \ \ \forall \bss \, \in \, \mS \setminus \mT, \ \ \sum_{r \in p_{\bss}} y_r  = \sum_{r' \in p'_{\bss}} y_{r'} \textrm{ (sums considered with multiplicity)} \right. \right\}.
\end{equation*}
A vector $\left(y_{\bss}\right)_{\bss \in \mS}$ in $\textrm{Im}(M_{\mS})$ can be parametrized as
\[y_{\bss} = x_{n+1} + \sum_{i=1}^n \pi_i(\bss)x_i\] 
for $\bss \in \mS$ and indeterminates $x_i$. For any path $q$ on the steps in $\mS$, the coefficient of $x_{n+1}$ in the sum $\sum_{\br \in q} y_{\br}$ is the length of $q$, while the coefficient of $x_k$ for $1 \leq k \leq n$ is the $k$th coordinate of the endpoint of $q$.  As the paths $p_{\bss}$ and $p'_{\bss}$ coincide at their endpoints and have the same length for every $\bss \in \mS \setminus \mathcal T$, we see that $(y_{\bss})$ belongs to $E$. In other words, $\textrm{Im}\left(M_{\mS}\right) \subseteq E$. 

The equality $\textrm{Im}\left(M_{\mS}\right) = E$ follows from considering the dimensions of these linear spaces. On one hand, the dimension of $\textrm{Im}\left(M_{\mS}\right)$ is $n+1$ by Lemma~\ref{lemma:rank}. On the other hand, the dimension of $E$ is also $n+1$ since it is the intersection of the $|S|-n-1$ hyperplanes
\[ \sum_{\br \in p_{\bss}} y_{\br}  - \sum_{\br' \in p'_{\bss}} y_{\br'} = 0, \qquad \bss \in \mS \setminus \mT\] 
where the hyperplane defined by $p_{\bss}$ is the only one containing the coordinate $y_{\bss}$ (by the conditions listed in Proposition~\ref{prop:p1p2}). Therefore, $\textrm{Im}\left(M_{\mS}\right) = E$.  Applying logarithms to~\eqref{eq:prodas} shows that $(\log(a_{\bss}))_{\bss \in \mS}$ belongs to $E$, and thus to $\textrm{Im}\left(M_{\mS}\right)$, so that there exist constants $\alpha_1,\dots,\alpha_d,$ and $\beta$ satisfying Condition~\ref{(ii)} of Theorem~\ref{thm:centralweighting}.

\item{$\ref{(ii)} \boldsymbol \Rightarrow$ Definition~\ref{def:central}.}
Let $w$ be a walk of length $k$ in $\mathbb{N}^n$ beginning at the point $\bi$ and ending at the point $\bj$, and for each step $\bss \in \mS$ let $r_{\bss}(w)$ denote the number of copies of $\bss$ in $w$.  Assuming Condition~\ref{(ii)} of Theorem~\ref{thm:centralweighting}, the weight of $w$ is
\[ \prod_{\bss \in \mS}a_{\bss}^{r_{\bss}(w)} 
= \beta^{\sum_{\bss \in \mS} r_{\bss}(w)} \prod_{l=1}^n \alpha_l^{\sum_{\bss \in \mS} \pi_l(\bss)\,r_{\bss}(w)} 
= \beta^k \prod_{l=1}^n \alpha_l^{j_l-i_l}, \]
which depends only on $\bi,\bj,$ and $k$.  Thus, the weighting $\ba$ is central. \qedhere
\end{description}
\end{proof}

\subsection{An Example and Generalization}

We illustrate how these results allow one to efficiently characterize central weightings of a step set by way of an example.

\begin{figure}
\center
\includegraphics[width=.32\textwidth]{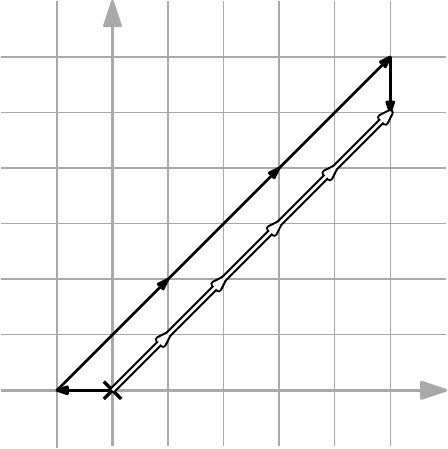} 
\caption[An illustration of Proposition~\ref{prop:p1p2}]{Given the step set $\mS = \{(2,2),(1,1),(-1,0),(0,1)\}$, the black and white paths satisfy the conditions of Proposition~\ref{prop:p1p2}. Note that they encode the left- and right-hand sides of Equation~\eqref{eq:geom}, respectively.}
\label{fig:paths}
\end{figure}

\begin{example}
Consider the (non-small) step set $\mS = \{(2,2),(1,1),(-1,0),(0,1)\}$. Taking the set $\mT = \{(1,1),(-1,0),(0,1)\}$ and $\bss = (2,2)$, we define the path $p_{\bss}$ to be the sequence of steps $(-1,0),(2,2),(2,2),(2,2),(0,-1)$ and $p'_{\bss}$ to be five copies of the step $(1,1)$. This pair of paths satisfies the hypotheses of Proposition~\ref{prop:p1p2}, so Proposition~\ref{prop:productweights} implies \textit{all} central weightings on these steps satisfy
\begin{equation}\label{eq:geom}
     a_{-1,0}a_{2,2}^3a_{0,-1} = a_{1,1}^5.
\end{equation}
Figure~\ref{fig:paths} shows the paths $p_{\bss}$ and $p'_{\bss}$. We can compute $\alpha_1,\alpha_2,$ and $\beta$ in terms of $a_{2,2},a_{0,-1},$ and $a_{1,1}$ by solving the system
\begin{equation*}
\left(\begin{array}{ccc}
2 & 2 & 1 \\ 1 & 1 & 1 \\ -1 & 0 & 1 \\ 0 & -1 & 1
\end{array}\right)
\,
\left(\begin{array}{l}
\log(\alpha_1) \\ \log(\alpha_2) \\ \log(\beta)
\end{array}\right)
=
\left(\begin{array}{l}
\log(a_{2,2}) \\ \log(a_{1,1}) \\ \log(a_{-1,0}) \\ \log(a_{0,-1})
\end{array}\right)
\end{equation*}
to find 
\[ \alpha_1 = a_{2,2}^2\,a_{0,-1} / a_{1,1}^3,  \quad \alpha_2 = a_{1,1}^2 /a_{2,2}\,a_{0,-1}, \quad \beta = a_{1,1}^2 /a_{2,2}.\] 
\end{example}

For a fixed step set $\mS$, $\bi,\bj \in \mathbb{N}^n,$ and $k \in \mathbb{N}$, we let $\mW_{\bi \rightarrow \bj}(k)$ denote the set of walks on the steps in $\mS$ from $\bi$ to $\bj$.  If $\ba$ is a positive weighting of the steps in $\mS$ then the weighted probability of $w \in \mW_{\bi \rightarrow \bj}(k)$ under $\ba$ is
\[ \text{Pr}_{\ba}(w) = \frac{\prod_{\bss \in w}a_{\bss} }{\sum_{w' \in \mW_{\bi \rightarrow \bj}(k)} \prod_{\bss' \in w'}a_{\bss'}}. \]
Theorem \ref{thm:centralweighting} can be generalized, enabling us to define equivalence classes among weighted models of walks.  

\begin{theorem}
Let $\mS \subset\mathbb{Z}^n$ be a finite non-singular step set and consider two positive weightings $(a_{\bss})_{\bss \in \mS}$ and $(a'_{\bss})_{\bss \in \mS}$ of $\mS$. Then the following statements are equivalent:
\begin{enumerate}
\item\label{item(i):equivweight} The weighted probability of any path in $\mathbb{N}^n$ starting at the origin is the same under $(a_{\bss})$ and $(a'_{\bss})$;
\item\label{item(ii):equivweight} There exist constants $\alpha_1,\dots,\alpha_n,$ and $\beta$ such that the weight assigned to each step $\bss \in \mS$ has the form   $a_{\bss} = a'_{\bss} \cdot \beta \, \prod_{j=1}^n \alpha_j^{\pi_j(\bss)}$;
\item\label{item(iii):equivweight} If $(p_{\bss},p'_{\bss})_{\bss \in \mS \setminus \mT}$ denote pairs of paths satisfying Proposition~\ref{prop:p1p2}, then for every $\bss \in \mS \setminus \mT$,
\begin{equation}
     \prod_{\br \in p_{\bss}} \frac{a_{\br}}{a'_{\br}} = \prod_{\br' \in p'_{\bss}} \frac{a_{\br'}}{a'_{\br'}}, \label{eq:equivweight}
\end{equation}
where the steps are considered with multiplicity.
\end{enumerate}
Two weightings of $\mS$ are said to be \emph{equivalent} if they satisfy one of the above statements.
\label{th:equivweight}
\end{theorem}

Every walk of length $k$ between any two fixed points has the same weight (and thus probability) when all weights are 1, so a central weighting is a weighting that is equivalent to the unweighted model $(1)_{\bss \in \mS}$.  The proof of Theorem~\ref{th:equivweight} is very similar to the proof of Theorem~\ref{thm:centralweighting} and Proposition~\ref{prop:productweights}, but we sketch it here.

\begin{proof}
\begin{description}

\item{$\ref{item(i):equivweight} \boldsymbol \Rightarrow$ \ref{item(iii):equivweight}.}
Again, since $\mS$ is non-singular we can find a path starting at the origin and ending at a point arbitrarily far from both the $x$- and $y$-axes to which we can append both $p_{\bss}$ and $p'_{\bss}$ and stay in $\mathbb{N}^n$, which implies
\begin{align*}
\frac{\prod_{\br \in p_{\bss}}a_{\br} }{\sum_{w' \in \mW_{\bi \rightarrow \bj}(k)} \prod_{\bss' \in w'}a_{\bss'}} &= 
\frac{\prod_{\br \in p_{\bss}}a'_{\br} }{\sum_{w' \in \mW_{\bi \rightarrow \bj}(k)} \prod_{\bss' \in w'}a'_{\bss'}} \\
\frac{\prod_{\br' \in p'_{\bss}}a_{\br'} }{\sum_{w' \in \mW_{\bi \rightarrow \bj}(k)} \prod_{\bss' \in w'}a_{\bss'}} &= 
\frac{\prod_{\br' \in p'_{\bss}}a'_{\br'} }{\sum_{w' \in \mW_{\bi \rightarrow \bj}(k)} \prod_{\bss' \in w'}a'_{\bss'}} \\
\end{align*}
so that
\[ \prod_{\br \in p_{\bss}} \frac{a_{\br}}{a'_{\br}} = \prod_{\br' \in p'_{\bss}} \frac{a_{\br'}}{a'_{\br'}}. \]

\item{$\ref{item(iii):equivweight} \boldsymbol \Rightarrow \ref{item(ii):equivweight}$.}
Condition~\ref{item(ii):equivweight} of Theorem~\ref{th:equivweight} holds if and only if there exist $\alpha_1,\dots,\alpha_n,\beta$ such that
\[ \begin{pmatrix} \log(a_{\bss_1}/a'_{\bss_1}) \\ \log(a_{\bss_2}/a'_{\bss_2}) \\ \vdots \\ \log(a_{\bss_m}/a'_{\bss_m}) \end{pmatrix}  =  M_{\mS}  \begin{pmatrix} \log(\alpha_1) \\ \vdots \\ \log(\alpha_n) \\ \log(\beta) \end{pmatrix}, \] 
where $\mS = \{\bss_1,\dots,\bss_m\}$ and $M_{\mS}$ is the matrix defined by Equation~\eqref{eq:MatrixS}.  As $M_{\mS}$ does not depend on any weights, it is still true that its image is the set
\begin{equation*}
E = \left\{ (y_{\bss})_{\bss \in \mS} \ \left| \ \ \forall \bss \, \in \, \mS \setminus \mT, \ \ \sum_{\br \in p_{\bss}} y_r  = \sum_{\br' \in p'_{\bss}} y_{r'} \textrm{ (sums considered with multiplicity)} \right. \right\}.
\end{equation*}
Applying logarithms to~\eqref{eq:equivweight} shows that $(\log(a_{\bss_j}/a'_{\bss_j}))_{j=1,\dots,m}$ belongs to $E$, and thus to $\textrm{Im}\left(M_{\mS}\right)$.

\item{$\ref{item(ii):equivweight} \boldsymbol \Rightarrow \ref{item(i):equivweight}$.}
For any $w \in \mW_{\bi \rightarrow \bj}(k)$ recall that $r_{\bss}(w)$ denotes the number of times the step $\bss$ occurs in $w$, and note that
 \[ \prod_{\bss \in w}\left(\beta \prod_{j=1}^n \alpha_j^{\pi_j(\bss)}\right) = \beta^k \prod_{l=1}^n \alpha_l^{\sum_{\bss \in \mS} \pi_l(\bss)\,r_{\bss}(w)} = \beta^k \prod_{l=1}^n \alpha_j^{j_l-i_l} \]
does not depend on $w$.  Thus, assuming Condition~\ref{item(ii):equivweight}, the probability of $w$ under the weighting $\ba$ satisfies
\begin{align*} 
\frac{\prod_{\bss \in w}a_{\bss} }{\sum_{w' \in \mW_{\bi \rightarrow \bj}(k)} \prod_{\bss' \in w'}a_{\bss'}}
 &=  \frac{\prod_{\bss \in w}a'_{\bss}\left( \beta \prod_{j=1}^n \alpha_j^{\pi_j(\bss)} \right)}{\sum_{w' \in \mW_{\bi \rightarrow \bj}(k)} \prod_{\bss' \in w'}a'_{\bss'} \left(\beta \prod_{j=1}^n \alpha_j^{\pi_j(\bss')}\right)} \\
 &= \frac{\prod_{\bss \in w}a'_{\bss} }{\sum_{w' \in \mW_{\bi \rightarrow \bj}(k)} \prod_{\bss' \in w'}a'_{\bss'} },
 \end{align*}
which is the probability of $w$ under the weighting $\ba'$.
\qedhere
\end{description}
\end{proof}

\subsection{Generating Function Relations}
One of the motivations for introducing central weightings is that they induce very simple relations between the generating functions of a centrally weighted model and an unweighted model (or, more generally, between equivalent weightings of a step set). Recall the weighted generating function $Q_{\ba}(\bz,t)$ counting walks beginning at the origin and staying in $\mathbb{N}^n$.  

\begin{proposition} 
Let $\ba$ and $\ba'$ be equivalent weightings on a finite non-singular step set $\mS$.  Then there exist constants $\alpha_1,\dots,\alpha_n,\beta>0$ such that
\begin{equation}
Q_{\ba}(\bz,t) = Q_{\ba'}(\alpha_1 z_1,\dots, \alpha_n z_n, \beta t).
\label{eq:Qm=Q}
\end{equation}
\label{prop:Qm=Q}
\end{proposition}
\vspace{-0.4in}

\begin{proof}
Theorem~\ref{th:equivweight} implies the existence of non-zero constants $\alpha_1,\dots,\alpha_n,$ and $\beta$ such that $a_{\bss} = a'_{\bss} \beta \prod_{j=1}^n \alpha_j^{\pi_j(\bss)}$ for all $\bss \in \mS$. Then for $(\bi,k) \in \mathbb{N}^{n+1}$,
\begin{align*}
[\bz^{\bi}t^k]Q_{\ba}(\bz,t) = \sum_{\substack{w\textrm{ walk ending}\\ \textrm{at }\bi \textrm{ of length } k}} \left( \prod_{\bss \in \mS} a_{\bss}^{r_{\bss}(w)} \right)
&= \sum_{\substack{w\textrm{ walk ending}\\ \textrm{at }\bi \textrm{ of length } k}} \,\, \prod_{\bss \in \mS} \left[a'_{\bss}\beta \prod_{j=1}^n \alpha_j^{\pi_j(\bss)} \right]^{r_{\bss}(w)} \\
&= \sum_{\substack{w\textrm{ walk ending}\\ \textrm{at }\bi \textrm{ of length } k}} \,\, \left[\,\,\prod_{\bss \in \mS} \left(a'_{\bss}\right)^{r_{\bss}(w)}\right] \beta^k \prod_{j=1}^n \alpha_j^{\sum_{\bss \in \mS} r_{\bss}(w)\pi_j(\bss)} \\
&= \sum_{\substack{w\textrm{ walk ending}\\ \textrm{at }\bi \textrm{ of length } k}} \,\, \left[\,\,\prod_{\bss \in \mS} \left(a'_{\bss}\right)^{r_{\bss}(w)}\right] \beta^k \prod_{j=1}^n \alpha_j^{i_j} \\[+2mm]
&= [\bz^{\bi}t^k]Q_{\ba'}(\alpha_1 z_1,\dots, \alpha_n z_n, \beta t).
\end{align*}
\end{proof}

From this quick observation, we obtain the following result.

\begin{corollary}\label{thm:bothornone}
The multivariate generating functions for any equivalent weightings with rational weights are either both D-finite or both non-D-finite. 
\end{corollary}

The generating function for weighted excursions under a weighting $\ba$ is given by $Q_{\ba}(\bzer,t)$.  Proposition~\ref{prop:Qm=Q} and Theorem~\ref{th:equivweight} then imply the following.

\begin{corollary} 
\label{thm:excursions}
Given two equivalent weightings $\ba$ and $\ba'$ on a finite non-singular step set $\mS$ there exists $\beta>0$ such that the number of weighted excursions $e_{\ba}(k)$ and $e_{\ba'}(k)$ of length $k$ under each weighting satisfy 
\[e_{\ba}(k) = \beta^k\,e_{\ba'}(k).\]
\end{corollary}

The fact that the asymptotics of centrally weighted Gouyou-Beauchamps excursions given in Theorem~\ref{thm:GB_excursion_asympt} do not depend on $a$ and $b$ (once $i$ and $j$ are specialized to 0) was derived by an ACSV analysis, but also follows from this result.

\subsection{Finding Families of D-Finite Models}

Corollary~\ref{thm:bothornone} illustrates how to derive a family of models with D-finite generating functions from one non-singular model with D-finite generating function.  In a recent paper, Kauers and Yatchak~\cite{KauersYatchak2015} use Gröbner Basis techniques to calculate all two-dimensional quarter plane models with short steps having a finite group of order at most 8 (note that the group of a model must have even order).  

These models, which Kauers and Yatchak show to have D-finite generating functions, fall into a finite number of families which are described by relations of the form appearing in Equation~\eqref{eq:equivweight}. In fact, with the exception of the models which have a group of order 4 (the smallest order possible) the families of walks described in their paper correspond to equivalence classes associated to fixed models (the family of models having order 4 fall into an infinite number of equivalence classes). Representatives for the families of models with groups of orders 6 and 8 are shown in Figure~\ref{fig:kaya}.

\begin{figure}\center
\includegraphics[width=.9\textwidth]{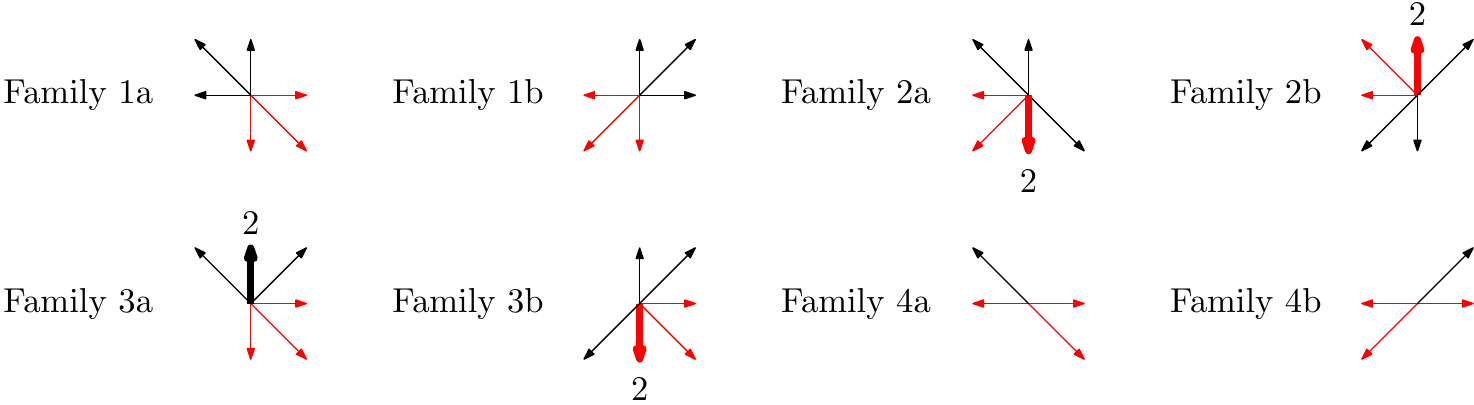} 
\caption[Equivalence class representatives for models with D-finite generating functions given by Kauers and Yatchak]{Equivalence class representatives corresponding to the families of walks with D-finite generating functions given by Kauers and Yatchak~\cite{KauersYatchak2015}; unlabeled steps have weight $1$. Steps in red denote one choice of the set $\mT \subset \mS$ in Theorem~\ref{th:equivweight}.}
\label{fig:kaya}
\end{figure}

In addition to determining the models with groups of size at most 8, Kauers and Yatchak performed a computational search for models with larger groups among those having steps with weights 1, 2, and 3.  This search found only 3 models, all of which admit groups of order 10.  Using Theorem~\ref{th:equivweight} we are thus able to determine three families of weighted models with D-finite generating functions corresponding to the equivalence classes of weightings generated by these 3 models.  For example, one of these models is the weighted step set
\begin{figure}[ht!]
\center
\includegraphics[width=.1\textwidth]{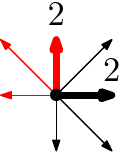} 
\end{figure}

\noindent
where the unlabeled (skinny) steps have weight 1.  Taking the set $\mT$ in Theorem~\ref{th:equivweight} to be the steps coloured red, we can calculate that the equivalence class corresponding to this step set is defined by 
\[ a_{0,-1} = \frac{a_{0, 1}a_{-1, 0}^2}{2a_{-1,1}^2}, \qquad 
a_{1, -1} = \frac{a_{-1, 0}^2a_{0, 1}^2}{4a_{-1, 1}^3}, \qquad 
a_{1, 0} = \frac{a_{-1, 0}a_{0, 1}^2}{2a_{-1, 1}^2}, \qquad 
a_{1, 1} = \frac{a_{0, 1}^2}{4a_{-1, 1}}, \]
where $a_{0, 1}, a_{-1, 0},$ and $a_{-1,1}$ are free parameters.  It is also possible to computationally determine these equivalence classes from the weighted models found by Kauers and Yatchak\footnote{Kauers and Yatchak had previously computed these families of models, but omitted them from their paper due to space constraints.}.

\subsection{A Conjecture About the Converse}

A natural question is whether the converse of Proposition~\ref{prop:Qm=Q} also holds; that is, if Equation~\eqref{eq:Qm=Q} is satisfied by two positive weightings, must they be equivalent.  This is not clear, and we formulate the following conjecture.

\begin{conjecture} 
\label{conj1}
Let $\mS \subset \mathbb{Z}^n$ be a finite non-singular step set and $\ba$ and $\ba'$ be two positive weightings of the steps of $\mS$. If there exist constants $\alpha_1,\dots,\alpha_n,$ and $\beta$ such that
\[ Q_{\ba}(\bz,t) = Q_{\ba'}(\alpha_1 z_1,\dots, \alpha_n z_n, \beta t), \]
then $\ba$ and $\ba'$ are equivalent, and $a_{\bss} = a'_{\bss} \beta \, \prod_{j=1}^n \alpha_j^{\pi_j(\bss)}$ for all $\bss \in \mS$.  In particular, if $Q_{\ba}(\bz,t) = Q(\alpha_1 z_1,\dots, \alpha_n z_n, \beta t)$ then the weighting $\ba$ is central.
\end{conjecture}

The statement on central weightings in Conjecture~\ref{conj1} was proven for certain types of random walks in root systems of Lie algebras by Lecouvey and Tarrago~\cite{LecouveyTarrago2017}. Further discussion about this conjecture can be found in the article~\cite{CourtielMelczerMishnaRaschel2016} on which this Chapter is based, but is not our focus here.

\subsection{General Universality Classes}

We now discuss how to generalize the universality classes discussed above for centrally weighted Gouyou-Beauchamps models.  As there are models with non-D-finite generating functions, one cannot use rational diagonal asymptotics in the general case.  Our discussion relies on a conjecture of Garbit, Mustapha, and Raschel~\cite{GarbitMustaphaRaschel2017} for critical exponents of weighted non-singular models in cones, following work of Garbit and Raschel~\cite{GarbitRaschel2016} which proved formulas for the exponential growth of general non-singular models in cones.  The results of Section~\ref{sec:GBresults} prove this conjecture for centrally weighted Gouyou-Beauchamps models, providing strong evidence for its truth and illustrating possible connections between the very general probabilistic approach and the more fine-tuned ACSV approach (which usually allows for stronger results, like leading asymptotic constants, in the cases where it applies).

\subsubsection*{The conjecture of Garbit, Mustapha and Raschel} 
Let $\mS \subset\mathbb{Z}^2$ be a non-singular finite step set and let $(a_{i,j})_{(i,j) \in \mS}$ be a non-negative weighting with weighted characteristic polynomial $S_{\ba}(x,y)$.  

\begin{remark} 
The conjecture of Garbit, Mustapha and Raschel is formulated in terms of the \emph{Laplace transform\/} $L_{\ba}(x,y)=\sum_{(i,j)\in \mS} a_{i,j} e^{ix+jy}$, but in order to be consistent with our previous notation we state it in terms of $S_{\ba}(x,y)$.  Observe that these functions are linked by the relation $S_{\ba}(x,y)=L_{\ba}(\ln(x),\ln(y))$, so the translation is straightforward. 
\end{remark}

As $\mS$ is non-singular, $S_{\ba}(x,y)$ has a unique positive critical point\footnote{This is well known and follows from the strict convexity of the Laplace transform (see, for instance, Garbit and Raschel~\cite{GarbitRaschel2016} or Denisov and Wachtel~\cite[Section 1.5]{DenisovWachtel2015}).} $(\xc, \yc)$ that satisfies 
\[ (\partial S_{\ba}/\partial x)(\xc, \yc)=(\partial S_{\ba}/\partial y)(\xc, \yc)=0;\] 
note that the critical point is a function of the weights. For the Gouyou-Beauchamps model with weights given by Figure~\ref{fig:GB_weights}, we have $(\xc,\yc)=\left(a^{-1},b^{-1}\right)$.  We define the \emph{covariance factor}
\begin{equation}
c= \frac{(\partial^2 S_{\ba}/\partial x \partial y)(\xc, \yc)}{\sqrt{(\partial^2 S_{\ba}/\partial x\partial x)(\xc, \yc) \cdot (\partial^2 S_{\ba}/ \partial y\partial y)(\xc,\yc)}}.
\label{eq:c}
\end{equation}
The value of~$c$ is used to determine the exponential growth, and also appears in the conjectured formula for the critical exponent. When we consider only central weightings, the value of~$c$ in Equation~\eqref{eq:c} does not depend on the weights $a_{i,j}$. For example, every centrally weighted Gouyou-Beauchamps model has $c=-\frac{\sqrt{2}}{2}$.

\begin{lemma} 
\label{lem:c_independent_weights}
For any non-singular centrally weighted model, the covariance factor $c$ in \eqref{eq:c} does not depend on the weights.
\end{lemma}
\begin{proof}
Let $\mS$ be a non-singular centrally weighted model.  By Theorem~\ref{thm:centralweighting}, the weight assigned to each step $(i,j) \in \mS$ has the form $a_{i,j} = \beta a^i b^j$ and we may assume $a,b,\beta>0$ as $a_{i,j}>0$ (if they are not positive, replace $a,b,$ and $\beta$ with their absolute values).  If $(\xc, \yc)$ is the critical point of $S(x,y)=\sum_{(i,j)\in \mS} x^iy^j$ and $(\xc', \yc')$ is the critical point of $S_{\ba}(x,y)=\beta S(a x, b y)$, then $(\xc', \yc')= \left(\frac{\xc}{a}, \frac{\yc}{b}\right).$ Thus, 
\begin{equation*}
(\partial^2 S_{\ba} / \partial x\partial y)(\xc', \yc') = ab\beta (\partial^2 S / \partial x\partial y)(\xc, \yc),
\quad
(\partial^2 S_{\ba} / \partial x\partial x)(\xc', \yc') = a^2\beta(\partial^2 S / \partial x\partial x)(\xc,\yc),
\]
\[ (\partial^2 S_{\ba} / \partial y\partial y)(\xc', \yc') =  b^2\beta(\partial^2 S / \partial y\partial y)(\xc, \yc),
\end{equation*}
and the result follows immediately upon substituting this into formula~\eqref{eq:c} for~$c$.
\end{proof}

Let $\mathcal{Q}$ be the set\footnote{The set $\mathcal{Q}$ is the image of the quarter plane $\left(\mathbb{R}_{\geq 0}\right)^2$ under the transformation $(x,y) \mapsto \left(e^x,e^y\right)$, coming from the fact that the results of Garbit and Raschel~\cite{GarbitRaschel2016} and Garbit et al.~\cite{GarbitMustaphaRaschel2017} minimize the Laplace transform $L_{\ba}$ instead of the characteristic polynomial $S_{\ba}$.  For a walk confined to a more general cone $K$, one must compute the minimum of the characteristic polynomial on the image of the dual cone $K^*$ under the map $(x,y) \mapsto \left(e^x,e^y\right)$. Note that the quarter plane is self-dual. }
\[ \mathcal{Q} = \left\{ (x,y) \in \mathbb{R}^2 \ \big\vert \ x \geq 1\textrm{ and } y \geq 1 \right\}.\]
Garbit and Raschel~\cite{GarbitRaschel2016} show that the exponential growth of a model's asymptotics is determined by $(x^*,y^*)$, the minimum of the characteristic polynomial $S_{\ba}$ on $\mathcal{Q}$:
\[S_{\ba}(x^*,y^*)=\min_{(x,y) \in \mathcal{Q}}S_{\ba}(x,y).\] 
The minimizing point is seen to be unique by strict convexity~\cite[Section 2.3]{GarbitRaschel2016} of the Laplace transform $L_{\ba}(x,y)$. The minimum is achieved at the critical point $(\xc, \yc)$ when this point is in $\mathcal{Q}$, otherwise it is achieved on the boundary of $\mathcal{Q}$.  
The conjecture of Garbit, Mustapha, and Raschel~\cite{GarbitMustaphaRaschel2017} states that the critical exponent of a model can also be determined from the point $(x^*,y^*)$.

\begin{conjecture}[Garbit et al.~\cite{GarbitMustaphaRaschel2017}]
\label{conj:GMR}
Suppose that $\mS \subset \mathbb{Z}^2$ is a non-singular step set. Then $[t^k]Q(1,1;t)$ has exponential growth and critical exponent defined by Table~\ref{tab:KR}. 
\end{conjecture}

\begin{table}
\center \renewcommand{\arraystretch}{1.5}
{\small
\begin{tabular}{c|ccc }
\begin{minipage}{2.5cm}
\ 
\end{minipage}
& 
\begin{minipage}{3.9cm}
\begin{center}
$\nabla S(x^*,y^*) =0$ \\
(i.e., $(x^*,y^*)=(\xc,\yc)$)
\end{center}
\end{minipage}
& 
\begin{minipage}{4cm}
\begin{center}
$(\partial S/\partial x)(x^*,y^*)=0$ or $(\partial S/\partial y)(x^*,y^*)=0$

\end{center}
\end{minipage}
& 
\begin{minipage}{3.2cm}
\begin{center}
$(\partial S/\partial x)(x^*,y^*)>0$ and $(\partial S/\partial y)(x^*,y^*)>0$
\end{center}
\end{minipage}
\\[2mm]\hline
$(x^*,y^*)= (1,1)$ &

\begin{minipage}{3.5cm}
\vspace*{0.1cm}
\begin{center}
$S(1, 1)^k\,k^{-p_1/2}$ \\ {\color{gray} balanced}
\end{center}
\end{minipage}
&

\begin{minipage}{4.9cm}
\vspace*{0.1cm}
\begin{center}
$S(1,1)^k\,k^{-1/2}$ \\{\color{gray}axial}
\end{center}
\end{minipage}
 & 
\begin{minipage}{3.3cm}
\vspace*{0.1cm}
\begin{center}
$S(1,1)^k k^0$\\{\color{gray}free}
\end{center}
\end{minipage}
 \\[3mm]\hline
$x^*=1$ or $y^*=1$ &  
\begin{minipage}{3.5cm}
\begin{center}
\vspace*{0.1cm}
$S(\xc, \yc)^k\, k^{-p_1/2-1}$\\{\color{gray}transitional}\\ \vspace*{-0.4cm} \ 
\end{center}
\end{minipage}
&
\begin{minipage}{4.9cm}
\begin{center}
$\min\{S(x_1, 1),S(1, y_1)\}^k\,k^{-3/2}$ \\ {\color{gray} directed} \\
\vspace*{-0.4cm} \ 
\end{center}
\end{minipage}
& (not possible)\\[3mm]\hline
$x^*>1$ and $y^*>1$ & 
\begin{minipage}{3.5cm}
\begin{center}
\vspace*{0.1cm}
$S(\xc, \yc)^k k^{-p_1 -1}$\\{\color{gray}reluctant}\\ \vspace*{-0.4cm} \ 
\end{center}
\end{minipage}
& (not possible) & (not possible)\\[3mm]\hline
\end{tabular}
}
\caption[Values of the exponential growth and conjectured values of the critical exponent for weighted models in $\mathbb{N}^2$.]{Values of the exponential growth and conjectured values of the critical exponent for two-dimensional models in $\mathbb{N}^2$.  Here $(\xc, \yc)$ is the unique positive critical point of the characteristic polynomial $S_{\ba}(x,y)$, $S_{\ba}(x,1)$ is minimized at $x_1$, $S_{\ba}(1,y)$ is minimized at $y_1$, and $p_1= \pi/\arccos(-c)$ where $c$ is the covariance factor.  By `or' we mean one condition or the other, but not both.}
\label{tab:KR}
\end{table}

Table~\ref{tab:KR} allows us to define universality classes for general centrally weighted models (on non-singular step sets) using the unique minimizer of $S_{\ba}(x,y)$ on $\mathcal{Q}$.  For the centrally weighted Gouyou-Beauchamps models, $p_1=\pi/\arccos(\sqrt{2}/2)=4$ and the different regions of universality classes denoted here match our alternative definition in Section~\ref{sec:GBresults}.  

Duraj~\cite{Duraj2014} has proven the asymptotics listed in Table~\ref{tab:KR} when the point $(x^*, y^*)$ lies in the interior of $\mathcal{Q}$, which corresponds to a \emph{reluctant} universality class.

\subsection{Connecting Back to ACSV}
In Section~\ref{sec:GBproofs} we saw that the generating function of the Gouyou-Beauchamps model has a strong structure coming from its representation as a rational diagonal, which allowed us to determine asymptotics. Since we have rational diagonal expressions for all transcendentally D-finite two-dimensional models with short steps in the quarter plane, it should be possible to verify Conjecture~\ref{conj:GMR} using the techniques of analytic combinatorics in several variables.

In fact, when we can write the weighted generating function $Q_{a,b}(1,1,t)$ as the diagonal of a rational function $F_{a,b}(x,y,t)$, there is often a direct translation from the conditions listed in Table~\ref{tab:KR} to properties of the singular variety of $F_{a,b}$.

\subsubsection*{Universality Classes of Highly Symmetric Models}
Let $\mS$ be one of the 4 highly symmetric two-dimensional models with short steps studied in Chapter~\ref{ch:SymmetricWalks}.  Applying the kernel method shows that the generating function $Q_{a,b}(1,1,t)$ of a central weighting\footnote{Note that the weightings allowed in Chapter~\ref{ch:SymmetricWalks}, which assign weights that are symmetric over every axis, are \emph{not} in general central weightings (and central weightings do not typically have this symmetry property).} defined by $a_{i,j} = a^ib^j$ for $(i,j) \in \mS$ can be expressed as
\[ Q_{a,b}(1,1,t) = \Delta\left(\frac{(x^2-a^2)(y^2-b^2)}{a^2b^2(1-txyS_{a,b}(\ox,\oy))(1-x)(1-y)}\right), \]
where $S_{a,b}$ is the weighted characteristic polynomial of the model. Defining
\[ H_1 = 1-txyS_{a,b}(\ox,\oy), \qquad H_2 = 1-x, \qquad H_3 = 1-y,  \]
we take our usual stratification of the singular variety and note that the only critical points occur on the strata $\mV_1,\mV_{12},\mV_{13},\mV_{123}$ (as long as the corresponding factors in the denominator do not cancel for specific choices of the weights $a$ and $b$). The characterization of minimal points given in Lemma~\ref{lem:GBminpts} for weighted Gouyou-Beauchamps models still holds when the characteristic polynomial $S_{a,b}$ is changed to match the model under consideration.  Thus, to find the minimal point(s) giving the best upper bound $|xyt|^{-1}$ on the exponential growth of the diagonal sequence, it is sufficient to minimize
\[ |xyt|^{-1} = S_{a,b}(\ox,\oy) \]
subject to the constraints $0<x\leq 1$ and $0<y \leq 1$ (except for special cases of the weights where the factors $H_2$ and $H_3$ in the denominator cancel). This is, of course, analogous to the process of finding a minimizer of $S_{a,b}(x,y)$ over the set $\mathcal{Q} = \{x\geq1,y\geq1\}$ in the probabilistic approach. 
\smallskip

Since $H_1 = 1-txyS_{a,b}(\ox,\oy)$, the condition
\[ x(\partial H_1/\partial x) = t(\partial H_1/\partial t) \]
is equivalent to $(\partial S_{a,b}/\partial x)(\ox,\oy) = 0$ for non-zero $x$ and $y$, and 
\[ y(\partial H_1/\partial y) = t(\partial H_1/\partial t) \]
can be expressed as $(\partial S_{a,b}/\partial y)(\ox,\oy) = 0$.  If $(\xc,\yc)$ is the unique critical point (in the calculus sense) of $S_{a,b}(x,y)$, then this shows the unique critical point of $F_{a,b}$ (in the ACSV sense) on $\mV_1$ with positive coordinates satisfies $x=1/\xc$ and $y=1/\yc$, and this point is minimal when $\xc >1$ and $\yc > 1$.  Similarly, the unique critical point of $F_{a,b}$ (in the ACSV sense) on $\mV_{12}$ with positive coordinates satisfies $x=1$ and $(\partial S_{a,b}/\partial y)(1,\oy) = 0$.  Analogous results hold for the critical points on $\mV_{13}$ and $\mV_{123}$.  

These arguments show that, for these models, the conditions defining the columns of Table~\ref{tab:KR} determine which strata contain minimal critical points minimizing $|xyt|^{-1}$ on $\overline{\mD}$ (and should thus determine diagonal coefficient asymptotics of $F_{a,b}$).  The rows of that table then correspond to special conditions on the weights $a$ and $b$ which cause the factors of $H_2 = 1-x$ and $H_3 = 1-y$ to cancel with factors in the numerator, changing the critical exponent of asymptotics. Because of their symmetries, each of the 4 highly symmetric models admit the unique critical points with positive coordinates
\[ c_1 = \left(a,b,\frac{1}{abS_{a,b}(a,b)}\right), \quad c_{12} = \left(1,b,\frac{1}{bS_{a,b}(1,b)} \right),\] 
\[c_{13} = \left(a,1,\frac{1}{aS_{a,b}(a,1)} \right), \quad c_{123} = \left(1,1,\frac{1}{S_{a,b}(1,1)} \right), \] 
on the strata $\mV_1,\mV_{12},\mV_{13},\mV_{123}$, respectively.  The conditions defining universality classes determine when these points are minimal, together with the order of vanishing of the numerator of $F_{a,b}$ at these points.  Furthermore, arguments mirroring those for the Gouyou-Beauchamps model above show that $c_1$ is a finitely minimal point.  Thus, asymptotics can then be calculated using the methods of Chapters~\ref{ch:SmoothACSV} and~\ref{ch:NonSmoothACSV}, proving the conjecture of Garbit et al.~for these models where there are non-smooth non-finitely minimal critical points and the numerator of $F_{a,b}$ vanishes at these points.

In particular, the balanced $(a<1,b<1)$, transitional $(a=1,b<1 \text{ or } a<1,b=1)$, and reluctant $(a<1,b<1)$ models admit smooth finitely minimal critical points, while the axial $(a=1,b>1 \text{ or } a>1,b=1)$ and directed $(a>1,b<1 \text{ or } a<1,b>1)$ cases admit convenient minimal points.  Finally, the free case $(a>1,b>1)$ corresponds to a minimal critical multiple point where $\mV$ forms a complete intersection.  

Note that there may be other critical points with the same coordinate-wise moduli as the points $c_1,c_{12},c_{13},c_{123}$, which must be accounted for in the standard way (by summing the contributions of such points).  These additional singularities do not figure into the probabilistic framework, which is reflected in the fact that the probabilistic approach does not say anything about leading asymptotic constants (or possible periodic behaviour of such terms).  Explicit formulas for asymptotics can be found in an accompanying Maple worksheet\footnote{Available at \websiteurl.}. 
\smallskip

\subsubsection*{Difficulties for the Remaining Models}

Unfortunately, the analysis outlined above does not immediately generalize to all non-highly symmetric models with short steps in the quarter plane.  For all positive drift, and some negative drift, step sets which are symmetric over one axis, the generating function $Q_{a,b}(1,1,t)$ for a centrally weighted model can be represented as the diagonal of a rational function whose denominator has the factors $H_1,H_2,$ and $H_3$ above, together with another factor equal to $x^2+a^2$ or $a^2+ax+x^2$.  This complicates the analysis, and leads to the possibility that non-minimal critical points could determine dominant asymptotics (although it may be possible to work around on a case-by-case basis).  

Rather then go through long and technical arguments for each of the remaining cases, current work in progress is aimed at better understanding the quasi-local cycles which appear in the multivariate residue approach to analytic combinatorics in several variables.  Such an approach should allow for a simpler and more uniform approach to lattice path asymptotics, as alluded to at the end of Section~\ref{sec:GBproofs}.  It is promising to see connections between probabilistic results which hold very generally (even for models with non-D-finite generating functions) and the methods of ACSV, which have the potential to give stronger conclusions when they apply.  This link deserves further study.

%\part{Other Variants of Lattice Paths}
%\label{part:LatticePathVariants}

%\include{Chapters/12ThreeDWalks}
%\include{Chapters/13YoungTableau}

\part{Conclusion}
\label{part:Conclusion}

\chapter{Conclusion}
\label{ch:Summary}

\setlength{\epigraphwidth}{3.4in}
\epigraph{I too have wholeheartedly pursued science passionately, as one would a beloved woman. I was a slave, and sought no other sun in my life. Day and night I crammed myself, bending my back, ruining myself over my books; I wept when I beheld others exploiting science for personal gain. But I was not long enthralled. The truth is every science has a beginning, but never an end -- they go on for ever like periodic fractions.\footnotemark}{Anton Chekhov, \emph{On the Road}} \footnotetext{Translation from the Russian by Ian Porteous in the Forward of Arnold et al.~\cite{ArnoldGusein-ZadeVarchenko2012}}

We end by giving some perspectives on the above work, together with directions for future research.

\section{Effective Asymptotics}

This thesis examines effective techniques in enumeration from the perspective of computer algebra, giving the first complete algorithms and complexity results for the methods of analytic combinatorics in several variables which work in any number of dimensions. Although most of our assumptions on the rational functions which can be analyzed by these algorithms hold generically, there are still limitations to the asymptotics they can capture.  In particular, our algorithms can only represent asymptotic behaviour which is a finite sum of terms of the form $C \cdot (\pi k)^{\alpha} \cdot \rho^k$, where $C$ is an algebraic number, $\alpha \in \mathbb{Z}/2$, and $\rho$ is algebraic.  

As shown in Corollary~\ref{cor:diagAsm}, rational diagonal sequences can admit a much larger range of asymptotic behaviour, and it is natural to wonder if it is decidable to determine asymptotics of an arbitrary rational diagonal.  From an algorithmic perspective, a natural next step is a generalization of the results in Chapter~\ref{ch:EffectiveACSV} to the non-smooth cases presented in Chapter~\ref{ch:NonSmoothACSV}.  In addition to dealing with stratifications of non-smooth varieties to determine critical points, efficient algorithms which determine minimizers of $|z_1 \cdots z_n|^{-1}$ among minimal points of the singular variety will be needed.  

From a theoretical point of view, there are several ways the theory of ACSV can be enriched.  Pemantle~\cite[Conjecture 2.11]{Pemantle2010} makes a conjecture related to the properness\footnote{The methods of stratified Morse theory typically require the function $h$ to be proper, meaning the inverse image of any compact set is compact.  This usually does not hold in the ACSV setting, but Pemantle~\cite[Conjecture 2.11]{Pemantle2010} conjectured a weaker condition which would be sufficient for the purposes of ACSV.} of the map $h(\bz) = |z_1 \cdots z_n|^{-1}$ on the singular variety of a rational function which, if true, would allow the techniques of stratified Morse theory to be rigorously applied to the problem of determining diagonal asymptotics.  For example, this would give explicit representations of the quasi-local cycles which can be taken for transverse multiple points in the residue approach to ACSV.  Other ways forward include examining rational diagonals which admit degenerate critical points, and extending the types of singular behaviour at critical points which can be handled to determine asymptotics.

Ultimately, one would also like to connect this study back to the more general problem of D-finite coefficient asymptotics, perhaps through Christol's conjecture that every globally bounded D-finite function is a rational diagonal.

\section{Lattice Path Enumeration}

After an intense period of study, the problem of enumerating lattice walks in the quadrant admitting short steps has more or less been solved for models with D-finite generating functions.  Here we have studied several generalizations of this topic, including variants of weighted walks and walks in higher dimension.  Although we were able to derive many asymptotic results, including the first proof of the asymptotic conjectures of Bostan and Kauers, a better understanding of the residue approach to ACSV should help give a simpler and more uniform approach to lattice path asymptotics through rational diagonals.  Furthermore, the connection between probabilistic studies of walks in cones and lattice path enumeration discussed at the end of Chapter~\ref{ch:WeightedWalks} could hint at a deeper connection between ACSV and lattice path problems\footnote{For instance, perhaps a generalization of the kernel method allows one to represent the generating functions of some (conjecturally) non-D-finite lattice path models in the quarter plane as diagonals of (non-D-finite) meromorphic functions, to which the methods of ACSV could be applied.}. 

Aside from their many applications, lattice path models are useful for studying the theory of ACSV as they give a large family of concrete examples with a wide variety of behaviour.  It is our hope that new applications will help guide the theory, and make it more accessible to a larger audience.  We conclude by discussing work currently in progress on this topic.

%\subsection*{Walks on Other Lattices} 

\subsection*{Almost Highly Symmetric Models}
In Chapter~\ref{ch:QuadrantLattice} we gave a uniform diagonal expression for lattice path models which are symmetric over one axis. In fact, much of this approach can be generalized to higher dimensional walks in orthants which are symmetric over all but one axis.  A step set $\mS =\subset \{\pm1,0\}^n \setminus \{\bzer\}$ is called \emph{almost highly symmetric} if the characteristic polynomial
 \[ S(\bz) := \sum_{\bi \in \mS} \bz^{\bi} \]
is invariant under the substitutions $z_r \mapsto \oz_r$ for $r=1,\dots,n-1$.  The group of transformations which arises in the kernel method can be determined explicitly for these models, leading to a uniform diagonal expression depending only on $S(\bz)$.  As in the two dimensional case, the analysis of minimal critical points breaks down into two cases depending on whether more steps in $\mS$ point towards or away from the positive orthant.  Unfortunately (as seen in the two dimensional case) the zeroes of the Laurent polynomial $[z_n]S(\bz)$ give points in the singular variety, and this becomes harder to deal with in higher dimensions.  Again, a detailed study using the multivariate residue approach to analytic combinatorics in several variables is needed, which is currently ongoing.  

We also note that such a uniform study cannot be undertaken for models defined by step sets which are symmetric over all but two axes.

\begin{proposition} 
\label{prop:concalmostsym}
For every dimension $n \geq 2$ there exists a step set $\mS_n \subset \{\pm1,0\}^n$, which is symmetric over all but two axes, such that the generating function of the lattice path model in the non-negative orthant defined by $\mS$ is non-D-finite (and therefore not a rational diagonal).
\end{proposition} 

To prove Proposition~\ref{prop:concalmostsym}, one can take the step set 
\[ \mS_2 := \{(-1,-1),(0,-1),(0,1),(1,0),(-1,0)\} \]
and define 
\[ \mS_n := \mS_2 \times \{\pm1\}^{n-2} \]
for $n \geq 3$.  Results of Bostan, Raschel, and Salvy~\cite{BostanRaschelSalvy2014} and Duraj~\cite{Duraj2014} imply that the sequence counting walks in the first quadrant using steps in $\mS_2$ which begin at the origin and end anywhere satisfies $a_k \sim C k^{\alpha} \rho^k$, where $\alpha$ is irrational.  If $d_k$ denotes the number of unrestricted Dyck paths of length $k$, Example~\ref{ex:freeDyck} shows that $d_k \sim (2/\pi)^{1/2}k^{-1/2}2^k$, and a simple combinatorial argument implies that the number of walks on the steps in $\mS_n$ staying in the non-negative orthant is
\[ a_k \cdot d_k^{n-2} \sim C (2/\pi)^{n/2} \cdot k^{\alpha - n/2-1} \cdot (2^n \rho)^k. \]
As $\alpha - n/2-1$ is irrational, Theorem~\ref{thm:DfinAsm} shows that the coefficient sequence of a D-finite function cannot have this asymptotic growth.

\subsection*{Walks with Longer Steps}  

The kernel method for walks restricted to a quadrant, as presented in Chapter~\ref{ch:KernelMethod}, requires models with short steps in order to define the group $\mG$ of birational transformations which fix the characteristic polynomial $S(x,y)$. Work of Bostan, Bousquet-Mélou, and Melczer~\cite{BostanBousquet-MelouMelczer2017}, currently in preparation, attempts to provide a framework for studying models with larger steps.  We end with an example from that work, which illustrates how rational diagonal expressions arise in a similar manner to the short step case.

\begin{example}[Bostan, Bousquet-Mélou, and Melczer~\cite{BostanBousquet-MelouMelczer2017}]
Consider the quarter plane model defined by the step set 
\[ \mS = \{(1,0),(-1,0),(-2,1),(0,-1)\}.\]  
Similar to the short step case, a recursive decomposition of a walk of length $k$ into a walk of length $k-1$ plus a single step implies that the multivariate generating function $Q(x,y,t)$ marking endpoint and length satisfies
\begin{equation} K(x,y,t)Q(x,y,t)=1-t\ox (1+\ox y)Q(0,y,t)-t\ox y Q_1(y,t)-t\oy Q(x,0,t),\label{eq:conLong} \end{equation}
where $Q_1(y,t) = [x^1]Q(x,y,t)$ and
\[ K(x,y,t) = 1-t\sum_{(i,j) \in \mS}x^iy^j = 1-t(x+\ox+\ox^2y + \oy). \]
Because $\mS$ contains steps with coordinates of modulus larger than 1, the construction of the group $\mG$ used in the short step case will not work here.  Nonetheless, to solve this equation for $Q(x,y,t)$ we search for substitutions of the variables $x$ and $y$ which fix the kernel $K(x,y,t)$ and change only one unknown term on the right hand side of Equation~\eqref{eq:conLong}.

\begin{figure}
\centering
\includegraphics[width=0.5\linewidth]{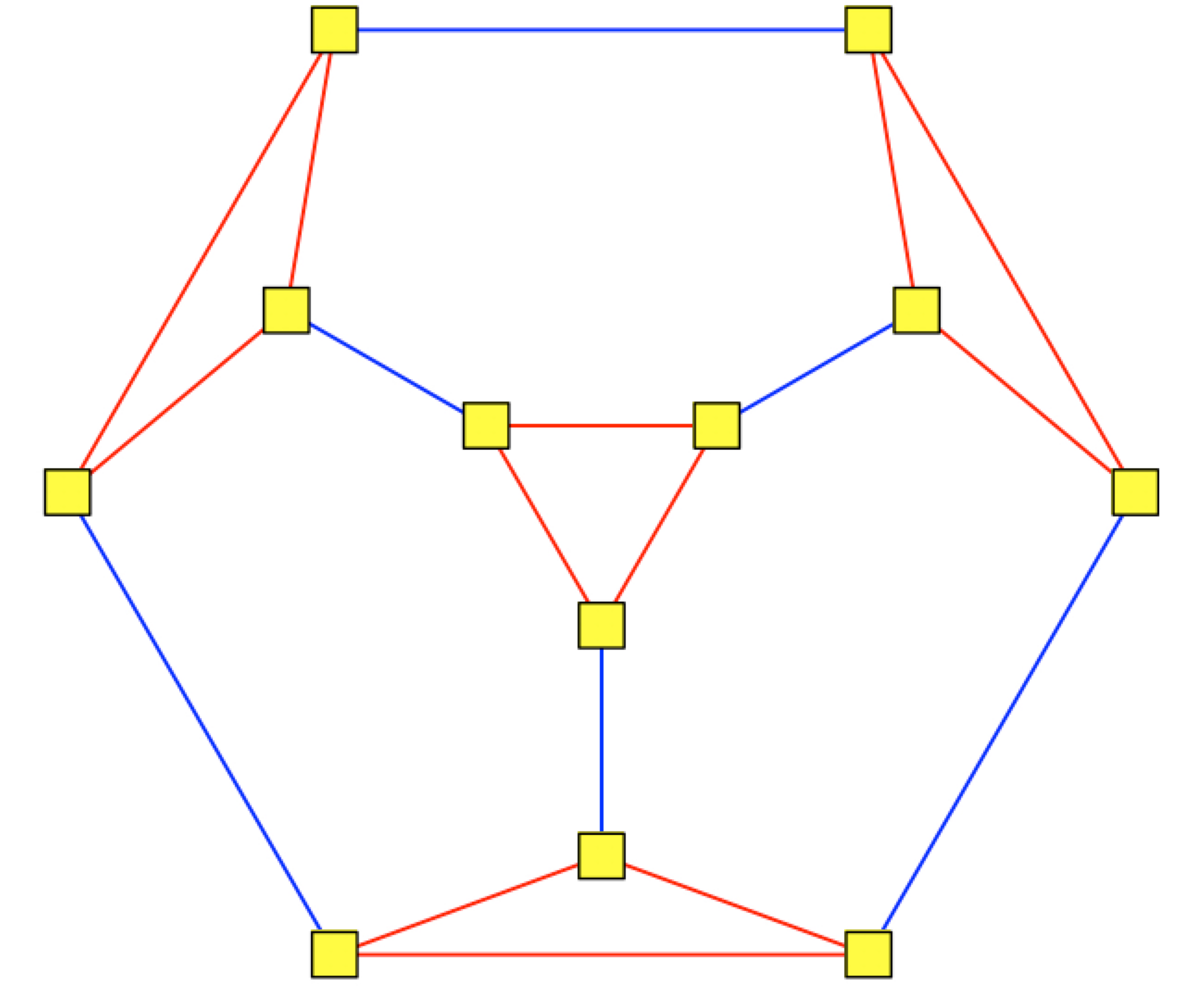}
\caption[Visualization of the orbit for a non-short step model]{Visualization of the set $\mG$, where each vertex is an element $(a,b) \in \mG$ and another vertex representing $(c,d) \in \mG$ is connected by a blue edge if $a=c$ and a red edge if $b=d$.}
\label{fig:conLong}
\end{figure}

The equation $K(X,y,t) = K(x,y,t)$ admits three solutions $X = x,x_1,$ and $x_2$, where 
\[ x_{1,2}= \frac{x+y\pm\sqrt{(x+y)^2+4x^3y}}{2x^2}.\]
Similarly, the equation $K(x,Y,t) = K(x,y,t)$ admits the two solutions $Y = y$ and $x^2\oy$.  Thus, we have a collection 
\[ \mG = \{ (x, y), \quad (x_1 , y), \quad (x_2 , y), \quad (x, x^2\oy ) \} \]
of pairs of elements in the algebraic closure $\overline{\mathbb{Q}(x,y)}$ such that $K(x',y')=K(x,y)$ for all $(x',y') \in \mG$.  Furthermore, for each $(x',y') \in \mG$ we can solve the equations 
\[ K(X,y',t)=K(x',y',t) \qquad\qquad K(x',Y,t)=K(x',y',t)\] 
to determine additional elements of $\overline{\mathbb{Q}(x,y)}^2$ fixing $K(x,y,t)$ which differ from other members of $\mG$ in one coordinate only.  This is done implicitly, using resultants, and repeated until the set $\mG$ stabilizes (if it does).  For this example, $\mG$ stabilizes in the set
\begin{align*} 
\mG &= \{ (x, y), (x_1 , y),  (x_2 , y), (x, x^2\oy),  (-\bxun , x^2\oy),  (-\bxde  , x^2\oy), \\
& \qquad (x_1 , x_1 ^2\oy),  (-\ox, x_1 ^2\oy), (-\bxde  , x_1 ^2\oy), (x_2 , x_2 ^2\oy),  (-\ox, x_2 ^2\oy), (-\bxun , x_2 ^2\oy) \},
\end{align*}
where $\overline{x}_i=1/x_i$.  

Substituting these pairs into the kernel equation gives 12 equations with 14 unknown evaluations of the unknown functions on the right hand side of Equation~\eqref{eq:conLong} (6 specializations of $Q(x,0,t)$ and 4 specializations each of $Q_1(y,t)$ and $Q(0,y,t)$).  Fortunately, there is still a linear combination of these equations which kills all unknown functions on the right hand side.  A (somewhat tedious) generating function argument then shows that the generating $Q(x,y,t)$ is given by the positive series extraction
\[ Q(x,y,t) = [x^{\geq 0}][y^{\geq 0}] 
\frac { \left( {x}^{2}+1 \right) \left( x+y \right)  \left( y-x \right)
\left( x^2y-2\,x-y \right) \left(x^3-x-2\,y \right)}{x^7y^3 \left(1-t(x+\ox+\ox^2 y +\oy) \right)},\]
and Proposition~\ref{prop:postodiag} implies 
\[ Q(1,1;t) = \Delta \left( \frac{(x^2+1)(x^2+2xy-1)(2x^3+x^2y-y)(y^2-x^2)}{x^2y(1-x)(1-y)(1-t(x^3+x^2y+xy^2+y))} \right). \]
This rational function admits the finitely minimal smooth critical points $c_1 = \left(3^{-1/2},3^{-1/2},3^{-1/2}/2\right)$ and $c_2 = \left(-3^{-1/2},-3^{-1/2},-3^{-1/2}/2\right)$, and Corollary~\ref{cor:smoothAsm} allows one to calculate the dominant asymptotic expansion
\[ [t^k]Q(1,1,t) = \frac{(2\sqrt{3})^k}{k^4}\left(C_k + O\left(\frac{1}{k}\right)\right), \]
where
\[ C_k = \begin{cases} \frac{5616\sqrt{3}}{\pi} &: k \text{ even} \\ \frac{9720}{\pi} &: k \text{ odd} \end{cases}. \]
\end{example}

The article of Bostan, Bousquet-Mélou, and Melczer~\cite{BostanBousquet-MelouMelczer2017} generalizes this argument and conducts a systematic study of models in the quarter plane containing steps in $\{-2,\pm1,0\}^2$.
%\include{Chapters/15QuoteSources} 

%----------------------------------------------------------------------
% END MATERIAL
%----------------------------------------------------------------------

% B I B L I O G R A P H Y
% -----------------------

\cleardoublepage 
\phantomsection 
\renewcommand*{\bibname}{References}

% Add the References to the Table of Contents
\addcontentsline{toc}{chapter}{\textbf{References}}
 
%\bibliography{bibl}
\printbibliography

% The \appendix statement indicates the beginning of the appendices.
\appendix
% Add a title page before the appendices and a line in the Table of Contents
\chapter*{APPENDICES}
\addcontentsline{toc}{chapter}{APPENDICES}
%======================================================================
\chapter{Values of the Periodic Constant for Gouyou-Beauchamps Walks}
%======================================================================
\label{appendix:GB}
Here we list the function $V^{[k]}(i,j)$ appearing in Theorem~\ref{thm:GB_main_asymptotic_result} for the different universality classes.

\paragraph{Balanced $(a=b=1)$}
\[ V^{[k]}(i,j) = \frac{4}{\pi} \cdot \frac{(i+1)(j+1)(i+j+2)(i+2j+3)}{3}. \]

\paragraph{Free $\left(\sqrt{b}<a<b\right)$} 
\begin{align*} 
V^{[k]}(i,j) = a^{-(4+2i+2j)}b^{-(2+2j)}&\left(\left(a^{1+j}-1\right)\left(a^{1+j}+1\right)\left(a^{2+i+j}-b^{2+i+j}\right)\left(a^{2+i+j}+b^{2+i+j}\right)b^{-i-1} \right.  \\
&\left. \qquad\qquad - \left(a^{2+i+j}-1\right)\left(a^{2+i+j}+1\right)\left(a^{1+j}-b^{1+j}\right)\left(a^{1+j}+b^{1+j}\right) \right).
\end{align*}

\paragraph{Reluctant $\left(a<1,b<1\right)$}
{\small
\begin{align*} 
V^{[k]}(i,j) = \frac{64}{\pi(b-1)^4}\cdot\frac{(1+j)(1+i)(3+i+2j)(2+i+j)}{a^ib^j} &\left(\frac{a^2b^2+a^2b-4ab+b+1}{(a-1)^4}\right.\\
&\left.\qquad\qquad+ (-1)^{k+i}\frac{a^2b^2+a^2b+4ab+b+1}{(a+1)^4}\right).
\end{align*}
}

\paragraph{Axial 1 $\left(a=b>1\right)$}
\[ V^{[k]}(i,j) =  \frac{b+1}{\sqrt{b\pi}} \cdot \left( (j+1)\left(1-b^{-2(2+i+j)}\right)+b^{-i-1}(i+2+j)\left(b^{-2(1+j)}-1\right) \right). \]

\paragraph{Axial 2 $\left(b=a^2>1\right)$}
\[ V^{[k]}(i,j) =  \frac{\sqrt{2}\left(\left(a^6-a^{-2i-4j}\right)(1+i)+\left(a^{2-2i-2j}-a^{4-2j}\right)(3+i+2j)\right)}{a^6\sqrt{\pi}}. \]

\paragraph{Transitional 1 $\left(a=1,b<1\right)$}
\[ V^{[k]}(i,j) =  \frac{16}{3\pi(1-b)^2} \cdot (j+1)(i+1)(i+3+2j)(i+2+j)b^{-j}. \]

\paragraph{Transitional 2 $\left(b=1,a<1\right)$}
\[ V^{[k]}(i,j) =  \frac{8}{3\pi} \cdot a^{-i}(j+1)(i+1)(i+3+2j)(i+2+j)\left(\frac{1}{(1-a)^2} + \frac{(-1)^{k+i}}{(1+a)^2} \right) \]

\paragraph{Directed 1 $\left(b>1, \sqrt{b}>a\right)$}
{\small
\[ V^{[k]}(i,j) = \frac{\sqrt{2}}{\sqrt{\pi}b^2} \cdot \left(\frac{b^{3+i+2j}(1+i)+\left(b^{1+j}-b^{2+i+j}\right)(3+i+2j)-i-1}{a^ib^{i/2+2j}}\right) \left(\frac{1}{(\sqrt{b}-a)^2}+(-1)^{i+k}\frac{1}{(\sqrt{b}+a)^2}\right). \]
}

\paragraph{Directed 2 $\left(a>1, a>b\right)$}
\[ V^{[k]}(i,j) =  \frac{(a+1)^3\sqrt{a} \cdot \left( (2+i+j)\left(a^{-2-j}-a^j\right)b^{-j}a^{-1-i} + (1+j)\left(1-a^{-4-2i-2j}\right)b^{-j}a^j \right)}{2\sqrt{\pi}(a-b)^2}.\]

\end{document}